%% file: twomodelsforinfinityoperads.tex
\documentclass[a4paper]{article}
\usepackage{graphicx}
\usepackage{amsmath}
\usepackage{amssymb}
\usepackage{amsfonts}
\usepackage{amsthm}
\usepackage[all]{xy}
\usepackage{makeidx}
\usepackage{eucal}
\usepackage{tikz}
\usepackage[nottoc]{tocbibind}

\usetikzlibrary{trees, positioning, calc}

\theoremstyle{plain}
\newtheorem{theorem}{Theorem}[subsection]
\newtheorem{lemma}[theorem]{Lemma}
\newtheorem{proposition}[theorem]{Proposition}
\newtheorem{corollary}[theorem]{Corollary}
\newtheorem*{proposition*}{Proposition}

\theoremstyle{definition}
\newtheorem{definition}[theorem]{Definition}
\newtheorem{example}[theorem]{Example}

\newtheorem{remark}[theorem]{Remark}

\makeatletter
\@addtoreset{theorem}{section}
\@addtoreset{subsection}{section}
\makeatother

\DeclareMathAlphabet{\mathpzc}{OT1}{pzc}{m}{it}
\begin{document}

\begin{center}

{ \LARGE \bfseries On the equivalence between Lurie's model and the dendroidal model for infinity-operads}\\[0.5cm]
\textsc{ \large  Gijs Heuts, Vladimir Hinich and Ieke Moerdijk } \\ [2cm]

\end{center}

%\title{\LARGE \bfseries The equivalence between Lurie's model and the dendroidal model for infinity-operads}
%\date{}
%\author{ \textsc{Gijs Heuts} \\ \small{\textsc{Harvard University}} \\ \\ \textsc{Vladimir Hinich} \\ \small{\textsc{University of Haifa}} \\ \\ \textsc{Ieke Moerdijk} \\ \small{\textsc{Radboud University Nijmegen}} }

%\author[I. Moerdijk]{Ieke Moerdijk}
%\address{Mathematisch Instituut\\
%Universiteit Utrecht\\
%PO.Box 80.010\\
%{3508~TA~Utrecht}\\
%The~Netherlands}
%\email{moerdijk@math.uu.nl}
%\urladdr{http://www.math.uu.nl/people/moerdijk/}

%\maketitle

\begin{abstract}
We compare two approaches to the homotopy theory of $\infty$-operads. One of them, the theory of dendroidal sets, is based on an extension of the theory of simplicial sets and $\infty$-categories which replaces simplices by trees. The other is based on a certain homotopy theory of marked simplicial sets over the nerve of Segal's category $\Gamma$. In this paper we prove that for operads without constants these two theories are equivalent, in the precise sense of the existence of a zig-zag of Quillen equivalences between the respective model categories.
\end{abstract}

\tableofcontents

\newpage

\include{introduction}

\include{mainresults}

\include{forestsets}

\include{markedfsets}

\include{dendrification}

\include{properties}

\newpage
\bibliographystyle{plain}
\bibliography{biblio}

\textsc{Gijs Heuts} \\
\textsc{\small Harvard University, Department of Mathematics, 1 Oxford Street, 02138 Cambridge, Massachusetts, USA} \\
\emph{E-mail address:} gheuts@math.harvard.edu \\
\par 

\textsc{Vladimir Hinich} \\
\textsc{\small University of Haifa, Department of Mathematics, Mount Carmel, Haifa 31905, Israel} \\
\emph{E-mail address:} hinich@math.haifa.ac.il \\
\par 

\textsc{Ieke Moerdijk} \\
\textsc{\small Radboud Universiteit Nijmegen, Institute for Mathematics, Astrophysics and Particle Physics, Heyendaalseweg 135, 6525 AJ Nijmegen, The Netherlands} \\
\emph{E-mail address:} i.moerdijk@math.ru.nl \par 

\end{document}

%% file: introduction.tex
\section{Introduction}
The goal of this paper is to compare two rather different approaches to the theory of higher operads. Both theories are based on and to some extent parallel the theory of higher categories. Ordinary category theory arose in algebraic topology, in the analysis of functoriality of constructions in the homotopy category of spaces or spectra and related categories like the derived category of an abelian category. In this context, it was soon realised that the naive notions of limit and colimit are of little practical use. Instead, one needs the notions of homotopy limits and colimits, the description of which requires higher categorical structure. One of the standard solutions is to equip the category of spaces (or spectra, or chain complexes, etc.) with the additional structure of a Quillen model category \cite{quillenrational,quillenHA}. Another and closely related way of encoding much of the same information is by a simplicial category constructed as the Dwyer-Kan localisation \cite{dwyerkan}. \par

Geometric problems have subsequently led to the analysis of the totality of homotopy categories --- these are, for example, problems of homotopical descent, where one needs to consider a homotopy category which is `glued' from `smaller' homotopy theories consisting of locally given objects. To efficiently study these questions, one needs a `homotopy theory' of these higher structures, or what is sometimes referred to as a `homotopy theory of homotopy theories'. As a consequence various concepts have arisen, among which we mention Rezk's theory of complete Segal spaces \cite{rezk}, the category of simplicial categories equipped with their so-called Dwyer-Kan model category structure \cite{bergner}, as well as the category of simplicial sets itself, but endowed with a weaker model structure than the classical Quillen one, namely the Joyal model structure \cite{joyal}. These approaches are all equivalent, at least to the extent that they can be related by (zig-zags of) Quillen equivalences.  All these approaches yield what is now called a theory of $\infty$-categories, or more precisely of $(\infty,1)$-categories: morally, they describe higher categorical objects for which nontrivial arrows of all degrees exist, but all higher arrows are invertible up to homotopy. \par 

The existence of a zig-zag of Quillen equivalences does not, unfortunately, allow one to automatically translate constructions from one formalism of $(\infty,1)$-categories to another.  For some specific applications a certain formalism may be more convenient than another.  The Joyal model is, in a sense, the most `economical'. Moreover, it bears a close relation to classical ideas of weak and categorical structures in homotopy theory of Boardman and Vogt \cite{boardmanvogt}, since its fibrant objects are precisely the weak Kan complexes of loc. cit. The effectiveness of this model is shown by recent applications to the theory of higher topoi and higher algebra, as for example in Lurie's books \cite{htt,higheralgebra}. It plays an important role in current advances in derived algebraic geometry, specifically chiral homology \cite{francis,higheralgebra}, geometric representation theory \cite{arinkingaitsgory,benzvinadler2,benzvinadler1,francisgaitsgory,gaitsgory,gaitsgoryrozenblyum} and mirror symmetry \cite{nadler}. \par 

As several of these references illustrate, any attempt to study algebraic structures in the context of $\infty$-categories leads one to the notion of an $\infty$-operad. Here it is to some extent possible to work with simplicial operads, equipped with a model structure which extends the one on simplicial categories mentioned above \cite{cisinskimoerdijk3}. However, this approach has several difficulties. First and foremost, to be able to work with algebras over simplicial operads, one has to convert the $\infty$-category under consideration into a simplicial category, using one of the Quillen equivalences mentioned above. Also, for a well-behaved homotopy theory of algebras over an operad, one often needs a cofibrant (or `almost cofibrant') resolution of this operad. Many naturally occurring simplicial operads are not cofibrant and the necessary (almost) cofibrant replacement is a non-trivial procedure. A third point concerns the Boardman-Vogt tensor product of operads. It plays an important role in the study of the little cubes operads $\mathbb{E}_n$ (see \cite{boardmanvogt, may}); roughly speaking, tensor products of such operads again yield little cubes operads \cite{dunn,fiedorowiczvogt}. Unfortunately, the Boardman-Vogt tensor product is not compatible with the model structure on simplicial operads (i.e. this is not a monoidal model structure), which complicates the study of its homotopical properties. \par 

Two approaches to the theory of higher operads will be discussed in this paper, namely Lurie's theory of preoperads \cite{higheralgebra} and the theory of dendroidal sets \cite{moerdijkweiss,cisinskimoerdijk1}. These address the issues raised above as follows. Both Lurie's theory and the theory of dendroidal sets are naturally adapted to Joyal's model structure on simplicial sets; they allow one to work directly with algebras in an $\infty$-category. (In fact, there are variants of the theory of dendroidal sets \cite{cisinskimoerdijk2} adapted to the theory of complete Segal spaces \cite{rezk} and Segal categories \cite{hirschowitzsimpson}.) Also, all objects in Lurie's model category of preoperads are cofibrant. This is not quite true for dendroidal sets, but there cofibrant objects are easily recognised and cofibrant replacement is an easy and explicit procedure. Both Lurie's category and the category of dendroidal sets carry a tensor product. In Lurie's category, this tensor product is compatible with the model structure, but is not symmetric. In the dendroidal category, it is symmetric, but only compatible with the model structure on the subcategory modelling operads without constants. In this paper we compare Lurie's approach to the denroidal approach and to the theory of simplicial operads. While the first two are both based on the theory of $\infty$-categories \cite{joyal,htt}, these two theories have rather different starting points. \par 

At a rather naive level, these starting points can already be explained within ordinary category theory. On the one hand, a coloured operad in the category of sets can be seen as a generalization of a category, where instead of arrows $f: x \rightarrow y$ with one input one has arrows $f: x_1,\ldots,x_n \rightarrow y$ with multiple inputs. With this picture in mind, a search for a homotopy-coherent notion of operad leads to the dendroidal theory. On the other hand, a coloured operad can be seen as a weak kind of monoidal (or tensor) category, in which tensor products $x_1 \otimes \cdots \otimes x_n$ are defined only as covariant functors $y \mapsto \mathrm{Hom}(x_1 \otimes \cdots \otimes x_n, y)$ which are not necessarily representable. This leads one to a homotopy-coherent notion of a coloured operad as a weakened version of the notion of a symmetric monoidal infinity-category and to Lurie's approach \cite{higheralgebra}. These two points of view are reflected in the various terms used to refer to coloured operads, such as (symmetric) `multicategories' \cite{lambek} and `pseudo-tensor categories' \cite{beilinsondrinfeld}. \par 

The category of dendroidal sets is designed to bear the same relation to the category of operads (in $\mathbf{Sets}$) as the category of simplicial sets bears to the category of (small) categories. In particular, there is a nerve functor from operads to dendroidal sets, extending the usual nerve functor from categories to simplicial sets. To achieve this, the simplex category $\mathbf{\Delta}$ is replaced by a category $\mathbf{\Omega}$ of finite rooted trees, which contains $\mathbf{\Delta}$ as a full subcategory. The category of dendroidal sets is the category of presheaves on $\mathbf{\Omega}$ and carries a model structure which extends (in a precise sense) the Joyal model structure on presheaves on $\mathbf{\Delta}$, i.e. on simplicial sets. This model structure is used to develop a theory of $\infty$-operads, which can now simply be defined as the fibrant objects in this model structure on the category of dendroidal sets. This dendroidal approach to $\infty$-operads has several advantages. For example, it is completely parallel to the simplicial theory of $\infty$-categories. An important aspect of this dendroidal theory is that every $\infty$-operad can be strictified, in the sense of being equivalent to the homotopy coherent nerve of an ordinary (simplicial or topological) coloured operad \cite{cisinskimoerdijk3}. A disadvantage of the theory, at least in its current state, is that laying the groundwork for it requires the analysis of rather a lot of delicate combinatorial properties of finite trees, surely not unlike those of simplices and shuffles from the early days of (semi-)simplicial topology and homological algebra in the 1950's and 1960's, but more involved. The reader will see some illustrations of this phenomenon in this paper as well, for example in the proofs of Propositions \ref{prop:poprodinneranodyne}, \ref{prop:poprodrootanodyne} and \ref{prop:poprodleafanodyne}. \par 

Lurie's theory on the other hand does not parallel the theory of $\infty$-categories, but builds structure on top of it. Following an old idea of Graeme Segal \cite{segal}, one can define a symmetric monoidal category as a (pseudo-)functor (satisfying some conditions) from the category of finite pointed sets to the the category of small categories. One can deal with simplicial or topological symmetric monoidal categories in the same way. Alternatively, such a structure can be presented by a (simplicial) category cofibered over (i.e., endowed with a coCartesian fibration to) the category of finite pointed sets. In a similar way, the more general notion of an operad can be modelled as a category that is `partially' cofibered (partially, much like a vector bundle can carry a partial connection) over the category of finite pointed sets. This leads to the definition of an $\infty$-operad as a simplicial set which is partially cofibered (in the appropriate weak, up-to-homotopy sense) over the simplicial set defined as the nerve of the category of finite pointed sets. In order to be able to efficiently work with such objects, a Quillen model category structure is constructed on the ambient category of so-called preoperads --- marked simplicial sets over the nerve of the category of finite pointed sets. One can then model $\infty$-operads as the fibrant objects in this model category. \par 

From the beginning of the development of these two theories, the general feeling was that they should be equivalent in the precise sense of there being a Quillen equivalence between the two model categories. This was already stated explicitly in the early installments of Lurie's DAG-series \cite{luriedagiii} and later in his Higher Algebra \cite{higheralgebra}. In this paper we will establish a (zig-zag of) Quillen equivalence(s), under the assumption that the $\infty$-operads have no constant (i.e. nullary) operations. No direct comparison seems to be possible; there are several different aspects to our somewhat indirect approach. In hindsight, the first step is a quite logical one: in Lurie's approach, the representable objects are much like those in dendroidal sets, with one big difference, namely that they correspond to `forests' (i.e. disjoint unions of trees), rather than just trees. To bring the two categories more in line with each other, we first develop a theory of `forest sets', close to dendroidal sets and Quillen equivalent to it. It is somewhat non-trivial to develop such a theory and the proof that it is equivalent to dendroidal sets requires the theory of dendroidal complete Segal spaces and its forest analogue, which takes up a large part of the paper (Chapter \ref{chap:forestsets}). \par 

A second difference is that in the theory of dendroidal sets, `equivalences' are treated by means of the infinite-dimensional sphere $J$ like in Joyal's original approach \cite{joyalpaper}, while Lurie deals with equivalences through markings on simplical sets. Again, to bring these in line, we extend the theory of dendroidal sets and of forest sets to marked dendroidal and forest sets. A slightly different (in fact, more general) theory of marked dendroidal sets had already been developed earlier in \cite{heuts}. \par 
 
Finally, a somewhat awkward feature is that in Lurie's approach there is a non-trivial zero-object $\langle 0 \rangle$, which is not the initial object, but acts as an initial object only in a homotopy-theoretic sense. However, this complication is easily overcome by moving to the Quillen equivalent slice category of objects under $\langle 0 \rangle$. \par
 
These constructions together result in the following diagram of model categories, the arrows between which we will comment on below. The names of the categories in this diagram are as follows: $\mathbf{dSets}$ for dendroidal sets, $\mathbf{fSets}$ for forest sets and $\mathbf{POp}$ for Lurie's category of $\infty$-preoperads. A superscript plus indicates that the objects of the category are endowed with `markings'. A subscript $o$ indicates the restriction to subcategories modelling the theory of $\infty$-operads without constants.

\[
\xymatrix@C=35pt@R=35pt{
\mathbf{dSets}_o^{\mathbf{\Delta}^{\mathrm{op}}} & \mathbf{fSets}_o^{\mathbf{\Delta}^{\mathrm{op}}} \ar[l]_{u^*} & \\ 
\mathbf{dSets}_o \ar[u]\ar[d]_{(-)^\flat} & \mathbf{fSets}_o \ar[u]\ar[l]_{u^*} \ar[d]_{(-)^\flat} & \\
\mathbf{dSets}_o^+ & \mathbf{fSets}_o^+ \ar[l]_{u^*} \ar[dr]^{\bar\omega^*} & \mathbf{POp}_o \ar[l]_{\omega_!} \ar[d]^{\langle 0 \rangle_!} \\
& & \langle 0 \rangle /\mathbf{POp}_o
}
\] 

The arrows in this diagram all denote left Quillen equivalences that we will construct. The functors $(-)^\flat$ are equivalences which are left adjoint to the right Quillen equivalences which forget the markings, exactly as in Chapter 3 of \cite{htt}. The functor $u^*$ is an obvious restriction functor from presheaves on forests to presheaves on trees, but we are only able to show that it is a left Quillen equivalence by passing through the categories of complete dendroidal and forest Segal spaces on top of the diagram. The main functors connecting the `Lurie side' of the diagram to the dendroidal side are the functors $\omega_!$ and $\bar \omega^*$, which we will construct in Chapter \ref{chap:dendrification}. It is in the proofs that these are left Quillen functors where much of the combinatorial aspects of our work lie, see Sections \ref{sec:omega!leftQ} and \ref{sec:omega*leftQ}. \par 

As mentioned above, both Lurie's model and the dendroidal model come with a notion of tensor product. Roughly speaking, the tensor product $\mathbf{P} \otimes \mathbf{Q}$ of two $\infty$-operads can be characterized by the fact that algebras over $\mathbf{P} \otimes \mathbf{Q}$ correspond to $\mathbf{P}$-algebras in the category of $\mathbf{Q}$-algebras, or equivalently $\mathbf{Q}$-algebras in the category of $\mathbf{P}$-algebras. We will show that (the derived functors of) $\omega_!$ and $\bar\omega^*$ respect these tensor products up to weak equivalence. Although the monoidal structure on the category $\mathbf{POp}$ of preoperads is not symmetric in the usual sense, it is symmetric up to weak equivalence. This observation can be exploited to give the homotopy category of $\mathbf{POp}$ a symmetric monoidal structure. We will demonstrate that our equivalence between the two models gives an equivalence of symmetric monoidal homotopy categories (see Section \ref{sec:tensorproduct}). \par 

One useful and immediate corollary of our work is a strictification result for Lurie's $\infty$-operads. Indeed, the model category of dendroidal sets is known to be Quillen equivalent to the model category of simplicial operads \cite{cisinskimoerdijk3}. Therefore our results produce a zig-zag of Quillen equivalences between the model category of simplicial operads without constants and Lurie's model category $\mathbf{POp}_o$ of preoperads without constants. However, there is also a straightforward direct construction of a functor from the category of (fibrant) simplicial operads to the category $\mathbf{POp}_o$. In Section \ref{sec:strictification} we compare this functor to the zig-zag just described and prove that they are equivalent in an appropriate sense, thereby obtaining a direct equivalence between the associated homotopy categories.

\subsection*{Acknowledgements}
Parts of this paper were written while the second author was visiting the IHES, and the third author was visiting the Universit\'{e} de Paris VII, and the Newton Institute and St. John's College in Cambridge, respectively. We are grateful to these institutions for their hospitality and excellent working conditions. In addition, we would like to thank the Dutch Science Foundation (NWO) for supporting several mutual visits and the Acad\'{e}mie des Sciences for supporting Moerdijk's visit to Paris through a Descartes-Huygens Prize. The authors would like to thank Jacob Lurie for several useful conversations.

%% file: mainresults.tex
\section{Several models for the theory of $\infty$-operads}
\label{chap:mainresults}

We briefly review three different models for the theory of $\infty$-operads, namely the $\infty$-operads in the sense of Lurie, dendroidal sets and simplicial operads. The equivalence of the latter two approaches has already been shown in \cite{cisinskimoerdijk3}. The goal of this paper is to establish an equivalence between the first two. At the end of this chapter we describe our results. The rest of the paper is devoted to their proofs.

\subsection{Operads}
Throughout this paper the term \emph{operad} will always mean \emph{symmetric coloured operad}. An operad $\mathbf{P}$ in a given closed symmetric monoidal category $\mathbf{V}$ with tensor unit $I$ consists of a set of colours $\mathrm{col}(\mathbf{P})$ and, for each tuple $(c_1, \ldots, c_n, d)$ of such colours, an object
\begin{equation*}
\mathbf{P}(c_1, \ldots, c_n; d)
\end{equation*}
of $\mathbf{V}$. This object is to be thought of as parametrizing operations of $\mathbf{P}$ with $n$ inputs of the respective colours $c_1, \ldots, c_n$ and an output of colour $d$. (The set of inputs is allowed to be empty.) There should be composition maps
\begin{equation*}
\mathbf{P}(d_1, \ldots, d_n; e) \otimes \mathbf{P}(c_1^1, \ldots, c_1^{m_1}; d_1) \otimes \cdots \otimes \mathbf{P}(c_n^1, \ldots, c_n^{m_n}; d_n) \longrightarrow \mathbf{P}(c_1^1, \ldots, c_n^{m_n}; e)
\end{equation*}
and, for each $c \in \mathrm{col}(\mathbf{P})$, an identity (or unit)
\begin{equation*}
I \longrightarrow \mathbf{P}(c;c).
\end{equation*}
Finally, permutations $\sigma \in \Sigma_n$ should act on the right by transformations
\begin{equation*}
\mathbf{P}(c_1, \ldots, c_n; d) \longrightarrow \mathbf{P}(c_{\sigma(1)}, \ldots, c_{\sigma(n)}; d).
\end{equation*}
All of these data are required to satisfy various well-known associativity, equivariance and unit axioms. A morphism of operads $f: \mathbf{P} \longrightarrow \mathbf{Q}$ consists of a map $f: \mathrm{col}(\mathbf{P}) \longrightarrow \mathrm{col}(\mathbf{Q})$ together with a collection of morphisms
\begin{equation*}
\mathbf{P}(c_1, \ldots, c_n; d) \longrightarrow \mathbf{Q}(f(c_1), \ldots, f(c_n); f(d))
\end{equation*}
which are compatible with the given compositions, units and symmetric group actions. The cases of most interest to us here will be those where $\mathbf{V}$ is either the category of sets or that of simplicial sets, the symmetric monoidal structure coming from the categorical product in both cases. We denote the category of operads in sets (resp. simplicial sets) by $\mathbf{Op}$ (resp. $\mathbf{sOp}$). We will say an operad is \emph{non-unital} if $\mathbf{P}(-;d) = \varnothing$ for every colour $d$ of $\mathbf{P}$; in other words, if $\mathbf{P}$ does not contain any nullary operations. A special role will be played by the operad $\mathbf{Com}^-$ parametrizing non-unital commutative algebras; it has one colour, one operation of every strictly positive arity and no nullary operations. Observe that the category of non-unital operads in $\mathbf{Sets}$ is precisely the slice category $\mathbf{Op}/\mathbf{Com}^-$, and similarly for non-unital simplicial operads. We will denote those categories by $\mathbf{Op}_o$ and $\mathbf{sOp}_o$ respectively. Note that these are full subcategories of $\mathbf{Op}$ and $\mathbf{sOp}$. \par 
When $\mathbf{V} = \mathbf{Sets}$ we get special examples of (non-unital) operads from categories. Indeed, if $\mathbf{C}$ is a (small) category we can define an operad $\iota_!\mathbf{C}$ whose colours are the objects of $\mathbf{C}$ by setting
\begin{equation*}
\iota_!\mathbf{C}(c_1, \ldots, c_n; d) := \begin{cases} \mathbf{C}(c_1, d) & \text{if } n=1 \\
\varnothing & \text{otherwise.} \end{cases}
\end{equation*}
This procedure is part of an adjunction 
\[
\xymatrix@C=40pt{
\iota_!: \mathbf{Cat} \ar@<.5ex>[r] & \mathbf{Op}: \iota^* \ar@<.5ex>[l]
}
\]
between the category of small categories and the category of operads (which in fact factors uniquely through $\mathbf{Op}_o$). The right adjoint $\iota^*$ is given by discarding all non-unary operations. Note that the left adjoint $\iota_!$ is fully faithful. \par 
For later use, we will introduce the construction of the \emph{category of operations} associated to an operad. First we need some notation.
\begin{definition}
Given finite sets $A$ and $B$, a \emph{partial map} $f: A \longrightarrow B$ is a pair $(A', f')$, where $A' \subseteq A$ is a subset of $A$ and $f': A' \longrightarrow B$ is an ordinary map of sets. We will use the notation $\langle n \rangle$ for the set $\{1, \ldots, n\}$. Denote by $\mathbf{F}$ the category which has as objects the sets $\langle n \rangle$ for $n \geq 0$ (where $\langle 0 \rangle$ is the empty set by convention) and as morphisms the partial maps between those sets.
\end{definition}
Note that $\mathbf{F}$ is a skeleton of the category of all finite sets and partial maps between them, which in turn is the opposite of Segal's category $\Gamma$. In \cite{higheralgebra} Lurie uses the category $\mathcal{F}\mathrm{in}_*$, which is a skeleton of the category of pointed finite sets. There is a canonical functor
\begin{equation*}
\mathcal{F}\mathrm{in}_* \longrightarrow \mathbf{F}
\end{equation*}
given by forgetting the basepoint and assigning to a map $f: A \longrightarrow B$ of pointed finite sets the obvious partial map with domain of definition $f^{-1}(B\backslash \{*\})$. This functor is an isomorphism of categories. 

\begin{definition}
A morphism $f: A \longrightarrow B$ in $\mathbf{F}$ is said to be \emph{inert} if the preimage of any element of $B$ consists of exactly one element of $A$.  A morphism $f: A \longrightarrow B$ in $\mathbf{F}$ is \emph{active} if its domain of definition is all of $A$. For $n \geq 1$ and $1 \leq i \leq n$ denote by $\rho^i: \langle n \rangle \longrightarrow \langle 1 \rangle$ the unique inert partial map whose domain of definition is precisely $\{i\}$.
\end{definition}
 
Let us now describe the functor which assigns to an operad in $\mathbf{Sets}$ its category of operations. Given $\mathbf{P} \in \mathbf{Op}$ we define a category $\mathrm{cat}(\mathbf{P})$ as follows:
\begin{itemize}
\item[(1)] The objects of $\mathrm{cat}(\mathbf{P})$ are (possibly empty) tuples $(c_1, \ldots, c_m)$ of colours of $\mathbf{P}$.
\item[(2)] A morphism
\begin{equation*}
f: (c_1, \ldots, c_m) \longrightarrow (d_1, \ldots, d_n)
\end{equation*}
in $\mathrm{cat}(\mathbf{P})$ is a morphism $\phi: \langle m \rangle \longrightarrow \langle n \rangle$ in $\mathbf{F}$ together with a collection of operations
\begin{equation*}
f_i \in \mathbf{P}\bigl((c_j)_{j \in \phi^{-1}\{i\}}; d_i \bigr)
\end{equation*}
for $1 \leq i \leq n$.
\item[(3)] The composition in $\mathrm{cat}(\mathbf{P})$ is given by composition in $\mathbf{F}$ and use of the composition maps of the operad $\mathbf{P}$.
\end{itemize}
There is an obvious functor
\begin{equation*}
\pi_{\mathbf{P}}: \mathrm{cat}(\mathbf{P}) \longrightarrow \mathbf{F}
\end{equation*}
To provide motivation for one of the definitions of an $\infty$-operad to be given later on, we make the following observations:
\begin{itemize}
\item[(1)] Suppose we are given an inert morphism $\phi: \langle m \rangle \longrightarrow \langle n \rangle$ in $\mathbf{F}$ and an object $(c_1, \ldots, c_m)$ of $\mathrm{cat}(\mathbf{P})$. These data canonically give rise to a morphism $(\phi, \{f_i\}_{1 \leq i \leq n})$ in $\mathrm{cat}(\mathbf{P})$ where the $f_i$ are all identities. This morphism has the special property that it is $\pi_{\mathbf{P}}$-coCartesian.
\item[(2)] Let $(c_1, \ldots, c_m)$ and $(d_1, \ldots, d_n)$ be two objects of $\mathrm{cat}(\mathbf{P})$ and let $f: \langle m \rangle \longrightarrow \langle n \rangle$ be a partial map. Recall the inert morphisms $\rho^i: \langle n \rangle \longrightarrow 1$ described above. Consider the canonical lifts (as described in (1)) of these maps to morphisms
\begin{equation*}
(d_1, \ldots, d_n) \longrightarrow (d_i) 
\end{equation*}  
Then these morphisms induce bijections
\begin{equation*}
\mathrm{cat}(\mathbf{P})\bigl( (c_1, \ldots, c_m), (d_1, \ldots, d_n) \bigr)_f \longrightarrow \prod_{i=1}^n \mathrm{cat}(\mathbf{P})\bigl( (c_1, \ldots, c_m), d_i \bigr)_{\rho^i \circ f}
\end{equation*}
The subscript $f$ on the left-hand side indicates that one only considers morphisms projecting to $f$ under $\pi_{\mathbf{P}}$, the subscript on the right has the analogous meaning. 
\item[(3)] There is a canonical equivalence (even an isomorphism) of categories
\begin{equation*}
\pi_{\mathbf{P}}^{-1}(\langle n \rangle) \longrightarrow \pi_{\mathbf{P}}^{-1}(\langle 1 \rangle)^{\times n}
\end{equation*}
\end{itemize}

The construction of the category of operations admits a straightforward extension to the case where $\mathbf{P}$ is a simplicial operad, in which case $\mathrm{cat}(\mathbf{P})$ is a simplicial category over $\mathbf{F}$, the latter now regarded as a discrete simplicial category. We will return to this construction later. Also, observe that the category of operators of the non-unital commutative operad $\mathbf{Com}^-$ is precisely the category of finite sets and \emph{surjective} partial maps. We will denote this category by $\mathbf{F}_o$.

\subsection{The category of $\infty$-preoperads}
\label{sec:POp}

In \cite{higheralgebra} Lurie introduces a formalism for the theory of $\infty$-operads. He organizes his $\infty$-operads into an $\infty$-category $\mathbf{Op}_\infty$ and exhibits this category as the underlying  $\infty$-category of a simplicial model category $\mathbf{POp}_\infty$, the category of so-called \emph{$\infty$-preoperads}. We will review the relevant definitions now. Also, we will abbreviate the notation $\mathbf{POp}_\infty$ to $\mathbf{POp}$ from now on. \par 
We will be interested in the category $\mathbf{sSets}/N\mathbf{F}$ of simplicial sets over the nerve of $\mathbf{F}$. Given an object $p: X \longrightarrow N\mathbf{F}$ of that category and an object $\langle n \rangle$ of $\mathbf{F}$, we will use the shorthand $X_{\langle n \rangle}$ to denote the fiber of $p$ over the corresponding vertex of $N\mathbf{F}$. For $p$ an inner fibration, the reader should also recall \cite{htt} the notion of a $p$-coCartesian edge of $X$, whose definition we do not repeat here.
\begin{definition}
A (Lurie) $\infty$-operad is an inner fibration of simplicial sets $p: \mathcal{O} \longrightarrow N\mathbf{F}$ which satisfies the following:
\begin{itemize}
\item[(1)] For every inert morphism $f: \langle m \rangle \longrightarrow \langle n \rangle$ in $\mathbf{F}$ and every vertex $C \in \mathcal{O}_{\langle m \rangle}$ there exists a $p$-coCartesian edge $f': C \longrightarrow C'$ in $\mathcal{O}$ such that $p(f') = f$. In particular, we can associate to $f$ a map of simplicial sets $f_!: \mathcal{O}_{\langle m \rangle} \longrightarrow \mathcal{O}_{\langle n \rangle}$, uniquely up to homotopy.
\item[(2)] Let $C \in \mathcal{O}_{\langle m \rangle}$ and $C' \in \mathcal{O}_{\langle n \rangle}$ be two vertices and let $f: \langle m \rangle \longrightarrow \langle n \rangle$ be a partial map. Let $\mathrm{Map}_{\mathcal{O}}(C, C')_f$ be the preimage of $f \in \mathbf{F}(\langle m \rangle, \langle n \rangle)$ under $p$. Choose $p$-coCartesian lifts $C' \longrightarrow C'_i$ of the maps $\rho^i: \langle n \rangle \longrightarrow \langle 1 \rangle$ defined above. We obtain a map (unique up to homotopy) as follows:
\begin{equation*}
\mathrm{Map}_{\mathcal{O}}(C, C')_f \longrightarrow \prod_{i=1}^n \mathrm{Map}_{\mathcal{O}}(C, C'_i)_{\rho^i \circ f}
\end{equation*}
This is a homotopy equivalence.
\item[(3)] Using (1) we obtain for each $n \geq 0$ a collection of maps $\{\rho^i_!: \mathcal{O}_{\langle n \rangle} \longrightarrow \mathcal{O}_{\langle 1 \rangle} \}_{1 \leq i \leq n}$. These induce equivalences of $\infty$-categories
\begin{equation*}
\mathcal{O}_{\langle n \rangle} \longrightarrow \mathcal{O}_{\langle 1 \rangle}^{\times n}
\end{equation*}
\end{itemize}
\end{definition}

We will now introduce the terminology necessary to describe the model category $\mathbf{POp}$. A \emph{marked simplicial set} is a pair $(X, \mathcal{E})$, where $X$ is a simplicial set and $\mathcal{E}$ is a subset of the set of 1-simplices of $X$. We require $\mathcal{E}$ to contain all the degenerate edges of $X$. The category of marked simplicial sets, which we will denote by $\mathbf{sSets}^+$, has as objects marked simplicial sets and as morphisms those maps of simplicial sets that map marked edges to marked edges. We define the marked simplicial set 
\begin{equation*}
N\mathbf{F}^{\natural} := (N\mathbf{F}, \mathcal{I})
\end{equation*}
where $N\mathbf{F}$ is the nerve of the category $\mathbf{F}$ and $\mathcal{I}$ is the collection of all inert morphisms in $\mathbf{F}$. The category of $\infty$-preoperads is defined by
\begin{equation*}
\mathbf{POp} := \mathbf{sSets}^+/N\mathbf{F}^\natural
\end{equation*}
This category is naturally tensored over simplicial sets. Indeed, for $X \in \mathbf{POp}$ and $K \in \mathbf{sSets}$ one sets $X \otimes K := X \times K^\sharp$, where $K^\sharp$ denotes $K$ with all its edges marked. By adjunction this tensoring induces the structure of a simplicial category on $\mathbf{POp}$. \par 
Given an $\infty$-operad $p: \mathcal{O} \longrightarrow N\mathbf{F}$, we will say that an edge $f$ of $\mathcal{O}$ is \emph{inert} if it is a $p$-coCartesian lift of an inert morphism of $\mathbf{F}$. Set $\mathcal{O}^\natural := (\mathcal{O}, \mathcal{I}_{\mathcal{O}})$, where $\mathcal{I}_{\mathcal{O}}$ is the collection of inert edges of $\mathcal{O}$. The following is due to Lurie \cite{higheralgebra}:

\begin{proposition}
\label{prop:modelstructpop}
There exists a model structure on $\mathbf{POp}$ which is characterized by the following properties:
\begin{itemize}
\item[(C)] A morphism is a cofibration precisely if its underlying map of simplicial sets is a monomorphism.
\item[(F)] Fibrant objects are precisely objects of the form $\mathcal{O}^\natural$, for $\mathcal{O}$ an $\infty$-operad.
\end{itemize}
Furthermore this model structure is left proper, combinatorial and simplicial with respect to the simplicial structure described above.
\end{proposition}

As we noted earlier, the construction of the category of operators can be extended to simplicial operads. For a given simplicial operad $\mathbf{P}$, this construction now yields a \emph{simplicial} category over $\mathbf{F}$, the latter regarded as a discrete simplicial category. Let 
\begin{equation*}
N: \mathbf{sCat} \longrightarrow \mathbf{sSets}
\end{equation*}
denote the homotopy-coherent nerve. The following result (see \cite{higheralgebra}) provides many examples of $\infty$-operads:

\begin{proposition}
Let $\mathbf{P}$ be a fibrant simplicial operad, i.e. an operad for which the simplicial sets $\mathbf{P}(c_1, \ldots, c_n; d)$ are all Kan complexes. Then
\begin{equation*}
N(\mathrm{cat}(\mathbf{P})) \longrightarrow N\mathbf{F}
\end{equation*}
is an $\infty$-operad. 
\end{proposition}

In this paper we will mostly restrict our attention to \emph{non-unital} $\infty$-operads. To be precise, we will say an $\infty$-operad $p: \mathcal{O} \longrightarrow N\mathbf{F}$ is non-unital if $p$ factors through $N\mathbf{F}_o$. Since the map $N\mathbf{F}_o \rightarrow N\mathbf{F}$ is a monomorphism, such a factorization is necessarily unique (if it exists). We will denote by $\mathbf{POp}_o$ the slice category
\begin{equation*}
\mathbf{POp}/N\mathbf{F}_o^\natural
\end{equation*}  
and refer to it as the category of non-unital $\infty$-preoperads. Note that this category inherits a model structure from $\mathbf{POp}$ to which the description of cofibrations and fibrant objects of Proposition \ref{prop:modelstructpop} still applies. Let $(\mathbf{sOp}_o)_f$ denote the full subcategory of $\mathbf{sOp}_o$ spanned by the fibrant open simplicial operads. Then using the previous result we obtain a functor
\begin{equation*}
\nu: (\mathbf{sOp}_o)_f \longrightarrow \mathbf{POp}_o: \mathbf{P} \longmapsto \Bigl(N(\mathrm{cat}(\mathbf{P}))^\natural \longrightarrow N(\mathbf{F}_o)^\natural \Bigr).
\end{equation*}
We will see later that this functor in fact induces an equivalence of homotopy categories.

\subsection{Dendroidal sets}

\subsubsection{The category of dendroidal sets}
In this section we review the basic definitions concerning the category of dendroidal sets. For more details we refer the reader to \cite{cisinskimoerdijk1}, \cite{moerdijklectures} and \cite{moerdijkweiss}. As in these references, we write $\mathbf{\Omega}$ for the following category of trees. Objects of $\mathbf{\Omega}$ are finite rooted trees. Such a tree has \emph{internal} (or \emph{inner}) and \emph{external} (or \emph{outer}) edges. Internal edges connect two vertices, while external edges are attached to only one vertex. One of the external edges is designated as being the \emph{root}, all the others are called \emph{leaves}. The choice of root gives a canonical notion of direction on the tree (namely `towards the root'), which allows us to speak of the \emph{input edges} and \emph{output edge} of every vertex. The number of input edges is called the \emph{valence} of the vertex. We refer to the vertex connected to the root edge as the \emph{root vertex} and to a vertex all of whose inputs are leaves as a \emph{leaf vertex}. A vertex with no input edges is called a \emph{stump}. The collection of \emph{external} (or \emph{outer}) vertices is formed by the leaf vertices, the stumps and the root vertex. For example, the tree

\[
\begin{tikzpicture} 
[level distance=10mm, 
every node/.style={fill, circle, minimum size=.1cm, inner sep=0pt}, 
level 1/.style={sibling distance=20mm}, 
level 2/.style={sibling distance=10mm}, 
level 3/.style={sibling distance=5mm}]
\node[style={color=white}] {} [grow'=up] 
child {node (p) {} 
	child{ node (q){}
		child
		child
	}
	child{ node (r) {}
	}
	child{
	}
};
\tikzstyle{every node}=[]

\node at ($(p) + (255:15pt)$) {$a$};
\node at ($(p) + (-10:6pt)$) {$p$};
\node at ($(p) + (150:20pt)$) {$b$};
\node at ($(p) + (105:15pt)$) {$e$};
\node at ($(p) + (35:21pt)$) {$f$};
\node at ($(q) + (200:5pt)$) {$q$};
\node at ($(q) + (125:15pt)$) {$c$};
\node at ($(q) + (60:16pt)$) {$d$};
\node at ($(r) + (180:5pt)$) {$r$};

\end{tikzpicture} 
\]

with the root edge $a$ drawn at the bottom, has three leaves $c$, $d$ and $f$ and three vertices $p$, $q$ and $r$ which are all external and have valence 3, 2 and 0 respectively. There exists one tree which has no vertices at all, in which the root edge is also a leaf; it is pictured as

\[
\begin{tikzpicture} 
[level distance=10mm, 
every node/.style={fill, circle, minimum size=.1cm, inner sep=0pt}, 
level 1/.style={sibling distance=20mm}, 
level 2/.style={sibling distance=10mm}, 
level 3/.style={sibling distance=5mm}]
\node[style={color=white}] {} [grow'=up] 
child; 
\end{tikzpicture} 
\]

We will denote this tree by $\eta$. \par 
Each tree $T$ in $\mathbf{\Omega}$ generates a (symmetric, coloured) operad $\Omega(T)$ in $\mathbf{Sets}$. The colours of this operad are the edges of the tree and the operations are generated by the vertices. One way to select a set of generators is by fixing a planar structure on the tree $T$. For example, for the tree pictured above with the planar structure as drawn, the natural generators are
\begin{eqnarray*}
p & \in & \Omega(T)(b, e, f; a) \\
q & \in & \Omega(T)(c, d; b) \\
r & \in & \Omega(T)( - ; e)
\end{eqnarray*} 
and the other operations are either identities or obtained from $p$, $q$ and $r$ by symmetrization and composition. Thus, for example, $\Omega(T)$ also has operations like

\[
\renewcommand{\arraystretch}{2} 
\begin{tabular}{l c l l}
$1_b$ & $\in$ & $\Omega(T)(b; b)$ & (identity) \\
$p \circ_e r$ & $\in$ & $\Omega(T)(b, f; a)$  & (composition) \\
$q \cdot \tau$ & $\in$ & $\Omega(T)(d, c; b)$ & (symmetry)
\end{tabular}
\]

where $\tau$ is the nontrivial element in the symmetric group $\Sigma_2$, etc. Another planar structure on the tree $T$ defines a different set of generators, but the \emph{same} operad $\Omega(T)$. \par 
Arrows in the category $\mathbf{\Omega}$ from a tree $S$ to a tree $T$ are maps of operads $\Omega(S) \longrightarrow \Omega(T)$. This completes the definition of the category $\mathbf{\Omega}$. \par 
The simplex category $\mathbf{\Delta}$ admits a natural inclusion into $\mathbf{\Omega}$ by the functor
\begin{equation*}
i: \mathbf{\Delta} \longrightarrow \mathbf{\Omega}
\end{equation*}
which sends an object $[n]$ to the linear tree with $n$ vertices and $n+1$ edges, labelled $0, \ldots, n$, where $0$ is the leaf and $n$ is the root:
\[
\begin{tikzpicture} 
[level distance=6mm, 
every node/.style={fill, circle, minimum size=.1cm, inner sep=0pt}, 
level 1/.style={sibling distance=20mm}, 
level 2/.style={sibling distance=10mm}, 
level 3/.style={sibling distance=5mm}]
\node[style={color=white}] {} [grow'=up] 
child {node (zero){} 
		child {}
};
\node[style={color=white}, below=3cm of zero] {} [grow'=up] 
child {node (n){}
		child{} 
};
\draw[dashed] ($(zero) + (-90:6mm)$) --($(n) + (90:6mm) $);

\tikzstyle{every node}=[]

\node at ($(zero) + (120:10pt)$) {$0$};
\node at ($(zero) + (-120:10pt)$) {$1$};
\node at ($(n) + (-120:10pt)$) {$n$};

\end{tikzpicture} 
\]

Just like in the category $\mathbf{\Delta}$, the arrows in the category $\mathbf{\Omega}$ are generated by a family of arrows that one can describe in simple terms. In $\mathbf{\Omega}$ there are \emph{faces} and \emph{degeneracies}, extending the corresponding notions in $\mathbf{\Delta}$, and also \emph{isomorphisms} of trees. Any arrow $S \longrightarrow T$ decomposes as a composition of degeneracies followed by an isomorphism followed by a composition of faces. For example, with the tree $T$ as pictured above in the centre, we have the following morphisms:

\[
\begin{tikzpicture} 
[level distance=10mm, 
every node/.style={fill, circle, minimum size=.1cm, inner sep=0pt}, 
level 1/.style={sibling distance=20mm}, 
level 2/.style={sibling distance=10mm}, 
level 3/.style={sibling distance=5mm}]

%centre tree
\node (centretree)[style={color=white}] {} [grow'=up] 
child {node (p) {} 
	child{ node (q) {}
		child
		child
	}
	child{ node {}
	}
	child{
	}
};

%tree above
\node(treeabove)[style={color=white}, above=4cm of centretree] {} [grow'=up] 
child {node(abovep) {} 
	child{ node {}
		child
		child
	}
	child{ node {}
	}
	child{
	} 
};
\node[above left = .6cm of abovep]{};

%tree to the left
\node(lefttree)[style={color=white}, left=4cm of centretree] {} [grow'=up] 
child {node (leftp) {} 
	child
	child{ node {}
	}
	child{
	} 
};

%tree to the right
\node(righttree)[style={color=white}, right=4cm of centretree] {} [grow'=up] 
child {node {} 
	child{ node (w){}
		child
		child
	}
	child{ node {}
	}
	child{
	} 
};

%tree below
\node(treebelow)[style={color=white}, below=4cm of centretree] {} [grow'=up] 
child {node (v) {} 
	child
	child
	child{ node {}
	}
	child
};

\tikzstyle{every node}=[]

%arrows
\draw[->] ($(treeabove) + (0,-5pt)$) -- node[right]{$\sigma_s$} ($(centretree) + (0, 80pt)$);
\draw[->] ($(treebelow) + (0,65pt)$) -- node[right]{$\partial_b$} ($(centretree) + (0, -15pt)$);
\draw[->] ($(lefttree) + (35pt,30pt)$) -- node[above]{$\partial_q$} ($(centretree) + (-35pt, 30pt)$);
\draw[->] ($(righttree) + (-35pt,30pt)$) -- node[above]{$\tau$} ($(centretree) + (35pt, 30pt)$);

%labels
%centre tree
\node at ($(p) + (-10:6pt)$) {$p$};
\node at ($(p) + (150:20pt)$) {$b$};
\node at ($(q) + (125:15pt)$) {$c$};
\node at ($(q) + (60:16pt)$) {$d$};
\node at ($(q) + (200:5pt)$) {$q$};

%upper tree
\node at ($(abovep) + (140:30pt)$) {$b_1$};
\node at ($(abovep) + (155:11pt)$) {$b_2$};
\node at ($(abovep) + (122:21pt)$) {$s$};

%left tree
\node at ($(leftp) + (150:20pt)$) {$b$};

%lower tree
\node at ($(v) + (-10:6pt)$) {$v$};
\node at ($(v) + (157:25pt)$) {$c$};
\node at ($(v) + (127:22pt)$) {$d$};

%right tree
\node at ($(w) + (125:16pt)$) {$d$};
\node at ($(w) + (60:14pt)$) {$c$};
\node at ($(w) + (200:6pt)$) {$w$};

\end{tikzpicture} 
\]

The arrow $\sigma_s$ is a degeneracy; as a map of operads, it sends the generating operation $s$ to the identity operation of the edge $b$.  The arrow $\partial_q$ corresponds to chopping off the vertex $q$ and is called an \emph{external face} of $T$. As a map of operads, it is simply the obvious inclusion. (Such an external face exists for any vertex with exactly one inner edge attached to it; any leaf vertex satisfies this condition, the root vertex might or might not.) The arrow $\partial_b$ corresponds to contracting the inner edge $b$ and is called an \emph{inner face}. As a map of operads, it sends the generator $v$ to $p \circ_b q$. The arrow $\tau$ is the isomorphism of trees interchanging $c$ and $d$. As a map of operads, it sends the generator $w$ to $q \cdot \tau$. \par 
The category of \emph{dendroidal sets} is the category of presheaves on $\mathbf{\Omega}$:
\begin{equation*}
\mathbf{dSets} := \mathbf{Sets}^{\mathbf{\Omega}^{\mathrm{op}}}
\end{equation*} 
The inclusion $i$ induces an adjoint pair (left adjoint on the left)
\[
\xymatrix@C=40pt{
i_!: \mathbf{sSets} \ar@<.5ex>[r] & \mathbf{dSets}: i^* \ar@<.5ex>[l]
}
\]
The functor $i_!$ is fully faithful and allows us to regard any simplicial set as a dendroidal set. In the other direction, each dendroidal set $X$ has an \emph{underlying simplicial set} $i^*X$. Let us list several examples of dendroidal sets. \par 

\begin{example}
Every tree $T \in \mathbf{\Omega}$ gives rise to a representable dendroidal set, which we denote by $\Omega[T]$. This notation resembles the notation $\Delta[n]$ for representable simplicial sets and we have
\begin{equation*}
i_!\Delta[n] = \Omega[i[n]]
\end{equation*}
\end{example}

\begin{example}
For a tree $T \in \mathbf{\Omega}$, the \emph{boundary} $\partial \Omega[T]$ of $T$ is the subpresheaf of $\Omega[T]$ obtained as the union of all proper monomorphisms (i.e. monomorphisms which aren't isomorphisms) into $\Omega[T]$. The map $\partial \Omega[T] \longrightarrow \Omega[T]$ can be obtained as the union of all the face inclusions
\begin{equation*}
\partial_x: \Omega[S] \longrightarrow \Omega[T]
\end{equation*}
where $x$ ranges over inner edges and those outer vertices of $T$ attached to only one inner edge (i.e. all leaf vertices and possibly the root vertex).
\end{example}

\begin{example}
For an inner edge $e$ in a tree $T$, the \emph{inner horn} $\Lambda^e[T]$ corresponding to $e$ is the subpresheaf of $\Omega[T]$ obtained as the union of all proper monomorphisms into $\Omega[T]$ having the edge $e$ in their image. It can be obtained as the union of all faces of $T$ \emph{except} the one given by contracting $e$.
\end{example}

\begin{example}
For an operad $\mathbf{P}$ in $\mathbf{Sets}$, its (dendroidal) nerve $N_d(\mathbf{P})$ is the dendroidal set defined by
\begin{equation*}
N_d(\mathbf{P})(T) := \mathbf{Op}(\Omega(T), \mathbf{P})
\end{equation*}
This defines a fully faithful functor
\begin{equation*}
N_d: \mathbf{Op} \longrightarrow \mathbf{dSets}
\end{equation*}
which has a left adjoint denoted
\begin{equation*}
\tau_d: \mathbf{dSets} \longrightarrow \mathbf{Op}
\end{equation*}
These functors are compatible with the similar pair $\tau$ and $N$ relating categories and simplicial sets, in the sense that the following two squares, of right and left adjoints respectively, commute:
\[
\xymatrix@C=40pt@R=40pt{
\mathbf{sSets} \ar@<.5ex>^\tau[r]\ar@<-.5ex>_{i_!}[d] & \mathbf{Cat} \ar@<.5ex>^N[l]\ar@<-.5ex>_{\iota_!}[d] \\
\mathbf{dSets} \ar@<-.5ex>_{i^*}[u]\ar@<.5ex>^{\tau_d}[r] & \mathbf{Op}\ar@<-.5ex>_{\iota^*}[u]\ar@<.5ex>^{N_d}[l]
}
\]
\end{example}

We need some discussion of \emph{open} dendroidal sets. First of all, we will say a tree $T$ is open if it contains no stumps (i.e. nullary vertices). Denote by $\mathbf{\Omega}_o$ the full subcategory of $\mathbf{\Omega}$ on the open trees. We will refer to the category of presheaves on $\mathbf{\Omega}_o$ as the category of open dendroidal sets and denote it by $\mathbf{dSets}_o$. The inclusion $\mathbf{\Omega}_o \rightarrow \mathbf{\Omega}$ induces a fully faithful functor $\mathbf{dSets}_o \rightarrow \mathbf{dSets}$, which canonically factors through $\mathbf{dSets}/N_d\mathbf{Com}^-$. In fact, this gives an isomorphism of categories
\begin{equation*}
\mathbf{dSets}_o \simeq \mathbf{dSets}/N_d\mathbf{Com}^-
\end{equation*}
and we will often blur the distinction between these two categories, regarding dendroidal sets as either presheaves on $\mathbf{\Omega}_o$ or dendroidal sets equipped with a (necessarily unique) map to $N_d\mathbf{Com}^-$. The reader should note that the dendroidal nerve of a non-unital operad in $\mathbf{Sets}$ is an open dendroidal set. Also, the embedding $i_!: \mathbf{sSets} \rightarrow \mathbf{dSets}$ factors canonically through the category of open dendroidal sets.

\subsubsection{A model structure on dendroidal sets}

\begin{definition}
A dendroidal set $X$ is called \emph{normal} if for each tree $T$, the action of $\mathrm{Aut}(T)$ on $X(T)$ is free. More generally, a monomorphism $X \longrightarrow Y$ of dendroidal sets is called \emph{normal} if for each tree $T$, the group $\mathrm{Aut}(T)$ acts freely on the complement of the image of $X(T)$ in $Y(T)$. 
\end{definition}

\begin{definition}
A map $X \longrightarrow Y$ of dendroidal sets is called an \emph{inner Kan fibration}, or just an \emph{inner fibration}, if it has the right lifting property with respect to all inner horn inclusions
\begin{equation*}
\Lambda^e[T] \longrightarrow \Omega[T]
\end{equation*}
for all trees $T \in \mathbf{\Omega}$ and all inner edges $e$ of $T$. A \emph{dendroidal inner Kan complex} is a dendroidal set $X$ for which the map $X \longrightarrow 1$ to the terminal object is an inner Kan fibration. These dendroidal inner Kan complexes are also referred to more briefly as \emph{(dendroidal) $\infty$-operads}.
\end{definition}

Together with Cisinski, the third author established the following (cf. \cite{cisinskimoerdijk1}):
\begin{theorem}
\label{thm:CMmodelstruct}
There exists a model structure on the category $\mathbf{dSets}$ characterized by the following two properties:
\begin{itemize}
\item[(C)] The cofibrations are the normal monomorphisms.
\item[(F)] The fibrant objects are the dendroidal $\infty$-operads.
\end{itemize}
\end{theorem}

We should recall the following additional properties of this model structure:
\begin{itemize}
\item[(a)] The model structure is combinatorial (so in particular cofibrantly generated) and left proper. The boundary inclusions $\partial\Omega[T] \longrightarrow \Omega[T]$ form a set of generating cofibrations.
\item[(b)] For the representable dendroidal set $\eta$, the slice category $\mathbf{dSets}/\eta$ is isomorphic to the category of simplicial sets, by an isomorphism which identifies the forgetful functor
\begin{equation*}
\mathbf{dSets}/\eta \longrightarrow \mathbf{dSets}
\end{equation*}
with the functor
\begin{equation*}
i_!: \mathbf{sSets} \longrightarrow \mathbf{dSets}.
\end{equation*}
Under this isomorphism, the induced model structure on $\mathbf{dSets}/\eta$ corresponds to the Joyal model structure on $\mathbf{sSets}$.
\item[(c)] The fibrations between fibrant objects can be characterized explicitly as those inner fibrations $X \longrightarrow Y$ with the additional property that the functor $\tau i^*(X) \longrightarrow \tau i^*(Y)$ is a categorical fibration. Recall that a functor $f: \mathbf{A} \longrightarrow \mathbf{B}$ is a categorical fibration if, for any isomorphism $\theta: b \simeq b'$ in $\mathbf{B}$, any lift of $b$ to an object $a$ of $\mathbf{A}$ (i.e. $f(a) = b$) can be extended to a lift of $\theta$ to an isomorphism $a \simeq a'$ for some object $a'$ of $\mathbf{A}$. 
\end{itemize}

\begin{remark}
The category $\mathbf{dSets}$ carries a symmetric tensor product related to the Boardman-Vogt \cite{boardmanvogt} tensor product of operads (see \cite{moerdijklectures,moerdijkweiss}). In particular, using this tensor product and the functor $i_!$, the model category $\mathbf{dSets}$ of the theorem becomes enriched in $\mathbf{sSets}$ with the Joyal model structure, in a sense explicitly discussed in Section \ref{subsec:weakenrichment}. In addition, the tensor product restricts to a tensor product on the category $\mathbf{dSets}_o$ (because of the fact that $N_d\mathbf{Com}^- \otimes N_d\mathbf{Com}^- = N_d\mathbf{Com}^-$), and this tensor product is compatible with the restricted model structure on $\mathbf{dSets}_o$ (see \cite{cisinskimoerdijkerr}).
\end{remark}

\subsection{Simplicial operads}

The category $\mathbf{sOp}$ of simplicial operads carries a model structure \cite{cisinskimoerdijk3} analogous to the Bergner model structure \cite{bergner} on the category of simplicial categories. The functor
\begin{equation*}
\Omega(-): \mathbf{\Omega} \longrightarrow \mathbf{Op}
\end{equation*} 
can be lifted to a functor
\begin{equation*}
W: \mathbf{\Omega} \longrightarrow \mathbf{sOp} 
\end{equation*}
by means of the Boardman-Vogt $W$-resolution with respect to the simplicial interval $\Delta^1$:
\begin{equation*}
W(T) := W(\Delta^1,\Omega(T))
\end{equation*}
where the right-hand side corresponds to the notation of \cite{bergermoerdijk}. This functor $W$ induces an adjoint pair
\[
\xymatrix@C=40pt{
W_!: \mathbf{dSets} \ar@<.5ex>[r] & \mathbf{sOp}: W^* \ar@<.5ex>[l].
}
\]
We will refer to $W_!(X)$ as the \emph{Boardman-Vogt resolution} of the dendroidal set $X$ and to $W^*(\mathbf{P})$ as the \emph{homotopy-coherent nerve} of the simplicial operad $\mathbf{P}$. When restricted to simplicial sets on the left and simplicial categories on the right, the above adjunction reduces to the adjoint pair that is denoted $(\mathfrak{C}, N)$ in \cite{htt}. The following result, which can be viewed as a strictification result for dendroidal $\infty$-operads, was proved in \cite{cisinskimoerdijk3}:

\begin{theorem}
The adjoint pair $(W_!, W^*)$ defines a Quillen equivalence between the category of dendroidal sets equipped with the model structure of Theorem \ref{thm:CMmodelstruct} and the category of simplicial operads equipped with the model structure established in \cite{cisinskimoerdijk3}.
\end{theorem}

\begin{remark}
\label{rmk:W!nonunital}
The pair $(W_!,W^*)$ restricts to an adjunction between the categories of open dendroidal sets and non-unital simplicial operads, providing a Quillen equivalence between these model categories.
\end{remark}

\begin{remark}
\label{rmk:Sigmacofibrant}
We should mention one more fact concerning simplicial operads. A simplicial operad $\mathbf{P}$ is called \emph{$\Sigma$-cofibrant} if the symmetric group actions inherent in the definition of $\mathbf{P}$ are all free. A cofibrant simplicial operad $\mathbf{P}$ is $\Sigma$-cofibrant, but the converse of this statement generally fails to hold. It is not hard to verify that if $\mathbf{P}$ is $\Sigma$-cofibrant, then the dendroidal set $W^*\mathbf{P}$ is normal and thus cofibrant in the model structure on $\mathbf{dSets}$ discussed above.
\end{remark}

\subsection{Main results}
\label{sec:mainresults}
The goal of this paper is to show that there exists a chain of Quillen equivalences connecting the categories $\mathbf{dSets}_o$ and $\mathbf{POp}_o$, both equipped with their respective model structures as described above. A key ingredient is the construction of an auxiliary category $\mathbf{fSets}$, the category of $\emph{forest sets}$. Just like $\mathbf{dSets}$, this is a presheaf category. The indexing category $\mathbf{\Phi}$ is a category of \emph{forests}. There exists a fully faithful functor
\begin{equation*}
u: \mathbf{\Omega} \longrightarrow \mathbf{\Phi}
\end{equation*}  
which by left and right Kan extension induces adjunctions
\[
\xymatrix@C=40pt{
u_!: \mathbf{dSets} \ar@<.5ex>[r] & \mathbf{fSets}: u^* \ar@<.5ex>[l]
}
\]
and
\[
\xymatrix@C=40pt{
u^*: \mathbf{fSets} \ar@<.5ex>[r] & \mathbf{dSets}: u_* \ar@<.5ex>[l].
}
\]
We will define the category $\mathbf{fSets}$ in detail in Chapter \ref{chap:forestsets}, as well as its full subcategory $\mathbf{fSets}_o$ of \emph{open forest sets}. The main result there is:

\begin{theorem}
\label{thm:forestsets}
The category of forest sets carries a model structure, homotopically enriched (see Section \ref{subsec:weakenrichment}) over the Joyal model structure on simplicial sets, for which the adjoint pair $(u^*, u_*)$ forms a Quillen equivalence with the category of dendroidal sets.
\end{theorem}

Later, in Section \ref{sec:tensorproduct}, we will show that when restricted to open forest sets, this model structure as well as the Quillen pair $(u^*,u_*)$ are compatible with tensor products.

To continue, two more auxiliary categories are needed. These are $\mathbf{dSets}^+$ and $\mathbf{fSets}^+$ (and their `open' variants), the categories of \emph{marked dendroidal sets} and \emph{marked forest sets} respectively. We will construct these categories in Chapter \ref{chap:markedfsets}. Both these categories are closely related to their unmarked analogues, and in fact the main result of that chapter will be:

\begin{theorem}
\label{thm:markeddSetsfSets}
There exists a commutative square of left Quillen functors as follows:
\[
\xymatrix@C=40pt@R=40pt{
\mathbf{dSets} \ar[d]_{(-)^\flat} & \mathbf{fSets} \ar[d]^{(-)^\flat}\ar[l]_{u^*} \\
\mathbf{dSets}^+ & \mathbf{fSets}^+ \ar[l]^{u^*}
}
\]
All these functors induce Quillen equivalences and are compatible with tensor products.
\end{theorem}

With all these preliminaries in place, we can finally relate the category of open dendroidal sets to the category $\mathbf{POp}_o$. In Chapter \ref{chap:dendrification} we construct the \emph{dendrification functor}
\begin{equation*}
\omega: \mathbf{\Delta}/N\mathbf{F}_o \longrightarrow \mathbf{fSets}_o.
\end{equation*}
Here $\mathbf{\Delta}/N\mathbf{F}_o$ denotes the Grothendieck construction of the simplicial set $N\mathbf{F}_o$, also called its \emph{category of simplices}. Roughly speaking, one can visualize a simplex in $N\mathbf{F}_o$ by drawing a picture of a \emph{layered forest}. For example, we can draw the 2-simplex
\begin{equation*}
A: \Delta^2 \longrightarrow N\mathbf{F}_o
\end{equation*} 
given by

\[
\xymatrix{
\langle 6 \rangle \ar[r]^f & \langle 3 \rangle \ar[r]^g & \langle 1 \rangle
}
\]
\[
\xymatrix{
f(1)=f(2)=f(3)= 1, & f(4) = 2, & f(5)=f(6)= 3, & g(1)=g(2)= 1
}
\]

as follows:

\[
\begin{tikzpicture} 
[level distance=10mm, 
every node/.style={fill, circle, minimum size=.1cm, inner sep=0pt}, 
level 1/.style={sibling distance=20mm}, 
level 2/.style={sibling distance=10mm}, 
level 3/.style={sibling distance=5mm}]

%left tree
\node (lefttree)[style={color=white}] {} [grow'=up] 
child {node (level1) {} 
	child{ node (level2) {}
		child
		child
		child
	}
	child{ node {}
		child
	}
};

%right tree
\node (righttree)[style={color=white}, right = 1.5cm of lefttree] {};
\node (righttreestart)[style={color=white}, above = .88cm of righttree] {} [grow'=up] 
child {node {} 
	child
	child
};

\tikzstyle{every node}=[]

%lines
\draw[dashed] ($(level1) + (-1.5cm, 0)$) -- ($(level1) + (2.5cm, 0)$);
\draw[dashed] ($(level1) + (-1.5cm, 1cm)$) -- ($(level1) + (2.5cm, 1cm)$);

%labels
\node at ($(level1) + (-1.5cm, 1.5cm)$) {$0$};
\node at ($(level1) + (-1.5cm, .5cm)$) {$1$};
\node at ($(level1) + (-1.5cm, -.5cm)$) {$2$};

\end{tikzpicture} 
\]

The forest $\omega(A)$ is then simply the forest obtained from this picture by forgetting the layered structure. \par 
Using $\omega$, we construct an adjunction
\[
\xymatrix@C=40pt{
\omega_!: \mathbf{POp}_o \ar@<.5ex>[r] & \mathbf{fSets}_o^+: \omega^* \ar@<.5ex>[l]
}
\]

The main result of this paper is:

\begin{theorem}
\label{thm:omegaequivalence}
The pair $(\omega_!, \omega^*)$ is a Quillen equivalence.
\end{theorem}

\begin{corollary}
\label{cor:equivalencePOpdSets}
There is a zig-zag of Quillen equivalences as follows (left adjoints on top):
\[
\xymatrix@C=35pt{
\mathbf{dSets}_o \ar@<.5ex>[r]^{(-)^\flat} & \mathbf{dSets}_o^+ \ar@<.5ex>[l] \ar@<-.5ex>[r] & \mathbf{POp}_o \ar@<-.5ex>[l]_-{u^*\omega_!}. 
}
\]
\end{corollary}

Chapter \ref{chap:dendrification} is devoted to the proof of Theorem \ref{thm:omegaequivalence}. In Section \ref{sec:tensorproduct} we will investigate the behaviour of the relevant functors with respect to tensor products. We will prove the following:

\begin{theorem}
\label{thm:omega!monoidal}
The equivalence of Corollary \ref{cor:equivalencePOpdSets} is monoidal on the level of homotopy categories. More precisely, for $X, Y \in \mathbf{POp}_o$ there exists a natural weak equivalence of cofibrant marked forest sets as follows:
\[
\xymatrix{
\omega_!(X \odot Y) \ar[r]^-{\simeq} & \omega_!(X) \otimes \omega_!(Y)
}
\]
where $\odot$ (resp. $\otimes$) denotes the tensor product on $\mathbf{POp}_o$ (resp. $\mathbf{fSets}_o^+$).
\end{theorem}

The following corollary is not a purely formal consequence of this theorem, but will follow easily once we have studied the functor $\omega^*$.

\begin{corollary}[See Corollary \ref{cor:omega*monoidal2}]
\label{cor:omega*monoidal}
For cofibrant objects $P, Q \in \mathbf{fSets}_o^+$ there is a natural weak equivalence
\begin{equation*}
\omega^*(P) \odot \omega^*(Q) \longrightarrow \omega^*(P \otimes Q). 
\end{equation*}
\end{corollary}

The tensor product on $\mathbf{POp}_o$ is `symmetric up to weak equivalence'. This can be used to construct a symmetric monoidal structure on the homotopy category $\mathrm{Ho}(\mathbf{POp}_o)$. We will show in Section \ref{sec:tensorproduct} that the tensor product of open dendroidal sets, which is symmetric, is associative up to weak equivalence (in a precise sense), which endows the homotopy category $\mathrm{Ho}(\mathbf{dSets}_o)$ with a symmetric monoidal structure as well. We will finish Section \ref{sec:tensorproduct} by relating these structures:

\begin{proposition}
\label{prop:omegasymmonoidal}
The zig-zag of Quillen equivalences between $\mathbf{dSets}_o$ and $\mathbf{POp}_o$ induces an equivalence of symmetric monoidal categories between $\mathrm{Ho}(\mathbf{dSets}_o)$ and $\mathrm{Ho}(\mathbf{POp}_o)$.
\end{proposition}

\begin{remark}
With a little more care, one can extract symmetric monoidal $\infty$-categories from the model categories $\mathbf{dSets}_o$ and $\mathbf{POp}_o$ and show that our Quillen adjunctions induce equivalences between these. We will not belabour the details of such a construction here.
\end{remark}

\subsection{Strictification}

The chain of Quillen equivalences between the categories of open dendroidal sets and non-unital $\infty$-preoperads, as expressed by Corollary \ref{cor:equivalencePOpdSets}, allows us to transfer various properties of the model category of dendroidal sets to that of preoperads and vice versa. By way of illustration, we will in this section give an example of this, namely the strictification of non-unital (Lurie) $\infty$-operads (i.e. fibrant objects in $\mathbf{POp}_o$). \par 
Recall that Lurie's definition of an $\infty$-operad $\mathcal{O}$ involves various choices: for an inert 1-simplex $f: \langle m \rangle \rightarrow \langle n \rangle$ one has to choose coCartesian 1-simplices of $\mathcal{O}$ lying over it and one uses these to construct a map
\begin{equation*}
f_!: \mathcal{O}_{\langle m \rangle} \longrightarrow \mathcal{O}_{\langle n \rangle}
\end{equation*}
of simplicial sets. This map is only unique up to homotopy (or rather, up to a contractible space of choices) and functorial in the weak sense that for another inert morphism $g: \langle l \rangle \rightarrow \langle m \rangle$, the composition $f_! g_!$ is homotopic, not necessarily equal, to $(fg)_!$. In analogy with the theory of (co)fibered categories, we call an $\infty$-operad \emph{split} if it comes equipped with explicit choices of coCartesian 1-simplices over inert maps $f: \langle m \rangle \rightarrow \langle n \rangle$, as well as explicit choices of corresponding maps $f_!: \mathcal{O}_{\langle m \rangle} \rightarrow \mathcal{O}_{\langle n \rangle}$, functorial in the sense that $f_!g_! = (fg)_!$. If $\mathcal{P}$ is an arbitrary $\infty$-operad, a \emph{strictification} of $\mathcal{P}$ is a weak equivalence $\mathcal{P} \rightarrow \mathcal{P}'$, where $\mathcal{P}'$ is a split $\infty$-operad.

\begin{theorem}
Every non-unital $\infty$-operad admits a strictification.
\end{theorem}

This result follows from our equivalence between open dendroidal sets and non-unital preoperads, since the image of any open dendroidal $\infty$-operad under this equivalence admits a canonical strictification. In fact, even more is true: for $\mathcal{O}$ an $\infty$-operad obtained from a dendroidal set, the natural choice of strictification will induce an \emph{isomorphism}
\begin{equation*}
\mathcal{O}_{\langle n \rangle} \simeq \mathcal{O}_{\langle 1 \rangle}^{\times n}
\end{equation*}
rather than just an equivalence. \par 
Our results also provide a zig-zag of Quillen equivalences between the category $\mathbf{POp}_o$ and the category of non-unital simplicial operads, by composing the equivalence
\begin{equation*}
W^*: \mathbf{sOp}_o \longrightarrow \mathbf{dSets}_o
\end{equation*}
with the chain of equivalences of Corollary \ref{cor:equivalencePOpdSets}. A careful inspection (cf. Section \ref{sec:strictification}) of the functors involved will show that, on the level of homotopy categories, this equivalence between simplicial operads and $\mathbf{POp}_o$ agrees with the functor $\nu$ described in Section \ref{sec:POp}, so that we obtain a further `strictification' of Lurie's $\infty$-operads:

\begin{theorem}
The functor
\begin{equation*}
\nu: (\mathbf{sOp}_o)_f \longrightarrow \mathbf{POp}_o: \mathbf{P} \longmapsto \bigl(N\mathrm{cat}(\mathbf{P})^\natural \rightarrow N\mathbf{F}_o^\natural \bigr)
\end{equation*}
induces an equivalence on the level of homotopy categories.
\end{theorem}

\begin{remark}
In fact, the result we will prove is stronger. It shows that the functor $\nu$ induces an equivalence of relative categories, in the language of \cite{barwickkan1}, or equivalently, an equivalence between the simplicial localizations of the categories involved \cite{dwyerkan,barwickkan2}.
\end{remark}

%% file: forestsets.tex
\section{Forest sets}
\label{chap:forestsets}

In this chapter we will introduce another model for the homotopy theory of $\infty$-operads, closely related to dendroidal sets, but with trees replaced by forests. The plan for this chapter is as follows. First, we will introduce the category $\mathbf{\Phi}$ of forests. The category of presheaves on $\mathbf{\Phi}$ is the category of \emph{forest sets}. Next, we discuss a special class of maps between forest sets, namely the \emph{normal monomorphisms}. Afterwards, we will establish a model category structure on this presheaf category. The chapter will end with a proof that the model category of forest sets is Quillen equivalent to dendroidal sets.

\subsection{The category $\mathbf{\Phi}$ of forests}

We recall the category $\mathbf{\Omega}$ of trees from Chapter \ref{chap:mainresults}. Its objects are trees, its arrows between trees $S \longrightarrow T$ are maps $\Omega(S) \longrightarrow \Omega(T)$ between the operads freely generated by $S$ and $T$. Any tree induces a natural partial order on its edges, where $e \leq e'$ if the unique path from $e'$ to the root of $T$ contains $e$. (There is of course a similar partial order on the vertices of $T$.) Two edges of $T$ are called \emph{incomparable}, or \emph{independent}, if they are not related in this partial order. Two sets of edges $A$ and $B$ are called \emph{independent} if any two edges $a \in A$ and $b \in B$ are incomparable. Thus, a collection $\{A_i\}$ of sets of edges of $T$ is pairwise independent if any path from the root of $T$ to any leaf of $T$ intersects at most one of the sets $A_i$. \par 
We can now define the category $\mathbf{\Phi}$, which can be thought of as obtained from $\mathbf{\Omega}$ by freely adjoining \emph{sums} of trees and ``independent'' maps. An object of $\mathbf{\Phi}$ is a finite non-empty collection
\begin{equation*}
F \, = \, \{S_i \in \mathbf{\Omega} \, | \, i \in I \}
\end{equation*}
We will call such objects \emph{forests} and we will also write
\begin{equation*}
F \, = \, \bigoplus_{i \in I} S_i
\end{equation*}
while referring to such an $F$ as the \emph{direct sum} of the trees $S_i$. If $G = \bigoplus_{j \in J} T_j$ is another forest, an arrow
\begin{equation*}
(\alpha, f): F \longrightarrow G
\end{equation*}
is a pair consisting of a function $\alpha: I \longrightarrow J$ and for each $i \in I$ a map $f_i: S_i \longrightarrow T_{\alpha(i)}$ in $\mathbf{\Omega}$. Moreover, if $\alpha(i) = j = \alpha(i')$, where $i \neq i'$, then $f_i$ and $f_{i'}$ should have independent images in $T_j$. In other words, if $e \in S_i$ and $e' \in S_{i'}$ are two edges, then $f_i(e)$ and $f_{i'}(e')$ are incomparable in the partial order on the edges of $T_j$. \par 
Observe that the operation assigning to two forests $F$ and $G$ their direct sum $F \oplus G$ equips $\mathbf{\Phi}$ with the structure of a (non-unital) symmetric monoidal category. Note, however, that this operation is \emph{not} a coproduct in $\mathbf{\Phi}$. Indeed, a would-be codiagonal $S \oplus S \longrightarrow S$ does not satisfy the independence condition on morphisms and is therefore not an arrow in $\mathbf{\Phi}$. \par 
There is an obvious full and faithful functor
\begin{equation*}
u: \mathbf{\Omega} \longrightarrow \mathbf{\Phi}
\end{equation*}
which sends a tree $T$ to the forest $u(T)$ consisting of only the tree $T$. We will often be somewhat informal and view $\mathbf{\Omega}$ as a subcategory of $\mathbf{\Phi}$ and we'll sometimes just write $T$ for $u(T)$ when it is clear that we are considering the tree $T$ as an object of $\mathbf{\Phi}$. However, some care is needed when it comes to the discussion of \emph{faces} (and in the next section, of \emph{boundaries} and \emph{horns}), as we will now explain. \par 
The arrows in $\mathbf{\Omega}$ are generated by ``elementary'' face maps, degeneracy maps and isomorphisms. In fact, every arrow can uniquely be written as a composition of degeneracies, followed by an isomorphism, followed by a composition of face maps (see \cite{moerdijklectures}). The elementary faces of a tree $T$ in $\mathbf{\Omega}$ come in two kinds: inner faces given by the contraction of an inner edge in $T$ and external faces chopping off a vertex on the top of a tree, or, in case the root vertex has only one internal edge attached to it, the face obtained by deleting the root vertex and all external edges attached to it. In $\mathbf{\Phi}$, however, there is a root face of a different kind: regardless of the number of inner edges of $T$ attached to the root vertex, we can delete the root vertex and the root edge and what remains is a forest which we denote by $\partial_{\mathrm{root}}(u(T))$, or by $\partial_r(u(T))$ if it is clear that $r$ is the root vertex. Note that there is an evident inclusion in the category  $\mathbf{\Phi}$,
\begin{equation*}
\partial_{\mathrm{root}}(u(T)) \longrightarrow u(T),
\end{equation*}
which looks like
\[
\begin{tikzpicture} 
[level distance=10mm, 
every node/.style={fill, circle, minimum size=.1cm, inner sep=0pt}, 
level 1/.style={sibling distance=20mm}, 
level 2/.style={sibling distance=15mm}, 
level 3/.style={sibling distance=14mm}]

%left tree
\node (lefttree)[style={color=white}] {} [grow'=up] 
child {node (left) {} 
	child
	child
};

%middle tree
\node (middletree)[style={color=white}, right = 2cm of lefttree] {} [grow'=up] 
child {node (middle) {} 
	child
	child
};

%right tree
\node (righttree)[style={color=white}, right = 2cm of middletree] {} [grow'=up] 
child {node (right) {} 
	child
	child
};

%forest
\node(forest)[style={color=white}, right = 4cm of righttree]{};
\node[style={color=white}, below = .5cm of forest] {} [grow'=up] 
child {node {} 
	child {node(fleft){} 
		child
		child}
	child {node(fmiddle){} 
		child
		child}
	child {node(fright){} 
		child
		child}
};

\tikzstyle{every node}=[]

%pluses
\node at ($(left) + (1.05cm, 0)$) {$\oplus$};
\node at ($(middle) + (1.05cm, 0)$) {$\oplus$};

%dots
\node at ($(left) + (0,.6cm)$) {$\cdots$};
\node at ($(middle) + (0,.6cm)$) {$\cdots$};
\node at ($(right) + (0,.6cm)$) {$\cdots$};
\node at ($(fleft) + (0,.6cm)$) {$\cdots$};
\node at ($(fmiddle) + (0,.6cm)$) {$\cdots$};
\node at ($(fright) + (0,.6cm)$) {$\cdots$};

%arrow
\draw[->] ($(right) + (1cm, 0)$) -- ($(right) + (2cm, 0)$);

\end{tikzpicture} 
\]

In other words, the tree $T$ viewed as a forest $u(T)$ has top external faces just like $T$ in $\mathbf{\Omega}$ and moreover it will always have a root face, at least if $T$ is not the tree $\eta$ or the unique tree with one edge and one vertex of valence zero (the ``stump''). This root face is a proper forest (i.e. an object of $\mathbf{\Phi}$ not in the image of $\mathbf{\Omega} \longrightarrow \mathbf{\Phi}$), unless the root of $T$ is a unary vertex. Also, if $\partial_{\mathrm{root}}(T)$ does exist in $\mathbf{\Omega}$, then there is a map
\begin{equation*}
u(\partial_{\mathrm{root}}(T)) \longrightarrow \partial_{\mathrm{root}}(u(T))
\end{equation*}
which is an isomorphism only if the root vertex is unary. \par 
We should be explicit about our conventions concerning corollas, i.e. trees with just one vertex. For the corolla $C_n$ with leaves $1, \ldots, n$ and root edge $0$, there are $n+1$ faces in $\mathbf{\Omega}$,
\[
\xymatrix{
\eta \, \ar@<-.2ex>@{^(->}[r]^i & C_n 
}
\]
which are all external. In $\mathbf{\Phi}$ there is one such for $i = 0$ and if $n > 0$ there is one other, namely the $n$-fold direct sum of copies of $\eta$, as follows:
\[
\xymatrix{
\eta \oplus \cdots \oplus \eta = \partial_{\mathrm{root}}(C_n) \, \ar@<-.2ex>@{^(->}[r] & C_n
}
\] 

The following lemma also explains some aspects of the difference between the category $\mathbf{\Omega}$ of trees and the category $\mathbf{\Phi}$ of forests.

\begin{lemma}
The category $\mathbf{\Phi}$ is obtained from $\mathbf{\Omega}$ as follows. The objects of $\mathbf{\Phi}$ are obtained by formally closing the objects of $\mathbf{\Omega}$ under non-empty finite direct sums. The arrows are generated by
\begin{itemize}
\item[(i)] All arrows $u(S) \longrightarrow u(T)$ arising from arrows $S \longrightarrow T$ in $\mathbf{\Omega}$
\item[(ii)] Inclusions $F \longrightarrow F \oplus G$ of summands
\item[(iii)] Inclusions of the form $\partial_{\mathrm{root}}(u(T)) \longrightarrow u(T)$
\end{itemize}
subject to the condition that $\oplus$ is functorial in both variables, symmetric and associative.
\end{lemma}
\begin{proof}
Consider a map
\begin{equation*}
(\alpha, f): \bigoplus_{i \in I} S_i \longrightarrow \bigoplus_{j \in J} T_j
\end{equation*}
in $\mathbf{\Phi}$. Using maps as in (ii) and the stated condition on $\oplus$, such a map can be obtained from maps where $J$ is a singleton. So consider a map
\begin{equation*}
\bigoplus_{i \in I} S_i \longrightarrow T
\end{equation*}
If $I$ has precisely one element, then it is a map of type (i). If $I$ has more than one element, then the independence condition on morphisms in $\mathbf{\Phi}$ implies that our map factors as a composition
\begin{equation*}
\bigoplus_{i \in I} S_i \longrightarrow \partial_{\mathrm{root}} T \longrightarrow T
\end{equation*}
One can now finish the proof by induction on the size of the fibers of $\alpha$.  $\Box$
\end{proof}

\begin{definition}
\begin{itemize}
\item[(i)] If $S \longrightarrow S'$ is an elementary degeneracy in $\mathbf{\Omega}$ (i.e. a map identifying two adjacent edges of $S$), then we call any map of the form
\begin{equation*}
S \longrightarrow S' \quad\quad \text{or} \quad\quad S \oplus F \longrightarrow S' \oplus F
\end{equation*}
in $\mathbf{\Phi}$ an \emph{elementary degeneracy}, or just a \emph{degeneracy}. (We sometimes use the word elementary to stress the fact that $S$ has exactly one more vertex than $S'$ and to distinguish this from a composition of several degeneracies.)
\item[(ii)] For an object of $\mathbf{\Phi}$ consisting of a single tree $S$, an \emph{elementary face} of $S$ is a map in $\mathbf{\Phi}$ of one of the following two kinds:
\begin{itemize}
\item[(a)] A map $S' \longrightarrow S$ which is induced by an internal face or a leaf face in $\mathbf{\Omega}$.
\item[(b)] The root face inclusion $\partial_{\mathrm{root}} S \longrightarrow S$.
\end{itemize}
More generally, an \emph{elementary face} of a forest $S \oplus F$ (where $S$ is a tree) is a map in $\mathbf{\Phi}$ of one of the following three kinds:
\begin{itemize}
\item[(a)] A map $S \oplus F \longrightarrow S' \oplus F$ induced by a map $S \longrightarrow S'$ which is an internal face or a leaf face in $\mathbf{\Omega}$.
\item[(b)] A map of the form $\partial_{\mathrm{root}} S \oplus F \longrightarrow S \oplus F$ induced by the root face inclusion of $S$, regarded as a forest.
\item[(c)] A map of the form $F \longrightarrow \eta \oplus F$ which is the identity on the $F$ summand.
\end{itemize} 
\end{itemize}
\end{definition}

Note that elementary degeneracies are surjective on edges and reduce the number of edges by one. Elementary faces are injective on edges and increase the number of vertices by one or in case (c) keep the number of vertices equal but increase the number of connected components of the forest by one. Exactly as in $\mathbf{\Omega}$, one has the following factorization of arrows in $\mathbf{\Phi}$:

\begin{lemma}
Any arrow $F \longrightarrow G$ in $\mathbf{\Phi}$ can be decomposed uniquely as
\[
\xymatrix{
F \ar@{->>}[r] & F' \ar^{\simeq}[r] & G' \,\,\ar@{>->}[r] & G,
}
\]
where the first map is a composition of degeneracies, the second map is an isomorphism and the third map is a composition of faces. Note that a map in $\mathbf{\Phi}$ is an isomorphism if it induces a bijection on connected components and the restriction to every component is an isomorphism in $\mathbf{\Omega}$. 
\end{lemma}
\begin{proof}
Consider an arbitrary map $(\alpha, f): \bigoplus_{i \in I} S_i \longrightarrow \bigoplus_{j \in J} T_j$ as before. Factor each $f_i: S_i \longrightarrow T_{\alpha(i)}$ as
\[
\xymatrix{
S_i \ar@{->>}[r] & S_i' \ar^-{\simeq}[r] & T_{\alpha(i)}' \,\, \ar@{>->}[r] & T_{\alpha(i)}
}
\]
using the known factorization of morphisms in $\mathbf{\Omega}$. This gives a composition of maps in $\mathbf{\Phi}$ as follows:
\[
\xymatrix{
\bigoplus_{i \in I} S_i \ar@{->>}[r] & \bigoplus_{i \in I} S_i' \ar^-{\simeq}[r] & \bigoplus_{i \in I} T_{\alpha(i)}' \,\, \ar@{>->}[r] & \bigoplus_{j \in J} T_{\alpha(i)}
}
\]
The first map is clearly a composition of degeneracies. The last one is a composition of maps of the form
\begin{equation*}
\bigl(\bigoplus_{i \in \alpha^{-1}(j)} T_{\alpha(i)}' \bigr) \oplus G \longrightarrow T_j \oplus G
\end{equation*}
Using an induction as in the proof of the previous lemma, we can write each $\bigoplus_{i \in \alpha^{-1}(j)} T_{\alpha(i)}' \longrightarrow T_j$ as a composition of faces, where one uses elementary faces of type (b) if $\alpha^{-1}(j)$ has more than one element and of type (c) if $\alpha^{-1}(j)$ is empty. Uniqueness follows straightforwardly by using the uniqueness of the factorization in $\mathbf{\Omega}$. $\Box$
\end{proof}

\begin{remark}
The previous lemma in fact shows that, like $\mathbf{\Omega}$, the category $\mathbf{\Phi}$ is a dualizable generalized Reedy category in the sense of \cite{bergermoerdijkReedy}. Explicitly, one defines $\mathbf{\Phi}^+$ to consist of maps which are injective on edges and $\mathbf{\Phi}^-$ as consisting of those which are surjective. We will use the resulting Reedy model structure on simplicial presheaves in Section \ref{sec:equivfsetsdsets}.
\end{remark}

\subsection{Presheaves on the category of forests}
\label{sec:fsetsproperties}

In this section we discuss some constructions in, and properties of, the category of set-valued presheaves on $\mathbf{\Phi}$. (Presheaves with values in simplicial sets will feature in Section \ref{sec:equivfsetsdsets}.) We will refer to such presheaves as \emph{forest sets} and denote the category of these as
\begin{equation*}
\mathbf{fSets} := \mathbf{Sets}^{\mathbf{\Phi}^{\mathrm{op}}}
\end{equation*}
Let us notice right away that the inclusion functor
\begin{equation*}
u: \mathbf{\Omega} \longrightarrow \mathbf{\Phi}
\end{equation*}
induces a triple of adjoint functors relating forest sets to dendroidal sets:
\[
\xymatrix@C=40pt{
\mathbf{dSets} \ar@/^/@<.75ex>[r]^{u_!} \ar@/_/@<-.75ex>[r]_{u_*} & \mathbf{fSets} \ar[l]_{u^*}
}
\]
Also notice that since $u$ is fully faithful, so are $u_!$ and $u_*$. In particular, for any dendroidal set $X$ the canonical maps
\begin{equation*}
u^*u_*X \longrightarrow X \longrightarrow u^*u_!X
\end{equation*}
are isomorphisms. \par 
The functors $u_!$ and $u_*$ provide many examples of forest sets coming from dendroidal sets. Also, each forest $F$ defines a representable forest set which we denote by $\Phi[F]$. Thus, for a tree $T$, we have the relation
\begin{eqnarray*}
u_!\Omega[T] & = & \Phi[u(T)] \quad\quad (\text{or simply } \Phi[T])
\end{eqnarray*}
When no confusion can arise, we will often just write $T$ for $\Omega[T]$ and $uT$ or $u_!T$ for $\Phi[u(T)]$. \par 
\textit{Direct sums.} The category $\mathbf{fSets}$ has all (small) colimits, so we can extend the operation $\oplus$ on $\mathbf{\Phi}$ to a symmetric monoidal structure on $\mathbf{fSets}$ as follows. We first define it on representables as
\begin{equation*}
\Phi[F] \oplus \Phi[G] \, := \, \Phi[F \oplus G]
\end{equation*}
Next, for a fixed forest $F$, we view $\Phi[F] \oplus \Phi[-]$ as a functor
\begin{equation*}
\mathbf{\Phi} \longrightarrow \Phi[F]/\mathbf{fSets}
\end{equation*}
and extend it (in a way that is unique up to unique isomorphism) to a colimit preserving functor
\begin{equation*}
\Phi[F] \oplus - : \mathbf{fSets} \longrightarrow \Phi[F]/\mathbf{fSets}
\end{equation*}
This defines $\Phi[F] \oplus X$ for any forest $F$ and any object $X$ of $\mathbf{fSets}$. Note that $\Phi[F] \oplus X$ comes equipped with a map $X \longrightarrow \Phi[F] \oplus X$, naturally in $F$. Thus we have a functor
\begin{equation*}
- \oplus X : \mathbf{\Phi} \longrightarrow X / \mathbf{fSets}
\end{equation*}
which we can again extend to a colimit preserving functor
\begin{equation*}
- \oplus X : \mathbf{fSets} \longrightarrow X / \mathbf{fSets}
\end{equation*}
This procedure defines a symmetric monoidal structure on the category $\mathbf{fSets}$ which we will refer to as \emph{direct sum}, with the initial object $\varnothing$ as the unit. Also note that there is a canonical monomorphism
\begin{equation*}
X \amalg Y \longrightarrow X \oplus Y
\end{equation*}
from the coproduct to the direct sum, which is never an isomorphism if $X$ and $Y$ are nonempty.

\begin{remark}
The functor $X \oplus - : \mathbf{fSets} \longrightarrow X/\mathbf{fSets}$ has a right adjoint, denoted
\begin{equation*}
(X \rightarrow Z) \longmapsto Z \ominus X
\end{equation*}
Thus, there is a natural bijective correspondence between maps $X \oplus Y \longrightarrow Z$ under $X$ and maps $Y \longrightarrow Z \ominus X$. Since $\oplus$ is symmetric, these also correspond to maps $X \longrightarrow Z \ominus Y$ if $Z$ is viewed as an object under $Y$.
\end{remark}

\textit{Tensor product.} The (``Boardman-Vogt'') tensor product on dendroidal sets induces another tensor product on $\mathbf{fSets}$, completely determined up to unique isomorphism by the following conditions on $X \otimes Y$ for forest sets $X$ and $Y$: 
\begin{itemize}
\item[(i)] $X \otimes Y$ preserves colimits in each variable separately. 
\item[(ii)] The functor $X \otimes - $ distributes over $\oplus$.
\item[(iii)] The functor $u_!: \mathbf{dSets} \longrightarrow \mathbf{fSets}$ preserves the tensor product (up to natural isomorphism).
\end{itemize}
More explicitly, for forests $F = \bigoplus_{i \in I} S_i$ and $G = \bigoplus_{j \in J} T_j$, one defines
\begin{equation*}
F \otimes G \, := \, \bigoplus_{(i,j) \in I \times J} u_!(S_i \otimes T_j)
\end{equation*}
and one then extends this operation from representable objects $F$ and $G$ to arbitrary objects in $\mathbf{fSets}$, by writing the latter as colimits of representables. If one extends the definition of shuffles of trees (as in \cite{moerdijklectures,moerdijkweiss}) to forests, then the tensor product $F \otimes G$ can also be described as the union of all shuffles of the forests $F$ and $G$, just like for dendroidal sets. \par 
For later reference we summarize some of the properties of these structures on $\mathbf{fSets}$ and their relations to the corresponding notions on $\mathbf{dSets}$:

\begin{proposition}
\label{prop:fsetsproperties}
The category $\mathbf{fSets}$ carries two symmetric tensor products, $\otimes$ and $\oplus$, satisfying the following properties:
\begin{itemize} 
\item[(i)] $\otimes$ distributes over $\oplus$. 
\item[(ii)] There are canonical maps $X \rightarrow X \oplus Y \leftarrow Y$ and the functor 
\begin{equation*}
X \oplus - : \mathbf{fSets} \longrightarrow X / \mathbf{fSets}
\end{equation*}
has a right adjoint.
\item[(iii)] The functor $u_!: \mathbf{dSets} \longrightarrow \mathbf{fSets}$ is compatible  with $\otimes$, i.e. there is a natural isomorphism
\begin{eqnarray*}
u_!(X \otimes Y) & \simeq & u_!(X) \otimes u_!(Y)
\end{eqnarray*}
for any two dendroidal sets $X$ and $Y$.
\item[(iv)] The functor $u^*: \mathbf{fSets} \longrightarrow \mathbf{dSets}$ is compatible with $\otimes$ and sends direct sums to coproducts:
\begin{eqnarray*}
(a) \quad\quad u^*(X \otimes Y) & \simeq & u^*(X) \otimes u^*(Y) \\
(b) \quad\quad u^*(X \oplus Y) & \simeq & u^*(X) \amalg u^*(Y)
\end{eqnarray*}
\end{itemize}
\end{proposition}
\begin{proof}
Only property (iv) has not been discussed before. Since $u^*$ preserves colimits, for (b) it suffices to prove that for a collection of trees $S_1, \ldots, S_n$ we have
\begin{equation*}
u^*(S_1 \oplus \cdots \oplus S_n) \, = \, u^*(S_1) \amalg \ldots \amalg u^*(S_n)
\end{equation*}
This is clear from the definitions. Since $\otimes$ distributes over $\oplus$, (iv)(a) now follows from $u^*u_! = \mathrm{id}$. $\Box$
\end{proof}

\begin{remark}
One can define a Grothendieck topology on the category $\mathbf{\Phi}$, generated by covering families of the form
\begin{equation*}
\{S_j \longrightarrow \bigoplus_{i \in I} S_i \}_{j \in I}
\end{equation*}
The topos $\mathrm{Sh}(\mathbf{\Phi})$ of sheaves for this topology is canonically equivalent to $\mathbf{dSets}$. We will use a homotopy theoretic version of this observation later on, when we compare $\mathbf{dSets}$ and $\mathbf{fSets}$ as model categories.
\end{remark}

As for dendroidal sets before, there is a full subcategory $\mathbf{\Phi}_o$ of $\mathbf{\Phi}$ of open forests, i.e. forests whose constituent trees are open. We will write $\mathbf{fSets}_o$ for the full subcategory of $\mathbf{fSets}$ consisting of presheaves on $\mathbf{\Phi}_o$. It is again a slice category of $\mathbf{fSets}$ over a subobject of the terminal object, namely $u_*N_d(\mathbf{Com}^-)$. Note that the functors $u_!$, $u^*$, $u_*$, as well direct sums and tensor products all restrict to open objects.

\subsection{Normal monomorphisms and boundaries in $\mathbf{fSets}$}
\label{sec:normalmonos}

Exactly as for dendroidal sets, we will call a monomorphism $X \longrightarrow Y$ between forest sets \emph{normal} if for every forest $F$, the group $\mathrm{Aut}(F)$ acts freely on the complement of the image of $X(F) \longrightarrow Y(F)$. An object $Y$ in $\mathbf{fSets}$ is called \emph{normal} if $\varnothing \longrightarrow Y$ is a normal monomorphism, i.e. if $\mathrm{Aut}(F)$ acts freely on $Y(F)$ for every $F$ in $\mathbf{\Phi}$. The following is clear from the definition:

\begin{lemma}
\label{lem:normalovernormal}
If $X \longrightarrow Y$ is a map of forest sets and $Y$ is normal, then $X$ is normal as well.
\end{lemma}

\begin{remark}
\label{rmk:sumcofibration}
Given normal forest sets $X$ and $Y$, the map
\begin{equation*}
X \amalg Y \longrightarrow X \oplus Y
\end{equation*}
is a normal monomorphism.
\end{remark}

\begin{lemma}
The functor $u^*: \mathbf{fSets} \longrightarrow \mathbf{dSets}$ sends normal monomorphisms to normal monomorphisms.
\end{lemma}
\begin{proof}
This is clear from the identities $u^*(X)(T) = X(uT)$ and $\mathrm{Aut}_{\mathbf{\Omega}}(T) = \mathrm{Aut}_{\mathbf{\Phi}}(uT)$, the second one following from the fact that $u$ is fully faithful. $\Box$
\end{proof}

\begin{remark}[Warning]
\label{rmk:warning}
The functor $u_!: \mathbf{dSets} \longrightarrow \mathbf{fSets}$ does \emph{not} send normal monomorphisms to normal monomorphisms. In fact, it does not even send them to monomorphisms in general. Consider the following example. Let $T$ be the tree
\[
\begin{tikzpicture} 
[level distance=10mm, 
every node/.style={fill, circle, minimum size=.1cm, inner sep=0pt}, 
level 1/.style={sibling distance=20mm}, 
level 2/.style={sibling distance=20mm}, 
level 3/.style={sibling distance=20mm},
level 4/.style={sibling distance=18mm}]
\node[style={color=white}] {} [grow'=up] 
child {node (r) {} 
	child{ node (u){}
		child{ node (v){}
			child
			child
			}
		child{ node (w){}
			child
			child
			}
	}
};
\tikzstyle{every node}=[]

\node at ($(r) + (.2cm,0)$) {$r$};
\node at ($(r) + (-.2cm,-.5cm)$) {$a$};
\node at ($(r) + (-.2cm,.5cm)$) {$b$};
\node at ($(u) + (.2cm,-.1cm)$) {$u$};
\node at ($(u) + (-.7cm,.5cm)$) {$c$};
\node at ($(u) + (.7cm,.5cm)$) {$f$};
\node at ($(v) + (.2cm,0)$) {$v$};
\node at ($(v) + (-.68cm,.5cm)$) {$d$};
\node at ($(v) + (.65cm,.5cm)$) {$e$};
\node at ($(w) + (.2cm,0)$) {$w$};
\node at ($(w) + (-.68cm,.5cm)$) {$g$};
\node at ($(w) + (.65cm,.5cm)$) {$h$};

\end{tikzpicture} 
\]

Now, the map $\partial_b(T) \cup \partial_r(T) \longrightarrow T$ is a normal monomorphism in $\mathbf{dSets}$; in fact, every mono into a representable is. On the other hand, consider 
\begin{equation*}
u_!(\partial_b(T) \cup \partial_r(T)) \, = \, \varinjlim_{R} u_!(R)
\end{equation*}
where the colimit is over all $R \longrightarrow T$ in $\mathbf{\Omega}$ which factor through $\partial_b(T)$ or $\partial_r(T)$ (or both). The two corollas with vertices $v$ and $w$ give rise to two \emph{different} maps
\begin{equation*}
u_!C_2 \oplus u_!C_2 \longrightarrow u_!(\partial_b(T) \cup \partial_r(T))
\end{equation*}
Indeed, there is one factoring through $u_!(\partial_b(T))$ and another one factoring through $u_!(\partial_r(T))$; these two maps only agree on the subobject
\begin{equation*}
u_!C_2 \amalg u_!C_2 \subseteq u_!C_2 \oplus u_!C_2
\end{equation*}
Hence the map
\begin{equation*}
u_!(\partial_b(T) \cup \partial_r(T)) \longrightarrow u_!(T)
\end{equation*}
is not a monomorphism. \par 
On the other hand, one easily checks that the composition
\[
\xymatrix{
\mathbf{sSets} \ar[r]^{i_!} & \mathbf{dSets} \ar[r]^{u_!} & \mathbf{fSets}
}
\]
does send monos to normal monos.
\end{remark}

We will now discuss the skeletal filtration of a normal forest set. Its description in Proposition \ref{prop:skeletalfiltration} below makes use of the notion of nondegenerate elements and of boundaries of forests, which we discuss first. \par 

\textit{Boundaries.} For a forest $F$, we will write $\partial \Phi[F]$, or simply $\partial F$, for
\begin{equation*}
\varinjlim_{G \rightarrowtail F} \Phi[G]
\end{equation*}
where the colimit ranges over all maps $G \longrightarrow F$ which strictly increase the number of edges. Thus, for direct sums we have
\begin{equation*}
\partial(F \oplus G) \, = \, \partial F \oplus G \cup F \oplus \partial G ,
\end{equation*}
so the calculation of the boundary of a forest reduces to that of the boundaries of its constituent trees $T$. There we have
\begin{equation*}
\partial(uT) \, = \, \bigcup_{F \rightarrowtail T} F
\end{equation*}
where $F$ ranges over the faces of $T$. Compared to the boundary of $T$ as computed in dendroidal sets, the only new face which arises is the root face, except in the two special cases $T = \eta$ and $T = C_0$ where there is no root face. Note that we have
\begin{equation*}
\partial \eta \, = \, \varnothing
\end{equation*}
and for a corolla $C_p$ we have
\begin{equation*}
\partial(u_!C_p) \, = \, \eta \amalg p \cdot \eta
\end{equation*} 
where the first copy of $\eta$ corresponds to the root of $C_p$ and 
\begin{equation*}
p \cdot \eta \, := \, \bigoplus_{i=1}^p \eta
\end{equation*}
is the ``crown'' of $C_p$, i.e. the direct sum of its leaves. If $p=0$, this crown is empty. Also notice that from these formulas and Proposition \ref{prop:fsetsproperties} it follows easily that
\begin{equation*}
u^*(\partial F) \, = \, \partial u^*(F)
\end{equation*}
for any forest $F$. \par 

\textit{Non-degenerate elements.} Let $X$ be a forest set and $F \in \mathbf{\Phi}$ a forest. An element $x \in X(F)$ is called \emph{degenerate} if there exists an $\alpha: F \longrightarrow G$ in $\mathbf{\Phi}$ and a $y \in X(G)$ with $x = \alpha^*(y)$, while $G$ has strictly fewer edges than $F$. Notice that if this is the case, the generalized Reedy structure on $\mathbf{\Phi}$ allows us to factor $\alpha$ as
\[
\xymatrix{
F \ar@{->>}[r]^\beta & H \,\,\ar@{>->}[r]^\gamma & G,
}
\]
where $\beta \in \mathbf{\Phi}^-$ and $\gamma \in \mathbf{\Phi}^+$. Therefore $x = \alpha^*(y) = \beta^*(z)$ with $z = \gamma^*(y) \in X(H)$. Thus $x \in X(F)$ is degenerate if and only if there is a nontrivial degeneracy $\beta: F \longrightarrow H$ such that $x$ is the restriction of an element in $X(H)$ along $\beta$. \par 
Clearly, when writing $X$ as a colimit of a diagram consisting of representables, we only need to take representables into account which correspond to non-degenerate elements of $X$ and it suffices to take just one in each isomorphism class. \par 

\textit{Skeletal filtration.} To set up a useful skeletal filtration, we need a notion of \emph{size} of a forest $F$, in such a way that a face of $F$ has strictly smaller size than $F$. We cannot just count vertices (as we do in $\mathbf{dSets}$), because of face inclusions like
\[
\xymatrix{
F \, \ar@{^(->}@<-.3ex>[r] & F \oplus \eta
}
\] 
and we cannot just count edges because of the face
\[
\xymatrix{
\eta \, \ar@{^(->}@<-.3ex>[r] & C_0
}
\] 
Therefore, let us define the size $|F|$ as the sum of the number of edges and the number of vertices of $F$. \par 
Let $X$ be a forest set. As noted above, $X$ can be written canonically as a colimit of representables corresponding only to non-degenerate elements. For $n \geq 0$, let $X^{(n)} \subseteq X$ be the subobject obtained as the colimit of the subdiagram of this canonical diagram consisting only of forests of size at most $n+1$. This yields an exhaustive filtration
\begin{equation*}
X^{(0)} \subseteq X^{(1)} \subseteq X^{(2)} \subseteq \cdots \quad\quad\quad \bigcup_{i=0}^\infty X^{(i)} \, = \, X
\end{equation*} 

\begin{proposition}
\label{prop:skeletalfiltration}
Let $X$ be a normal forest set. Then for each $n \geq 0$ the following diagram is a pushout:
\[
\xymatrix{
\coprod_{[e]} \partial F_e \ar[d]\ar[r] & X^{(n-1)} \ar[d] \\
\coprod_{[e]} F_e \ar[r] & X^{(n)}
}
\]
Here the coproduct ranges over all isomorphism classes of elements $e \in X^{(n)}$, corresponding to maps $e: F_e \longrightarrow X$ where $F_e$ is a forest of size exactly $n+1$. We have adopted the convention $X^{(-1)} = \varnothing$.
\end{proposition}
\begin{proof}
The forest set $X^{(0)}$ is a disjoint union of copies of $\eta$ and the diagram is clearly a pushout for $n = 0$. We proceed by induction. We'll write $P^{(0)} = X^{(0)}$ and $P^{(n)}$ (if $n > 0$) for the pushout
\[
\xymatrix{
\coprod_{[e]} \partial F_e \ar[d]\ar[r] & P^{(n-1)} \ar[d] \\
\coprod_{[e]} F_e \ar[r] & P^{(n)}
}
\]
Then it suffices to prove for each $n \geq 0$ that the evident map $P^{(n)} \longrightarrow X$ is mono. Assuming this is the case for all $k < n$ (so that $P^{(k)} = X^{(k)}$ in those cases), the fact that $P^{(n)} \longrightarrow X^{(n)}$ is also mono follows from the following two assertions:
\begin{itemize}
\item[(a)] For each $e$ as above, the diagram
\[
\xymatrix{
\partial F_e \ar[d]\ar[r] & X^{(n-1)} \ar[d] \\
F_e \ar[r] & X
}
\]
is a pullback.
\item[(b)] If $e_1: F_{e_1} \longrightarrow X$ and $e_2: F_{e_2} \longrightarrow X$ are two non-isomorphic elements of $X^{(n)}$, then 
\begin{equation*}
F_{e_1} \times_X F_{e_2} \, \subseteq \, X^{(n-1)} \times_X X^{(n-1)}
\end{equation*}
Note that the latter object is simply $X^{(n-1)}$.
\end{itemize}
To prove (a), suppose 
\[
\xymatrix{
G \ar[r]^{\alpha} & F_e \ar[r]^e & X
}
\] 
factors through $X^{(n-1)}$. Then $x = e \circ \alpha$ can also be obtained as $x = z \circ \beta$ as in
\[
\xymatrix{
G \ar[r]^\alpha\ar[d]_{\beta}\ar[dr]^x & F_e \ar[d]^{e} \\
G' \ar[r]^z & X
}
\]
where $G'$ has size strictly less than $n+1$. If $\alpha$ factors through $\partial F_e$ we are done, so we may assume $\alpha$ is surjective. Choose a section $\sigma$ of $\alpha$ and factor $\beta \circ \sigma: G \longrightarrow G'$ as
\[
\xymatrix{
F_e \ar@{->>}^\epsilon[r] & H \,\,\ar@{>->}^\delta[r] & G'
}
\]
Then
\begin{equation*}
e \, = \, e \alpha \sigma \, = \, x \sigma \, = \, z \beta \sigma = z \delta \epsilon
\end{equation*}
contradicting the fact that $e$ is non-degenerate. \par 
To prove (b), suppose $x \in X^{(n)}$ can be written in two ways, say $e_1 \alpha = x = e_2 \beta$ as in
\[
\xymatrix{
G \ar[r]^\beta\ar[dr]^x\ar[d]_\alpha & F_{e_2} \ar[d]^{e_2} \\
F_{e_1} \ar[r]_{e_1} & X
}
\]
We can assume $\alpha$ and $\beta$ are surjective, because otherwise $x \in X^{(n-1)}$ and there is nothing to prove. Choose sections $u$ of $\alpha$ and $v$ of $\beta$. Then $e_2 = e_2 \beta v = e_1 \alpha v$, so $\alpha v$ must be an isomorphism because $e_2$ is non-degenerate (and $F_{e_1}$ and $F_{e_2}$ have the same size). But then $e_1$ and $e_2$ are isomorphic, contradicting the assumption. $\Box$
\end{proof}

\begin{remark}
Since we're counting edges and vertices, the skeleta grow somewhat differently from the way they do in dendroidal sets. For example, for the corolla $C_p$ viewed as a forest set -- let us write $u_!(C_p)$ for emphasis -- we have
\begin{eqnarray*}
u_!(C_p)^{(0)} & = & \coprod_{i=0}^p \eta \\
u_!(C_p)^{(p-1)} & = & \eta \amalg p \cdot \eta \, = \, \partial(u_!(C_p)) \\
u_!(C_p)^{(p)} & = & u_!(C_p)
\end{eqnarray*}
More generally, consider any dendroidal set $X$ and form the colimit
\begin{equation*}
V \, := \, \varinjlim F_e
\end{equation*}
over all $e: F_e \longrightarrow u_!X$ where $F_e$ has no vertices. This can be a much more complicated object than just a disjoint union of copies of $\eta$, which is what we would get by forming a similar colimit in $\mathbf{dSets}$, giving the 0-skeleton of $X$ in that category. Indeed, the colimit diagram for $V$ can contain objects of the form
\begin{equation*}
p \cdot \eta \, = \, \eta \oplus \cdots \oplus \eta
\end{equation*}
and maps between them. If $X$ is normal, these maps are all monomorphisms. These monos are all obtained by pushout and composition of monos of the form
\begin{equation*}
\partial(p \cdot \eta) \longrightarrow p \cdot \eta 
\end{equation*}
as expressed by Proposition \ref{prop:skeletalfiltration}. The same need not be true if $X$ is not normal, as one sees by considering objects of the form $(\eta \oplus \eta) /\Sigma_2$, where $\Sigma_2$ acts by interchanging the two copies of $\eta$.
\end{remark}

In exactly the same way as Proposition \ref{prop:skeletalfiltration} one can prove the following:
\begin{proposition}
Let $f: X \longrightarrow Y$ be a normal monomorphism and form the relative skeleta
\begin{equation*}
Y^{(n)}_X \, = \, Y^{(n)} \cup_f X \subseteq Y
\end{equation*}
Then for each $n$, the diagram
\[
\xymatrix{
\coprod_{[e]} \partial F_e \ar[d]\ar[r] & Y^{(n-1)}_X \ar[d] \\
\coprod_{[e]} F_e \ar[r] & Y^{(n)}_X 
}
\]
is a pushout, where the coproduct ranges over isomorphism classes of non-degenerate elements $e \in Y(F_e) - X(F_e)$ and where $F_e$ has size exactly $n+1$.
\end{proposition}

\begin{corollary}
\label{cor:gennormalmonos}
The class of normal monomorphism in $\mathbf{fSets}$ is the saturation of the set of boundary inclusions $\partial F \longrightarrow F$. More specifically, every normal monomorphism is a transfinite composition of pushouts of maps of the form $\partial F \longrightarrow F$.
\end{corollary}

Applying Quillen's small object argument, we get:
\begin{corollary}
\label{cor:normalmonofactorization}
Every map $X \longrightarrow Y$ in $\mathbf{fSets}$ can be factored as $X \rightarrowtail Z \rightarrow Y$, where $X \rightarrowtail Z$ is a normal monomorphism and $Z \rightarrow Y$ has the right lifting property with respect to all normal monomorphisms.
\end{corollary}

As for dendroidal sets, this corollary leads to the following definition:

\begin{definition}
A \emph{normalization} of a forest set $Y$ is a map $Y' \longrightarrow Y$ from a normal object $Y'$, having the right lifting property with respect to all normal monomorphisms.
\end{definition}

In particular, the previous corollary shows that every forest set admits a normalization.

\subsection{Tensor products and normal monomorphisms}
\label{sec:tensorprodsnormals}

In this section we investigate the behaviour of normal monomorphisms with respect to tensor products. The arguments are of a rather technical nature; the reader might want to skip this section on first reading, only noting the following crucial result:

\begin{proposition}
\label{prop:newnormalmonopoprod}
Let $X \rightarrow Y$ and $U \rightarrow V$ be normal monomorphisms between forest sets and assume one of the following two conditions is satisfied:
\begin{itemize}
\item[(i)] Either $Y$ or $V$ is a simplicial set, i.e. is in the essential image of the functor $u_! \circ i_!: \mathbf{sSets} \rightarrow \mathbf{fSets}$.
\item[(ii)] Both $Y$ and $V$ are open forest sets.
\end{itemize}
Then the pushout-product
\begin{equation*}
X \otimes V \cup_{X \otimes U} Y \otimes U \longrightarrow Y \otimes V
\end{equation*}
is a normal monomorphism.
\end{proposition}

By standard arguments, this proposition is a consequence of the following result:

\begin{proposition}
\label{prop:newnormalmonopoprod2}
Let $F$ and $G$ be forests and assume one of the following two conditions is satisfied:
\begin{itemize}
\item[(i)] Either $F$ or $G$ is a simplex, i.e. is in the essential image of the functor $u \circ i: \mathbf{\Delta} \rightarrow \mathbf{\Phi}$.
\item[(ii)] Both $F$ and $G$ are open forests.
\end{itemize}
Then the pushout-product
\begin{equation*}
\partial F \otimes G \cup_{\partial F \otimes \partial G} F \otimes \partial G \longrightarrow F \otimes G
\end{equation*}
is a normal monomorphism.
\end{proposition}

The rest of this section is devoted to proving the previous proposition. We will, as before, suppress the functors $u_!$ and $i_!$ from the notation, simply writing $\Delta^n$ for the forest set obtained by applying $u_!$ and $i_!$ to the $n$-simplex. First, we need an easy way to establish that certain maps we encounter are monomorphisms.

\begin{definition}
An operad $\mathbf{P}$ in $\mathbf{Sets}$ is called \emph{thin} if for every tuple $(c_1, \ldots, c_n, d)$ of colours of $\mathbf{P}$, the set of operations $\mathbf{P}(c_1, \ldots, c_n; d)$ is either empty or a singleton.
\end{definition}

%add more words about tensor products of thin operads
Examples of thin operads are the operads $\Omega(T)$ freely generated by trees in $\Omega$. Observe that the class of thin operads is closed under small limits and, as a consequence of the Boardman-Vogt relation, contains tensor products of the form $\Omega(S) \otimes \Omega(T)$. We will make frequent use of the following obvious lemma:

\begin{lemma}
\label{lem:thinmono}
A map of thin operads is a monomorphism if and only if it is injective on colours.
\end{lemma}

Note that monomorphisms of operads give rise to monomorphisms of forest sets by applying the functor $u_* \circ N_d$, which we will in this section also refer to as the \emph{nerve}. To prove Proposition \ref{prop:newnormalmonopoprod2}, we need some discussion of the intersections between different faces of a forest. So, let $F$ be a forest and let $H_1$ and $H_2$ be two elementary faces of $F$. For simplicity, assume $F$ consists of only one tree (although the general case is no more difficult). There is one `exceptional' and one `generic' case to consider: \par 
\emph{Case 1}. The forest $F$ has a leaf vertex $v$ attached to an inner edge $e$ and we have $H_1 = \partial_v F$ and $H_2 = \partial_e F$. To describe their intersection, let us denote by $w$ the vertex attached to the bottom of $e$. One of the leaves of $w$ is $e$; label the others by $l_1, \ldots, l_n$ (the \emph{siblings} of $e$). Denote the outgoing edge of $w$ by $r$. Let us write $F/l_i$ for the maximal subtree of $F$ with $l_i$ as its root and $r/F$ for the tree obtained from $F$ by chopping off everything above the edge $r$. Then we have
\begin{equation*}
\partial_v F \cap \partial_e F = r/F \amalg (F/l_1 \oplus \cdots \oplus F/l_n).
\end{equation*}
In particular, this is \emph{not} a representable forest set, unless $e$ is the only leaf of the vertex $w$. \par
\emph{Case 2}. For any choice of $H_1$ and $H_2$ which is not of the type described in Case 1, the intersection of the two \emph{is} representable, i.e. is just a forest, which is simultaneously a face of $H_1$ and of $H_2$. \par 
Proposition \ref{prop:newnormalmonopoprod2} is a consequence of the following three lemmas.

\begin{lemma}
\label{lem:tensorintersection}
Let $F$ and $G$ be forests, assume $G$ is open and let $H_1$, $H_2$ be elementary faces of $F$. Then the natural map
\begin{equation*}
(H_1 \times_F H_2) \otimes G \longrightarrow (H_1 \otimes G) \times_{F \otimes G} (H_2 \otimes G)
\end{equation*}
is an isomorphism. In words, tensoring with $G$ preserves the intersection of $H_1$ and $H_2$.
\end{lemma}
\begin{proof}
For simplicity, we will use the symbol $\cap$ for intersections instead of writing pullbacks as in the statement of the lemma; this should not cause confusion. To avoid cluttering up the exposition, let us assume that both $F$ and $G$ consist of a single tree. The modifications for the general case are trivial. Also, we will assume $F$ has at least two vertices; the cases where $F$ is either $\eta$ or a corolla are trivial. Recall the discussion above about the intersection of faces. If we are in Case 2 discussed there, both the forest sets mentioned in the map above are the nerves of thin operads. It is therefore immediate from Lemma \ref{lem:thinmono} that the stated map is a monomorphism. To prove surjectivity, suppose $S$ is a shuffle of the tensor product $H_1 \otimes G$. If we can prove that the intersection $S \cap (H_2 \otimes G)$ is contained in $(H_1 \cap H_2) \otimes G$, we are done (as far as Case 2 is concerned). Let us distinguish the following possibilities:
\begin{itemize}
\item[(a)] The face $H_2$ is obtained by contracting an inner edge $e$ of $F$, which is also an inner edge of $H_1$. Then the intersection $S \cap (H_2 \otimes G)$ is the forest obtained from $S$ by contracting all the edges of the form $e \otimes g$ in $S$. Note that these are indeed inner edges and that the resulting forest is a shuffle of the tensor product $(H_1 \cap H_2) \otimes G$.
\item[(b)] The face $H_2$ is obtained by contracting an inner edge $e$ of $F$, which is \emph{not} an inner edge of $H_1$. Since we are in Case 2, this means $H_1$ must be the root face of $F$; we can then interchange the roles of $H_1$ and $H_2$ and move to (c) below.
\item[(c)] The face $H_2$ is chops off the root vertex of $F$. Let us call the root that is being deleted $r$. The shuffle $S$ has a connected subtree containing the root, containing precisely all the edges of $S$ whose colour is of the form $r \otimes g$ for some colour $g$ of $G$. By taking iterated root faces, we may delete all these edges. The intersection $S \cap (H_2 \otimes G)$ is the forest resulting from this procedure. Again, it is clear that this forest is precisely a shuffle of the tensor product $(H_1 \cap H_2) \otimes G$.
\item[(d)] The face $H_2$ chops off a leaf vertex $v$ of $F$, with leaves $l_1, \ldots, l_n$. Since we are in Case 2, $v$ is also a leaf vertex of $H_1$. The shuffle $S$ potentially contains inner edges of the form $l_i \otimes g$. First contract all these. There are then potentially leaf corollas of $S$ left with leaves of the form $l_i \otimes g$. Take the iterated outer face chopping of these leaf corollas. The intersection $S \cap (H_2 \otimes G)$ is precisely the resulting tree. This tree is a shuffle of $(H_1 \cap H_2) \otimes G$.
\end{itemize}
The reader should observe that so far we haven't used the assumption that $G$ is open. This assumption will only play a role when the choice of $H_1$ and $H_2$ is as in Case 1 above, which we will deal with now. We use the same notation introduced there, with the addition that we label the leaves of $v$ by $k_1, \ldots, k_m$ (in case $v$ has any leaves). In the case at hand, the left-hand side of the map stated in the lemma is not quite the nerve of a thin operad, but rather a coproduct of such nerves, and it is still clear that the stated map is mono. To establish surjectivity, we should argue that for any shuffle $S$ of $H_1 \otimes G$, the intersection $S \cap (H_2 \otimes G)$ splits as
\begin{equation*}
S \cap \bigl((r/F \amalg F/l_1 \oplus \cdots \oplus F/l_n) \otimes G\bigr).
\end{equation*}
This follows if we can show that there is no dendrex of $S \cap (H_2 \otimes G)$ whose root edge is of the form $r \otimes g$ and whose leaves are of the form $l_i \otimes g_i$, for colours $g, g_i$ of $G$. First assume $v$ is not nullary. By our assumption that $G$ is open, any dendrex of $H_2 \otimes G$ with a leaf of colour $l_i \otimes g_i$ must also contain a leaf of colour $k_i \otimes g'_i$ (in fact, for each $1 \leq i \leq m$). Since these colours are not in $S$, the intersection can have no such dendrex. In the case where $v$ is a nullary vertex, observe that any dendrex of $S$ with root edge of the form $r \otimes g$ and at least one leaf of the form $l_i \otimes g_i$ must also have a leaf of the form $e \otimes g'$ (again using the assumption that $G$ is open). But such an edge is not in $H_2 \otimes G$, so that such a dendrex cannot be in the intersection of $S$ with $H_2 \otimes G$. $\Box$
\end{proof}

\begin{remark}
The assumption on $G$ is necessary (cf. \cite{cisinskimoerdijkerr}). A counterexample to the statement of the lemma in the case of a non-open $G$ is the following:
\[
\begin{tikzpicture} 
[level distance=5mm, 
every node/.style={fill, circle, minimum size=.1cm, inner sep=0pt}, 
level 1/.style={sibling distance=10mm}, 
level 2/.style={sibling distance=5mm}, 
level 3/.style={sibling distance=5mm}]

%left tree
\node (lefttree)[style={color=white}] {} [grow'=up] 
child {node (Llevel1) {} 
	child{ node (Llevel2) {}
		child
	}
	child{
	}
};

%right simplex
\node (righttree)[style={color=white}, right = 3cm of lefttree] {};
\node (righttreestart)[style={color=white}, above = .4cm of righttree] {} [grow'=up] 
child {node(Rlevel1)[draw,fill=none]{} 
};

\tikzstyle{every node}=[]

%labels
\node at ($(Llevel1) + (-.25cm, .2cm)$) {$e$};
\node at ($(Llevel1) + (-.45cm, .5cm)$) {$v$};
\node at ($(Llevel1) + (-.15cm, -.25cm)$) {$r$};
\node at ($(Llevel1) + (.25cm, .2cm)$) {$f$};
\node at ($(Rlevel1) + (-.15cm, -.25cm)$) {$g$};
%\node at ($(Rlevel1) + (-1cm, -.5cm)$) {$1$};
\node at ($(Llevel1) + (-1cm, .5cm)$) {$F:$};
\node at ($(Rlevel1) + (-1cm, 0cm)$) {$G:$};

\end{tikzpicture} 
\]
Let $H_1 = \partial_v F$ and $H_2 = \partial_e F$. Then $H_1 \cap H_2$ is the disjoint union of two $\eta$'s, corresponding to the edges $r$ and $f$. Hence
\begin{equation*}
(H_1 \cap H_2) \otimes G \, \simeq \, C_0 \amalg C_0.
\end{equation*}  
On the other hand, one verifies that $H_1 \otimes G \cap H_2 \otimes G$ is the following tree:
\[
\begin{tikzpicture} 
[level distance=5mm, 
every node/.style={fill, circle, minimum size=.1cm, inner sep=0pt}, 
level 1/.style={sibling distance=10mm}, 
level 2/.style={sibling distance=5mm}, 
level 3/.style={sibling distance=5mm}]

%left tree
\node (lefttree)[style={color=white}] {} [grow'=up] 
child {node (Llevel1) {} 
	child{ node (Llevel2) {}
	}
};

\tikzstyle{every node}=[]

%labels
\node at ($(Llevel1) + (.45cm, -.25cm)$) {$r \otimes g$};
\node at ($(Llevel1) + (.45cm, .25cm)$) {$f \otimes g$};

\end{tikzpicture} 
\]
\end{remark}

\begin{lemma}
Let $F = \Delta^n$ and let $0 \leq i < j \leq n$. Then for any forest $G$, the natural map
\begin{equation*}
\partial_i\partial_j\Delta^n \otimes G \longrightarrow \partial_i\Delta^n \otimes G \cap \partial_j\Delta^n \otimes G
\end{equation*}
is an isomorphism.
\end{lemma}
\begin{proof}
Note that both forest sets appearing in the map above are nerves of thin operads. Therefore Lemma \ref{lem:thinmono} shows that the stated map is a monomorphism. To establish surjectivity, observe that for a tuple of colours
\begin{equation*}
(k_1 \otimes c_1, \ldots, k_m \otimes c_m, l \otimes d) 
\end{equation*} 
of $\Delta^n \otimes G$, there exists an operation
\begin{equation*}
(k_1 \otimes c_1, \ldots, k_m \otimes c_m) \longrightarrow l \otimes d
\end{equation*}
of (the operad underlying) $\partial_i\Delta^n \otimes G \cap \partial_j\Delta^n \otimes G$ if and only if none of the $k$'s and $l$ equal $i$ or $j$, all of the $k$'s are less than or equal to $l$ and there exists an operation
\begin{equation*}
(c_1, \ldots, c_m) \longrightarrow d
\end{equation*}
of $G$. But clearly such an operation also exists in $\partial_i\partial_j\Delta^n \otimes G$. $\Box$
\end{proof}

\begin{lemma}
Let $F$ and $G$ be forests and suppose at least one of the two is open. (In particular, this is the case if one of the two is a simplex.) Let $\partial_x F$ be a face of $F$ and $\partial_y G$ a face of $G$. Then
\begin{equation*}
\partial_x F \otimes G \cap F \otimes \partial_y G = \partial_x F \otimes \partial_y G. 
\end{equation*}
\end{lemma}
\begin{proof}
Without loss of generality, assume $F$ is open. Again, a straightforward application of Lemma \ref{lem:thinmono} shows that the natural map
\begin{equation*}
\partial_x F \otimes \partial_y G \longrightarrow \partial_x F \otimes G \cap F \otimes \partial_y G
\end{equation*}
is a monomorphism. To establish surjectivity, consider a shuffle $S$ of $\partial_x F \otimes G$. We need to show that the intersection $S \cap (F \otimes \partial_y G)$ is contained in $\partial_x F \otimes \partial_y G$. Using the same procedure as in Case 2 of the proof of Lemma \ref{lem:tensorintersection}, depending on what type of face $\partial_y G$ is, we form an associated face of $S$ (possibly of high codimension) and observe that it is a shuffle of the tensor product $\partial_x F \otimes \partial_y G$. $\Box$
\end{proof}

\begin{remark}
Again, the assumption that one of the two forests is open is necessary. A counterexample without this assumption is given by setting $F = G = C_0$, the 0-corolla.
\end{remark}

\begin{remark}
\label{rmk:tensorface}
In the proofs of the previous lemmas, we repeatedly made the following type of observation. Suppose $F$ is a forest, $S$ is a shuffle of the tensor product $\Delta^n \otimes F$ and $R$ is a face of $S$, satisfying one of the following conditions:
\begin{itemize}
\item[-] The forest $R$ does not contain any edges of the form $i \otimes e$, for some fixed colour $e$ of $F$.
\item[-] The forest $R$ does not contain any vertices of the form $i \otimes v$, for some fixed vertex $v$ of $F$.
\end{itemize}
Then $R$ is contained in $\Delta^n \otimes A$, for $A$ a face of $F$ corresponding to the relevant case above; in particular, $R$ is contained in $\Delta^n \otimes \partial F$. By a standard argument, the analogous observation holds if $\Delta^n$ is replaced by an arbitrary simplicial set.
\end{remark}

\subsection{Homotopically enriched model categories}
\label{subsec:weakenrichment}

Before establishing a model structure on the category $\mathbf{fSets}$, we will have to construct simplicial mapping objects $\mathbf{hom}(X,Y)$ between forest sets $X$ and $Y$. These mapping objects will not quite be part of a simplicial enrichment, but a slightly weaker structure. In this section we discuss the general setup of such weakly simplicial categories (or weakly enriched categories), as well as the compatibility of such a structure with a model structure, in order to facilitate later discussion.

Let $\mathcal{E}$ be a category and $\mathcal{S}$ a monoidal category. In our motivating example, $\mathcal{E}$ will be $\mathbf{fSets}$ and $\mathcal{S}$ will be $\mathbf{sSets}$. For convenience we will denote the tensor product and unit of $\mathcal{S}$ by $\times$ and $1$ respectively, although it is irrelevant whether the monoidal structure on $\mathcal{S}$ is Cartesian.

We will assume that $\mathcal{E}$ is \emph{weakly enriched, tensored and cotensored} over $\mathcal{S}$, meaning that it is equipped with functors
\begin{eqnarray*}
\mathcal{E}^{\mathrm{op}} \times \mathcal{E} \longrightarrow \mathcal{S}: && (X,Y) \longmapsto \mathbf{hom}(X,Y) \\
\mathcal{S} \times \mathcal{E} \longrightarrow \mathcal{E}: && (M,X) \longmapsto M \otimes X \\
\mathcal{S}^{\mathrm{op}} \times \mathcal{E} \longrightarrow \mathcal{E}: && (M,X) \longmapsto X^M
\end{eqnarray*}
which are adjoint in the sense that there exist isomorphisms
\begin{equation*}
\mathcal{S}(M, \mathbf{hom}(X,Y)) \simeq \mathcal{E}(M \otimes X, Y) \simeq \mathcal{E}(X, Y^M)
\end{equation*}
natural in the objects $M \in \mathcal{S}$ and $X, Y \in \mathcal{E}$. Furthermore, there should be natural associativity and unit maps
\begin{eqnarray*}
\alpha_{M,N,X}: (M \times N) \otimes X & \longrightarrow & M \otimes (N \otimes X) \\
\alpha_X: 1 \otimes X & \longrightarrow & X 
\end{eqnarray*}
satisfying the following conditions:
\begin{itemize}
\item[(i)] The map $\alpha_X: 1 \otimes X \rightarrow X$ is an isomorphism.
\item[(ii)] The following associativity diagram, using the associator of the monoidal structure of $\mathcal{S}$, commutes:
\[
\xymatrix@C=0pt{
& (L \times (M \times N)) \otimes X \ar[rr]^\alpha \ar[dl] & & L \otimes ((M \times N) \otimes X) \ar[dr]^{L \otimes \alpha} & \\
((L \times M) \times N) \otimes X \ar[drr]_{\alpha} & & & & L \otimes (M \otimes (N \otimes X)) \\
& & (L \times M) \otimes (N \otimes X) \ar[urr]_{\alpha} & & 
}
\]
\end{itemize}

In case $\mathcal{S}$ is additionally a \emph{symmetric} monoidal category, we will say that the weak enrichment of $\mathcal{E}$ is \emph{symmetric} if it is equipped with natural isomorphisms
\begin{equation*}
M \otimes (N \otimes X) \longrightarrow N \otimes (M \otimes X)
\end{equation*}
which make the following diagram, involving the symmetry of the monoidal structure on $\mathcal{S}$, commute:
\[
\xymatrix{
(M \times N) \otimes X \ar[r]^\alpha \ar[d] & M \otimes (N \otimes X) \ar[d] \\
(N \times M) \otimes X \ar[r]_\alpha & N \otimes (M \otimes X).
}
\]

We will now describe our main examples of weakly enriched categories, namely $\mathbf{fSets}$ and $\mathbf{dSets}$. Recall that we have an embedding
\[
\xymatrix{
\mathbf{sSets} \ar[r]^-{i_!} & \mathbf{dSets}.
}
\]
Using the tensor product of dendroidal sets, we can then define the following functor:
\begin{equation*}
- \otimes - : \mathbf{sSets} \times \mathbf{dSets} \longrightarrow \mathbf{dSets}: (M, X) \longmapsto i_!(M) \otimes X.
\end{equation*}
This functor preserves colimits in each variable separately. We can then define $\mathbf{hom}(X,Y)$ and $X^M$, for dendroidal sets $X,Y$ and a simplicial set $M$, by the adjointness formulas above. Next, we should define maps
\begin{equation*}
\alpha_{M,N,X}: (M \times N) \otimes X \rightarrow M \otimes (N \otimes X).
\end{equation*}
Let us assume $M$, $N$ and $X$ are representable; the general case then follows by extending by colimits. So, set $M = \Delta^m$, $N = \Delta^n$ and $X = T$. These objects are the dendroidal nerves of the operads $[m]$, $[n]$ and $\Omega(T)$ respectively. Moreover, there is a natural isomorphism
\begin{equation*}
\Delta^m \otimes (\Delta^n \otimes T) \simeq N_d([m] \otimes ([n] \otimes \Omega(T))),
\end{equation*}
where the tensor product on the right is the Boardman-Vogt tensor product of operads. By adjunction, supplying a map 
\begin{equation*}
(\Delta^m \times \Delta^n) \otimes T \longrightarrow N_d([m] \otimes ([n] \otimes \Omega(T)))
\end{equation*}
is equivalent to supplying a map
\begin{equation*}
\tau_d((\Delta^m \times \Delta^n) \otimes T) \longrightarrow [m] \otimes ([n] \otimes \Omega(T)).
\end{equation*}
But on the left-hand side $\tau_d$ distributes over the tensor product, so that this expression is naturally isomorphic to
\begin{equation*}
(\tau_d(\Delta^m) \otimes \tau_d(\Delta^n)) \otimes \tau_d(T) = ([m] \otimes [n]) \otimes \Omega(T).
\end{equation*}
By associativity of the tensor product of operads, a natural map to $[m] \otimes ([n] \otimes \Omega(T))$ exists (and is in fact an isomorphism).

\begin{remark}
The map $\alpha_{M,N,X}$ is typically \emph{not} an isomorphism. A more elaborate discussion of the tensor product of dendroidal sets and its associativity properties is given in Section \ref{subsec:unbiased}.
\end{remark}

Furthermore, there is an evident natural isomorphism
\begin{equation*}
\alpha_X: 1 \otimes X \longrightarrow X.
\end{equation*}

Using the distributivity of the tensor product over direct sums, we can make completely analogous definitions with the category of dendroidal sets replaced by that of forest sets. The proof of the following proposition is a straightforward verification given the definitions above; we leave the details to the reader.

\begin{proposition}
\label{prop:weakenrichment}
With the structure described above, the categories $\mathbf{dSets}$ and $\mathbf{fSets}$ become weakly enriched, tensored and cotensored over $\mathbf{sSets}$. Furthermore, these weak enrichments are symmetric.
\end{proposition}

Now suppose $\mathcal{E}$ is a model category, which is weakly tensored, cotensored and enriched over a monoidal model category $\mathcal{S}$. To finish this section, we describe how the weak enrichment can interact with these model structures. First we introduce an analogue of the usual axiom of enriched model categories.

\begin{definition}
Under the assumptions above, we say that the weak enrichment of $\mathcal{E}$ over $\mathcal{S}$ satisfies axiom (H1) if, for any cofibrations $i: M \rightarrow N$ in $\mathcal{S}$ and $j: X \rightarrow Y$ in $\mathcal{E}$, the pushout-product
\begin{equation*}
M \otimes Y \cup_{M \otimes X} N \otimes X \longrightarrow N \otimes Y
\end{equation*}
is a cofibration in $\mathcal{E}$, which is trivial if either $i$ or $j$ is trivial. 
%Furthermore, we will say $\mathcal{E}$ is \emph{strongly homotopically enriched} if in addition the following axiom is satisfied: for cofibrations $M \rightarrow M'$, $N \rightarrow N'$ in $\mathcal{S}$ and a cofibration $X \rightarrow Y$ in $\mathcal{E}$, the natural map
%\begin{equation*}
%(M' \times N') \otimes Y \cup M' \otimes (N \otimes Y) \cup M \otimes (N' \otimes Y) \cup M' \otimes (N' \otimes X) \longrightarrow M \otimes (N \otimes Y)
%\end{equation*}
%is a trivial cofibration in $\mathcal{E}$.
\end{definition}

There is one further axiom we might impose. Consider cofibrations $i: M \rightarrow M'$, $j: N \rightarrow N'$ in $\mathcal{S}$ and a cofibration $k: X \rightarrow X'$ in $\mathcal{E}$. For the sake of brevity, let us write $i \, \hat{\times} \, j$ for the pushout-product of the maps $i$ and $j$ with respect to the monoidal structure of $\mathcal{S}$ and $j \, \hat{\otimes} \, k$ for the pushout-product of $j$ and $k$ with respect to the (weak) tensoring of $\mathcal{E}$ over $\mathcal{S}$, as in the previous definition. Now write $f: A \rightarrow B$ for the map
\begin{equation*}
(i \, \hat{\times} \, j) \, \hat{\otimes} \, k
\end{equation*}
and similarly write $g: C \rightarrow D$ for the map
\begin{equation*}
i \, \hat{\otimes} \, (j \, \hat{\otimes} \, k).
\end{equation*}
Note that $\alpha$ induces maps $A \rightarrow C$ and $B \rightarrow D$, which give a commutative diagram
\[
\xymatrix{
A \ar[d]_f \ar[r] & C \ar[d] \ar@/^/[ddr]^g & \\
B \ar[r]\ar@/_/[drr] & B \cup_A C \ar[dr]^h & \\
& & D.
}
\]

\begin{definition}
\label{def:stronghomotenrich}
The weak enrichment of $\mathcal{E}$ over $\mathcal{S}$ satisfies axiom (H2) if, for any choice of cofibrations $i$, $j$ and $k$ as above, the map 
\begin{equation*}
h: B \cup_A C \longrightarrow D
\end{equation*}
as just constructed is a trivial cofibration. We say that $\mathcal{E}$ is \emph{homotopically enriched} over $\mathcal{S}$ if it satisfies (H1) and (H2).
\end{definition}

If $\mathcal{E}$ is homotopically enriched over $\mathcal{S}$, then by taking the domains of $i$, $j$ and $k$ to be initial objects we see that a map of the form
\begin{equation*}
\alpha_{M,N,X}: (M \times N) \otimes X \longrightarrow M \otimes (N \otimes X)
\end{equation*}
is a trivial cofibration for any choice of cofibrant objects $M$, $N$ and $X$.

Let us note a straightforward consequence of our definitions.  For objects $X, Y \in \mathcal{E}$ and $M \in \mathcal{S}$, there is a natural map
\begin{equation*}
\beta_{M,X,Y}: \mathbf{hom}(X,Y^M) \longrightarrow \mathbf{hom}(X,Y)^M.
\end{equation*}
Indeed, for $N \in \mathcal{S}$, the map $\alpha_{M,N,X}$ allows us to form the sequence of maps
\begin{equation*}
\mathcal{S}(N, \mathbf{hom}(X,Y^M)) \simeq \mathcal{E}(M \otimes (N \otimes X), Y) \longrightarrow \mathcal{E}((M \times N) \otimes X, Y) \simeq \mathcal{S}(N, \mathbf{hom}(X,Y)^M),
\end{equation*}
which, by the Yoneda lemma, defines the map $\beta_{M,X,Y}$.

\begin{lemma}
\label{lem:exphomotenriched}
Suppose the weak enrichment of $\mathcal{E}$ over $\mathcal{S}$ satisfies axiom (H2). If $M \in \mathcal{S}$ and $X \in \mathcal{E}$ are cofibrant and $Y \in \mathcal{E}$ is fibrant, then $\beta_{M,X,Y}$ is a trivial fibration. 
\end{lemma}
\begin{proof}
For $K \rightarrow L$ a cofibration in $\mathcal{S}$, the lifting problem
\[
\xymatrix{
K \ar[d]\ar[r] & \mathbf{hom}(X,Y^M) \ar[d] \\
L \ar[r]\ar@{-->}[ur] & \mathbf{hom}(X,Y)^M
}
\]
is equivalent to the lifting problem
\[
\xymatrix{
M \otimes (K \otimes X) \cup_{(M \times K) \otimes X} (M \times L) \otimes X \ar[d] \ar[r] & Y. \\
M \otimes (L \otimes X) \ar@{-->}[ur]
}
\]
The latter admits a solution since $Y$ is fibrant and the left vertical map is a trivial cofibration by assumption. $\Box$
\end{proof}

\subsection{Tensor products, inner horns and Segal cores}

% new section:
% 3.6: Tensor products, inner horns and Segal cores
% 3.7.11 (pushout-product inner horns) & 3.7.20 (Segal cores)
% \alpha is in the `two-out-of-three'-class containing the Segal cores

In the next section we will establish a model structure on the category of forest sets. Before we can do so, we need to understand the behaviour of \emph{inner horn inclusions} with respect to tensor products. Also, we will investigate \emph{Segal cores} and their relation to inner horns. We will later need these Segal cores to obtain a convenient description of the trivial cofibrations between forest sets.

To begin with, let us be more precise about these inner horns. Recall for a forest $F$ its boundary $\partial F \rightarrowtail F$, the union of all its faces, as well as the fact that this operation satisfies a `derivation rule':
\begin{equation*}
\partial(F \oplus G) \, = \, \partial F \oplus G \cup F \oplus \partial G
\end{equation*}
If $e$ is an inner edge of $F$, i.e. an inner edge in one of the constituent trees of $F$, then $\Lambda^e[F]$ (the \emph{inner horn} associated to $e$) is defined to be the union of all the faces of $F$ \emph{except} the one given by contraction of $e$, or equivalently, the union of all the faces whose image contains the edge $e$. Notice that for an inner edge in such a forest $F$, one has the identity
\begin{equation*}
\Lambda^e[F \oplus G] \, = \, \Lambda^e[F] \oplus G \cup F \oplus \partial G
\end{equation*}
In particular, if the inner edge $e$ lies in the tree $T$ where $F \simeq uT \oplus G$, then
\begin{equation}
\label{eq:innerhornforest}
\Lambda^e[F] \, = \, \Lambda^e[uT] \oplus G \cup uT \oplus \partial G
\end{equation}

We will call a map of forest sets \emph{inner anodyne} if it can be written as composition of pushouts of inner horn inclusions. By induction over skeleta, we can immediately conclude the following:

\begin{corollary}
\label{cor:posuminneranod}
For any normal monomorphism $A \longrightarrow B$ of forest sets, the map
\begin{equation*}
\Lambda^e[F] \oplus B \cup F \oplus A \longrightarrow F \oplus B
\end{equation*}
is an inner anodyne map. In particular, for a normal forest set $A$ the map $\Lambda^e[F] \oplus A \longrightarrow F \oplus A$ is again inner anodyne. More generally, for any inner anodyne map $C \longrightarrow D$, the map
\begin{equation*}
C \oplus B \cup D \oplus A \longrightarrow D \oplus B
\end{equation*}
is inner anodyne again.
\end{corollary}

We should emphasize that if $F = uT$ consists of a single tree, this `forestial' inner horn is generally larger than the `dendroidal' one, because of the extra root face (cf. Section \ref{sec:normalmonos}). In general, we have
\begin{equation*}
\Lambda^e[uT] \, = \, \mathrm{Im}(u_!(\Lambda^e[T])) \cup \partial_{\mathrm{root}}(uT)
\end{equation*}
where $\mathrm{Im}$ denotes the image as a subpresheaf of $uT$, while $\Lambda^e[uT] = \mathrm{Im}(u_!(\Lambda^e[T]))$ only if the root vertex in $T$ is unary. \par 

The statements in Propositions \ref{prop:poprodinneranodyne} and \ref{prop:Segcoregeninneran} below are analogues of basic facts about dendroidal sets, cf. \cite{cisinskimoerdijk2, moerdijkweiss}. However, their proofs do not carry over to the present setting, because of the difference between the boundary of a tree in $\mathbf{dSets}$ and its boundary in $\mathbf{fSets}$.

\begin{proposition}
\label{prop:poprodinneranodyne} 
Let $F$ and $G$ be two forests. Suppose that one of them is a simplex or both are open. For an inner edge $e$ in $F$, the map
\begin{equation*}
\Lambda^e[F] \otimes G \cup F \otimes \partial G \longrightarrow F \otimes G
\end{equation*}
is inner anodyne.
\end{proposition}

The proof of this proposition requires a fair amount of combinatorics; we will first set up some terminology.

\begin{definition}
\label{def:pruning}
Let $T$ be a tree. A \emph{pruning} of $T$ is a subtree $P \subseteq T$ such that the root of $P$ coincides with the root of $T$ and so that the inclusion map of $P$ into $T$ can be written as a composition of outer face maps. In other words, $P$ is obtained from $T$ by iteratively chopping off leaf corollas.
\end{definition}

\begin{proof}[Proof of Proposition \ref{prop:poprodinneranodyne}]
Observe first that by Proposition \ref{prop:newnormalmonopoprod}, the map in the statement of the proposition is a normal monomorphism. Let us first prove this proposition in the case where $F$ and $G$ are just trees, say $S$ and $T$ respectively. The vertices of the constituent shuffles of the tensor product $S \otimes T$ are all of the form $v \otimes t$ or $s \otimes w$, where $v$ (resp. $w$) is a vertex of $S$ (resp. $T$) and $s$ (resp. $t$) is a colour of $S$ (resp. $T$). We will loosely refer to vertices of the first kind as `vertices of $S$' and vertices of the second kind as `vertices of $T$'. Throughout this proof we will draw vertices of $S$ as being black and vertices of $T$ white, as follows:

\[
\begin{tikzpicture} 
[level distance=10mm, 
every node/.style={fill, circle, minimum size=.1cm, inner sep=0pt}, 
level 1/.style={sibling distance=20mm}, 
level 2/.style={sibling distance=20mm}, 
level 3/.style={sibling distance=5mm}]

%vertex of S
\node(anchorS)[style={color=white}] {} [grow'=up] 
child {node(vertexS) {} 
	child
	child
};

%vertex of T
\node[style={color=white}, right=4cm of anchorS] {} [grow'=up] 
child {node(vertexT)[draw,fill=none] {} 
	child
	child
};

\tikzstyle{every node}=[]

%labels
%vertex of S
\node at ($(vertexS) + (-1.5cm,0)$) {$\text{vertex of } S$};
\node at ($(vertexS) + (0,1cm)$) {$\vdots$};
\node at ($(vertexS) + (0,-1.2cm)$) {$\vdots$};

%vertex of T
\node at ($(vertexT) + (-1.5cm,0)$) {$\text{vertex of } T$};
\node at ($(vertexT) + (0,1cm)$) {$\vdots$};
\node at ($(vertexT) + (0,-1.2cm)$) {$\vdots$};

\end{tikzpicture} 
\]

The set of shuffles of the tensor product $S \otimes T$ has a natural partial ordering in which the smallest element is the shuffle given by grafting copies of $T$ onto the leaves of $S$:

\[
\begin{tikzpicture} 
[level distance=15mm, 
every node/.style={fill, circle, minimum size=.1cm, inner sep=0pt}, 
level 1/.style={sibling distance=35mm}, 
level 2/.style={sibling distance=30mm}, 
level 3/.style={sibling distance=20mm}]

%vertex of S
\node(anchorS)[style={color=white}] {} [grow'=up] 
child {node(vertexS) {} 
	child{ node(vertexT1)[draw,fill=none] {}
		child
		child
	}
	child{ node(vertexT2)[draw,fill=none] {}
		child
		child
	}
};

\tikzstyle{every node}=[]

%labels
%S
\node at ($(vertexS) + (0,.7cm)$) {$S$};

%T
\node at ($(vertexT1) + (0,.9cm)$) {$T$};
\node at ($(vertexT2) + (0,.9cm)$) {$T$};
\node at ($(vertexS) + (0, 1.5cm)$) {$\cdots$};

%lines
\draw ($(vertexT1) + (.3cm,-.3cm)$) -- ($(vertexT2) + (-.3cm,-.3cm)$);
\draw ($(vertexT1) + (-1.03cm,1.5cm)$) -- ($(vertexT1) + (1.03cm,1.5cm)$);
\draw ($(vertexT2) + (-1.03cm,1.5cm)$) -- ($(vertexT2) + (1.03cm,1.5cm)$);

\end{tikzpicture} 
\]

If a shuffle $R_2$ is obtained from another shuffle $R_1$ by percolating a vertex of $T$ down through a vertex of $S$, as in the picture below, then $R_1 < R_2$ in this partial order;

\[
\begin{tikzpicture} 
[level distance=10mm, 
every node/.style={fill, circle, minimum size=.1cm, inner sep=0pt}, 
level 1/.style={sibling distance=20mm}, 
level 2/.style={sibling distance=20mm}, 
level 3/.style={sibling distance=10mm}]

%R_1
\node(anchorR1)[style={color=white}] {} [grow'=up] 
child {node(vertexR1) {} 
	child{node[draw,fill=none] {}
		child
		child
 	}
	child{node[draw,fill=none] {}
		child
		child
 	}
};

%R_2
\node[style={color=white}, right=5cm of anchorR1] {} [grow'=up] 
child {node(vertexR2)[draw,fill=none] {} 
	child{node {}
		child
		child
 	}
	child{node {}
		child
		child
 	}
};

\tikzstyle{every node}=[]

%labels
%R1
\node at ($(vertexR1) + (-1.5cm,.5cm)$) {$R_1$};
\node at ($(vertexR1) + (0,2.4cm)$) {$\vdots$};
\node at ($(vertexR1) + (0,-1.2cm)$) {$\vdots$};
\node at ($(vertexR1) + (2.5cm,.5cm)$) {$\longrightarrow$};

%R2
\node at ($(vertexR2) + (-1.5cm,.5cm)$) {$R_2$};
\node at ($(vertexR2) + (0,2.4cm)$) {$\vdots$};
\node at ($(vertexR2) + (0,-1.2cm)$) {$\vdots$};

\end{tikzpicture} 
\]

Now let $v_e$ denote the bottom vertex attached to the inner edge $e$ in $S$. The shuffle $R$ will contain one or several vertices of the form $v_e \otimes t$, where $t$ is a colour of $T$. Each such vertex has a leaf (or incoming edge) $e \otimes t$ and we will refer to these edges of $R$ as \emph{special edges}.
\par  
With all this terminology set up, we can begin our induction. Define
\begin{equation*}
A_0 \, := \, \Lambda^e[uS] \otimes uT \cup uS \otimes \partial(uT)
\end{equation*}
Choose a linear ordering on the set of shuffles of $S \otimes T$ extending the partial order described above. By adjoining these shuffles one by one, we obtain a filtration
\begin{equation*}
A_0 \subseteq A_1 \subseteq \cdots \subseteq \bigcup_{i} A_i = uS \otimes uT
\end{equation*}
We will show that each of the inclusions in this filtration is inner anodyne. Say $A_{i+1}$ is obtained from $A_i$ by adjoining a shuffle $R$. Define a further filtration
\begin{equation*}
A_i \, =: \, A_i^0 \subseteq A_i^1 \subseteq \cdots \subseteq \bigcup_{j} A_i^j \, = \, A_{i+1}
\end{equation*}
by adjoining all prunings of $R$ one by one, in an order that extends the partial order of size (i.e. number of vertices) of prunings. Consider an inclusion $A_i^j \subseteq A_i^{j+1}$ given by adjoining a pruning $P$ of $R$. 
Let $\Sigma_P$ denote the intersection of the set of special edges of $R$ with the set of inner edges of $P$. We may assume this intersection is non-empty, because otherwise $P$ is already contained in $A_0$. Define
\begin{equation*}
\mathcal{H}_{P} \, := \, I(P) - \Sigma_{P}
\end{equation*} 
where $I(P)$ denotes the set of inner edges of $P$. For each subset $H \subseteq \mathcal{H}_{P}$, define the tree $P^{[H]}$ as the tree obtained from $P$ by contracting all edges in $\mathcal{H}_P - H$. Pick a linear order on the subsets of $\mathcal{H}_P$ extending the partial order of inclusion and adjoin the trees $P^{[H]}$ to $A_i^j$ in this order to obtain a filtration
\begin{equation*}
A_i^j \, =: \, A_i^{j,0} \subseteq A_i^{j,1} \subseteq \cdots \subseteq \bigcup_k A_i^{j,k} \, = \, A_i^{j+1}
\end{equation*}
Finally, consider one of the inclusions $A_i^{j,k} \subseteq A_i^{j,k+1}$ in this filtration, given by adjoining a tree $P^{[H]}$. If the map
\begin{equation*}
P^{[H]} \longrightarrow uS \otimes uT
\end{equation*}
factors through $A_i^{j,k}$, then the inclusion under consideration is the identity and there is nothing to prove. If it doesn't, we can say the following:
\begin{itemize}
\item[-] Any outer face chopping off a leaf corolla factors through $A_i^j$ by our induction on the size of the prunings.
\item[-] The outer face chopping off the root of $P^{[H]}$ factors through $A_0$.
\item[-] An inner face contracting an edge that is not special (i.e. not contained in $\Sigma_P$) factors through $A_i^{j,k}$ by our induction on the size of $H$.
\item[-] An inner face contracting a special edge, or a composition of inner faces contracting several special edges, cannot factor through an earlier stage of the filtration. Indeed, it cannot factor through an earlier shuffle by the way special edges are defined. Given this, it is clear that it also cannot factor through $A_i^{j'}$ for $j' \leq j$ because of the size of the pruning $P$ under consideration or through $A_i^{j,k'}$ for $k' \leq k$ by the definition of the $P^{[H']}$. 
\end{itemize}
We conclude that the map $A_i^{j,k} \subseteq A_i^{j,k+1}$ is a pushout of the map
\begin{equation*}
\Lambda^{\Sigma_P}[uP^{[H]}] \longrightarrow uP^{[H]} 
\end{equation*}
where $\Lambda^{\Sigma_P}[uP^{[H]}]$ denotes the union of all the faces of the tree $uP^{[H]}$ except for the inner faces contracting edges in $\Sigma_P$. This map is easily seen to be a composition of pushouts of inner horn inclusions (cf. Lemma \ref{lem:moreinneranodynes}(b) below), which finishes the proof in this case, i.e. under the assumption that $F$ and $G$ are trees $S$ and $T$. \par 
Let us now show how to remove the restrictions on $F$ and $G$. First, let $F$ be a forest with an inner edge $e$, but let us still assume that $G$ equals a tree $T$. One verifies that the proof given above carries over to this case verbatim, using the fact that tensor products distribute over direct sums, with the caveat that a pruning of a forest is now a subforest obtained by iteratively chopping off leaf corollas of the constituent trees. The only phrase that needs altering is the second item in the last step, which becomes:
\begin{itemize}
\item[-] An outer face chopping off the root of any of the constituent trees of $P^{[H]}$ factors through $A_0$.
\end{itemize}
Notice that this observation \emph{does} need the fact that $G$ is still just a single tree $T$. If it is not, this assertion could be false. \par 
Now assume $F$ is as above and $G$ is a forest. Write $G = T \oplus G_1$ by arbitrarily singling out a tree $T$ in this forest. By what was proved above, the map
\begin{equation*}
\Lambda^e[F] \otimes T \cup F \otimes \partial T \longrightarrow F \otimes T 
\end{equation*}
is inner anodyne. Let us use the abbreviations
\begin{eqnarray*}
A & := & \Lambda^e[F] \otimes G \cup F \otimes \partial G \\
A_T & := & \Lambda^e[F] \otimes T \cup F \otimes \partial T \\
A_1 & := & \Lambda^e[F] \otimes G_1 \cup F \otimes \partial G_1
\end{eqnarray*}
Now observe that (invoking Corollary \ref{cor:posuminneranod}) all maps in the following pushout square are inner anodyne:
\[
\xymatrix{
\Lambda^e[F] \otimes (T \oplus G_1) \ar[d]\ar[r] & (F \otimes T) \oplus (\Lambda^e[F] \otimes G_1) \ar[d] \\
(\Lambda^e[F] \otimes T) \oplus (F \otimes G_1) \ar[r] & (\Lambda^e[F] \otimes T) \oplus (F \otimes G_1) \cup (F \otimes T) \oplus (\Lambda^e[F] \otimes G_1)
}
\]
Now we push out this square along the inclusion $\Lambda^e[F] \otimes (T \oplus G_1) \rightarrow A$ to obtain the following pushout square consisting of inner anodyne maps:
\[
\xymatrix{
A \ar[r]\ar[d] & A \cup \bigl((\Lambda^e[F] \otimes T) \oplus (F \otimes G_1) \bigr) \ar[d] \\
A \cup \bigl((F \otimes T) \oplus (\Lambda^e[F] \otimes G_1)\bigr) \ar[r] & \bigl(A_T \oplus (F \otimes G_1)\bigr) \cup \bigl((F \otimes T) \oplus A_1\bigr)
}
\]
In particular, the map
\begin{equation*}
A \longrightarrow \bigl(A_T \oplus (F \otimes G_1)\bigr) \cup \bigl((F \otimes T) \oplus A_1\bigr)
\end{equation*}
is inner anodyne. Since $A_T \longrightarrow F \otimes T$ is inner anodyne and $A_1 \longrightarrow F \otimes G_1$ is a cofibration, we may use Corollary \ref{cor:posuminneranod} again to conclude that 
\begin{equation*}
\bigl(A_T \oplus (F \otimes G_1)\bigr) \cup \bigl((F \otimes T) \oplus A_1\bigr) \longrightarrow (F \otimes T) \oplus (F \otimes G_1) \, =  \, F \otimes G
\end{equation*}
is inner anodyne, which also leads us to conclude that
\begin{equation*}
A \longrightarrow F \otimes G
\end{equation*}
is inner anodyne. $\Box$
\end{proof}

We now wish to give a more efficient description of the class of inner anodyne maps using the notion of \emph{Segal core}. Recall from \cite{cisinskimoerdijk2} that for a tree $T$, its Segal core $\mathrm{Sc}(T) \rightarrowtail T$ in the category $\mathbf{dSets}$ is the union of all the corollas contained in $T$. Its analogue for forests is the following:

\begin{definition}
Let $F$ be a forest. Its \emph{(forest) Segal core} 
\begin{equation*}
\mathrm{fSc}(F) \rightarrowtail F
\end{equation*}
is the colimit over all embeddings $G \rightarrowtail F$ of subforests (i.e. compositions of outer face maps) whose constituent trees all have at most one vertex. (In other words, the trees in $G$ are all either a copy of the unit tree $\eta$ or a corolla.)
\end{definition}

\begin{remark}
\label{rmk:Segalcore}
(a). As an example, consider the following tree $T$:
\[
\begin{tikzpicture} 
[level distance=10mm, 
every node/.style={fill, circle, minimum size=.1cm, inner sep=0pt}, 
level 1/.style={sibling distance=20mm}, 
level 2/.style={sibling distance=15mm}, 
level 3/.style={sibling distance=5mm}]

\node (anchor)[style={color=white}] {} [grow'=up] 
child {node (r) {} 
	child {node (p) {}
		child
		child
		child
		child
	}
	child {node (q) {}
		child
		child
		child
	}
};

\tikzstyle{every node}=[]

%labels
\node at ($(r) + (.2cm, -.05cm)$) {$r$};
\node at ($(r) + (-.65cm, .5cm)$) {$a$};
\node at ($(r) + (.65cm, .5cm)$) {$b$};
\node at ($(q) + (.2cm, -.05cm)$) {$q$};
\node at ($(p) + (-.2cm, -.05cm)$) {$p$};

\end{tikzpicture} 
\]

Then its dendroidal Segal core is
\begin{equation*}
C_2 \cup C_3 \cup C_4 \rightarrowtail T
\end{equation*}
where the union is the pushout under the copy of $\eta$ corresponding to the edge $a$, respectively $b$. The forest Segal core of $T$ is
\begin{equation*}
C_2 \cup_{\eta \oplus \eta} (C_3 \oplus C_4) \rightarrowtail uT
\end{equation*}
(b). The obvious formula
\begin{equation*}
\mathrm{fSc}(F \oplus G) = \mathrm{fSc}(F) \oplus \mathrm{fSc}(G)
\end{equation*}
reduces the calculation of Segal cores to trees. \par 
(c). For a tree $T$, one has the following inductive formulas for its Segal core:
\begin{eqnarray*}
\mathrm{fSc}(\eta) & = & \eta \\
\mathrm{fSc}(C_p) & = & C_p \\
\mathrm{fSc}(C_p \star (T_1, \ldots, T_p)) & = & C_p \cup_{p \cdot \eta} \bigl(\mathrm{fSc}(T_1) \oplus \cdots \oplus \mathrm{fSc}(T_p)\bigr)
\end{eqnarray*}
Here $C_p$ is the corolla with $p$ leaves and $C_p \star (T_1, \ldots, T_p)$ is the tree obtained by gluing the trees $T_1, \ldots, T_p$ onto the leaves of this corolla, while $p \cdot \eta$ denotes the forest $\eta \oplus \cdots \oplus \eta$, as before.
\end{remark}

\begin{proposition}
\label{prop:Segalcoreinnanod}
For any forest $F$, the inclusion $\mathrm{fSc}(F) \rightarrowtail F$ is inner anodyne (cf. \cite{cisinskimoerdijk2} for the dendroidal case).
\end{proposition}

For the proof of this proposition we need a few simple observations concerning faces and boundaries in $\mathbf{fSets}$. These are analogues of similar facts in the dendroidal case, cf. \cite{cisinskimoerdijk2}.

\begin{lemma}
\label{lem:moreinneranodynes}
Let $T$ be a tree and let $A$ be a non-empty set of inner edges of $T$. Write $\Lambda^A[uT]$ for the union of all the faces of $uT$ \emph{except} the ones given by contraction of an edge in $A$.
\begin{itemize}
\item[(a)] For any other inner edge $e \notin A$,
\begin{equation*}
\partial_e(uT) \cap \Lambda^{A \cup \{e\}}[uT] \, = \, \Lambda^A \partial_e(uT)
\end{equation*}
\item[(b)] The map $\Lambda^A[uT] \rightarrowtail uT$ is inner anodyne.
\item[(c)] For any tree $T$ with at least one inner edge, the inclusion
\begin{equation*}
\partial^{\mathrm{ext}}(uT) \rightarrowtail uT
\end{equation*}
of the union of all external faces is inner anodyne.
\end{itemize}
\end{lemma}
\begin{proof}
(a). Clearly $\Lambda^A\partial_e(uT) \subseteq \partial_e(uT) \cap \Lambda^A[uT]$. For the reverse inclusion, one checks the different kinds of faces $\partial_x(uT)$ involved in forming $\Lambda^A[uT]$. If $x$ is an internal edge other than $e$, or an external (leaf or root) vertex not attached to $e$, then
\begin{equation*}
\partial_e(uT) \cap \partial_x(uT) \, = \, \partial_x\partial_e(uT)
\end{equation*}
If $e$ is attached to the root, write $T_d$ for the subtree of $T$ with root edge $d$ (the tree `above' $d$), where $d$ is any input edge of the root vertex in $T$. Then
\begin{eqnarray*}
\partial_{\mathrm{root}}(uT) \cap \partial_e(uT) & = & \partial_{\mathrm{root}}(uT_e) \oplus \bigl(\bigoplus_{d \neq e} uT_d\bigr) \\
& = & \partial_{\mathrm{root}}\partial_e(uT)
\end{eqnarray*}
Finally, if $x$ is a leaf vertex of the tree $T$ and $e$ is attached to $x$, then $\partial_x(uT) \cap \partial_e(uT)$ is more complicated; but any forest contained in it will also be contained in $\partial_{\mathrm{root}}\partial_e(uT)$ or in $\partial_y\partial_e(uT)$ for some leaf vertex $y$ in $T$ other than $x$, so $\partial_x(uT) \cap \partial_e(uT) \subseteq \Lambda^A\partial_e(uT)$ as well. \par 
(b). From part (a) we conclude that for $B = A \cup \{e\}$ one has a pushout diagram
\[
\xymatrix{
\Lambda^A\partial_e(uT) \ar[d]\ar[r] & \Lambda^B[uT] \ar[d] \\
\partial_e(uT) \ar[r] & \Lambda^A[uT]
}
\]
and fact (b) follows by induction on the size of $A$, the case where $A$ has one element being true by definition. \par 
(c). This is the special case of (b) where $A$ is the set of all inner edges. $\Box$
\end{proof}

\begin{proof}[Proof of Proposition \ref{prop:Segalcoreinnanod}]
By Remark \ref{rmk:Segalcore}(b) and Corollary \ref{cor:posuminneranod} above, it suffices to check this for trees. Notice that $\mathrm{fSc}(uT) \rightarrowtail uT$ is an isomorphism if $T$ has at most one vertex and an inner horn if $T$ has two vertices. For larger $T$ we consider \emph{subforests} of $T$, i.e. maps $F \rightarrowtail T$ obtained as a composition of external faces. Write $A_{n,k}$ for the union of all subforests with at most $k$ vertices, in which every constituent tree has at most $n$ vertices. Write
\begin{equation*}
A_n \, = \, \bigcup_{k \geq 0} A_{n,k}
\end{equation*}
This is in fact a finite union of course, bounded by the number $N$ of vertices in $T$. Also $A_{n,k} \subseteq A_{n-1}$ if $k<n$ and
\begin{equation*}
A_1 \, = \, \mathrm{fSc}(uT) \rightarrowtail uT
\end{equation*}
while
\begin{equation*}
A_{N-1} \, = \, \partial^{\mathrm{ext}}(uT) \rightarrowtail uT
\end{equation*}
which is inner anodyne by the previous Lemma. So it suffices to prove by induction 
\begin{equation*}
A_{n-1} \cup A_{n,k} \rightarrowtail A_{n-1} \cup A_{n,k+1}
\end{equation*}
is inner anodyne. Let $F_0, \ldots, F_p$ be all the subforests with exactly $k+1$ vertices, in which every tree has at most $n$ vertices and in which at least one tree has exactly $n$ vertices. Write
\begin{equation*}
S_j \, = \, F_0 \cup \cdots \cup F_j \quad\quad\quad (j = 0, \ldots, p)
\end{equation*}
and write $A = A_{n-1} \cup A_{n,k}$ just for now. We claim that each of the maps
\begin{equation*}
A \rightarrow A \cup S_0 \rightarrow A \cup S_1 \rightarrow \cdots \rightarrow A \cup S_p = A_{k-1} \cup A_{n,k+1}
\end{equation*}
is inner anodyne. Indeed, $A \rightarrow A \cup S_0$ is a pushout of $A \cap S_0 \rightarrow S_0$, i.e. of $A \cap F_0 \rightarrow F_0$, and
\begin{eqnarray*}
A \cap F_0 & = & (A_{n-1} \cap F_0) \cup (A_{n,k} \cap F_0) \\
& = & \partial^{\mathrm{ext}}(F_0)
\end{eqnarray*}
because $A_{n-1} \cap F_0$ is contained in $\partial^{\mathrm{ext}}(F_0)$. Similarly $A \cup S_{j-1} \rightarrow A \cup S_j$ is a pushout of $(A \cup S_{j-1}) \cap F_j \rightarrow F_j$ and 
\begin{eqnarray*}
(A \cup S_{j-1}) \cap F_j & = & (A_{n-1} \cap F_j) \cup (A_{n,k} \cap F_j) \cup \bigcup_{j<i} F_j \cap F_i
\end{eqnarray*}
which is $\partial^{\mathrm{ext}}(F_j)$ again. This proves the proposition. $\Box$
\end{proof}

We say a class $\mathcal{A}$ of normal monomorphisms in $\mathbf{fSets}$ is \emph{hypersaturated} if it is closed under pushouts, retracts, (transfinite) composition, direct sums and also satisfies the following cancellation property: if
\[
\xymatrix{
A \ar[r]^i & B \ar[r]^j & C
}
\]
are normal monomorphisms such that $i$ and $ji$ are in $\mathcal{A}$, then $j$ is also in $\mathcal{A}$.  As a kind of converse to Proposition \ref{prop:Segalcoreinnanod}, we now prove that the Segal cores generate the inner horn inclusions in the following sense:

\begin{proposition}
\label{prop:Segcoregeninneran}
Let $\mathcal{A}$ be a hypersaturated class of normal monomorphisms containing all the Segal cores $\mathrm{fSc}(uT) \rightarrowtail uT$ of trees. Then $\mathcal{A}$ contains all inner horn inclusions $\Lambda^e[uT] \rightarrowtail uT$.
\end{proposition}
\begin{proof}
For the duration of this proof, let us simply write $T$ for $uT$, which should not cause confusion; everything we do is to be considered in $\mathbf{fSets}$. We will argue by induction on $T$ and prove first that all three of the inclusions
\begin{equation*}
\mathrm{fSc}(T) \longrightarrow \mathrm{fSc}(T) \cup \partial_{\mathrm{root}}(T) \longrightarrow \partial^{\mathrm{ext}}(T) \longrightarrow \Lambda^e[T]
\end{equation*}
belong to $\mathcal{A}$. This is clear if $T$ has at most two vertices, since all these maps are then isomorphisms. For a larger tree $T$, write
\begin{equation*}
T \, = \, C_p \star (T_1, \ldots, T_p)
\end{equation*}
where $p$ is the valence of the root vertex $r$ of $T$ and $\partial_r(T) = T_1 \oplus \cdots \oplus T_p$. Then by the inductive assumption and the fact that $\mathcal{A}$ is assumed to be closed under direct sums, we find that
\begin{equation*}
\mathrm{fSc}(T_1) \oplus \cdots \oplus \mathrm{fSc}(T_p) \longrightarrow T_1 \oplus \cdots \oplus T_p = \partial_r(T)
\end{equation*}
belongs to $\mathcal{A}$ and hence by a pushout so does its union with $C_p$, the corolla at the root:
\begin{equation*}
\mathrm{fSc}(T) = C_p \cup (\mathrm{fSc}(T_1) \oplus \cdots \oplus \mathrm{fSc}(T_p)) \longrightarrow C_p \cup (T_1 \oplus \cdots \oplus T_p) = \mathrm{fSc}(T) \cup \partial_{\mathrm{root}}(T)
\end{equation*}
This proves that the first of the three maps above is in $\mathcal{A}$. Now, let $v$ be any leaf vertex and let $\partial_v(T)$ be the corresponding external face. Consider the pushout
\[
\xymatrix{
C_p \cup \bigl(\partial_v(T) \cap \partial_{\mathrm{root}}(T)\bigr) \ar[d]\ar[r] & C_p \cup \partial_{\mathrm{root}}(T) \ar[d] \\
\partial_v(T) \ar[r] & C_p \cup \partial_{\mathrm{root}}(T) \cup \partial_v(T)
}
\]
Since $\partial_v(T) \cap \partial_{\mathrm{root}}(T) = \partial_{\mathrm{root}}(\partial_v(T))$ and $C_p \cup \partial_{\mathrm{root}}(\partial_v(T)) = \mathrm{fSc}(\partial_v(T)) \cup \partial_{\mathrm{root}}(\partial_v(T))$, the left-hand vertical map belongs to $\mathcal{A}$ by induction. Hence so does the right-hand vertical map. Next, if we have shown that 
\begin{equation*}
C_p \cup \partial_{\mathrm{root}}(T) \longrightarrow \partial_{\mathrm{root}}(T) \cup \partial_{v_1}(T) \cup \cdots \cup \partial_{v_k}(T)
\end{equation*}
belongs to $\mathcal{A}$ for a sequence of leaf vertices $v_1, \ldots, v_k$ of $T$, we can adjoin another leaf face $\partial_{v_{k+1}}(T)$ in exactly the same way. Having done this for all the leaf faces, we conclude that the map
\begin{equation*}
\mathrm{fSc}(T) \cup \partial_{\mathrm{root}}(T) = C_p \cup \partial_{\mathrm{root}}(T) \longrightarrow \partial^{\mathrm{ext}}(T)
\end{equation*}
belongs to $\mathcal{A}$. \par 
Finally, we will adjoin the inner faces $\partial_{a_i}(T)$ for all the inner edges $a_1, \ldots, a_n$ in $T$ other than $e$ and show that each of
\begin{equation*}
\partial^{\mathrm{ext}}(T) \cup \partial_{a_1}(T) \cup \cdots \cup \partial_{a_i}(T) \longrightarrow \partial^{\mathrm{ext}}(T) \cup \partial_{a_1}(T) \cup \cdots \cup \partial_{a_{i+1}}(T)
\end{equation*}
belongs to $\mathcal{A}$, for $i = 0, \ldots, n-1$. Indeed, this map is a pushout of
\begin{equation*}
\bigl(\partial^{\mathrm{ext}}(T) \cup \partial_{a_1}(T) \cup \cdots \cup \partial_{a_i}(T)\bigr) \cap \partial_{a_{i+1}}(T) \, = \, \partial^{\mathrm{ext}}\partial_{a_{i+1}}(T) \cup \partial_{a_1}\partial_{a_{i+1}}(T) \cup \cdots \cup \partial_{a_i}\partial_{a_{i+1}}(T) \longrightarrow \partial_{a_{i+1}}(T),
\end{equation*}
so the assertion follows by induction on $T$ and $i$, since the base of the induction was already established at the start of our proof. \par 
We have now shown that in the following diagram, the vertical and skew maps are in $\mathcal{A}$:
\[
\xymatrix{
\mathrm{fSc}(T) \ar[d]\ar[dr] & \\
\Lambda^e[T] \ar[r] & T
}
\]
By the assumed closure property of $\mathcal{A}$, we conclude that it also contains the inner horn inclusion $\Lambda^e[T] \rightarrowtail T$. $\Box$
\end{proof}

The following result is crucial when showing that the weak enrichment of the category of forest sets satisfies axiom (H2) of Definition \ref{def:stronghomotenrich}.

\begin{proposition}
\label{prop:forestsetsH2}
For simplices $\Delta^m$, $\Delta^n$ and a forest $F$, consider the map
\begin{equation*}
\alpha: (\Delta^m \times \Delta^n) \otimes F \longrightarrow \Delta^m \otimes (\Delta^n \otimes F).
\end{equation*}
If $\mathcal{A}$ is a hypersaturated class of normal monomorphisms containing all the Segal cores $\mathrm{fSc}(uT) \rightarrowtail uT$ of trees, then $\mathcal{A}$ contains $\alpha$.
\end{proposition}
\begin{proof}
Form the following commutative square:
\[
\xymatrix{
\bigl(\mathrm{fSc}(\Delta^m) \times \mathrm{fSc}(\Delta^n)\bigr) \otimes \mathrm{fSc}(F) \ar[r]\ar[d] & \mathrm{fSc}(\Delta^m) \otimes \bigl(\mathrm{fSc}(\Delta^n) \otimes \mathrm{fSc}(F)\bigr) \ar[d] \\
(\Delta^m \times \Delta^n) \otimes F \ar[r] & \Delta^m \otimes (\Delta^n \otimes F).
}
\]
The vertical maps are in $\mathcal{A}$ by Proposition \ref{prop:poprodinneranodyne} and the fact that $\mathcal{A}$ contains the inner anodynes, by Proposition \ref{prop:Segcoregeninneran}. Therefore it suffices to show that the top horizontal map is in $\mathcal{A}$. This is a composition of pushouts of maps involving only 1-simplices, corollas and sums of such. Since tensor products distribute over sums, it suffices to prove that the map
\begin{equation*}
\alpha: (\Delta^1 \times \Delta^1) \otimes C_k \longrightarrow \Delta^1 \otimes (\Delta^1 \otimes C_k)
\end{equation*} 
is inner anodyne (for every $k \geq 0$). For notational simplicity we treat the case $k=2$; the higher cases are completely analogous. The tensor product $\Delta^1 \otimes (\Delta^1 \otimes C_2)$ is the nerve of the operad $[1] \otimes ([1] \otimes \Omega(C_2))$. It can be described as the union of all the shuffles of the three trees $\Delta^1$, $\Delta^1$ and $C_2$. The forest set $(\Delta^1 \times \Delta^1) \otimes C_2$ is the union of only a subset of all these shuffles. Let us label the leaves of $C_2$ by $a$ and $b$. Then the `missing shuffles' in $(\Delta^1 \times \Delta^1) \otimes C_2$ are the following:

\[
\begin{tikzpicture} 
[level distance=10mm, 
every node/.style={fill, circle, minimum size=.1cm, inner sep=0pt}, 
level 1/.style={sibling distance=20mm}, 
level 2/.style={sibling distance=10mm}, 
level 3/.style={sibling distance=5mm},
level 4/.style={sibling distance=4mm}]

%shuffle 1
\node (shuffle1)[style={color=white}] {} [grow'=up] 
child {node[draw,fill=none] {} 
	child{ node {}
		child{ node {}
			child
		}
	}
	child{ node {}
		child{ node {}
			child
		}
	}
};

%shuffle2
\node (shuffle2)[style={color=white}, right = 3cm of shuffle1] {} [grow'=up] 
child {node[draw,fill=none] {} 
	child{ node {}
		child{ node {}
			child
		}
	}
	child{ node {}
		child{ node {}
			child
		}
	}
};

\tikzstyle{every node}=[]

%lines
%\draw ($(shuffle1) + (1cm, 1cm)$) -- ($(shuffle1) + (2cm, 1cm)$);

%labels
\node at ($(shuffle1) + (-.5cm, 0)$) {$S_1$};
\node at ($(shuffle1) + (-.6cm, 1.5cm)$) {$11a$};
\node at ($(shuffle1) + (-.85cm, 2.5cm)$) {$01a$};
\node at ($(shuffle1) + (-.85cm, 3.5cm)$) {$00a$};
\node at ($(shuffle1) + (.6cm, 1.5cm)$) {$11b$};
\node at ($(shuffle1) + (.85cm, 2.5cm)$) {$10b$};
\node at ($(shuffle1) + (.85cm, 3.5cm)$) {$00b$};

\node at ($(shuffle2) + (-.5cm, 0cm)$) {$S_2$};
\node at ($(shuffle2) + (-.6cm, 1.5cm)$) {$11a$};
\node at ($(shuffle2) + (-.85cm, 2.5cm)$) {$10a$};
\node at ($(shuffle2) + (-.85cm, 3.5cm)$) {$00a$};
\node at ($(shuffle2) + (.6cm, 1.5cm)$) {$11b$};
\node at ($(shuffle2) + (.85cm, 2.5cm)$) {$01b$};
\node at ($(shuffle2) + (.85cm, 3.5cm)$) {$00b$};

\end{tikzpicture} 
\]

In words, they are those shuffles of $\Delta^1 \otimes (\Delta^1 \otimes C_2)$ where we place one shuffle of $\Delta^1 \times \Delta^1$ on the leaf $a$ and the other shuffle of $\Delta^1 \times \Delta^1$ on the leaf $b$. (In this picture, a label like $00a$ is shorthand for $0 \otimes 0 \otimes a$.) We will demonstrate how to adjoin $S_1$ to $(\Delta^1 \times \Delta^1) \otimes C_2$ by an inner anodyne map; the argument for subsequently adjoining $S_2$ is completely analogous. Consider the following external face $R$ of $S_1$:

\[
\begin{tikzpicture} 
[level distance=10mm, 
every node/.style={fill, circle, minimum size=.1cm, inner sep=0pt}, 
level 1/.style={sibling distance=20mm}, 
level 2/.style={sibling distance=10mm}, 
level 3/.style={sibling distance=5mm},
level 4/.style={sibling distance=4mm}]

%shuffle 1
\node (shuffle1)[style={color=white}] {} [grow'=up] 
child {node[draw,fill=none] {} 
	child{ node {}
		child
	}
	child{ node {}
		child
	}
};

\tikzstyle{every node}=[]

%lines
%\draw ($(shuffle1) + (1cm, 1cm)$) -- ($(shuffle1) + (2cm, 1cm)$);

%labels
\node at ($(shuffle1) + (-.5cm, 0)$) {$R$};
\node at ($(shuffle1) + (-.6cm, 1.5cm)$) {$11a$};
\node at ($(shuffle1) + (-.85cm, 2.5cm)$) {$01a$};
\node at ($(shuffle1) + (.6cm, 1.5cm)$) {$11b$};
\node at ($(shuffle1) + (.85cm, 2.5cm)$) {$10b$};

\end{tikzpicture} 
\]

Then the map $(\Delta^1 \times \Delta^1) \otimes C_2 \rightarrow R \cup (\Delta^1 \times \Delta^1) \otimes C_2$ is a pushout of the map $\Lambda^E[R] \rightarrow R$, where $E$ is the set of inner edges $\{11a, 11b\}$. Subsequently, consider the external face $T_1$ (resp. $T_2$) obtained from $S_1$ by chopping off the leaf $00a$ and its adjacent vertex (resp. the leaf $00b$ and its adjacent vertex). Then $T_1$ (and then $T_2$) can be adjoined a pushout along the inner horn inclusion $\Lambda^E[T_1] \rightarrow T_1$ (and then along $\Lambda^E[T_2] \rightarrow T_2$). Finally, we can then adjoin $S_1$ itself by a pushout along the inner horn inclusion $\Lambda^E[S_1] \rightarrow S_1$.
$\Box$ 
\end{proof}

\subsection{The operadic model structure on $\mathbf{fSets}$}

In this section we will use the weak enrichment of the category of forest sets to establish a model structure on this category. 

\begin{definition}
An object $E$ of $\mathbf{fSets}$ is \emph{operadically local} if it satisfies the following three conditions:
\begin{itemize}
\item[(1)] For every normal monomorphism between normal forest sets $A \rightarrowtail B$,  the map
\begin{equation*}
\mathbf{hom}(B, E) \longrightarrow \mathbf{hom}(A, E)
\end{equation*}
is a categorical fibration of simplicial sets, i.e. a fibration in the Joyal model structure. 
\item[(2)] For any two normal forest sets $C$ and $D$, the map
\begin{equation*}
\mathbf{hom}(C \oplus D, E) \longrightarrow \mathbf{hom}(C \amalg D, E)
\end{equation*}
is a trivial fibration of simplicial sets.
\item[(3)] For every inner horn inclusion of a forest $\Lambda^e[F] \rightarrowtail F$, the map
\begin{equation*}
\mathbf{hom}(F, E) \longrightarrow \mathbf{hom}(\Lambda^e[F], E)
\end{equation*}
is a trivial fibration of simplicial sets.
\end{itemize}
\end{definition}

In particular, by applying condition (1) to the map $\varnothing \longrightarrow A$, the simplicial set $\mathbf{hom}(A, E)$ is an $\infty$-category for any normal forest set $A$ and operadically local object $E$.

We denote by $J$ the forest set which is the nerve of the groupoid interval, i.e. the category with two objects labelled $0, 1$ and an isomorphism between them. It comes with maps
\[
\xymatrix{
\{0\} \amalg \{1\} \ar[r]^-{i_0 \amalg i_1} & J \ar[r]^\varepsilon & \eta
}
\]
where $\{0\}$ and $\{1\}$ denote copies of $\eta$. We will use the short-hand notation
\begin{equation*}
\partial J := \{0\} \amalg \{1\}
\end{equation*}

Here is a reformulation and simplification of the previous definition:

\begin{lemma}
\label{lem:genbasicanodyne}
\begin{itemize}
\item[(i)] A forest set $E$ is an operadically local object if and only if $E$ has the right lifting property with respect to all maps of the following types:
\begin{eqnarray*}
(a) && \Lambda^n_k \otimes B \cup \Delta^n \otimes A \longrightarrow  \Delta^n \otimes B \\ 
(b) && J \otimes A \cup \{0\} \otimes B \longrightarrow J \otimes B \\
(c) && \Delta^n \otimes (C \amalg D) \cup \partial \Delta^n \otimes (C \oplus D) \longrightarrow \Delta^n \otimes (C \oplus D) \\
(d) && \Delta^n \otimes \Lambda^e[F] \cup \partial \Delta^n \otimes F \longrightarrow \Delta^n \otimes F.
\end{eqnarray*}
Here we assume $0 < k < n$, the map $A \rightarrowtail B$ is a normal monomorphism between normal objects, $C$ and $D$ are normal forest sets and $\Lambda^e[F] \rightarrow F$ is an inner horn inclusion.
\item[(ii)] A forest set $E$ is an operadically local object if and only if $E$ has the right lifting property with respect to all maps of the following types:
\begin{eqnarray*}
(a) && \Lambda^e[F] \longrightarrow F \\ 
(b) && J \otimes \partial F \cup \{0\} \otimes F \longrightarrow J \otimes F \\
(c) && \Delta^n \otimes (F \amalg G) \cup \partial \Delta^n \otimes (F \oplus G) \longrightarrow \Delta^n \otimes (F \oplus G).
\end{eqnarray*}
where $F$ and $G$ are representable forest sets and $\Lambda^e[F] \rightarrow F$ is an inner horn inclusion. 
\end{itemize}
\end{lemma}
\begin{proof}
Recall that a map between $\infty$-categories is a categorical fibration if and only if it has the right lifting property with respect to the maps
\begin{eqnarray*}
\Lambda^n_k & \longrightarrow & \Delta^n, \quad\quad\quad 0 < k < n, \\
\{0\} & \longrightarrow & J.
\end{eqnarray*}
The statement of part (i) of the lemma is then clear from the definitions. We should verify part (ii). First, note that by Proposition \ref{prop:poprodinneranodyne} the maps of (i)(a) and (i)(d) are inner anodyne, which allows us to replace them by (ii)(a). We can reduce (i)(b) to (ii)(b) by the characterization of normal monomorphisms given in Corollary \ref{cor:gennormalmonos}. Finally, (i)(c) can be reduced to (ii)(c) by a standard skeletal induction argument (see the proof of Proposition \ref{prop:operadicwes} for a typical example). $\Box$
\end{proof}

\begin{definition}
\label{def:genbasicanodyne}
The class of \emph{operadic anodyne maps} is the saturation of the class of maps occurring in the lemma, i.e. the closure of (i)(a)-(d), or equivalently (ii)(a)-(c), under pushouts, transfinite compositions and retracts. We will call the maps in (ii)(a)-(c) the \emph{generating operadic anodyne maps}. Note that by Proposition \ref{prop:newnormalmonopoprod}(i), these are all normal monomorphisms. When $A$ and $B$ range over \emph{simplicial} sets, recall that the saturation (within the category of simplicial sets) of the class of maps in (i)(a) and (i)(b) is called the class of \emph{$J$-anodyne maps} of simplicial sets.
\end{definition}

For later use, let us record the following elementary property:
\begin{lemma}
\label{lem:basicanodpoprod}
For an operadic anodyne map $A \rightarrow B$ between forest sets and a cofibration $M \rightarrow N$ between simplicial sets, the pushout-product
\begin{equation*}
M \otimes B \cup N \otimes A \longrightarrow N \otimes B
\end{equation*}
is again an operadic anodyne map.
\end{lemma}
\begin{proof}
Clearly it suffices to treat the case where $A \rightarrow B$ is a generating operadic anodyne. If it is of the form (ii)(a), then the desired conclusion follows from Proposition \ref{prop:poprodinneranodyne}. Now suppose $A \longrightarrow B$ is a map of the form
\begin{equation*}
\{0\} \otimes F \cup J \otimes \partial F \longrightarrow J \otimes F
\end{equation*}
for a forest $F$. We have to consider the map
\begin{equation*}
M \otimes B \cup N \otimes A \longrightarrow N \otimes B.
\end{equation*}
Explicitly, it can be written as
\begin{equation*}
M \otimes \bigl(J \otimes F\bigr) \cup N \otimes \bigl(\{0\} \otimes F \cup J \otimes \partial F \bigr) \longrightarrow N \otimes \bigl(J \otimes F\bigr).
\end{equation*}
By the symmetry of the enrichment, this is isomorphic to the map
\begin{equation*}
\{0\} \otimes \bigl(N \otimes F\bigr) \cup J \otimes \bigl(M \otimes F \cup N \otimes \partial F\bigr) \longrightarrow J \otimes \bigl(N \otimes F \bigr).
\end{equation*}
This is an operadic anodyne, since the map 
\begin{equation*}
M \otimes F \cup N \otimes \partial F \longrightarrow N \otimes F
\end{equation*}
is a normal monomorphism by virtue of Proposition \ref{prop:newnormalmonopoprod}. The case where $A \longrightarrow B$ is of the form
\begin{equation*}
\Delta^n \otimes (C \amalg D) \cup \partial \Delta^n \otimes (C \oplus D) \longrightarrow \Delta^n \otimes (C \oplus D)
\end{equation*}
is treated similarly. $\Box$
\end{proof}

Also, we immediately conclude the following:
\begin{proposition}
\label{prop:countable}
\begin{itemize}
\item[(i)] For any forest set $X$ there exists an operadic anodyne map $X \longrightarrow X_f$ into an operadically local object $X_f$.
\item[(ii)] For a monomorphism $X \longrightarrow Y$ and any choice of $X \longrightarrow X_f$ as in (i), there exists an operadic anodyne $Y \longrightarrow Y_f$ such that $Y_f$ is an operadically local object and there is a commutative square of monomorphisms as follows:
\[
\xymatrix{
X \ar[r]\ar[d] & Y \ar[d] \\
X_f \ar[r] & Y_f
}
\]
\item[(iii)] If $X$ is countable, $X_f$ can be chosen to be countable as well.
\item[(iv)] If $A \subseteq Y_f$ is countable, then there exists a countable $X$ and a commutative square of monomorphisms
\[
\xymatrix{
X \ar[r]\ar[d] & Y \ar[d] \\
X_f \ar[r] & Y_f
}
\]
such that the map $A \longrightarrow Y_f$ factors through $X_f$.
\end{itemize}
\end{proposition}
\begin{proof}
The first two parts follow from standard arguments involving the small object argument. The rest is clear from the finiteness of the objects involved in the maps of Lemma \ref{lem:genbasicanodyne}(ii)(a)-(c). $\Box$
\end{proof}

We now define the classes of maps involved in the operadic model structure.

\begin{definition}
\begin{itemize}
\item[(i)] A map $X \longrightarrow Y$ in $\mathbf{fSets}$ is called a \emph{cofibration} if it is a normal monomorphism.
\item[(ii)] A map $X \longrightarrow Y$ in $\mathbf{fSets}$ is called an \emph{operadic weak equivalence} if there exists a commutative diagram
\[
\xymatrix{
X' \ar@{->>}[d] \ar[r] & Y'\ar@{->>}[d] \\
X \ar[r] & Y
}
\]
where the vertical maps are normalizations and $X' \longrightarrow Y'$ induces an equivalence of $\infty$-categories
\begin{equation*}
\mathbf{hom}(Y', E) \longrightarrow \mathbf{hom}(X', E)
\end{equation*}
for every operadically local object $E$. (One could construct the diagram so that the map $X' \longrightarrow Y'$ is in addition a cofibration, in which case the stated map between $\infty$-categories will in fact be a trivial fibration.) 
\item[(iii)] A map $X \longrightarrow Y$ is called an \emph{operadic fibration} if it has the right lifting property with respect to all trivial cofibrations, i.e. those cofibrations that are also operadic weak equivalences.
\end{itemize}
\end{definition}

Another useful concept is that of $J$-homotopy:

\begin{definition}
Two maps $f,g: X \longrightarrow Y$ between forest sets are \emph{$J$-homotopic} if there exists a dashed arrow as indicated in the following diagram:
\[
\xymatrix{
\{0\} \amalg \{1\} \ar[d]\ar[r]^-{f \amalg g} & \mathbf{hom}(X,Y) \\
J \ar@{-->}[ur] &
}
\]
\end{definition}

\begin{remark}
The previous definition gives rise to an obvious notion of $J$-homotopy equivalence between forest sets. Using the fact that $J$-homotopy equivalences of simplicial sets are equivalences in the Joyal model structure, it is easy to see that a $J$-homotopy equivalence of normal forest sets is an operadic weak equivalence.
\end{remark}

The rest of this section and the next will be devoted to a proof of the following theorem. We will give a proof using fairly elementary methods to stress the essential simplicity of the arguments involved.

\begin{theorem}
\label{thm:basicmodelstructure}
\begin{itemize}
\item[(i)] The normal monomorphisms, operadic weak equivalences and operadic fibrations define a model structure on $\mathbf{fSets}$, to be referred to as the \emph{operadic model structure}.
\item[(ii)] The operadic model structure is cofibrantly generated and left proper.
\item[(iii)] The fibrant objects in this model structure are exactly the operadically local objects.
\item[(iv)] The fibrations between fibrant objects are precisely the maps having the right lifting property with respect to the operadic anodyne morphisms.
\item[(v)] The operadic model structure is homotopically enriched over the Joyal model structure on the category of simplicial sets.
\end{itemize}
\end{theorem}

Before embarking on the proof, let us draw an immediate consequence.

\begin{corollary}
The adjoint functors
\[
\xymatrix@C=40pt{
u^*: \mathbf{fSets} \ar@<.5ex>[r] & \mathbf{dSets}: u_* \ar@<.5ex>[l]
}
\]
form a Quillen pair between the operadic model structure on $\mathbf{fSets}$ and the model structure on $\mathbf{dSets}$ of Theorem \ref{thm:CMmodelstruct}.
\end{corollary}
\begin{proof}[Proof of Corollary]
It suffices to show that $u^*$ preserves cofibrations and that $u_*$ preserves fibrant objects and fibrations between fibrant objects. The fact that $u^*$ preserves cofibrations was already discussed in Section \ref{sec:normalmonos}. Part (iv) of the theorem now shows that it suffices to prove that $u^*$ sends operadic anodynes to trivial cofibrations in $\mathbf{dSets}$. Since the model structure on $\mathbf{dSets}$ is homotopically enriched over the Joyal model structure, this is clear from the fact that $u^*$ preserves tensor products and sends direct sums to coproducts. $\Box$
\end{proof}

Let us now turn our attention to the proof of Theorem \ref{thm:basicmodelstructure}. We will begin with several lemmas concerning the weak equivalences. The first one shows that the definition of operadic weak equivalence is independent of the chosen square involving normalizations of $X$ and $Y$.

\begin{lemma}
\label{lem:normalizationsindep}
If $X \longrightarrow Y$ is an operadic weak equivalence and we have a square 
\[
\xymatrix{
X'\ar@{->>}[d]\ar[r] & Y' \ar@{->>}[d] \\
X \ar[r] & Y
}
\]
in which the vertical maps are normalizations, then the induced map
\begin{equation*}
\mathbf{hom}(Y', E) \longrightarrow \mathbf{hom}(X', E)
\end{equation*}
is an equivalence of $\infty$-categories for any operadically local object $E$. 
\end{lemma}
\begin{proof}
First, construct a square
\[
\xymatrix{
X'' \ar@{->>}[d]\ar[r] & Y'' \ar@{->>}[d] \\
X \ar[r] & Y
}
\]
by choosing a normalization $X''$ of $X$ and then factoring the composite map $X'' \longrightarrow Y$ into a normal mono $X'' \longrightarrow Y''$ followed by a map $Y'' \longrightarrow Y$ having the right lifting property with respect to all normal monos, which is therefore a normalization of $Y$. We will now show that any other square of normalizations as described in the lemma is equivalent to this one in an appropriate sense. Choose lifts as indicated by the dashed arrows in the squares
\[
\xymatrix{
\varnothing \ar[d]\ar[r] & X'\ar[d] & & & X' \ar[d]\ar[r] & Y' \ar[d]\\ 
X''\ar@{-->}[ur]^f\ar[r] & X    & & & X' \amalg_{X''} Y''\ar[r]\ar@{-->}[ur] & Y 
}
\]
This gives a commutative diagram
\[
\xymatrix{
X'' \ar[dd]\ar[dr]\ar[rr] & & Y'' \ar'[d][dd]\ar[dr] & \\
& X' \ar[rr]\ar[dl] & & Y'\ar[dl] \\
X \ar[rr] && Y & 
}
\]
It now suffices to show that the induced map
\begin{equation*}
\mathbf{hom}(X', E) \longrightarrow \mathbf{hom}(X'', E)
\end{equation*}
is an equivalence of $\infty$-categories (and similary for the map induced by $Y'' \longrightarrow Y'$). This follows from the fact that normalizations are unique up to $J$-homotopy equivalence. Indeed, successively lifting in the squares
\[
\xymatrix{
\varnothing \ar[r]\ar[d] & X'' \ar[d] \\
X' \ar@{-->}[ur]^g \ar[r]& X         
}
\]
and
\[
\xymatrix@C=35pt{
X' \amalg X' \ar[d]\ar[r]^{\mathrm{id} \amalg fg} & X' \ar[d] && X'' \amalg X'' \ar[d]\ar[r]^{\mathrm{id} \amalg gf} & X'' \ar[d] \\
X' \otimes J \ar@{-->}[ur]\ar[r] & X   && X'' \otimes J \ar@{-->}[ur]\ar[r] & X
}
\]
produces such an equivalence between $X'$ and $X''$. The fact that the left vertical maps in both squares are cofibrations follows from Proposition \ref{prop:newnormalmonopoprod}. $\Box$
\end{proof}

\begin{lemma}
\label{lem:RLPwrtnormalmonos}
A map in $\mathbf{fSets}$ which has the right lifting property with respect to all normal monomorphisms is an operadic weak equivalence.
\end{lemma}
\begin{proof}
Let $f: Y \longrightarrow X$ be such a map. It will suffice to show the existence of a square
\[
\xymatrix{
Y' \ar@{->>}[d]\ar[r]^{f'} & X' \ar@{->>}[d] \\
Y \ar[r]_f & X
}
\]
in which the vertical maps are normalizations and $f'$ is a $J$-homotopy equivalence. To do this, choose a normalization $X' \longrightarrow X$ and lift in the square
\[
\xymatrix{
\varnothing \ar[r]\ar[d] & Y \ar[d]^f \\
X' \ar@{-->}[ur]^s\ar[r] & X
}
\]
Now factor the lift $s$ as 
\[
\xymatrix{
X'\ar[rr]^s\ar@{)->}[dr]_{i} && Y \\
& Y' \ar@{->>}[ur]_t &
}
\]
where $i$ is a normal mono and $t$ is a normalization. Finally, lift in
\[
\xymatrix{
X' \ar@{=}[r]\ar[d]_i & X' \ar[d] \\
Y' \ar@{-->}[ur]^{f'}\ar[r]_{ft} & X
}
\]
Then clearly $f'i = \mathrm{id}_{X'}$ and we claim that $if'$ is $J$-homotopic to $\mathrm{id}_{Y'}$. Indeed, $ft$ has the right lifting property with respect to all normal monomorphisms, so we can lift in 
\[
\xymatrix{
\partial J \otimes Y' \cup J \otimes X' \ar[r]^-\phi \ar[d] & Y' \ar[d]^{ft} \\
J \otimes Y' \ar@{-->}[ur] \ar[r]_-\psi & X
}
\] 
where $\phi = (if', \mathrm{id}_{Y'}) \cup i\varepsilon$ and $\psi = ft \varepsilon$, with $\varepsilon: J \longrightarrow \eta$ the obvious collapse map. Here we use the fact that the left vertical map is a normal mono, which follows from Proposition \ref{prop:newnormalmonopoprod}. Recall that $\partial J$ is shorthand for $\{0\} \amalg \{1\}$. $\Box$ 
\end{proof}

\begin{lemma}
\label{lem:pushoutbasictrivcof}
A pushout of an operadic trivial cofibration (i.e. a cofibration that is also an operadic weak equivalence) is again an operadic trivial cofibration. 
\end{lemma}
\begin{proof}
Consider a pushout
\[
\xymatrix{
A \,\, \ar@{>->}^{\sim}[r]\ar[d] & B \ar[d] \\
C \,\, \ar@{>->}[r] & D
}
\]
and suppose $A \longrightarrow B$ is a trivial cofibration as indicated. First assume all objects in this square are normal. If $E$ is an operadically local object, then in the pullback square
\[
\xymatrix{
\mathbf{hom}(D, E) \ar[d]\ar[r] & \mathbf{hom}(C, E) \ar[d] \\
\mathbf{hom}(B, E) \ar@{->>}^{\sim}[r] & \mathbf{hom}(A, E)
}
\]
the bottom horizontal map is a trivial fibration. Hence so is the top horizontal map, so that $C \longrightarrow D$ is an operadically weak equivalence. We now show how to reduce the general case to this one. Choose a normalization 
\[
\xymatrix{
D' \ar@{->>}[r] & D
}
\]
and pull back along this map to produce a cube
\[
\xymatrix{
A' \ar@{->>}[dd]\ar[rr]\ar[dr] &           & C' \ar@{->>}'[d][dd]\ar[dr]&     \\
                               & B'\ar[rr]\ar@{->>}[dd] &   & D'\ar@{->>}[dd] \\
A  \ar'[r][rr]\ar[dr]          &           & C \ar[dr]      &                 \\
                               & B \ar[rr] &                & D  
}
\]
In this cube, both horizontal faces are pushouts and all vertical faces are pullbacks. Using Lemma \ref{lem:normalovernormal} we see that all objects in the top face are normal, so that all vertical maps are in fact normalizations. Now $A' \longrightarrow B'$ is a trivial cofibration between normal objects and we use the argument above to conclude that $C' \longrightarrow D'$ is as well. $\Box$
\end{proof}

\begin{lemma}
\label{lem:basicanodtrivcof}
Operadic anodyne maps are trivial cofibrations.
\end{lemma}
\begin{proof}
Since compositions and retracts of trivial cofibrations are clearly trivial cofibrations again and the same is true for pushouts by the preceding lemma, it suffices to prove that the generating operadic anodynes of Definition \ref{def:genbasicanodyne} are trivial cofibrations. We claim that if $U \longrightarrow V$ is a generating operadic anodyne and $E$ an operadically local object, then
\begin{equation*}
\mathbf{hom}(V, E) \longrightarrow \mathbf{hom}(U, E)
\end{equation*}
is a trivial fibration of simplicial sets. Indeed, it has the right lifting property with respect to boundary inclusions $\partial \Delta^m \rightarrow \Delta^m$, since the pushout-product
\begin{equation*}
\partial \Delta^m \otimes V \cup \Delta^m \otimes U \longrightarrow \Delta^m \otimes V
\end{equation*}
is an operadically anodyne map, by Lemma \ref{lem:basicanodpoprod}. $\Box$
\end{proof}

We now turn to the study of arbitrary trivial cofibrations.

\begin{lemma}
\label{lem:trivcofgenbetweennormals}
Every trivial cofibration is a retract of a pushout of a trivial cofibration between normal objects.
\end{lemma}
\begin{proof}
Let $u: A \longrightarrow B$ be a trivial cofibration. Choose a normalization $B' \longrightarrow B$ of $B$ and form the pullback
\[
\xymatrix{
A' \,\,\ar@{>->}[r]^{u'}\ar@{->>}[d]_p & B'\ar@{->>}[d]^q \\ 
A \,\,\ar@{>->}[r]_u & B
}
\]
Then $u'$ is a trivial cofibration between normal objects. Now form the pushout
\[
\xymatrix{
A' \ar[r]^{u'}\ar[d]_p & B' \ar[d]^r \\
A \ar[r]_v & P
}
\]
which gives a canonical map $s: P \longrightarrow B$. It now suffices to prove that $s$ has the right lifting property with respect to all cofibrations, because this would make $u$ a retract of $v$ by lifting in the square 
\[
\xymatrix{
A \ar[r]^v\ar[d]_u & P \ar[d]^s \\ 
B \ar@{=}[r]\ar@{-->}[ur] & B
}
\]
So, consider a lifting problem
\[
\xymatrix{
\partial F \ar[d]\ar[r] & P \ar[d]^s \\
F \ar[r] & B
}
\]
Pull our previous pushout diagram back along $\partial F \longrightarrow P$ to form the cube 
\[
\xymatrix{
E \ar[rr]\ar[dd]\ar[dr] & & D \ar'[d][dd]\ar[dr]& \\
& A' \ar[rr]\ar@{->>}[dd] & & B'\ar@{->>}[dd] \\
C \ar'[r][rr]\ar[dr] & & \partial F \ar[dr] & \\
& A \ar[rr] & & P 
}
\]
in which the front and back face are pushouts and all other faces are pullbacks. Then $E \longrightarrow C$ is a normalization and all objects in the back face are normal, so $E \longrightarrow C$ has a section and hence so does the pushout $D \longrightarrow \partial F$. Using this section, we can form a commutative diagram
\[
\xymatrix{
\partial F \ar[d]\ar[r] & D \ar[r] & B'\ar[d]^q \\
F \ar@{-->}[urr]\ar[rr] & & B
}
\]
in which a lift as indicated exists, which also gives a solution to our previous lifting problem. $\Box$ 
\end{proof}

\begin{lemma}
A trivial cofibration between normal and operadically local objects is a $J$-deformation retract.
\end{lemma}
\begin{proof}
Let $f: A \longrightarrow B$ be such a trivial cofibration. Then the map 
\begin{equation*}
f^*: \mathbf{hom}(B, A) \longrightarrow \mathbf{hom}(A,A) 
\end{equation*}
is a trivial fibration of simplicial sets and therefore surjective on vertices. This allows us to pick a map $r: B \longrightarrow A$ such that $rf = \mathrm{id}_A$. We find a $J$-homotopy $fr \simeq \mathrm{id}_B$ by lifting in the diagram
\[
\xymatrix@C=35pt{
\partial J \ar[r]^-{(\mathrm{id}_B, fr)}\ar[d] & \mathbf{hom}(B,B) \ar[d]^{f^*} \\
J \ar[r]\ar@{-->}[ur] & \mathbf{hom}(A,B)
}
\]
where the bottom horizontal arrow is the constant map with value $f$. $\Box$
\end{proof}

\begin{lemma}
\label{lem:countablesubpresh}
Let $u: X \longrightarrow Y$ be a trivial cofibration between normal objects in $\mathbf{fSets}$. Then for any countable subpresheaves $A \subseteq X$ and $B \subseteq Y$, there exist intermediate countable subpresheaves $A \subseteq \tilde A \subseteq X$ and $B \subseteq \tilde B \subseteq Y$ which fit into a pullback diagram 
\[
\xymatrix{
\tilde A \ar[r]\ar[d] & X \ar[d]^u \\
\tilde B \ar[r] & Y
}
\]
and in which $\tilde A \longrightarrow \tilde B$ is also a trivial cofibration.
\end{lemma}
\begin{proof}
We use Proposition \ref{prop:countable} to complete $u: X \longrightarrow Y$ into a diagram
\[
\xymatrix{
X \ar[r]\ar[d]_u & X_f \ar[d]^v \\
Y \ar[r] & Y_f
}
\]
Then $X_f \longrightarrow Y_f$ is again a trivial cofibration (by Lemma \ref{lem:basicanodtrivcof} and the obvious two-out-of-three property of operadic weak equivalences) and hence a deformation retract by the previous lemma. Write $r: Y_f \longrightarrow X_f$ for the retraction and 
\begin{equation*}
h: J \otimes Y_f \longrightarrow Y_f
\end{equation*}
for the homotopy. Let $A_0 = A$ and $B_0 = B$. Then we can `close' $A_0$ and $B_0$ inside $X_f$ and $Y_f$ respectively, to find countable $A_0'$ and $B_0'$ with $A_0 \subseteq A_0' \subseteq X$ and $B_0 \subseteq B_0' \subseteq Y$ and
\begin{eqnarray*}
v^{-1}(B_0') & = & A_0' \\
r(B_0') & = & B_0' \\
h(J \otimes B_0') & = & B_0'
\end{eqnarray*}
In other words, the diagram
\[
\xymatrix{
A_0' \ar[r]\ar[d] & X_f \ar[d] \\
B_0' \ar[r] & Y_f
}
\]
is a pullback and $r$ and $h$ restrict to a deformation retract between $A_0'$ and $B_0'$. Next use Proposition \ref{prop:countable} to find countable $A_1 \subseteq X$ and $B_1 \subseteq Y$ with $u^{-1}(B_1) = A_1$ and
\begin{equation*}
A_0' \subseteq (A_1)_f \subseteq X_f \quad\quad \text{and} \quad\quad B_1' \subseteq (B_1)_f \subseteq Y_f
\end{equation*}
Repeat the above construction to find $A_1'$ and $B_1'$ with $A_0' \subseteq A_1' \subseteq X$ and $B_0' \subseteq B_1' \subseteq Y_f$ for which $v^{-1}(B_1') = A_1'$ while $r$ and $h$ restrict to a deformation retract between $A_1'$ and $B_1'$. Iterating this process countably many times, we obtain a ladder
\[
\xymatrix@!0{
A_0 \ar[rr]\ar[dr]\ar[dd] & & A_1 \ar[rr]\ar[dr]\ar'[d][dd] & & A_2 \ar[rr]\ar[dr]\ar'[d][dd] & & & \cdots & \ar[rr] & & X \ar[dr]\ar'[d][dd] \\
& A_0' \ar[rr]\ar[dd] & & A_1' \ar[rr]\ar[dd] & & A_2' \ar[rr]\ar[dd] & & & \cdots & \ar[rr] & & X_f \ar[dd] \\
B_0 \ar[dr]\ar'[r][rr] & & B_1 \ar[dr]\ar'[r][rr] & & B_2 \ar[dr]\ar'[r][rr] & & & \cdots & \ar[rr] & & Y \ar[dr] \\
& B_0' \ar[rr] & & B_1' \ar[rr] & & B_2' \ar[rr] & & & \cdots & \ar[rr] & & Y_f \\
}
\]
where the vertical maps in the front are all deformation retracts and where $A_n' \subseteq (A_{n+1})_f$ and $B_n' \subseteq (B_{n+1})_f$. Let $\widetilde A = \bigcup_n A_n$ and $\widetilde B = \bigcup_n B_n$. Then in the diagram
\[
\xymatrix@!0{
\widetilde A \ar[dd]\ar[rr]\ar[dr] & & X \ar'[d][dd]\ar[dr] & \\
& (\widetilde A)_f \ar[dd]\ar[rr] & & X_f \ar[dd] \\
\widetilde B \ar[dr]\ar'[r][rr] & & Y \ar[dr] & \\
& (\widetilde B)_f \ar[rr] & & Y_f
}
\]
the map $(\widetilde A)_f \longrightarrow (\widetilde B)_f$ is the colimit of the deformation retracts $A'_n \longrightarrow B'_n$, hence is itself a deformation retract (by the same maps $r$ and $h$). Thus $(\widetilde A)_f \longrightarrow (\widetilde B)_f$ is a weak equivalence, hence so is $\widetilde A \longrightarrow \widetilde B$. $\Box$
\end{proof}

\begin{lemma}
\label{lem:countablegentrivcof}
The class of trivial cofibrations is generated by the trivial cofibrations between countable and normal objects.
\end{lemma}
\begin{proof}
We already know that every trivial cofibration is a retract of a pushout of a trivial cofibration between normal objects (Lemma \ref{lem:trivcofgenbetweennormals}). Therefore it suffices to show that every trivial cofibration $X \rightarrowtail Y$ between normal objects lies in the saturation of the countable such maps. Well-order the elements of $Y - X$ as $\{y_\xi \, | \, \xi < \lambda \}$. By induction we will construct factorizations $X \rightarrowtail W_\xi \rightarrowtail Y$ of $X \rightarrowtail Y$ into trivial cofibrations, such that for $\xi < \xi '$ there is a commutative diagram
\[
\xymatrix{
X \ar[r]\ar[d] & W_\xi \ar[d]\ar[dl] \\
W_{\xi '} \ar[r] & Y
}
\]
and such that
\begin{itemize}
\item[-] $y_\xi \in W_{\xi + 1}$
\item[-] $X \rightarrowtail W_\xi$ lies in the saturation of the class of trivial cofibrations between countable normal objects.
\end{itemize} 
Then $W_{\lambda + 1}$ must equal $Y$, which completes the proof. If $W_\xi$ has been constructed for all $\xi < \zeta$, we construct $W_\zeta$ as follows. First, let 
\begin{equation*}
W_\zeta^- := \varinjlim_{\xi < \zeta} W_\xi
\end{equation*}
(Note that $W_{\xi + 1}^-$ is $W_\xi$, so this is only relevant if $\zeta$ is a limit ordinal.) Let 
\[
\xymatrix{
\widetilde A \ar[r]\ar[d] & W_\zeta^- \ar[d] \\
\widetilde B \ar[r] & Y
}
\]
be a pullback diagram as in Lemma \ref{lem:countablesubpresh}, with $y_\zeta \in \widetilde B$, and construct the pushout
\[
\xymatrix{
\widetilde A \ar[d]\ar[r] & W_\zeta^- \ar[d] \\
\widetilde B \ar[r] & W_\zeta 
}
\]
The universal property of the pushout gives us a unique map $W_\zeta \rightarrowtail Y$ compatible with the earlier maps. This map is mono since $W_\zeta^- \rightarrowtail Y$ is and since the previous square is a pullback. This finishes the proof. $\Box$
\end{proof}

We are now ready to complete the proof of parts (i)-(iii) of the main theorem. We defer the proofs of (iv) and (v) to the next section.

\begin{proof}[Proof of Theorem \ref{thm:basicmodelstructure}(i)-(iii)]
(i) We will check the usual axioms CM1-5 from \cite{quillenrational}. The axioms (CM1) for existence of limits and colimits, (CM2) for two-out-of-three for weak equivalences and (CM3) for retracts evidently hold (and in fact we have already used (CM2)). As to the factorization axiom (CM5), Corollary \ref{cor:normalmonofactorization} states that every map can be factored as a cofibration followed by a map having the right lifting property with respect to all cofibrations and the latter is a trivial fibration by Lemma \ref{lem:RLPwrtnormalmonos}. Similarly, any map $X \longrightarrow Y$ can be factored as $X \rightarrowtail Z \rightarrow Y$ where $X \rightarrowtail Z$ lies in the saturation of the class of trivial cofibrations between countable normal objects and $Z \rightarrow Y$ has the right lifting property with respect to this class. Lemma \ref{lem:countablegentrivcof} shows that $Z \rightarrow Y$ is a fibration. Finally, for the lifting axiom (CM4), consider a commutative square
\[
\xymatrix{
A \ar[d]_i\ar[r]^f & Y \ar[d]^p \\
B \ar[r]_g & X
}
\]
where $i$ is a cofibration and $p$ is a fibration. If $i$ is a weak equivalence, then a lift exists by definition of the fibrations. If $p$ is a weak equivalence, one applies the standard retract arguments: factor $Y \longrightarrow X$ as a cofibration $Y \rightarrowtail Z$ followed by a map $Z \longrightarrow X$ having the right lifting property with respect to all cofibrations. Then $Y \rightarrowtail Z$ is a trivial cofibration and successive liftings in
\[
\xymatrix{
A \ar[r]\ar[d] & Y \ar[r] & Z \ar[d] & & Y \ar[d]\ar@{=}[r] & Y \ar[d] \\
B \ar@{-->}[urr]\ar[rr] & & X               & & Z \ar@{-->}[ur]\ar[r] & X
}
\]
will give the required lift $B \longrightarrow Y$. \par 
(ii) We have already seen that the model structure is cofibrantly generated (Lemma \ref{lem:countablegentrivcof} and Corollary \ref{cor:gennormalmonos}). To see it is left proper, consider a pushout
\[
\xymatrix{
A \ar[r]\ar[d] & C \ar[d] \\
B \ar[r] & D
}
\]
in which $A \longrightarrow B$ is a weak equivalence and $A \longrightarrow C$ is a cofibration. We can `normalize' the pushout by pulling back along a normalization of $D$ (as in the proof of Lemma \ref{lem:pushoutbasictrivcof}) to get a cube
\[
\xymatrix@!0{
A' \ar[rr]\ar[dd]\ar[dr] & & C'\ar'[d][dd]\ar[dr] & \\  
& B' \ar[dd]\ar[rr] & & D'\ar[dd] \\
A \ar'[r][rr]\ar[dr] & & C \ar[dr] & \\
& B \ar[rr] & & D
}
\] 
in which the top square is again a pushout. Then the diagram of simplicial sets
\[
\xymatrix{
\mathbf{hom}(D', E) \ar[r]\ar[d] & \mathbf{hom}(B', E) \ar[d] \\
\mathbf{hom}(C', E) \ar[r] & \mathbf{hom}(A', E)
}
\]
is a pullback for any object $E$ and the bottom horizontal map is a categorical fibration. If $E$ is operadically local, all the simplicial sets in this diagram are $\infty$-categories and the right vertical map is an equivalence of $\infty$-categories. The square is then a homotopy pullback square in the Joyal model structure, so that the left vertical map is also an equivalence. \par 
(iii) Lemma \ref{lem:basicanodtrivcof} shows that any fibrant object is an operadically local object. Conversely, let $X$ be any operadically local object. Then $X$ has the right lifting property with respect to maps $A \longrightarrow B$ which are trivial cofibrations between normal objects, because in this case 
\begin{equation*}
\mathbf{hom}(B, X) \longrightarrow \mathbf{hom}(A, X)
\end{equation*}
is a trivial fibration of simplicial sets and therefore surjective on vertices. It now follows from Lemma \ref{lem:trivcofgenbetweennormals} that $X$ has the right lifting property with respect to arbitrary trivial cofibrations. $\Box$
\end{proof}

\subsection{Further properties of the operadic model structure} 

In this section we will establish parts (iv) and (v) of Theorem \ref{thm:basicmodelstructure}. Also, we give a convenient characterization of the trivial cofibrations in the operadic model structure in Proposition \ref{prop:trivcofoperadic} and a characterization of the weak equivalences between fibrant objects in Proposition \ref{prop:operadicwes}.

\begin{lemma}
\label{lem:forestsetsH2}
Let $i: M \rightarrow M'$ and $j: N \rightarrow N'$ be monomorphisms of simplicial sets and let $k: X \rightarrow X'$ be a normal monomorphism of forest sets. Then the map $h$ of Definition \ref{def:stronghomotenrich} is a trivial cofibration, i.e. the weak enrichment of the category of forest sets over the category of simplicial sets satisfies axiom (H2). 
\end{lemma}
\begin{proof}
By standard arguments it suffices to treat the case where $i$, $j$ and $k$ are of the form
\begin{equation*}
i: \partial \Delta^m \rightarrow \Delta^m, \quad j: \partial \Delta^n \rightarrow \Delta^n \quad \text{and} \quad k: \partial F \rightarrow F.
\end{equation*}
The map $h$ may somewhat informally be written as
\begin{equation*}
h: (\Delta^m \times \Delta^n) \otimes F \cup \partial \Delta^m \otimes (\Delta^n \otimes F) \cup \Delta^m \otimes (\partial \Delta^n \otimes F) \cup \Delta^m \otimes (\Delta^n \otimes \partial F) \longrightarrow \Delta^m \otimes (\Delta^n \otimes F).
\end{equation*}
To check that it is a normal monomorphism, we need to verify the following:
\begin{eqnarray*}
(\Delta^m \times \Delta^n) \otimes F \cap \partial \Delta^m \otimes (\Delta^n \otimes F) & = & (\partial \Delta^m \times \Delta^n) \otimes F, \\
(\Delta^m \times \Delta^n) \otimes F \cap \Delta^m \otimes (\partial \Delta^n \otimes F) & = & (\Delta^m \times \partial\Delta^n) \otimes F, \\
(\Delta^m \times \Delta^n) \otimes F \cap \Delta^m \otimes (\Delta^n \otimes \partial F) & = & (\Delta^m \times \Delta^n) \otimes \partial F.
\end{eqnarray*}
All of these identities follow by the same reasoning as was applied in Section \ref{sec:tensorprodsnormals}, specifically the type of observation mentioned in Remark \ref{rmk:tensorface}. \par 
By Proposition \ref{prop:forestsetsH2} we know that $\alpha_{M,N,X}$ is a trivial cofibration in case $M$, $N$ and $X$ are representable. It follows that $\alpha_{M,N,X}$ is a weak equivalence for any choice of cofibrant objects $M$, $N$ and $X$ by the usual induction on skeletal filtrations. We can now deduce formally that the map $h$ is a weak equivalence as well. Indeed, consider the square
\[
\xymatrix{
A \ar[d] \ar[r] & C \ar[d] \ar@/^/[ddr] & \\
B \ar[r]\ar@/_/[drr] & B \cup_A C \ar[dr]^h & \\
& & D.
}
\]
of Definition \ref{def:stronghomotenrich}. We know that the map $B \rightarrow D$ is a trivial cofibration, since it is the map $\alpha$ described above. Similarly, we deduce that the map $A \rightarrow C$ is a trivial cofibration, being a composition of pushouts of maps of the form $\alpha_{M,N,X}$. Therefore the map $B \rightarrow B \cup_A C$, being a pushout of a trivial cofibration, is a trivial cofibration itself. By two-out-of-three we conclude that $h$ is a weak equivalence. $\Box$ 
\end{proof}

\begin{lemma}
\label{lem:halfhomotenrich}
If $M \rightarrow N$ is a monomorphism of simplicial sets and $A \rightarrow B$ is a trivial cofibration between forest sets, then the pushout-product
\begin{equation*}
N \otimes A \cup M \otimes B \longrightarrow N \otimes B
\end{equation*}
is a trivial cofibration.
\end{lemma}
\begin{proof}
By Lemma \ref{lem:trivcofgenbetweennormals} there is no loss of generality if we assume that $A$ and $B$ are both normal. We should verify that for every operadically local $E$, the map 
\begin{equation*}
\mathbf{hom}(N \otimes B, E) \longrightarrow \mathbf{hom}(N \otimes A \cup M \otimes B, E) = \mathbf{hom}(N \otimes A, E) \times_{\mathbf{hom}(M \otimes A, E)} \mathbf{hom}(M \otimes B, E)
\end{equation*}
is a trivial fibration of simplicial sets. We already know it is a fibration, since by Proposition \ref{prop:newnormalmonopoprod} the pushout-product under consideration is a normal monomorphism between normal objects. Using the maps $\beta$ of Lemma \ref{lem:exphomotenriched} we may form the diagram
\[
\xymatrix{
\mathbf{hom}(N \otimes B, E) \ar[r]\ar[d] & \mathbf{hom}(B, E)^N \ar[d] \\
\mathbf{hom}(N \otimes A, E) \times_{\mathbf{hom}(M \otimes A, E)} \mathbf{hom}(M \otimes B, E) \ar[r] & \mathbf{hom}(A, E)^N \times_{\mathbf{hom}(A, E)^M} \mathbf{hom}(B, E)^M.
}
\]
By Lemmas \ref{lem:forestsetsH2} and \ref{lem:exphomotenriched} we know that the horizontal maps are trivial fibrations. Furthermore, the right vertical map is a trivial fibration. Indeed, $\mathbf{hom}(B,E) \rightarrow \mathbf{hom}(A,E)$ is a trivial fibration by assumption, so that this follows from the fact that the Kan-Quillen model structure on simplicial sets is Cartesian. We may now conclude that the left vertical map in our diagram is a trivial fibration as well. $\Box$
\end{proof}

We can now provide the promised characterization of fibrations between fibrant objects in the operadic model structure:

\begin{proof}[Proof of Theorem \ref{thm:basicmodelstructure}(iv)]
Let $Y \longrightarrow X$ be a map between fibrant objects. If it is a fibration, Lemma \ref{lem:basicanodtrivcof} shows that it has the right lifting property with respect to operadic anodyne maps. Conversely, suppose $Y \longrightarrow X$ has this right lifting property. Factor the map as 
\[
\xymatrix{
Y \ar[r]^i & Z \ar[r]^p & X
}
\]
where $p$ is a fibration and $i$ is a trivial cofibration. Since $Y$ is fibrant, the map $i$ has a retract $r: Z \longrightarrow Y$. Next, the map
\begin{equation*}
J \otimes Y \cup \partial J \otimes Z \longrightarrow J \otimes Z
\end{equation*}
is a trivial cofibration by Lemma \ref{lem:halfhomotenrich}. Thus we can lift in the diagram
\[
\xymatrix@C=45pt{
J \otimes Y \cup \partial J \otimes Z \ar[r]^-{f\varepsilon \cup (p, fr)} \ar[d] & X \\
J \otimes Z \ar@{-->}[ur] & 
}
\]
because $X$ is fibrant as well. This gives a homotopy $h$ from $fr$ to $p$ relative to $Y$. Finally, let $(J \otimes Y) \cup_Y Z$ be the pushout along $i_0: Y \longrightarrow J \otimes Y$ and lift in 
\[
\xymatrix@C=40pt{
(J \otimes Y) \cup_Y Z \ar[r]^-{\varepsilon \cup r}\ar[d] & Y \ar[d]^f \\
J \otimes Z \ar@{-->}[ur]^k \ar[r]_h & X
}
\]
This is possible because the map on the left is an operadic anodyne. Then $r' = k_1$ has the property that $fr' = h_1 = p$ and $r' i = \varepsilon i_1 = \mathrm{id}_Y$. So $f$ is a retract of $p$ over $X$ and hence a fibration, since $p$ is. $\Box$
\end{proof}

\begin{lemma}
\label{lem:basictrivcofs}
The class of operadic trivial cofibrations is the smallest class $\mathcal{C}$ of cofibrations between forest sets containing the operadic anodynes and satisfying the following cancellation property: if
\[
\xymatrix{
A \ar[r]^i & B \ar[r]^j & C
}
\] 
are cofibrations such that $j$ and $ji$ are in $\mathcal{C}$, then so is $i$.
\end{lemma}
\begin{proof}
These arguments are standard. Suppose $f: X \longrightarrow Y$ is a trivial cofibration between forest sets. Construct a square
\[
\xymatrix{
X \ar[r]\ar[d]_f & X' \ar[d]^g \\
Y \ar[r] & Y'
}
\]
in which the top and bottom horizontal maps are operadic anodyne, $X'$ and $Y'$ are fibrant and $g$ is a trivial cofibration. If we can prove that $g$ is an operadic anodyne then we are done. But this follows from the fact that all trivial cofibrations with fibrant codomain are operadic anodyne. Indeed, if $p: C \longrightarrow D$ is such a trivial cofibration, then factor it as an operadic anodyne $q: C \longrightarrow C'$ followed by a map $r: C' \longrightarrow D$ having the right lifting property with respect to all operadic anodynes. Since $D$ is fibrant, $C'$ is also fibrant. Now, by the characterization of fibrations between fibrant objects given in Theorem \ref{thm:basicmodelstructure}(iv) we conclude that $r$ is a trivial fibration. Lifting in the square
\[
\xymatrix{
C \ar[d]_p\ar[r]^q & C' \ar[d]^r \\
D \ar@{=}[r] & D
}
\]
exhibits $p$ as a retract of the operadic anodyne $q$. $\Box$
\end{proof}

The operadic trivial cofibrations can be characterized even a little more efficiently:

\begin{proposition}
\label{prop:trivcofoperadic}
In the operadic model structure on $\mathbf{fSets}$, the class of trivial cofibrations is the smallest hypersaturated class $\mathcal{C}$ containing the morphisms listed below. 
%and satisfying the following closure property (i.e. $\mathc): if $i: A \rightarrowtail B$ and $j: B \rightarrowtail C$ are cofibrations and two of the three maps $i$, $j$ and $ji$ are in $\mathcal{C}$, then so is the third.
\begin{itemize}
\item[(a)] The inner horn inclusions
\begin{equation*}
\Lambda^e[uT] \longrightarrow uT
\end{equation*}
for any tree $T$ and any inner edge $e$ of $T$.
\item[(b)] For any tree $T$, the map
\begin{equation*}
\{0\} \otimes uT \cup J \otimes \partial(uT) \longrightarrow J \otimes uT.
\end{equation*}
\item[(c)] For any non-empty sequence of trees $T_1, \ldots, T_k$, the map
\begin{equation*}
T_1 \amalg \cdots \amalg T_k \longrightarrow T_1 \oplus \cdots \oplus T_k. 
\end{equation*}
\end{itemize}
In fact, we may replace (a) by the following:
\begin{itemize}
\item[(a')] For any tree $T$, the Segal core
\begin{equation*}
\mathrm{fSc}(uT) \longrightarrow uT.
\end{equation*}
\end{itemize}
\end{proposition}
\begin{proof}
It is clear that all the stated maps are trivial cofibrations and that the class of trivial cofibrations has the stated closure property. Conversely, consider the smallest class $\mathcal{C}$ having the stated closure properties and containing (a) and (b) for any \emph{forest} $F$ (instead of $uT$) and the maps described in Lemma \ref{lem:genbasicanodyne}(ii)(c), i.e. the inclusions
\begin{equation*}
\Delta^n \otimes (F \amalg G) \cup \partial \Delta^n \otimes (F \oplus G) \longrightarrow \Delta^n \otimes (F \oplus G)
\end{equation*}
for forests $F$ and $G$. Then Lemmas \ref{lem:genbasicanodyne} and  \ref{lem:basictrivcofs} show that this class is in fact the class of trivial cofibrations in the operadic model structure. \par 
Let us first show that it suffices to include only the $n=0$ version of the maps just listed, i.e. only the maps
\begin{equation}
\label{eq:sums1}
F \amalg G \longrightarrow F \oplus G
\end{equation} 
Indeed, the more general map listed before is of the form
\begin{equation}
\label{eq:sums2}
A \amalg B \cup A' \oplus B' \longrightarrow A \oplus B
\end{equation}
for normal monomorphisms $A' \rightarrowtail A$ and $B' \rightarrowtail B$. Let us first treat the case where $B' = B$ (which, by symmetry, will also cover the case $A' = A$). Such a map is in the saturation of the class of maps of the form
\begin{equation}
\label{eq:sums3}
F \amalg G \cup \partial F \oplus G \longrightarrow F \oplus G
\end{equation}
Form the following diagram, in which the square is a pushout and the vertical maps and the top right horizontal map are in the saturation of the class of maps of the form in (\ref{eq:sums1}):
\[
\xymatrix{
\partial F \amalg G \ar[d]_{\ast} \ar[r] & F \amalg G \ar[r]^{\ast}\ar[d]_{\ast} & F \oplus G \\
\partial F \oplus G \ar[r] & F \amalg G \cup \partial F \oplus G \ar[ur] &
}
\]  
By two-out-of-three, we get the map of (\ref{eq:sums3}) and hence by saturation the maps of (\ref{eq:sums2}) in the special cases $A' = A$ or $B' = B$. To remove this restriction, consider arbitrary normal monos $A' \rightarrowtail A$ and $B' \rightarrowtail B$ and form the following diagram, in which the square is a pushout:
\[
\xymatrix{
A' \amalg B \cup A' \oplus B' \ar[r]^{\ast}\ar[d] & A' \oplus B \ar[d] \ar[dr] & \\
A \amalg B \cup A' \oplus B' \ar[r]_{\ast} & A \amalg B \cup A' \oplus B \ar[r]_-{\ast} & A \oplus B
}
\]
The top horizontal and right bottom horizontal map are of the special form just described. Composing the bottom two horizontal maps gives the map of (\ref{eq:sums2}), so we have succeeded in reducing to (\ref{eq:sums1}). We now wish to get the maps of (\ref{eq:sums1}) from the maps listed in (c) of the proposition. For this, write $F = \bigoplus_i S_i$ and $G = \bigoplus_j T_j$ and form the diagram
\[
\xymatrix{
(S_1 \oplus \cdots \oplus S_m) \amalg (T_1 \oplus \cdots \oplus T_n) \ar[r] & (S_1 \oplus \cdots \oplus S_m) \oplus (T_1 \oplus \cdots \oplus T_n) \\
(S_1 \amalg \cdots \amalg S_m) \amalg (T_1 \amalg \cdots \amalg T_n) \ar[u]\ar[ur] & 
}
\] 
The skew map is of the form (c) and the vertical map is a coproduct of such maps.

We still have to demonstrate that it suffices to include the maps (a) and (b) of the proposition, rather than their analogues with $uT$ replaced by a forest $F$. This is done similarly, using the two-out-of-three property and what we already know about sums. For example, for a direct sum $F = S \oplus T$ of two trees, the map
\begin{equation*}
\{0\} \otimes (S \oplus T) \cup J \otimes \partial (S \oplus T) \longrightarrow J \otimes (S \oplus T)
\end{equation*}
fits into a diagram
\[
\xymatrix{
\{0\} \otimes (S \oplus T) \cup J \otimes \partial (S \oplus T) \ar[r] & J \otimes (S \oplus T) \\
\{0\} \otimes (S \oplus T) \cup J \otimes (\partial S \oplus T \cup S \oplus \partial T) \ar@{=}[u] & \\
\{0\} \otimes (S \oplus T) \cup J \otimes (\partial S \amalg T \cup S \amalg \partial T) \ar[u]\ar[r] & J \otimes (S \amalg T) \ar[uu]
}
\]
where the lower left and the right map are in the saturation of the class (c) as just argued and the lower horizontal map is an isomorphism. We leave the remaining case (a) to the reader. To replace (a) by (a') one uses Propositions \ref{prop:Segalcoreinnanod} and \ref{prop:Segcoregeninneran}. $\Box$
\end{proof}

Let us now prove that the operadic model structure is homotopically enriched:

\begin{proof}[Proof of Theorem \ref{thm:basicmodelstructure}(v)]
By Lemma \ref{lem:forestsetsH2} we know that the weak enrichment of $\mathbf{fSets}$ satisfies axiom (H2). It remains to verify (H1). By definition, the statement we have to prove is that given a cofibration $Y \longrightarrow Z$ of forest sets and a cofibration $M \longrightarrow N$ of simplicial sets, the map
\begin{equation*}
N \otimes Y  \cup_{M \otimes Y} M \otimes Z \longrightarrow N \otimes Z
\end{equation*}
is a cofibration, which is trivial if either $Y \longrightarrow Z$ is an operadic weak equivalence or $M \longrightarrow N$ is a weak equivalence in the Joyal model structure. We already know it is a cofibration by Proposition \ref{prop:newnormalmonopoprod}. \par 
(i) Assume $Y \longrightarrow Z$ is trivial. Then the conclusion follows from Lemma \ref{lem:halfhomotenrich}. \par
(ii) Assume $M \longrightarrow N$ is trivial. As before, Lemma \ref{lem:trivcofgenbetweennormals} shows there is no loss of generality if we assume that $Y$ and $Z$ are both normal. Let us write $A \rightarrow B$ for the pushout-product map
\begin{equation*}
N \otimes Y \cup_{M \otimes Y} M \otimes Z \longrightarrow N \otimes Z.
\end{equation*}
Now consider an arbitrary $K \rightarrow L$ of simplicial sets and write $A' \rightarrow B'$ for the similar pushout-product map with $M \rightarrow N$ replaced by $K \rightarrow L$. By symmetry of the enrichment, 
\begin{equation*}
\mathbf{hom}(B, E) \longrightarrow \mathbf{hom}(A,E)
\end{equation*}
has the right lifting property with respect to $K \rightarrow L$ if and only if
\begin{equation*}
\mathbf{hom}(B', E) \longrightarrow \mathbf{hom}(A',E)
\end{equation*}
has the right lifting property with respect to $M \rightarrow N$. It has this lifting property because it is a fibration in the Joyal model structure, since $A' \rightarrow B'$ is a normal monomorphism between normal objects. Consequently, the map $\mathbf{hom}(B, E) \rightarrow \mathbf{hom}(A,E)$ is a trivial fibration of simplicial sets and $A \rightarrow B$ is an operadic weak equivalence, as was to be shown.
$\Box$
\end{proof}

We finish this section with a characterization of the weak equivalences between fibrant objects in the operadic model structure:

\begin{proposition}
\label{prop:operadicwes}
Let $X$ and $Y$ be operadically local objects of $\mathbf{fSets}$. A map $f: X \longrightarrow Y$ is a weak equivalence in the operadic model structure if and only if the following two conditions hold:
\begin{itemize}
\item[(i)] For every corolla $C_n$, the map
\begin{equation*}
\mathbf{hom}(C_n, X) \longrightarrow \mathbf{hom}(C_n, Y)
\end{equation*}
is an equivalence of $\infty$-categories.
\item[(ii)] The map $i^*u^*X \longrightarrow i^*u^*Y$ of underlying $\infty$-categories is an equivalence.
\end{itemize}
\end{proposition}
\begin{remark}
The `fully faithful' part of (ii) is already implied by (i), since $i^*u^*X = \mathbf{hom}(\eta, X)$ and $\mathbf{hom}(C_1, X) = i^*u^*X^{\Delta^1}$. So we may replace (ii) by the weaker condition that the functor 
\begin{equation*}
\tau i^*u^*X \longrightarrow \tau i^*u^*Y
\end{equation*}
between ordinary small categories is essentially surjective.
\end{remark}
\begin{proof}[Proof of Proposition \ref{prop:operadicwes}]
For the direct implication, note that the operadic model structure being homotopically enriched in particular implies that
\begin{equation*}
\mathbf{hom}(A, - ): \mathbf{fSets} \longrightarrow \mathbf{sSets}
\end{equation*}
is a right Quillen functor for any cofibrant forest set $A$ and therefore preserves weak equivalences between fibrants. For the converse, we will first show that 
\begin{equation*}
\mathbf{hom}(A, X) \longrightarrow \mathbf{hom}(A, Y)
\end{equation*}
is a weak equivalence for every normal forest set $A$.  In case $A$ is a tree $T$, consider the diagram
\[
\xymatrix{
\mathbf{hom}(T, X) \ar[r]\ar[d] & \mathbf{hom}(T, Y) \ar[d] \\
\mathbf{hom}(\mathrm{fSc}(T), X) \ar[r] & \mathbf{hom}(\mathrm{fSc}(T), Y)
}
\]
where the vertical maps are trivial fibrations. To prove that the top horizontal map is weak equivalence, it suffices to prove that the lower horizontal map is a weak equivalence. Now $\mathrm{fSc}(T)$ is a colimit of a finite diagram whose objects are direct sums of corollas and copies of $\eta$ and whose maps are normal monomorphisms. One deduces that this diagram is in fact a homotopy colimit and that similarly the diagram formed by applying $\mathbf{hom}(-,X)$ to it is a homotopy limit diagram. Therefore, to check that $\mathbf{hom}(\mathrm{fSc}(T), X) \longrightarrow \mathbf{hom}(\mathrm{fSc}(T), Y)$ is a weak equivalence, it suffices to check this assertion with $\mathrm{fSc}(T)$ replaced by the constituent objects of the homotopy colimit diagram for $\mathrm{fSc}(T)$. The map is then a weak equivalence by our assumptions (i) (for corollas) and (ii) (for $\eta$). Note that we are using the fact that $X$ and $Y$ are operadically local to reduce from sums of corollas to individual corollas here. Similarly, this also allows us to treat the case where $A$ is a forest $F$ rather than just a single tree $T$. To handle the case where $A$ is a general normal object, we proceed by skeletal induction. Indeed, suppose we are given a pushout
\[
\xymatrix{
\partial F \ar[d]\ar[r] & A \ar[d] \\
F \ar[r] & B
}
\]
and assume that the statement is true for $\partial F$, $F$ and $A$. Then in the cube
\[
\xymatrix{
\mathbf{hom}(B, E_1) \ar[rr]\ar[dr]\ar@{->>}[dd] & & \mathbf{hom}(B, E_2) \ar@{->>}'[d][dd]\ar[dr] & \\
& \mathbf{hom}(F, E_1) \ar@{->>}[dd]\ar[rr]^(.4)\sim & & \mathbf{hom}(F, E_2) \ar@{->>}[dd] \\
\mathbf{hom}(A, E_1) \ar'[r][rr]^(.1)\sim\ar[dr] & & \mathbf{hom}(A, E_2)\ar[dr] & \\
& \mathbf{hom}(\partial F, E_1) \ar[rr]^(.4)\sim & & \mathbf{hom}(\partial F, E_2)
}
\]
all vertices are $\infty$-categories and the vertical maps are all fibrations. Therefore the left and right squares, which are pullbacks, are in fact also homotopy pullbacks and the map
\begin{equation*}
\mathbf{hom}(B, E_1) \longrightarrow \mathbf{hom}(B, E_2)
\end{equation*}
must be an equivalence of $\infty$-categories as well, which finishes the induction. To deduce that $X \rightarrow Y$ is a weak equivalence, construct a square
\[
\xymatrix{
X' \ar@{->>}[d]\ar[r]^{f'} & Y' \ar@{->>}[d] \\
X \ar[r]_f & Y
}
\]
in which the vertical maps are normalizations. By what we just proved, the map
\begin{equation*}
f'_*: \mathbf{hom}(Y', X') \longrightarrow \mathbf{hom}(Y', Y')
\end{equation*}
is an equivalence of $\infty$-categories and hence a $J$-homotopy equivalence. Choose a homotopy inverse $\phi$ of $f'_*$ and set $g' = \phi(\mathrm{id}_{Y'})$. Then clearly $f'$ and $g'$ are part of a $J$-homotopy equivalence between $X'$ and $Y'$. In particular $f'$, and hence $f$, is an operadic weak equivalence. $\Box$
\end{proof}

\subsection{The equivalence of forest sets and dendroidal sets}
\label{sec:equivfsetsdsets}

The goal of this section is to prove the following result:

\begin{theorem}
\label{thm:equivfSetsdSets}
The Quillen pair
\[
\xymatrix@C=40pt{
u^*: \mathbf{fSets} \ar@<.5ex>[r] & \mathbf{dSets}: u_* \ar@<.5ex>[l]
}
\]
between dendroidal sets and forest sets equipped with the operadic model structure is a Quillen equivalence.
\end{theorem}

\begin{remark}
Note that although $u^*u_*$ is isomorphic to the identity functor, $u_*(X)$ is never cofibrant (for non-empty $X$), since elements of $X(T)$ restrict along the codiagonal $T \amalg T \longrightarrow T$ to elements of
\begin{equation*}
u_*X(T \oplus T) \, = \, \mathbf{dSets}(T \amalg T, X)
\end{equation*}
which are invariant under the twist isomorphism of $T \oplus T$. So we cannot conclude a similar identity for the composition $(\mathbf{L}u^*)(\mathbf{R}u_*)$ of derived functors. Similarly, $u^*$ of a fibrant object is rarely fibrant because $u_!$ does not send all inner horns to monomorphisms, cf. Remark \ref{rmk:warning} above. So the calculation of the composition $(\mathbf{R}u_*)(\mathbf{L}u^*)$ is far from that of $u_*u^*$. For these reasons, we'll have to prove the theorem above in a rather roundabout way, using simplicial presheaves.
\end{remark}

\emph{Simplicial presheaves on $\mathbf{\Omega}$.} We recall some results from \cite{cisinskimoerdijk2} concerning simplicial presheaves on $\mathbf{\Omega}$, i.e. \emph{dendroidal spaces}. The category of dendroidal spaces is of course identical to that of simplicial objects in dendroidal sets,
\begin{equation*}
\mathbf{dSets}^{\mathbf{\Delta}^{\mathrm{op}}} \, = \, \mathbf{Sets}^{(\mathbf{\Delta}\times\mathbf{\Omega})^{\mathrm{op}}} \, = \, \mathbf{sSets}^{\mathbf{\Omega}^{\mathrm{op}}},
\end{equation*}
so we can study its homotopy theory in two ways: one departing from the Reedy model structure on $\mathbf{dSets}^{\mathbf{\Delta}^{\mathrm{op}}}$, the other departing from the generalized Reedy model structure (cf. \cite{bergermoerdijkReedy}) on $\mathbf{sSets}^{\mathbf{\Omega}^{\mathrm{op}}}$. For the first of these, the adjoint functors
\[
\xymatrix@C=40pt{
\mathbf{dSets} \ar@<.5ex>[r]^{\mathrm{con}} & \mathbf{dSets}^{\mathbf{\Delta}^{\mathrm{op}}} \ar@<.5ex>[l]^{\mathrm{ev}_0}
}
\]
given by the constant simplicial objects and the evaluation at the object $[0]$ of $\mathbf{\Delta}$ form a Quillen pair. This Quillen pair can easily be turned into a Quillen equivalence by forcing the fibrant (i.e. `local') objects in $\mathbf{dSets}^{\mathbf{\Delta}^{\mathrm{op}}}$ to be homotopically constant. More precisely, we can consider the left Bousfield localization of $\mathbf{dSets}^{\mathbf{\Delta}^{\mathrm{op}}}$ whose local objects $X$ have the property that the face maps $d_i: X_n \longrightarrow X_{n-1}$, which are fibrations for any object that is Reedy fibrant, are actually trivial fibrations of dendroidal sets. Equivalently, one can force the maps 
\begin{equation*}
X_n \, = \, \mathbf{Hom}(\Delta^n, X) \longrightarrow \mathbf{Hom}(\Lambda^n_k, X)
\end{equation*}
to be trivial fibrations of dendroidal sets for $0 \leq k \leq n$. Thus, the localized model structure which makes $\mathbf{dSets}^{\mathbf{\Delta}^{\mathrm{op}}}$ equivalent to $\mathbf{dSets}$ is completely characterized by forcing three classes of normal monomorphisms (i.e. Reedy cofibrations) to be trivial cofibrations:
\begin{itemize}
\item[($\alpha$)] For any tree $T$, any inner edge $e$ in $T$ and any $n \geq 0$,
\begin{equation*}
\partial\Delta^n \times T \cup \Delta^n \times \Lambda^e[T] \longrightarrow \Delta^n \times T
\end{equation*}
\item[($\beta$)] For any tree $T$ and any $n \geq 0$,
\begin{equation*}
\partial\Delta^n \times (J \otimes T) \cup \Delta^n \times (\{0\} \otimes T \cup J \otimes \partial T) \longrightarrow \Delta^n \times (J \otimes T)
\end{equation*}
\item[($\gamma$)] For any tree $T$ and any $0 \leq k \leq n$,
\begin{equation*}
\Lambda_k^n \times T \cup \Delta^n \times \partial T \longrightarrow \Delta^n \times T
\end{equation*}
\end{itemize}
Indeed, the right lifting property with respect to the first two characterizes being Reedy fibrant, while the last one corresponds to being homotopically constant.  \par 
Now let us start from the category $\mathbf{sSets}^{\mathbf{\Omega}^{\mathrm{op}}}$ of dendroidal spaces. The fibrant objects for the (generalized) Reedy structure are now exactly the ones having the right lifting property with respect to the class of maps ($\gamma$). So if one localizes further, to ask local objects to have the right lifting property with respect to ($\alpha$) and ($\beta$), one obtains an identical model category. For a Reedy fibrant dendroidal space, the right lifting property with respect to ($\alpha$) means that the fibration of simplicial sets
\begin{equation*}
\mathbf{hom}(T, X) \longrightarrow \mathbf{hom}(\Lambda^e[T], X)
\end{equation*}
is a trivial fibration. The right lifting property with respect to ($\beta$) means that 
\begin{equation*}
\mathrm{ev}_0: X^J \longrightarrow X
\end{equation*}
is a trivial fibration. The first is a Segal condition, the second is a \emph{completeness} condition similar to the condition for Rezk's complete Segal spaces. Therefore, the Reedy fibrant dendroidal spaces which are local with respect to ($\alpha$) and ($\beta$) are called \emph{dendroidal complete Segal spaces}, cf. \cite{cisinskimoerdijk2}. In conclusion, we have the following theorem:

\begin{theorem}
\label{thm:dcomplSegspace}
There are Quillen equivalences
\[
\xymatrix@C=30pt{
\mathbf{dSets} \ar@<.5ex>[r]^-{\mathrm{con}} & \Bigl(\mathbf{dSets}^{\mathbf{\Delta^{\mathrm{op}}}}\Bigr)_{\mathrm{Reedy,con}} \ar@<.5ex>[l]^-{\mathrm{ev}_0} \ar@{=}[r] & \Bigl(\mathbf{sSets}^{\mathbf{\Omega^{\mathrm{op}}}}\Bigr)_{\mathrm{Reedy,Segal,complete}}
}
\]
where the middle category is equipped with the Reedy model structure localized for homotopically constant objects and the right-hand one is equipped with the generalized Reedy model structure localized for dendroidal complete Segal spaces. 
\end{theorem}

Exactly the same pattern of reasoning applies to the category $\mathbf{fSets}$ with its operadic model structure:

\begin{theorem}
\label{thm:fcomplSegspace}
There are Quillen equivalences
\[
\xymatrix@C=30pt{
\mathbf{fSets} \ar@<.5ex>[r]^-{\mathrm{con}} & \Bigl(\mathbf{fSets}^{\mathbf{\Delta^{\mathrm{op}}}}\Bigr)_{\mathrm{Reedy},\lambda} \ar@<.5ex>[l]^-{\mathrm{ev}_0} \ar@{=}[r] & \Bigl(\mathbf{sSets}^{\mathbf{\Phi^{\mathrm{op}}}}\Bigr)_{\mathrm{Reedy},\mu}
}
\]
where $\lambda$ is the localization of the Reedy model structure for homotopically constant objects and $\mu$ is the corresponding localization of the generalized Reedy model structure.
\end{theorem}

We will need the explicit descriptions of $\lambda$ and $\mu$ later and these will also constitute a proof of the theorem.

\begin{proof}
For a Reedy fibrant object $X$ in the middle category, the localization $\lambda$ forces each
\begin{equation*}
\mathbf{Hom}(\Delta^n, X) \longrightarrow \mathbf{Hom}(\Lambda_k^n, X)
\end{equation*}
to be a trivial fibration in $\mathbf{fSets}$. So the fibrant (i.e. local) objects in $\bigl(\mathbf{fSets}^{\mathbf{\Delta^{\mathrm{op}}}}\bigr)_{\mathrm{Reedy},\lambda}$ are completely characterized by having the right lifting property with respect to the following four classes of cofibrations in $\bigl(\mathbf{fSets}^{\mathbf{\Delta^{\mathrm{op}}}}\bigr)_{\mathrm{Reedy}}$:
\begin{itemize}
\item[($f\alpha$)]  For any tree $T$, any inner edge $e$ in $T$ and any $n \geq 0$,
\begin{equation*}
\partial\Delta^n \times uT \cup \Delta^n \times \Lambda^e[uT] \longrightarrow \Delta^n \times uT
\end{equation*}
\item[($f\beta$)] For any tree $T$ and any $n \geq 0$,
\begin{equation*}
\partial\Delta^n \times (J \otimes uT) \cup \Delta^n \times (\{0\} \otimes uT \cup J \otimes \partial uT) \longrightarrow \Delta^n \times (J \otimes uT)
\end{equation*}
\item[($f\beta '$)] For any $n \geq 0$ and any non-empty sequence of trees $T_1, \ldots, T_k$,
\begin{equation*}
\partial\Delta^n\times (uT_1 \oplus \cdots \oplus uT_k) \cup \Delta^n \times (uT_1 \amalg \cdots \amalg uT_k) \longrightarrow \Delta^n \times (uT_1 \oplus \cdots \oplus uT_k)
\end{equation*}
\item[($f\gamma$)] For any forest $F$ and any $0 \leq k \leq n$,
\begin{equation*}
\Lambda_k^n \times F \cup \Delta^n \times \partial F \longrightarrow \Delta^n \times F
\end{equation*}
\end{itemize} 
Here the right lifting property with respect to the first three classes expresses Reedy fibrancy with respect to the operadic model structure, while the last one relates to being homotopically constant. So we let $\lambda$ be the localization with respect to ($f\gamma$). Notice that in the presence of ($f\beta '$), it is equivalent to require ($f\gamma$) only for trees. \par
To prove the theorem, we have to describe a set $\mu$ which when added to the generating family of trivial cofibrations for the generalized Reedy model structure on $\mathbf{sSets}^{\mathbf{\Phi}^{\mathrm{op}}}$ yields the family given by ($f\alpha$),($f\beta$),($f\beta '$) and ($f\gamma$). But notice that ($f\gamma$) expresses precisely Reedy fibrancy in $\mathbf{sSets}^{\mathbf{\Phi}^{\mathrm{op}}}$, so we let $\mu$ be the class of maps given by ($f\alpha$) (for `Segal forest spaces'), ($f\beta$) (for `complete' ones) and ($f\beta '$) (for locality with respect to sums). This proves the theorem. $\Box$
\end{proof}

For the proof of Theorem \ref{thm:equivfSetsdSets}, we will first consider a different but Quillen equivalent model structure on simplicial presheaves, viz. the \emph{projective} one. For an arbitrary small category $\mathbf{C}$, the projective model structure on the category $\mathbf{sSets}^{\mathbf{C}^{\mathrm{op}}}$ of simplicial presheaves is characterized by the fact that a map $X \longrightarrow Y$ is a fibration or a weak equivalence precisely if it is `pointwise' so, i.e. if $X(C) \longrightarrow Y(C)$ is one for every object $C$ of $\mathbf{C}$ (with respect to the Quillen model structure on $\mathbf{sSets}$). Its generating cofibrations are of the form
\begin{equation*}
\partial\Delta^n \times C \longrightarrow \Delta^n \times C
\end{equation*}
where $C$ ranges over the objects of $\mathbf{C}$, viewed as representable presheaves.

\begin{lemma}
\label{lem:projQequiv}
\begin{itemize}
\item[(i)] With respect to the projective model structures, the embedding $u: \mathbf{\Omega} \longrightarrow \mathbf{\Phi}$ induces \emph{two} Quillen pairs:
\[
\xymatrix@C=30pt{
\mathbf{sSets}^{\mathbf{\Omega}^{\mathrm{op}}} \ar@<.5ex>[r]^{u_!} & \mathbf{sSets}^{\mathbf{\Phi}^{\mathrm{op}}} \ar@<.5ex>[l]^{u^*} \ar@<.5ex>[r]^{u^*} & \mathbf{sSets}^{\mathbf{\Omega}^{\mathrm{op}}} \ar@<.5ex>[l]^{u_*}
}
\]
\item[(ii)] For the left Bousfield localization of $\mathbf{sSets}^{\mathbf{\Phi}^{\mathrm{op}}}$ with respect to the maps ($f\beta '$), both pairs are Quillen equivalences.
\end{itemize}
\end{lemma}
\begin{proof}
(i). Since $u_!$ preserves representables, it is clear that $u^*$ preserves fibrations and weak equivalences. So $u^*$ and $u_!$ form a Quillen pair. But $u^*$ sends representables to coproducts of representables, hence sends generating cofibrations to cofibrations. Thus $u^*$ is also a left Quillen functor, so $u^*$ and $u_*$ form a Quillen pair as well. \par 
(ii). Since $u^*$ sends direct sums to coproducts, the pair $(u^*,u_*)$ is also a Quillen pair for the localized model structure on $\mathbf{sSets}^{\mathbf{\Phi}^{\mathrm{op}}}$, while this is automatic for the pair $(u_!,u^*)$. So again, $u^*$ is both a left and a right Quillen functor. Now let $X$ be a fibrant object in $\mathbf{sSets}^{\mathbf{\Omega}^{\mathrm{op}}}$ (always with respect to the projective model structure in this proof). To calculate $(\mathbf{L}u^*)(\mathbf{R}u_*)(X) = (\mathbf{L}u^*)u_*(X)$, take a cofibrant resolution $C \longrightarrow u_*(X)$.  Then this map is a pointwise weak equivalence and hence so is $u^*C \longrightarrow u^*u_*(X)$, which shows that the derived counit
\[
\xymatrix{
(\mathbf{L}u^*)(\mathbf{R}u_*)(X) = u^*(C) \ar[r]^-\sim & u^*u_*(X) \ar[r]^-{\simeq} & X 
}
\]
is a weak equivalence. On the other hand, for a generating cofibrant object $\Delta^n \times F$ in $\mathbf{sSets}^{\mathbf{\Phi}^{\mathrm{op}}}$, the counit map $u_!u^*(\Delta^n \times F) \longrightarrow \Delta^n \times F$ is of the form
\begin{equation*}
\Delta^n \times (T_1 \amalg \cdots \amalg T_k) \longrightarrow \Delta^n \times (T_1 \oplus \cdots \oplus T_k),
\end{equation*}
hence a weak equivalence for the localized structure. By the usual induction on skeletal filtrations, this shows that $u_!u^*(X) \longrightarrow X$ is a weak equivalence for any cofibrant object $X$. But then, if $X$ is cofibrant as well as fibrant, we have that 
\begin{equation*}
\mathbf{L}u_!\mathbf{R}u^*(X) = \mathbf{L}u_!\mathbf{L}u^*(X) = u_!u^*(X) \longrightarrow X
\end{equation*}
is a weak equivalence. Thus at the level of homotopy categories, the derived functor of $u^*$ has both a left and a right inverse (up to natural isomorphism), hence must be an equivalence of categories. $\Box$
\end{proof}

\begin{lemma}
\label{lem:locQequiv}
Let $f_!: \mathcal{F} \rightleftarrows \mathcal{E}: f^*$ be a Quillen pair and let $\lambda$ be a set of cofibrations between cofibrant objects. Write $\mathcal{F}_\lambda$ for the left Bousfield localization forcing the cofibrations in $\lambda$ to become trivial (assuming it exists) and write $f_!(\lambda)$ for the image of $\lambda$ under $f_!$.
\begin{itemize}
\item[(i)] A fibrant object $E$ in $\mathcal{E}$ is local with respect to $f_!(\lambda)$ if and only if $f^*(E)$ is local with respect to $\lambda$.
\item[(ii)] The same functors also define a Quillen pair $f_!: \mathcal{F}_\lambda \rightleftarrows \mathcal{E}_{f_!(\lambda)}: f^*$ (assuming the localizations exist).
\item[(iii)] If the original pair $\mathcal{F} \rightleftarrows \mathcal{E}$ is a Quillen equivalence, then so is the induced pair $\mathcal{F}_\lambda \rightleftarrows \mathcal{E}_{f_!(\lambda)}$. 
\end{itemize}
\end{lemma}
\begin{proof}
Property (i) is clear from the equivalence $\mathrm{Map}(f_!A, E) \simeq \mathrm{Map}(A, f^*E)$ and (ii) is immediate from the universal property of left Bousfield localization. For (iii), let us now write $\overline{f}_!: \mathcal{F}_\lambda \rightleftarrows \mathcal{E}_{f_!(\lambda)}: \overline{f}^*$ for the induced Quillen pair. Since the cofibrations and trivial fibrations haven't changed and since fibrant objects in $\mathcal{E}_{f_!(\lambda)}$ are a fortiori fibrant in $\mathcal{E}$, we find for such a fibrant object $X$ that
\begin{equation*}
(\mathbf{L}\overline{f}_!)(\mathbf{R}\overline{f}^*)(X) \, = \, (\mathbf{L}f_!)(\mathbf{R}f^*)(X) \, = \, (\mathbf{L}f_!)f^*(X)
\end{equation*}  
which maps to $X$ via a weak equivalence by assumption (even a weak equivalence in $\mathcal{F}$, without localizing). This shows that $\mathbf{R}\overline{f}^*$ is fully faithful when considered as a functor $\mathrm{Ho}\mathcal{E}_{f_!(\lambda)} \longrightarrow \mathrm{Ho}\mathcal{F}_{\lambda}$, so it suffices to prove it is also essentially surjective. Let $Y$ be a fibrant and cofibrant object in $\mathcal{F}_\lambda$. Since $f_!$ and $f^*$ form an equivalence, the map $Y \longrightarrow (\mathbf{R}f^*)(\mathbf{L}f_!)(X)$ is a weak equivalence in $\mathcal{F}$. In other words, if $f_!(Y) \longrightarrow W$ is a fibrant replacement in $\mathcal{E}$, then $Y \longrightarrow f^*W$ is a weak equivalence. But then $f^*W$ is local with respect to $\lambda$ since $Y$ is (because being local is evidently invariant under weak equivalence between fibrant objects). Hence $W$ is local with respect to $f_!\lambda$ by part (i), i.e. fibrant in $\mathcal{E}_{f_!(\lambda)}$. This proves that $\mathbf{R}\overline{f}^*$ is essentially surjective as a functor on homotopy categories. $\Box$
\end{proof}

\begin{proof}[Proof of Theorem \ref{thm:equivfSetsdSets}]
Considering the projective as well as the Reedy model structures on our presheaf categories, we have a diagram of left Quillen functors
\[
\xymatrix{
\Bigl(\mathbf{sSets}^{\mathbf{\Phi}^{\mathrm{op}}}\Bigr)_{\mathrm{proj}} \ar[r]^{u^*}\ar[d] & \Bigl(\mathbf{sSets}^{\mathbf{\Omega}^{\mathrm{op}}}\Bigr)_{\mathrm{proj}} \ar[d] \\
\Bigl(\mathbf{sSets}^{\mathbf{\Phi}^{\mathrm{op}}}\Bigr)_{\mathrm{Reedy}} \ar[r]^{u^*} & \Bigl(\mathbf{sSets}^{\mathbf{\Omega}^{\mathrm{op}}}\Bigr)_{\mathrm{Reedy}}
}
\]
where the vertical identity functors are Quillen equivalences. Now consider the localizations with respect to sums (i.e. ($f\beta '$)) on the left. By Lemmas \ref{lem:projQequiv} and \ref{lem:locQequiv} this turns the left and top functors in the diagram into Quillen equivalences. Hence we also obtain a left Quillen equivalence
\[
\xymatrix{
\Bigl(\mathbf{sSets}^{\mathbf{\Phi}^{\mathrm{op}}}\Bigr)_{\mathrm{Reedy},(f\beta ')} \ar[r]^-{u^*} & \Bigl(\mathbf{sSets}^{\mathbf{\Omega}^{\mathrm{op}}}\Bigr)_{\mathrm{Reedy}}
}
\]
Now observe that $u^*$ sends the maps in the classes $(f\alpha)$, $(f\beta)$ and $(f\gamma)$ to the similar classes $(\alpha)$, $(\beta)$ and $(\gamma)$. Lemma \ref{lem:locQequiv}(iii) yields an equivalence
\[
\xymatrix{
\Bigl(\mathbf{sSets}^{\mathbf{\Phi}^{\mathrm{op}}}\Bigr)_{\mathrm{Reedy},\mu} \ar[r]^-{u^*} & \Bigl(\mathbf{sSets}^{\mathbf{\Omega}^{\mathrm{op}}}\Bigr)_{\mathrm{Reedy},u^*\mu}
}
\]
Now consider the diagram
\[
\xymatrix@C=50pt{
\mathbf{fSets} \ar[r]^{u^*}\ar[d]_{\mathrm{con}} & \mathbf{dSets} \ar[d] \\
\Bigl(\mathbf{sSets}^{\mathbf{\Phi}^{\mathrm{op}}}\Bigr)_{\mathrm{Reedy},\mu} \ar[r]^-{u^*} & \Bigl(\mathbf{sSets}^{\mathbf{\Omega}^{\mathrm{op}}}\Bigr)_{\mathrm{Reedy},u^*\mu}
}
\]
in which the vertical functors are left Quillen equivalences by Theorems \ref{thm:dcomplSegspace} and \ref{thm:fcomplSegspace}. We conclude that the top functor is also a left Quillen equivalence. $\Box$
\end{proof}

%% file: markedfsets.tex
\section{Marked dendroidal and forest sets}
\label{chap:markedfsets}

Before we can set up a useful functor relating the category $\mathbf{POp}_o$ of non-unital preoperads to the category of open forest sets, we have to introduce markings into our categories. We will treat the categories $\mathbf{dSets}^+$ and $\mathbf{fSets}^+$ of \emph{marked dendroidal sets} and \emph{marked forest sets} respectively, which are defined analogously to the category of marked simplicial sets. We will establish (simplicial, combinatorial) model structures on these categories and show that they are Quillen equivalent to $\mathbf{dSets}$ and $\mathbf{fSets}$ (with their operadic model structures) respectively. The results of this chapter can be summarized in a commutative square of left Quillen functors, all of which are part of Quillen equivalences:
\[
\xymatrix{
\mathbf{dSets} \ar[d]_{(-)^\flat} & \mathbf{fSets} \ar[l]_-{u^*} \ar[d]^{(-)^\flat} \\
\mathbf{dSets}^+ & \mathbf{fSets}^+ \ar[l]^-{u^*}
}
\] 

The introduction of markings somewhat complicates notation. To not clutter things up too much, we will in this chapter mostly omit the functors $u_!$ and $i_!$ from the notation. When we write $T$ for a tree, it should be clear from the context whether it is to be interpreted as a representable dendroidal set or a representable forest set. We will also use the notations $C_1$ and $\Delta^1$ interchangeably for the 1-corolla as a representable forest set, which would strictly speaking have to be $u_!C_1$ and $u_!i_!\Delta^1$, respectively.

\subsection{Marked forest sets}
\label{sec:markedfsets}

For now, we will focus on the category of marked forest sets. In this section we will summarize the main definitions and results. Most proofs are deferred to the following sections. The corresponding results for marked dendroidal sets are established in completely analogous fashion; we will briefly summarize what we need in the last section of this chapter. \par 
A \emph{marked forest set} is a pair $(X, \mathcal{E})$ where $X$ is a forest set and $\mathcal{E}$ is a subset of its set of 1-corollas $X(C_1)$ containing all degenerate 1-corollas. A morphism of marked forest sets is a map of forest sets sending marked 1-corollas to marked 1-corollas. We denote the category of such marked forest sets by $\mathbf{fSets}^+$. A marked forest set is \emph{open} if its underlying forest set is open and we denote the full subcategory of open marked forest sets by $\mathbf{fSets}_o^+$. \par 
There is an obvious forgetful functor $a: \mathbf{fSets}^+ \longrightarrow \mathbf{fSets}$. This functor has a left adjoint $(-)^\flat$ and a right adjoint $(-)^\sharp$. These three functors obviously preserve the property of being open. For a forest set $X$, the marked forest set $X^\flat$ is $X$ with only degenerate 1-corollas marked and $X^\sharp$ is $X$ with all its 1-corollas marked. The tensor product on $\mathbf{fSets}$ can be used to define a tensor product on $\mathbf{fSets}^+$ by simply setting
\begin{equation*}
(X, \mathcal{E}_X) \otimes (X, \mathcal{E}_Y) \, := \, (X \otimes Y, \mathcal{E}_X \times \mathcal{E}_Y).
\end{equation*}
Similarly, we can use this tensor to supply $\mathbf{fSets}^+$ with a weak enrichment over the category of marked simplicial sets as follows. Define a marked simplicial set $\mathrm{Map}^+(X,Y)$ by
\begin{eqnarray*}
\mathrm{Map}^+(X,Y)_n & = & \mathbf{fSets}^+((\Delta^n)^\flat \otimes X, Y), \\
\mathcal{E}_{\mathrm{Map}^+(X,Y)} & = & \mathbf{fSets}^+((\Delta^1)^\sharp \otimes X, Y).
\end{eqnarray*}
Also, define simplicial mapping objects by
\begin{eqnarray*}
\mathrm{Map}^\flat(X,Y) & := & a\mathrm{Map}^+(X, Y) 
\end{eqnarray*}
(i.e. forget the markings) and let $\mathrm{Map}^\sharp(X,Y)$ be the simplicial subset of $\mathrm{Map}^\flat(X,Y)$ consisting of the simplices all of whose edges are marked in $\mathrm{Map}^+(X,Y)$. These mapping objects can be characterized by the following natural isomorphisms, for $K$ a simplicial set:
\begin{eqnarray*} 
\mathbf{sSets}(K,\mathrm{Map}^\flat(X,Y)) & \simeq & \mathbf{fSets}^+(K^\flat \otimes X, Y), \\
\mathbf{sSets}(K, \mathrm{Map}^\sharp(X,Y)) & \simeq & \mathbf{fSets}^+(K^\sharp \otimes X, Y).
\end{eqnarray*}
Note that on the right-hand side we interpret $K$ as a forest set via the embedding $u_!i_!: \mathbf{sSets} \longrightarrow \mathbf{fSets}$. For the second isomorphism we have used the fact that the functor $(-)^\sharp: \mathbf{sSets} \longrightarrow \mathbf{sSets}^+$ is left adjoint to the functor taking a marked simplicial set to the simplicial set consisting of all simplices whose edges are marked. \par 
Let us now introduce the terminology and notation necessary to describe the model structure on $\mathbf{fSets^+}$ that we need:

\begin{definition}
A map $f: X \longrightarrow Y$ of marked forest sets is called a \emph{normal monomorphism} if the underlying map $a(f)$ between forest sets is a normal monomorphism. Also, a marked forest set $X$ is \emph{normal} precisely if $aX$ is a normal forest set. A \emph{normalization} of a marked forest set $X$ is a map $X' \longrightarrow X$ from a normal marked forest set $X'$ to $X$, having the right lifting property with respect to all normal monomorphisms.
\end{definition}

\begin{remark}
\label{rmk:gencofmarkedfSets}
The class of normal monomorphisms in $\mathbf{fSets}^+$ is the smallest saturated class containing the following maps:
\begin{itemize}
\item[(i)] All boundary inclusions of forests, with minimal markings. In other words, for every forest $F$, the map $(\partial F)^\flat \longrightarrow F^\flat$.
\item[(ii)] The map $(C_1)^\flat \longrightarrow (C_1)^\sharp$.
\end{itemize}
By the small object argument, every map between marked forest sets can be factored into a normal mono followed by a map having the right lifting property with respect to all normal monos. In particular, every marked forest set admits a normalization.
\end{remark}

\begin{definition}
If $E$ is an operadically local forest set, then an \emph{equivalence in $E$} is a 1-corolla of $E$ which is an equivalence in the underlying $\infty$-category $i^*u^*E$ of $E$. We denote by $E^\natural$ the marked forest set obtained from $E$ by marking all 1-corollas which are equivalences in $E$.
\end{definition}

The following result should provide intuition for the role the markings play. We will prove it in Section \ref{sec:markedfsets3}.

\begin{proposition}
\label{prop:markedmappingspace}
Let $A$ be a marked forest set and let $E$ be an operadically local forest set. Suppose that $A$ is normal. Then $\mathrm{Map}^\flat(A, E^\natural)$ is an $\infty$-category and $\mathrm{Map}^\sharp(A, E^\natural)$ is the largest Kan complex contained in it.
\end{proposition}

\begin{definition}
A map $f: X \longrightarrow Y$ between marked forest sets is called a \emph{marked equivalence} if there exists a commutative square
\[
\xymatrix{
X' \ar@{->>}[d]\ar[r] & Y' \ar@{->>}[d] \\
X \ar[r] & Y
}
\]
where the vertical maps are normalizations and $X' \longrightarrow Y'$ induces an equivalence of $\infty$-categories
\begin{equation*}
\mathrm{Map}^\flat(Y', E^\natural) \longrightarrow \mathrm{Map}^\flat(X', E^\natural)
\end{equation*}
for every operadically local forest set $E$.
\end{definition} 

\begin{remark}
This definition of marked equivalence is independent of the choice of normalizations. More precisely, if $X \longrightarrow Y$ is a marked equivalence, then for \emph{any} commutative square
\[
\xymatrix{
X'' \ar@{->>}[d]\ar[r] & Y'' \ar@{->>}[d] \\
X \ar[r] & Y
}
\]
in which the vertical arrows are normalizations, the induced map
\begin{equation*}
\mathrm{Map}^\flat(Y'', E^\natural) \longrightarrow \mathrm{Map}^\flat(X'', E^\natural)
\end{equation*}
is an equivalence of $\infty$-categories. The proof of this fact is virtually identical to that of Lemma \ref{lem:normalizationsindep}, so we leave it to the reader.
\end{remark}

The following theorem is the main result of this chapter. We will prove it in Section \ref{sec:markedfsetsmodelstruct}, after treating the necessary technical preliminaries in Sections \ref{sec:markedfsets2} and \ref{sec:markedfsets3}.

\begin{theorem}
\label{thm:markedfSets}
There exists a left proper, cofibrantly generated model structure on the category $\mathbf{fSets}^+$ such that:
\begin{itemize}
\item[(C)] The cofibrations are the normal monomorphisms.
\item[(W)] The weak equivalences are the marked equivalences.
\end{itemize}
Furthermore, this model structure enjoys the following properties:
\begin{itemize}
\item[(i)] An object is fibrant if and only if it is of the form $E^\natural$, for an operadically local forest set $E$.
\item[(ii)] A map $f$ between fibrant objects is a fibration if and only if it has the right lifting property with respect to all marked anodyne morphisms (see Definition \ref{def:markedanodynes}).
\item[(iii)] With the simplicial structure on $\mathbf{fSets}^+$ corresponding to the mapping objects $\mathrm{Map}^\sharp(-,-)$, the model structure is homotopically enriched over simplicial sets with the Kan-Quillen model structure.
\end{itemize} 
\end{theorem}

\begin{corollary}
\label{cor:equivmarkedunmarked}
The adjunction
\[
\xymatrix@C=40pt{
(-)^\flat: \mathbf{fSets} \ar@<.5ex>[r] & \mathbf{fSets}^+: a \ar@<.5ex>[l]
}
\]
is a Quillen equivalence, as is its restriction to the corresponding subcategories of open objects.
\end{corollary}
\begin{proof}
Clearly $(-)^\flat$ preserves normal monomorphisms. Considering the definitions, we see that it suffices to show this functor preserves weak equivalences between cofibrant objects in order for it to be left Quillen. So let $X \longrightarrow Y$ be a weak equivalence between normal forest sets. We have to check that for every operadically local forest set $E$, the map
\begin{equation*}
\mathrm{Map}^\flat(Y^\flat, E^\natural) \longrightarrow \mathrm{Map}^\flat(X^\flat, E^\natural)
\end{equation*}
is an equivalence of $\infty$-categories. But note that we can canonically identify this map with the map
\begin{equation*}
\mathbf{hom}(Y, E) \longrightarrow \mathbf{hom}(X, E)
\end{equation*}
which is an equivalence by assumption. \par 
Now let us show the adjunction is in fact a Quillen equivalence. Suppose we are given a map 
\begin{equation*}
f: X^\flat \longrightarrow Y^\natural
\end{equation*}
where $X$ is a cofibrant forest set and $Y^\natural$ is a fibrant object of $\mathbf{fSets}^+$. We have to show that $f$ is a weak equivalence if and only if the adjoint map $X \longrightarrow a(Y^\natural) = Y$ is a weak equivalence. Again, making the canonical identifications
\begin{eqnarray*}
\mathrm{Map}^\flat(X^\flat, E^\natural) & \simeq & \mathbf{hom}(X, E) \\
\mathrm{Map}^\flat(Y^\natural, E^\natural) & \simeq & \mathbf{hom}(Y, E)
\end{eqnarray*}
it is clear that this is indeed the case. Note that for the second isomorphism we use the fact that a map between operadically local objects automatically preserves equivalences. This is immediate from the corresponding fact for $\infty$-categories. $\Box$ 
\end{proof}

\subsection{Equivalences in $\infty$-operads}
\label{sec:markedfsets2}

We will need some properties of equivalences in operadically local forest sets which are analogues of similar properties of equivalences in $\infty$-categories established by Joyal (cf. Proposition 1.2.4.3 of \cite{htt}) and equivalences in dendroidal $\infty$-operads (cf. Theorems 4.2 and A.7 of \cite{cisinskimoerdijk1}).

\begin{theorem}
\label{thm:Cartequivalences}
Suppose $E$ is an operadically local forest set. Also, suppose we have a forest $F$ containing a tree $T$ with at least two vertices and having a unary root corolla, whose root we denote by $r$. Then for any lifting problem
\[
\xymatrix{
\Lambda^r[F] \ar[d]\ar[r] & E \\
F \ar@{-->}[ur] &
}
\]
such that the root corolla corresponding to $r$ maps to an equivalence in $E$ under the horizontal map, a lift exists. 
\end{theorem}

\begin{remark}
In fact, the statement of the theorem is only interesting if the forest $F$ consists of only the one tree $T$. If it has multiple components, then a lift will always exist by the fact that $E$ is local with respect to sums. However, the given formulation of the theorem will be convenient in the next section.
\end{remark}

\begin{remark}
\label{rmk:Cartequivalences}
We can reformulate the theorem as follows: given an operadically local forest set $E$, the marked forest set $E^\natural$ has the right lifting property with respect to all maps of the form
\begin{equation*}
(\Lambda^r[F], \mathcal{E} \cap \Lambda^r[F]) \longrightarrow (F, \mathcal{E})
\end{equation*}
where $F$ is a forest as described above and $\mathcal{E}$ consists of all degenerate 1-corollas of $F$ together with the root corolla corresponding to $r$.
\end{remark}

For later use and ease of reference, let us make the following definition:

\begin{definition}
The class of \emph{root anodynes} is the smallest saturated class of morphisms containing the maps
\begin{equation*}
(\Lambda^r[F], \mathcal{E} \cap \Lambda^r[F]) \longrightarrow (F, \mathcal{E}) 
\end{equation*}
where $F$ is a forest containing a tree $T$ with a root corolla of valence one, $\Lambda^r[F]$ is the horn of $F$ corresponding to that root and $\mathcal{E}$ consists of all degenerate 1-corollas of $F$ together with that root corolla.
\end{definition}

Note that, unlike the formulation of Theorem \ref{thm:Cartequivalences}, we're not requiring the tree $T$ to have at least two vertices, so the class of root anodynes also includes the map
\begin{equation*}
\{1\} \longrightarrow (\Delta^1)^\sharp
\end{equation*} 

We will also need the following `dual' version of the previous result: 

\begin{theorem}
\label{thm:coCartequivalences}
Suppose $E$ is an operadically local forest set. Also suppose we have a forest $F$ containing a tree $T$ with at least two vertices and having a unary leaf corolla, whose leaf we denote by $l$. Then for any lifting problem
\[
\xymatrix{
\Lambda^l[F] \ar[d]\ar[r] & E \\
F \ar@{-->}[ur] &
}
\]
such that the leaf corolla corresponding to $l$ maps to an equivalence in $E$ under the horizontal map, a lift exists.
\end{theorem}

\begin{remark}
\label{rmk:coCartequivalences} 
This result admits a similar reformulation, this time in terms of lifting properties with respect to \emph{leaf anodynes}, i.e. compositions of pushouts of maps of the form
\begin{equation*}
(\Lambda^l[F], \mathcal{E} \cap \Lambda^l[F]) \longrightarrow (F, \mathcal{E})
\end{equation*}
where $F$ is a forest containing a tree $T$ with a leaf corolla of valence one, $\Lambda^l[F]$ is the horn of $F$ corresponding to that leaf and $\mathcal{E}$ consists of all degenerate 1-corollas of $F$ together with that leaf corolla. 
\end{remark}

Fortunately, both these theorems can be derived fairly easily from their dendroidal counterparts. We will show how to do this for the first one, the second is similar.

\begin{proof}[Proof of Theorem \ref{thm:Cartequivalences}]
We first recall a useful fact from the theory of model categories. Suppose we have a model category $\mathcal{C}$, a cofibration between cofibrant objects $i: A \longrightarrow B$, a fibrant object $X$ and a lifting problem
\[
\xymatrix{
A \ar[r]^f\ar[d]_i & X \\
B \ar@{-->}[ur] &
}
\]
If there exists a commutative diagram
\[
\xymatrix{
A \ar[r]^{[f]}\ar[d]_{[i]} & X \\
B \ar[ur] &
}
\]
in the homotopy category $\mathrm{ho}\,\mathcal{C}$ of $\mathcal{C}$ (i.e. a `lift up to homotopy'), then the actual lifting problem in $\mathcal{C}$ admits a solution. The proof of this fact is straightforward. If needed, it can be found in \cite{htt} as Proposition A.2.3.1. \par 
We will apply this fact as follows. First, recall that we only have to prove the theorem in case $F$ consists of a single tree $T$; otherwise a lift automatically exists since $E$ is local with respect to sums. Also, we may assume $E$ is cofibrant. Now consider the diagram 
\[
\xymatrix{
\Lambda^r[uT] \ar[d]\ar[r] & E \ar[r]^-\sim & (\mathbf{R}u_*)u^*E \\
uT \ar@{-->}[urr] & & 
}
\]
To solve the lifting problem in the theorem, it suffices (by the fact just mentioned) to find a dashed arrow as in this diagram. But this is equivalent to finding a lift in the following diagram in $\mathbf{dSets}$:
\[
\xymatrix{
\Lambda^r[T] \ar[d]\ar[r] & (u^*E)_f \\
T \ar@{-->}[ur] & 
}
\]
Here the subscript $f$ indicates a fibrant replacement of $u^*E$ and we have used the fact that $u^*\Lambda^r[uT] = \Lambda^r[T]$. Such a lift exists by Theorem 4.2 of \cite{cisinskimoerdijk1}. $\Box$
\end{proof}

Before we move on, it is worthwile to record an important property of the root and leaf anodynes mentioned above.

\begin{proposition}
\label{prop:poprodrootanodyne}
Given a root anodyne $f: X \longrightarrow Y$ between marked forest sets and a normal monomorphism $g: A \longrightarrow B$ between forest sets, the pushout-product
\begin{equation*}
X \otimes B^\flat \cup Y \otimes A^\flat \longrightarrow Y \otimes B^\flat
\end{equation*}
is a composition of root anodynes and inner anodynes, provided that $Y$ or $B$ is simplicial or both $Y$ and $B$ are open.
\end{proposition}

\begin{proposition}
\label{prop:poprodleafanodyne}
Given a leaf anodyne $f: X \longrightarrow Y$ between marked forest sets and a normal monomorphism $g: A \longrightarrow B$ between forest sets, the pushout-product
\begin{equation*}
X \otimes B^\flat \cup Y \otimes A^\flat \longrightarrow Y \otimes B^\flat
\end{equation*}
is a composition of leaf anodynes and inner anodynes, provided that $Y$ or $B$ is simplicial or both $Y$ and $B$ are open.
\end{proposition}

The proof of the second result is not formally dual to the first in any reasonable sense and in fact requires a little more care. Let us start with the first.

\begin{proof}[Proof of Proposition \ref{prop:poprodrootanodyne}]
By standard arguments, it suffices to prove this in the case where $f$ is of the form 
\begin{equation*}
(\Lambda^r[F], \mathcal{E} \cap \Lambda^r[F]) \longrightarrow (F, \mathcal{E})
\end{equation*}
as described in Remark \ref{rmk:Cartequivalences} and $g$ is of the form $\partial G \longrightarrow G$ for some forest $G$. Let us abbreviate notation by writing 
\begin{equation*}
f: \Lambda^r[F]^\diamond \longrightarrow F^\diamond.
\end{equation*}
In fact, we will now restrict our attention to the case where $F$ (resp. $G$) consists of a single tree $S$ (resp. $T$). The general case may be deduced from this one by a method completely analogous to the one at the end of the proof of Proposition \ref{prop:poprodinneranodyne}. Also, let us write $v_r$ for the root vertex of $S$. \par 
As in the proof of Proposition \ref{prop:poprodinneranodyne}, we consider the constituent shuffles of the tensor product $S^\diamond \otimes T^\flat$. This time, we pick the partial ordering on those such that the minimal element is obtained by grafting copies of $S$ onto the leaves of $T$:

\[
\begin{tikzpicture} 
[level distance=15mm, 
every node/.style={fill, circle, minimum size=.1cm, inner sep=0pt}, 
level 1/.style={sibling distance=35mm}, 
level 2/.style={sibling distance=30mm}, 
level 3/.style={sibling distance=20mm}]

%vertex of T
\node(anchorT)[style={color=white}] {} [grow'=up] 
child {node(vertexT)[draw,fill=none] {} 
	child{ node(vertexS1) {}
		child
		child
	}
	child{ node(vertexS2) {}
		child
		child
	}
};

\tikzstyle{every node}=[]

%labels
%T
\node at ($(vertexT) + (0,.7cm)$) {$T$};

%S
\node at ($(vertexS1) + (0,.9cm)$) {$S$};
\node at ($(vertexS2) + (0,.9cm)$) {$S$};
\node at ($(vertexT) + (0, 1.5cm)$) {$\cdots$};

%lines
\draw ($(vertexS1) + (.3cm,-.3cm)$) -- ($(vertexS2) + (-.3cm,-.3cm)$);
\draw ($(vertexS1) + (-1.03cm,1.5cm)$) -- ($(vertexS1) + (1.03cm,1.5cm)$);
\draw ($(vertexS2) + (-1.03cm,1.5cm)$) -- ($(vertexS2) + (1.03cm,1.5cm)$);

\end{tikzpicture} 
\]

The maximal element in this partial order can be pictured as follows:

\[
\begin{tikzpicture} 
[level distance=15mm, 
every node/.style={fill, circle, minimum size=.1cm, inner sep=0pt}, 
level 1/.style={sibling distance=35mm}, 
level 2/.style={sibling distance=30mm}, 
level 3/.style={sibling distance=20mm}]

%vertex of S
\node(anchorS)[style={color=white}] {} [grow'=up] 
child {node(vertexS) {}
 	child{ node(vertexT1)[draw,fill=none] {}
		child
		child
	}
	child{ node(vertexT2)[draw,fill=none] {}
		child
		child
	}
};

\tikzstyle{every node}=[]

%labels
%S
\node at ($(vertexS) + (0,.7cm)$) {$S$};
\node at ($(vertexS) + (.2,-1.1cm)$) {$r$};
\node at ($(vertexS) + (-.2,-.7cm)$) {$v_r$};

%T
\node at ($(vertexT1) + (0,.9cm)$) {$T$};
\node at ($(vertexT2) + (0,.9cm)$) {$T$};
\node at ($(vertexS) + (0, 1.5cm)$) {$\cdots$};

%root vertex
\draw node[fill, circle, minimum size=.1cm, inner sep=0pt] at ($(vertexS) + (0, -.7cm)$){}; 

%lines
\draw ($(vertexT1) + (.3cm,-.3cm)$) -- ($(vertexT2) + (-.3cm,-.3cm)$);
\draw ($(vertexT1) + (-1.03cm,1.5cm)$) -- ($(vertexT1) + (1.03cm,1.5cm)$);
\draw ($(vertexT2) + (-1.03cm,1.5cm)$) -- ($(vertexT2) + (1.03cm,1.5cm)$);

\end{tikzpicture} 
\]

Of course, all our shuffles in fact carry markings induced from $S^\diamond$ and $T^\flat$. We will not make this explicit in the notation. To begin with our induction, define
\begin{equation*}
A_0 := \Lambda^r[S]^\diamond \otimes T^\flat \cup S^\diamond \otimes \partial T^\flat 
\end{equation*}
and notice that the map $A_0 \rightarrow S^\diamond \otimes T^\flat$ is a normal monomorphism by Proposition \ref{prop:newnormalmonopoprod} and our assumptions. Choose a linear ordering on the set of shuffles of $S^\diamond \otimes T^\flat$ extending the partial order described above. Adjoin these shuffles one by one in that order to obtain a filtration
\begin{equation*}
A_0 \subseteq A_1 \subseteq \cdots \subseteq \bigcup_i A_i = S^\diamond \otimes T^\flat.
\end{equation*}
Consider one of the inclusions $A_i \subseteq A_{i+1}$ in this filtration. We have to distinguish two cases: 
\begin{itemize}
\item[Case 1.] The root vertex of the shuffle $R$ that we are adjoining to $A_i$ is \emph{not} of the form $v_r \otimes t$, where $t$ is a colour of $T$ and $v_r$ is the root vertex of $S$. In this case we will show that the map $A_i \subseteq A_{i+1}$ is inner anodyne. 
\item[Case 2.] The root vertex of the shuffle $R$ that we are adjoining to $A_i$ \emph{is} of the form $v_r \otimes t$. In this case we will show that the map $A_i \subseteq A_{i+1}$ is root anodyne.
\end{itemize}
\emph{Case 1.} (This case bears great similarity to what we did in the proof of \ref{prop:poprodinneranodyne}, the only difference being that we are `shuffling $S$ down through $T$' instead of the other way around and that all our trees carry markings, which will not concern us in this case.) The shuffle $R$ will have one or several vertices of the form $v_r \otimes t$, none of which are root vertices. We will refer to the outgoing edges $r \otimes t$ of these vertices as \emph{special edges}. Note that all these are inner edges of $R$, since the $v_r \otimes t$ are not root vertices. Now define a further filtration
\begin{equation*}
A_i =: A_i^0 \subseteq A_i^1 \subseteq \cdots \subseteq \bigcup_j A_i^j = A_{i+1}
\end{equation*}
by adjoining all prunings of $R$ one by one, in an order that extends the partial order of size (i.e. number of vertices of prunings). Consider an inclusion $A_i^j \subseteq A_i^{j+1}$ given by adjoining a pruning $P$ of $R$. Let $\Sigma_P$ denote the intersection of the set of special edges of $R$ with $I(P)$, the set of inner edges of $P$. We may assume this intersection is non-empty: if it is empty, then $P$ will in fact already be contained in $A_0$. Define 
\begin{equation*}
\mathcal{H}_P := I(P) - \Sigma_P 
\end{equation*} 
For each subset $H \subseteq \mathcal{H}_P$, define the tree $P^{[H]}$ as the tree obtained from $P$ by contracting all edges in $\mathcal{H}_P - H$. Pick a linear order on the subsets of $\mathcal{H}_P$ extending the partial order of inclusion and adjoin the trees $P^{[H]}$ to $A_i^j$ in this order to obtain a further filtration
\begin{equation*}
A_i^j =: A_i^{j,0} \subseteq A_i^{j,1} \subseteq \cdots \subseteq \bigcup_k A_i^{j,k} = A_i^{j+1}
\end{equation*}
Finally, consider one of the inclusions $A_i^{j,k} \subseteq A_i^{j,k+1}$ in this filtration, given by adjoining a tree $P^{[H]}$. If the map
\begin{equation*}
P^{[H]} \longrightarrow S^\diamond \otimes T^\flat
\end{equation*}
factors through $A_i^{j,k}$, then the inclusion under consideration is the identity and there is nothing to prove. If it doesn't, we can say the following:
\begin{itemize}
\item[-] Any outer face chopping off a leaf corolla factors through $A_i^j$ by our induction on the size of the prunings.
\item[-] The outer face chopping off the root of $P^{[H]}$ factors through $A_0$.
\item[-] An inner face contracting an edge that is not special (i.e. not contained in $\Sigma_P$) factors through $A_i^{j,k}$ by our induction on the size of $H$.
\item[-] An inner face contracting a special edge, or a composition of inner faces contracting several special edges, cannot factor through an earlier stage of the filtration. Indeed, it cannot factor through an earlier shuffle by the way special edges are defined. Given this, it is clear that it also cannot factor through $A_i^{j'}$ for $j' \leq j$ because of the size of the pruning $P$ under consideration or through $A_i^{j,k'}$ for $k' \leq k$ by the definition of the $P^{[H']}$. 
\end{itemize}
We conclude that the map $A_i^{j,k} \subseteq A_i^{j,k+1}$ is a pushout of the map
\begin{equation*}
\Lambda^{\Sigma_P}[P^{[H]}]^\flat \longrightarrow (P^{[H]})^\flat 
\end{equation*}
which is inner anodyne. \par 
\emph{Case 2.} The root vertex of the shuffle $R$ is of the form $v_r \otimes t$ and the root corolla is marked. Again, define a further filtration
\begin{equation*}
A_i =: A_i^0 \subseteq A_i^1 \subseteq \cdots \subseteq \bigcup_j A_i^j = A_{i+1}
\end{equation*}
by adjoining all prunings of $R$ one by one in an order that extends the partial order of size of prunings. Consider an inclusion $A_i^j \subseteq A_i^{j+1}$ given by adjoining a pruning $P$ of $R$. This time, we consider subsets $H \subseteq I(P)$ and corresponding trees $P^{[H]}$ given by contracting all edges of $P$ contained in $I(P) - H$. Adjoin all the trees $P^{[H]}$ to $A_i^j$ one by one in an order extending the partial order of inclusion of subsets to obtain a further filtration 
\begin{equation*}
A_i^j =: A_i^{j,0} \subseteq A_i^{j,1} \subseteq \cdots \subseteq \bigcup_k A_i^{j,k} = A_i^{j+1}
\end{equation*}
Now consider one of the inclusions $A_i^{j,k} \subseteq A_i^{j,k+1}$ given by adjoining a tree $P^{[H]}$. If the map
\begin{equation*}
P^{[H]} \longrightarrow S^\diamond \otimes T^\flat
\end{equation*}
factors through $A_i^{j,k}$ there is nothing to prove. Note that this is in particular the case if $H$ does not contain the incoming edge of the root vertex. Indeed, if this edge is contracted the resulting tree will factor through $A_0$ if the vertex above the root vertex is black, or through $A_i$ by the Boardman-Vogt relation in case the vertex above the root vertex is white. So let us assume $P^{[H]}$ does not factor through $A_i^{j,k}$ and therefore in particular that the root vertex of $P^{[H]}$ is of the form $v_r \otimes t$. We observe:
\begin{itemize}
\item[-] Any outer face chopping off a leaf corolla of $P^{[H]}$ factors through $A_i^j$ by induction on the size of the prunings.
\item[-] Any inner face factors through $A_i^{j,k}$ by our induction on the size of $H$.
\item[-] The outer face chopping off the unary root corolla of $P^{[H]}$ cannot factor through any earlier stage of the filtration. Indeed, it does not factor through an earlier shuffle. Also, it cannot factor through $A_i^{j'}$ for $j' \leq j$ because of the size of the pruning $P$ under consideration and it cannot factor through $A_i^{j,k'}$ for $k' \leq k$ by the definition of the $P^{[H']}$.
\end{itemize}
We conclude that the inclusion $A_i^{j,k} \subseteq A_i^{j,k+1}$ is a pushout of the map
\begin{equation*}
\Lambda^{\mathrm{root}}[P^{[H]}]^\diamond \longrightarrow (P^{[H]})^\diamond
\end{equation*}
where the superscript $\diamond$ again indicates that the only non-degenerate marked corolla is the root corolla of $P^{[H]}$. In particular, $A_i \subseteq A_{i+1}$ is a composition of pushouts of root anodynes and hence root anodyne. $\Box$
\end{proof}

\begin{proof}[Proof of Proposition \ref{prop:poprodleafanodyne}]
It suffices to prove this in the case where $f$ is of the form
\begin{equation*}
(\Lambda^l[F], \mathcal{E} \cap \Lambda^l[F]) \longrightarrow (F, \mathcal{E})
\end{equation*}
as described in Remark \ref{rmk:coCartequivalences} and $g$ is a boundary inclusion $\partial G \longrightarrow G$, for some forest $G$. Again, for the duration of this proof we abbreviate notation by writing
\begin{equation*}
f: \Lambda^l[F]^\diamond \longrightarrow F^\diamond.
\end{equation*} 
Also, the map of the proposition is a normal monomorphism. To avoid excessive bookkeeping, we again focus on the case where $F$ (resp. $G$) is just a single tree $S$ (resp. $T$). As before, one may use the method of the last part of the proof of Proposition \ref{prop:poprodinneranodyne} to deduce the general case from this one. Let us write $v_l$ for the vertex of the leaf corolla of $S$ under consideration. The leaf or incoming edge of this corolla is $l$ and we denote its outgoing edge by $m$. \par 
Consider the shuffles of the tensor product $S^\diamond \otimes T^\flat$ and put the partial order on these in which the minimal element is given by grafting copies of $T$ onto the leaves of $S$. This partial order is the opposite of the one considered in the previous proof, but coincides with the one used in the proof of Proposition \ref{prop:poprodinneranodyne}. The ideas we are going to employ are similar to what was done before, but for this proof we have to modify our definition of prunings slightly. Given a shuffle $R$, let us define an \emph{$l$-pruning} of $R$ to be a pruning of $R$, i.e. a subtree obtained by iteratively chopping off leaf corollas, satisfying the following extra property:
\begin{itemize}
\item[-] If there is a vertex $v_l \otimes t$ of $R$ whose outgoing edge $m \otimes t$ is contained in $P$, then $v_l \otimes t$ is itself contained in $P$.
\end{itemize}
Let us start our induction. Define 
\begin{equation*}
A_0 \, := \, \Lambda^l[S]^\diamond \otimes T^\flat \cup S^\diamond \otimes \partial T^\flat
\end{equation*} 
Choose a linear ordering on the shuffles of $S^\diamond \otimes T^\flat$ that extends the partial order we fixed before. Adjoin these shuffles one by one in this order to obtain a filtration
\begin{equation*}
A_0 \subseteq A_1 \subseteq \cdots \subseteq \bigcup_i A_i = S^\diamond \otimes T^\flat
\end{equation*}
Consider an inclusion $A_i \subseteq A_{i+1}$ in this filtration given by adjoining a shuffle $R$. Define a further filtration
\begin{equation*}
A_i =: A_i^0 \subseteq A_i^1 \subseteq \cdots \subseteq \bigcup_j A_i^j = A_{i+1}
\end{equation*}
by adjoining the $l$-prunings of $R$ one by one, in an order extending the partial order of size. Now consider an inclusion $A_i^j \subseteq A_i^{j+1}$ given by adjoining an $l$-pruning $P$ of $R$. We have to distinguish two cases:
\begin{itemize}
\item[Case 1.] The pruning $P$ does \emph{not} have any leaf vertices of the form $v_l \otimes t$, for $t$ a colour of $T$. In this case we will show that the map $A_i^j \subseteq A_i^{j+1}$ is inner anodyne.
\item[Case 2.] The pruning $P$ \emph{does} have at least one leaf vertex of the form $v_l \otimes t$ (which is then necessarily marked). In this case we will show that the map $A_i^j \subseteq A_i^{j+1}$ is leaf anodyne.
\end{itemize}
\emph{Case 1.} We may assume the tree $P$ has one or several vertices of the form $v_l \otimes t$, none of which are leaf vertices since we're in Case 1. (Indeed, if $P$ does not contain any such vertices then it is easily verified that $P$ is already contained in $A_0$: by the definition of $l$-pruning, $P$ cannot contain any edges of the form $m \otimes t$ and must therefore be contained in $\partial_m S^\diamond \otimes T^\flat$, which is itself contained in $\Lambda^l[S]^\diamond \otimes T^\flat$.) We will refer to the incoming edges $l \otimes t$ of the vertices $v_l \otimes t$ as \emph{special edges}. All of these are inner edges of $P$ and we denote the collection of these special edges by $\Sigma_P$. Define
\begin{equation*}
\mathcal{H}_P := I(P) - \Sigma_P
\end{equation*}
As usual, we consider trees $P^{[H]}$ obtained from $P$ by contracting the inner edges in $\mathcal{H}_P - H$, where $H$ ranges over the subsets of $\mathcal{H}_P$. These subsets are partially ordered by inclusion and we adjoin the trees $P^{[H]}$ one by one in an order extending this partial order to obtain a further filtration
\begin{equation*}
A_i^j =: A_i^{j,0} \subseteq A_i^{j,1} \subseteq \cdots \subseteq \bigcup_k A_i^{j,k} = A_i^{j,k+1}
\end{equation*}
Consider one of the inlcusions $A_i^{j,k} \subseteq A_i^{j,k+1}$ in this filtration, given by adjoining a tree $P^{[H]}$. If $P^{[H]}$ is already contained in $A_i^{j,k}$, there is nothing to prove. If it doesn't, we can say the following:
\begin{itemize}
\item[-] Any leaf face of $P^{[H]}$ will factor through $A_i^j$ by our induction on the size of $l$-prunings. Indeed, the leaf vertices of $P$ are assumed not to be of the form $v_l \otimes t$, so chopping a leaf vertex off of $P$ yields another $l$-pruning. 
\item[-] The root face of $P^{[H]}$ will factor through $A_0$.
\item[-] An inner face contracting an edge that is not in $\Sigma_P$ factors through $A_i^{j,k}$ by our induction on $H$.
\item[-] An inner face contracting a special edge or a composition of inner faces contracting several special edges cannot factor through any earlier stage of the filtration (cf. the proofs of Propositions \ref{prop:poprodinneranodyne} and \ref{prop:poprodrootanodyne}).
\end{itemize}
We conclude that $A_i^{j,k} \subseteq A_i^{j,k+1}$ is a pushout of
\begin{equation*}
\Lambda^{\Sigma_P}[P^{[H]}]^\flat \longrightarrow (P^{[H]})^\flat
\end{equation*}
and hence inner anodyne. \par 
\emph{Case 2.} The pruning $P$ has at least one (unary) leaf vertex of the form $v_l \otimes t$ and the corolla with this vertex is marked. Consider subsets $H \subseteq I(P)$ and corresponding trees $P^{[H]}$ given by contracting the edges in $I(P) - H$. Adjoin these trees to $A_i^j$ in an order compatible with the natural partial order on the subsets of $I(P)$ to obtain a filtration
\begin{equation*}
A_i^j =: A_i^{j,0} \subseteq A_i^{j,1} \subseteq \cdots \subseteq \bigcup_k A_i^{j,k} = A_i^{j+1}
\end{equation*}
Now consider an inclusion $A_i^{j,k} \subseteq A_i^{j,k+1}$ given by adjoining a tree $P^{[H]}$. If $P^{[H]}$ is contained in $A_i^{j,k}$ there is nothing to prove. Note that this is in particular the case if $H$ does not contain any edges of the form $m \otimes t$ corresponding to a leaf corolla $v_l \otimes t$: if all such edges are contracted, the resulting tree factors either through $\partial_m S^\diamond \otimes T^\flat$ (if they connect two black vertices), or through a previous shuffle, and hence through $A_i$, by the Boardman-Vogt relation (if they connect the black vertices $v_l \otimes t$ to white vertices). Hence, we may assume $P^{[H]}$ has at least one marked unary leaf corolla of the form $v_l \otimes t$. Let us denote the collection of such corollas by $L$. We find the following:
\begin{itemize}
\item[-] Any leaf face \emph{not} chopping off a vertex of the form $v_l \otimes t$ factors through $A_i^j$, by the induction on $l$-prunings.
\item[-] The root face of $P^{[H]}$ factors through $A_0$.
\item[-] Any inner face factors through $A_i^{j,k}$ by the induction on $H$.
\item[-] Any face chopping off a (marked) leaf corolla of the form $v_l \otimes t$ cannot factor through an earlier stage of the filtration. Indeed, such a face cannot factor through an earlier shuffle and chopping off such a corolla would not yield an $l$-pruning.
\end{itemize}
We conclude that the map $A_i^{j,k} \subset A_i^{j,k+1}$ is a pushout of the map
\begin{equation*}
\Lambda^{L}[P^{[H]}]^\diamond \longrightarrow (P^{[H]})^\diamond
\end{equation*}
where the superscript $\diamond$ indicates that the leaf corollas in $L$ are marked. It is easily verified that this is a composition of pushouts of leaf anodynes (analogous to Lemma \ref{lem:moreinneranodynes}(b)) and hence is itself leaf anodyne. $\Box$
\end{proof}

\subsection{Marked anodyne morphisms}
\label{sec:markedfsets3}

The main technical device in proving Theorem \ref{thm:markedfSets} is a good supply of `anodynes':

\begin{definition}
\label{def:markedanodynes}
The class of \emph{strong marked anodyne morphisms} is the smallest saturated class of maps in $\mathbf{fSets}^+$ containing the following:
\begin{itemize}
\item[($M_1$)] For any forest $F$ and any inner edge $e$ in $F$, the inner horn inclusion
\begin{equation*}
\Lambda^e[F]^\flat \longrightarrow F^\flat
\end{equation*}
\item[($M_2$)] The root anodynes, i.e. the inclusions
\begin{equation*}
(\Lambda^r[F], \mathcal{E} \cap \Lambda^r[F]) \longrightarrow (F, \mathcal{E}) 
\end{equation*}
where $F$ is a forest containing a tree $T$ with a root corolla of valence one, $\Lambda^r[F]$ is the horn of $F$ corresponding to that root and $\mathcal{E}$ consists of all the degenerate 1-corollas of $F$ together with that root corolla.
\item[($M_3$)] The map
\begin{equation*}
(\Lambda_1^2)^\sharp \cup_{(\Lambda_1^2)^\flat} (\Delta^2)^\flat \longrightarrow (\Delta^2)^\sharp
\end{equation*}
\item[($M_4$)] The inclusion $J^\flat \subseteq J^\sharp$.
\end{itemize}
Also, the class of \emph{marked anodyne morphisms} is the smallest saturated class containing the strong marked anodynes and the following maps:
\begin{itemize}
\item[($M_5$)] For any $n \geq 0$ and any non-empty sequence $T_1, \ldots, T_k$ of trees, the map
\begin{equation*}
(\partial \Delta^n)^\flat \otimes (T_1 \oplus \cdots \oplus T_k)^\flat \cup (\Delta^n)^\flat \otimes (T_1 \amalg \cdots \amalg T_k)^\flat \longrightarrow (\Delta^n)^\flat \otimes (T_1 \oplus \cdots \oplus T_k)^\flat,  
\end{equation*}
which is a normal monomorphism by \ref{prop:newnormalmonopoprod}.
\end{itemize}
\end{definition}

\begin{remark}
It is useful to note that for any marked anodyne morphism $f$ of marked \emph{simplicial} sets, as defined in \cite{htt}, the morphism $u_!i_!f$ is a strong marked anodyne morphism of marked forest sets.
\end{remark}

The following fact is immediate from Corollary 3.1.1.7 of \cite{htt} and the previous remark:
\begin{lemma}
\label{lem:horn22markedanod}
The map
\begin{equation*}
(\Lambda^2_2)^\sharp \cup_{(\Lambda^2_2)^\flat} (\Delta^2)^\flat \longrightarrow (\Delta^2)^\sharp 
\end{equation*}
is strong marked anodyne.
\end{lemma}

For ease of reference, we record the following crucial property of strong marked anodynes:

\begin{lemma}
\label{lem:poprodmarkedanodyne}
Let $f: X \longrightarrow Y$ be a strong marked anodyne and $g: A \longrightarrow B$ a normal mono. If $Y$ or $B$ is simplicial or both $Y$ and $B$ are open, then the pushout-product
\begin{equation*}
X \otimes B \cup Y \otimes A \longrightarrow Y \otimes B
\end{equation*}
is also strong marked anodyne. 
\end{lemma}
\begin{proof}
By standard arguments, we may restrict our attention to the case where $f$ is one of the generators listed in the previous definition and $g$ is of the form (i) (i.e. $\partial G^\flat \subseteq G^\flat$) or (ii) (i.e. $C_1^\flat \subseteq C_1^\sharp$) as described in Remark \ref{rmk:gencofmarkedfSets}. This gives us eight cases to check.
\begin{itemize} 
\item[($M_1$)(i)] In this case the pushout-product is again inner anodyne by Proposition \ref{prop:poprodinneranodyne}. 
\item[($M_1$)(ii)] The pushout-product is an isomorphism.
\item[($M_2$)(i)] The pushout-product is a composition of marked anodynes of types ($M_1$) and ($M_2$) by Proposition \ref{prop:poprodrootanodyne}.
\item[($M_2$)(ii)] If $F$ is just a 1-corolla, then the pushout-product is a composition of a pushout of a strong marked anodyne of type ($M_3$) followed by a strong marked anodyne of the kind described in Lemma \ref{lem:horn22markedanod}. If $F$ is bigger than that, the pushout-product is an isomorphism.
\item[($M_3$)(i)] If $G = \eta$, the pushout-product is isomorphic to the marked anodyne of type ($M_3$). If $G$ is bigger than that, the pushout-product is an isomorphism.
\item[($M_3$)(ii)] The pushout-product is a pushout of a marked anodyne of type ($M_3$).
\item[($M_4$)(i)] If $G = \eta$, the pushout-product is isomorphic to the marked anodyne of type ($M_4$). If $G$ is bigger than that, the pushout-product is an isomorphism.
\item[($M_4$)(ii)] The pushout-product is a (possibly transfinite) composition of pushouts of marked anodynes of type ($M_3$).
\end{itemize}
$\Box$  
\end{proof}

Of course, we also have the following:

\begin{lemma}
\label{lem:poprodmarkedanodyne2}
Let $f: A \longrightarrow B$ be a monomorphism of simplicial sets. Then the normal monomorphism obtained as the pushout-product of a marked anodyne of type ($M_5$) with the map $f^\flat$ is a marked equivalence.
\end{lemma}
\begin{proof}
We have to show that for any operadically local $E$, the marked forest set $E^\natural$ has the right lifting property with respect to such a pushout-product. But this follows directly from the fact that the operadic model structure is homotopically enriched over the Joyal model structure and the observation that the underlying map of forest sets associated to a map of type ($M_5$) is a trivial cofibration in the operadic model structure. $\Box$
\end{proof}

\begin{lemma}
\label{lem:markedmappingspace}
Suppose $A \longrightarrow B$ is a cofibration between marked forest sets and $E$ has the right lifting property with respect to all strong marked anodynes. Then the map
\begin{equation*}
\mathrm{Map}^\flat(B, E) \longrightarrow \mathrm{Map}^\flat(A, E)
\end{equation*}
is an inner fibration and
\begin{equation*}
\mathrm{Map}^\sharp(B, E) \longrightarrow \mathrm{Map}^\sharp(A, E)
\end{equation*}
is a right fibration.
\end{lemma}
\begin{proof}
Consider a lifting problem
\[
\xymatrix{
\Lambda_i^n \ar[d]\ar[r] & \mathrm{Map}^\flat(B, E) \ar[d] \\
\Delta^n \ar[r]\ar@{-->}[ur] & \mathrm{Map}^\flat(A, E)
}
\]
where $0 < i < n$. This is equivalent to the lifting problem
\[
\xymatrix{
(\Delta^n)^\flat \otimes A \cup (\Lambda_i^n)^\flat \otimes B \ar[d]\ar[r] & E \\
(\Delta^n)^\flat \otimes B \ar@{-->}[ur] &
}
\]
By Lemma \ref{lem:poprodmarkedanodyne} the left-hand map is a strong marked anodyne, so by our assumption on $E$ there exists a lift. To prove the second statement, we have to solve lifting problems of the form
\[
\xymatrix{
\Lambda_i^n \ar[d]\ar[r] & \mathrm{Map}^\sharp(B, E) \ar[d] \\
\Delta^n \ar[r]\ar@{-->}[ur] & \mathrm{Map}^\sharp(A, E)
}
\]
where $0 < i \leq n$. Note that $(\Lambda_i^n)^\sharp \longrightarrow (\Delta^n)^\sharp$ is strong marked anodyne (it is a pushout of a strong marked anodyne of type ($M_1$), respectively ($M_2$), for $i < n$, respectively $i=n$), so again by Lemma \ref{lem:poprodmarkedanodyne} we can find a lift in
\[
\xymatrix{
(\Delta^n)^\sharp \otimes A \cup (\Lambda_i^n)^\sharp \otimes B \ar[d]\ar[r] & E \\
(\Delta^n)^\sharp \otimes B \ar@{-->}[ur] &
}
\]
This completes the proof. $\Box$ 
\end{proof}

\begin{corollary}
\label{cor:markedmappingspace}
Let $A$ be a normal marked forest set and let $E$ be a marked forest set having the right lifting property with respect to all strong marked anodynes. Then $\mathrm{Map}^\flat(A, E)$ is an $\infty$-category and $\mathrm{Map}^\sharp(A, E)$ is the largest Kan complex contained in it.
\end{corollary}
\begin{proof}
For any normal marked forest set $A$, we can apply the previous lemma to the inclusion $\varnothing \longrightarrow A$ to conclude that $\mathrm{Map}^\flat(A, E)$ is an $\infty$-category and $\mathrm{Map}^\sharp(A, E)$ is a Kan complex. Indeed, a right fibration over a point (or in fact over any Kan complex) is a Kan fibration. Also, applying Lemma \ref{lem:poprodmarkedanodyne} above, we see that $\mathrm{Map}^+(A, E)$ has the right lifting property with respect to $J^\flat \subseteq J^\sharp$, so that every equivalence in $\mathrm{Map}^+(A, E)$ is marked. This shows the maximal Kan complex in $\mathrm{Map}^\flat(A, E)$ is contained in $\mathrm{Map}^\sharp(A, E)$ and the result follows. $\Box$ 
\end{proof}

From the previous lemma we can in fact prove the following stronger statement. 

\begin{proposition}
\label{prop:fibmarkedmappingspace}
Suppose $E$ has the right lifting property with respect to all strong marked anodynes. For a cofibration $A \longrightarrow B$, the map of simplicial sets
\begin{equation*}
\mathrm{Map}^\flat(B, E) \longrightarrow \mathrm{Map}^\flat(A, E)
\end{equation*}
is a categorical fibration (i.e. a fibration in the Joyal model structure) and 
\begin{equation*}
\mathrm{Map}^\sharp(B, E) \longrightarrow \mathrm{Map}^\sharp(A, E)
\end{equation*}
is a Kan fibration.
\end{proposition}
\begin{proof}
We know that $\mathrm{Map}^\sharp(A, E)$ is a Kan complex and that $\mathrm{Map}^\sharp(B, E) \longrightarrow \mathrm{Map}^\sharp(A, E)$ is a right fibration, so it is in fact a Kan fibration. To prove that $\mathrm{Map}^\flat(B, E) \longrightarrow \mathrm{Map}^\flat(A, E)$ is a categorical fibration, it only remains to show it has the right lifting property with respect to the map $\{1\} \longrightarrow J$. By Corollary \ref{cor:markedmappingspace} any map $J \longrightarrow \mathrm{Map}^\flat(A, E)$ factors through $\mathrm{Map}^\sharp(A, E)$, so it suffices to solve the lifting problem
\[
\xymatrix{
\{1\}\ar[d]\ar[r] & \mathrm{Map}^\sharp(B, E) \ar[d] \\
J \ar@{-->}[ur]\ar[r] & \mathrm{Map}^\sharp(A, E)
}
\]
The map on the right is a Kan fibration and the map on the left is a trivial cofibration in the Quillen model structure on simplicial sets, so a lift exists. $\Box$
\end{proof}

\begin{proposition}
\label{prop:RLPmarkedanodynes}
A marked forest set $E$ has the right lifting property with respect to all marked anodynes if and only if $aE$ is an operadically local forest set and $E = (aE)^\natural$, i.e. precisely the equivalences in $E$ are marked.
\end{proposition}
\begin{proof}
Suppose $E$ is a marked forest set having the right lifting property with respect to marked anodynes. By Proposition \ref{prop:fibmarkedmappingspace} above, $aE$ is an operadically local object. Since it has the right lifting property with respect to marked anodynes of type ($M_5$) it is also local with respect to sums. Then the fact that it has the right lifting property with respect to marked anodynes of type ($M_1$) implies it is operadically local. The right lifting property with respect to marked anodynes of type ($M_4$) implies that all equivalences in $E$ are marked. Also, given a marked 1-corolla of $E$, the existence of lifts against marked anodynes of type ($M_2$) implies it is an equivalence. (In fact, an easy exercise shows one only needs root horns of 2- and 3-simplices for this.) \par 
Now suppose $E$ is of the form $(aE)^\natural$ and we wish to show it has the right lifting property with respect to marked anodynes. Lifts against anodynes of types ($M_1$), ($M_4$) and ($M_5$) exist by assumption. Lifts with respect to ($M_2$) exist by Theorem \ref{thm:Cartequivalences} of the previous section. Lifts with respect to ($M_3$) exist because equivalences are closed under composition. $\Box$
\end{proof}

\begin{proof}[Proof of Proposition \ref{prop:markedmappingspace}]
Combine Corollary \ref{cor:markedmappingspace} with Proposition \ref{prop:RLPmarkedanodynes}. $\Box$
\end{proof}

\begin{lemma}
\label{lem:expfibmarkedfset}
Let $E$ be an operadically local forest set and let $M$ be a simplicial set. 
Then the cotensor $(E^\natural)^{M^\flat}$ has the right lifting property with respect to all marked anodynes. In particular, $E^M$ is operadically local and we have
\begin{equation*}
(E^\natural)^{M^\flat} \, = \, (E^M)^\natural
\end{equation*}
This statement can be rephrased by saying that the equivalences in $E^M$ are the `pointwise' equivalences.
\end{lemma}
\begin{proof}
Using Proposition \ref{prop:RLPmarkedanodynes}, note that it suffices to prove that for any marked anodyne map $X \longrightarrow Y$ the map $M^\flat \otimes X \longrightarrow M^\flat \otimes Y$ is again marked anodyne. For strong marked anodynes, this follows directly from Lemma \ref{lem:poprodmarkedanodyne}. For marked anodynes of type ($M_5$), this is clear by inspection. $\Box$
\end{proof}

\subsection{A model structure on $\mathbf{fSets}^+$}
\label{sec:markedfsetsmodelstruct}

Before establishing our model structure on the category of marked forest sets, we still need a few observations concerning the marked equivalences and the trivial cofibrations.

\begin{lemma}
\label{lem:countablegenmarkedtrivcof}
The class of marked trivial cofibrations (i.e. cofibrations that are also marked equivalences) is generated by the marked trivial cofibrations between countable and normal objects.
\end{lemma}
\begin{proof}
This is the direct analogue of Lemma \ref{lem:countablegentrivcof}. One can check that the proofs of that lemma and of its preliminaries can be applied to the present setting. The only necessary modification is to replace `operadic anodyne' by `marked anodyne' throughout. $\Box$
\end{proof}

\begin{lemma}
\label{lem:markedanodequiv}
Marked anodyne morphisms are marked trivial cofibrations.
\end{lemma}
\begin{proof}
Since compositions and retracts of marked trivial cofibrations are clearly marked trivial cofibrations again and the same is true for pushouts by the obvious analogue of Lemma \ref{lem:pushoutbasictrivcof}, it suffices to prove that the generating marked anodynes of Definition \ref{def:markedanodynes} are marked trivial cofibrations. Let $f: X \longrightarrow Y$ be such a generating marked anodyne. We wish to show that for any operadically local forest set $E$, the map
\begin{equation*}
\mathrm{Map}^\flat(Y, E^\natural) \longrightarrow \mathrm{Map}^\flat(X, E^\natural)
\end{equation*}
is a trivial fibration of simplicial sets. This is equivalent to $E$ having the right lifting property with respect to maps of the form
\begin{equation*}
N^\flat \otimes X \cup M^\flat \otimes Y \longrightarrow N^\flat \otimes Y
\end{equation*}
where $M \longrightarrow N$ is a monomorphism of simplicial sets. But this lifting property follows from Lemmas \ref{lem:poprodmarkedanodyne} and \ref{lem:poprodmarkedanodyne2}. 
$\Box$  
\end{proof}

\begin{lemma}
\label{lem:RLPwrtcofibrations}
Any map of marked forest sets having the right lifting property with respect to all cofibrations is a marked equivalence.
\end{lemma}
\begin{proof}
The proof is a straightforward adaptation of the proof of Lemma \ref{lem:RLPwrtnormalmonos}. $\Box$
\end{proof}

\begin{proof}[Proof of Theorem \ref{thm:markedfSets}]
First we establish a cofibrantly generated model structure as described in the statement of the Theorem. We check Quillen's axioms CM1-5. As usual, the axioms (CM1) for existence of limits and colimits, (CM2) for two-out-of-three for weak equivalences and (CM3) for retracts are obvious. For the factorization axiom (CM5), Remark \ref{rmk:gencofmarkedfSets} guarantees that every map can be factored as a normal monomorphism followed by a map having the right lifting property with respect to all normal monos and the latter is a trivial fibration by Lemma \ref{lem:RLPwrtcofibrations}. Also, any map $X \longrightarrow Y$ can be factored as $X \rightarrowtail Z \rightarrow Y$ where $X \rightarrowtail Z$ lies in the saturation of the class of trivial cofibrations between countable normal objects and $Z \rightarrow Y$ has the right lifting property with respect to this class. By Lemma \ref{lem:countablegenmarkedtrivcof}, this map is a fibration. It remains to verify the lifting axiom (CM4). Consider a commutative square
\[
\xymatrix{
A \ar[d]_i\ar[r] & X \ar[d]^p \\
B \ar[r] & Y
}
\]
where $i$ is a cofibration and $p$ is a fibration. If $i$ is a marked equivalence, then a lift exists by definition of the fibrations. If $p$ is a weak equivalence, then one applies the same standard retract argument as in the proof of Theorem \ref{thm:basicmodelstructure}. We defer the proof of left properness to Lemma \ref{lem:leftproper}. Let us now establish claims (i) and (ii). We prove (iii) further on in this section, in Lemma \ref{lem:markedmodelstrucsimplicial}, after having established a convenient characterization of the marked trivial cofibrations. \par 
(i). A fibrant object $X$ has the right lifting property with respect to all marked anodynes, by Lemma \ref{lem:markedanodequiv}, and must therefore be of the form $E^\natural$ for some operadically local forest set $E$ by Proposition \ref{prop:RLPmarkedanodynes}. Conversely, assume we have a marked forest set which has the right lifting property with respect to all marked anodynes, i.e. something of the form $E^\natural$. By Lemma \ref{lem:countablegenmarkedtrivcof}, we only have to show that $E^\natural$ has the right lifting property with respect to trivial cofibrations between (countable) normal objects. So let $A \longrightarrow B$ be such a cofibration. The map
\begin{equation*}
\mathrm{Map}^\sharp(B, E^\natural) \longrightarrow \mathrm{Map}^\sharp(A, E^\natural)
\end{equation*}
is a trivial fibration of simplicial sets. Indeed, it is a homotopy equivalence by assumption and a fibration by Proposition \ref{prop:fibmarkedmappingspace}. But a trivial fibration is surjective on vertices, so any lifting problem of the form
\[
\xymatrix{
A \ar[d]\ar[r] & E^\natural \\
B \ar@{-->}[ur] & 
}
\]
admits a solution. \par 
(ii). Let $f: X \longrightarrow Y$ be a map between fibrant objects. If it is a fibration, then it has the right lifting property with respect to marked anodynes. Conversely, suppose it has this right lifting property. Choose a factorization
\[
\xymatrix{
X \ar[r]^i & Z \ar[r]^p & Y 
}
\]
where $i$ is a trivial cofibration and $p$ is a fibration. Since $X$ is fibrant, the map $i$ has a retract $r: Z \longrightarrow X$. Next, note that the map
\begin{equation*}
(\Delta^1)^\sharp \otimes X \cup (\partial \Delta^1)^\sharp \otimes Z \longrightarrow (\Delta^1)^\sharp \otimes Z 
\end{equation*}
is a trivial cofibration, by the analogue of Lemma \ref{lem:halfhomotenrich}. Therefore we can find a lift $h$ in
\[
\xymatrix@C=50pt{
(\Delta^1)^\sharp \otimes X \cup (\partial \Delta^1)^\sharp \otimes Z \ar[r]^-{fs_0 \cup (p, fr)} \ar[d] & Y \\
(\Delta^1)^\sharp \otimes Z \ar@{-->}[ur]_h &
}
\]
because $Y$ is fibrant as well. (Note that this gives a `homotopy over $(\Delta^1)^\sharp$' from $p$ to $fr$ relative to $X$.) Finally, lift in
\[
\xymatrix@C=40pt{
(\Delta^1)^\sharp \otimes X \cup_{\{1\} \otimes X} \{1\} \otimes Z \ar[r]^-{s_0 \cup r}\ar[d] & X \ar[d]^f \\
(\Delta^1)^\sharp \otimes Z \ar@{-->}[ur]_k \ar[r]_h & Y
}
\]
This is possible because the map on the left is (strong) marked anodyne by Lemma \ref{lem:poprodmarkedanodyne}. Then $r' = k_0$ has the property that $fr' = h_0 = p$ and $r'i = \mathrm{id}_X$, so that the diagram
\[
\xymatrix{
X \ar[r]^i \ar[d]_f & Z \ar[r]^{r'}\ar[d]_p & X \ar[d]_f \\
Y \ar@{=}[r] & Y \ar@{=}[r] & Y
}
\]
exhibits $f$ as a retract of $p$. In particular, $f$ is a fibration. $\Box$ 
\end{proof}

To establish left properness we need the following:

\begin{proposition}
\label{prop:markedequivmapsharp}
Let $f: X \longrightarrow Y$ be a map in $\mathbf{fSets}^+$ and let 
\[
\xymatrix{
X' \ar[r]^{f'}\ar@{->>}[d] & Y' \ar@{->>}[d] \\
X \ar[r]^f & Y
}
\]
be a commutative square in which the vertical arrows are normalizations. Then the following are equivalent:
\begin{itemize}
\item[(i)] The map $f$ is a marked equivalence.
\item[(ii)] For every operadically local forest set $E$, the map $f'$ induces a homotopy equivalence of Kan complexes
\begin{equation*}
\mathrm{Map}^\sharp(Y', E^\natural) \longrightarrow \mathrm{Map}^\sharp(X', E^\natural)
\end{equation*}
\end{itemize}
\end{proposition}
\begin{proof}
Assume (i). Then the map
\begin{equation*}
\mathrm{Map}^\flat(Y', E^\natural) \longrightarrow \mathrm{Map}^\flat(X', E^\natural)
\end{equation*}
is an equivalence of $\infty$-categories and the map stated in (ii) is the induced map on maximal Kan complexes, so it is clear that (i) implies (ii). \par 
Conversely, let us assume (ii). First, recall that a map $\mathcal{C} \longrightarrow \mathcal{D}$ of $\infty$-categories is a categorical equivalence if and only if, for every simplicial set $M$, the map $\mathcal{C}^M \longrightarrow \mathcal{D}^M$ induces a homotopy equivalence between the maximal Kan complexes contained in the $\infty$-categories $\mathcal{C}^M$ and $\mathcal{D}^M$ (see for example Lemma 3.1.3.2 of \cite{htt}). We wish to show that
\begin{equation*}
\mathrm{Map}^\flat(Y', E^\natural) \longrightarrow \mathrm{Map}^\flat(X', E^\natural)
\end{equation*}
is an equivalence of $\infty$-categories. To this end, we will prove that for an arbitrary simplicial set $M$, the map
\begin{equation*}
\mathrm{Map}^\flat(Y', E^\natural)^M \longrightarrow \mathrm{Map}^\flat(X', E^\natural)^M
\end{equation*}
induces a homotopy equivalence on the maximal Kan complexes contained in these $\infty$-categories. Notice that this map fits into a square
\[
\xymatrix{
\mathrm{Map}^\flat(Y', (E^\natural)^{M^\flat}) \ar[d]\ar[r] & \mathrm{Map}^\flat(X', (E^\natural)^{M^\flat}) \ar[d] \\
\mathrm{Map}^\flat(Y', E^\natural)^M \ar[r] & \mathrm{Map}^\flat(X', E^\natural)^M.
}
\]
The vertical maps are trivial fibrations; indeed, this follows from the marked analogue of Proposition \ref{prop:forestsetsH2} (which has the same proof) and Lemma \ref{lem:exphomotenriched}. Also, the top map induces an equivalence on maximal Kan complexes by our assumption and the isomorphism
\begin{equation*}
(E^\natural)^{M^\flat} \simeq (E^M)^\natural 
\end{equation*}
following from Lemma \ref{lem:expfibmarkedfset}. Therefore, the bottom map induces an equivalence on maximal Kan complexes as well. $\Box$
\end{proof}

\begin{lemma}
\label{lem:leftproper}
The model structure of Theorem \ref{thm:markedfSets} is left proper.
\end{lemma}
\begin{proof}
Consider a pushout square
\[
\xymatrix{
A \ar[r]^\sim\ar[d] & B \ar[d] \\
C \ar[r] & D
}
\]
in which the top map is a marked equivalence and the left map is a cofibration. Choose a normalization $D' \longrightarrow D$ and pull the square back along this map to obtain another square
\[
\xymatrix{
A' \ar[r]^\sim\ar[d] & B' \ar[d] \\
C' \ar[r] & D'
}
\]
This square is still a pushout and all the objects in it are normal. Now let $E$ be an operadically local forest set and consider the pullback square
\[
\xymatrix{
\mathrm{Map}^\sharp(A', E^\natural) & \mathrm{Map}^\sharp(B', E^\natural) \ar[l]_\sim \\
\mathrm{Map}^\sharp(C', E^\natural) \ar[u] & \mathrm{Map}^\sharp(D', E^\natural) \ar[l]\ar[u]
}
\]
By assumption, the top map is a homotopy equivalence of simplicial sets. By Proposition \ref{prop:fibmarkedmappingspace}, the left map is a Kan fibration. Since the Quillen model structure on simplicial sets is right proper, the bottom map must then also be a homotopy equivalence. We now apply Proposition \ref{prop:markedequivmapsharp} to conclude that $C \longrightarrow D$ is a marked equivalence. $\Box$
\end{proof}

The following analogue of Proposition \ref{prop:trivcofoperadic} will be convenient when establishing Quillen adjunctions:

\begin{lemma}
\label{lem:markedtrivcof}
The class of marked trivial cofibrations is the smallest saturated class which is closed under two-out-of-three among cofibrations and contains the marked anodynes. In fact, it is enough to demand it contains the following morphisms:
\begin{itemize}
\item[(a)] For any tree $T$, the Segal core inclusion
\begin{equation*}
\mathrm{fSc}(T)^\flat \longrightarrow T^\flat
\end{equation*}
\item[(b)] The inclusions
\begin{equation*}
(\Lambda^r[T], \mathcal{E} \cap \Lambda^r[T]) \longrightarrow (T, \mathcal{E})
\end{equation*}
where $T$ is a tree with a root corolla of valence one, $\Lambda^r[T]$ is the horn of $T$ corresponding to that root and $\mathcal{E}$ consists of all degenerate 1-corollas of $T$ together with that root corolla. 
\item[(c)] The map
\begin{equation*}
(\Lambda_1^2)^\sharp \cup_{(\Lambda_1^2)^\flat} (\Delta^2)^\flat \longrightarrow (\Delta^2)^\sharp
\end{equation*}
\item[(d)] For any Kan complex $K$, the inclusion $K^\flat \subseteq K^\sharp$.
\item[(e)] For any non-empty sequence $T_1, \ldots, T_k$ of trees, the map
\begin{equation*}
(T_1 \amalg \cdots \amalg T_k)^\flat \longrightarrow (T_1 \oplus \cdots \oplus T_k)^\flat
\end{equation*}
\end{itemize}
\end{lemma}
\begin{proof}
The statement of the first sentence is proved in the same way as Proposition \ref{prop:basictrivcofs}, using the fact that fibrations between fibrant objects are `detected' by the marked anodynes. Next, reducing the inner anodynes to Segal core inclusions for (a) and reducing from marked anodynes of type ($M_5$) to the maps listed under (e) in the proposition was done in the previous chapter, cf. Corollary \ref{cor:operadictrivcofSegal}. Then for marked anodynes of the type
\begin{equation*}
(\Lambda^r[F], \mathcal{E} \cap \Lambda^r[F]) \longrightarrow (F, \mathcal{E})
\end{equation*}
we can restrict to the case where the forest $F$ consists of just one tree by what we already know about sums. $\Box$
\end{proof}

\begin{lemma}
\label{lem:markedmodelstrucsimplicial}
With the weak simplicial enrichment on $\mathbf{fSets}^+$ corresponding to the mapping objects $\mathrm{Map}^\sharp(-,-)$, the model structure of Theorem \ref{thm:markedfSets} is homotopically enriched over the Kan-Quillen model structure on simplicial sets.
\end{lemma}
\begin{proof}
Let $i: A \longrightarrow B$ be a cofibration of marked forest sets and let $j: M \longrightarrow N$ be a monomorphism of simplicial sets. We have to show that the pushout-product
\begin{equation*}
N^\sharp \otimes A \cup M^\sharp \otimes B \longrightarrow N^\sharp \otimes B
\end{equation*}
is a cofibration, which is trivial if either $i$ or $j$ is a weak equivalence. The fact that it is a cofibration follows from the corresponding fact for $\mathbf{fSets}$, since cofibrations are defined on the level of underlying forest sets. To prove the second part, first consider the case where $i$ is trivial. By Lemma \ref{lem:markedtrivcof} it suffices to treat the case where $i$ is marked anodyne. In this case the pushout-product is a trivial cofibration by Lemmas \ref{lem:poprodmarkedanodyne}, \ref{lem:poprodmarkedanodyne2} and \ref{lem:markedanodequiv}. In case $j$ is assumed to be a trivial cofibration of simplicial sets, we argue as follows. We will show that
\begin{equation*}
\mathrm{Map}^\sharp(N^\sharp \otimes B, E^\natural) \longrightarrow \mathrm{Map}^\sharp(N^\sharp \otimes A \cup M^\sharp \otimes B, E^\natural)
\end{equation*}
is a trivial fibration of simplicial sets, for any operadically local $E$. 
Consider a lifting problem
\[
\xymatrix@C=40pt{
\partial \Delta^n \ar[d]\ar[r] & \mathrm{Map}^\sharp(N^\sharp \otimes B, E^\natural) \ar[d] \\
\Delta^n \ar[r] & \mathrm{Map}^\sharp(N^\sharp \otimes A \cup M^\sharp \otimes B, E^\natural).
}
\]
By symmetry of the enrichment, this is equivalent to the lifting problem
\[
\xymatrix@C=40pt{
M \ar[d]\ar[r] & \mathrm{Map}^\sharp((\Delta^n)^\sharp \otimes B, E^\natural) \ar[d] \\
N \ar[r] & \mathrm{Map}^\sharp((\Delta^n)^\sharp \otimes A \cup (\partial \Delta^n)^\sharp \otimes B, E^\natural).
}
\]
The right vertical map is a Kan fibration by Proposition \ref{prop:fibmarkedmappingspace}, so a lift exists. $\Box$
\end{proof}

%\begin{lemma}
%\label{lem:markedmodelstrucmonoidal}
%The restriction of the model structure of Theorem \ref{thm:markedfSets} to the category $\mathbf{fSets}_o^+$ is monoidal.
%\end{lemma}
%\begin{proof}
%We have to check that for cofibrations $i: A \longrightarrow B$ and $j: X \longrightarrow Y$ between open marked forest sets, the pushout-product
%\begin{equation*}
%A \otimes Y \cup B \otimes X \longrightarrow B \otimes Y
%\end{equation*}
%is a cofibration, which is trivial if either $i$ or $j$ is. Again, the fact that it is a cofibration follows from the corresponding statement for the category $\mathbf{fSets}$. So assume that $i$ is a trivial cofibration. By Lemma \ref{lem:markedtrivcof} above, we may restrict our attention from arbitrary trivial cofibrations to marked anodyne maps. For the strong marked anodynes, i.e. marked anodynes of types ($M_1$) - ($M_4$), we know that the pushout-products with cofibrations are again (strong) marked anodyne by Lemma \ref{lem:poprodmarkedanodyne}. Again invoking Lemma \ref{lem:markedtrivcof} above (specifically, the form of (e) listed there), the only thing left to check is that the pushout-product of a map of the form
%\begin{equation*}
%(T_1 \amalg \cdots \amalg T_k)^\flat \longrightarrow (T_1 \oplus \cdots \oplus T_k)^\flat
%\end{equation*}
%with a cofibration $A \rightarrowtail B$ is a trivial cofibration. This is easy and we leave it for the reader to verify. $\Box$ 
%\end{proof}

Finally, let us also note the following, which we will need in the next chapter.

\begin{lemma}
\label{lem:leafanodmarkedequiv}
Leaf anodyne maps (see Remark \ref{rmk:coCartequivalences}) are marked equivalences.
\end{lemma}
\begin{proof}
Just as in the proof of \ref{lem:markedanodequiv} it suffices to check this for generating leaf anodynes. First, note that by Theorem \ref{thm:coCartequivalences} every fibrant object $E^\natural$ of $\mathbf{fSets}^+$ has the right lifting property with respect to leaf anodynes. Now, let $X \longrightarrow Y$ be a generating leaf anodyne map and let $E$ be an operadically local forest set. We have to show that the map
\begin{equation*}
\mathrm{Map}^\flat(Y, E^\natural) \longrightarrow \mathrm{Map}^\flat(X, E^\natural)
\end{equation*}
is a categorical equivalence. It is in fact a trivial fibration; indeed, this now follows from the fact that for any monomorphism $A \longrightarrow B$ of simplicial sets, the pushout-product
\begin{equation*}
 B^\flat \otimes X \cup A^\flat \otimes Y  \longrightarrow B^\flat \otimes Y
\end{equation*}
is leaf anodyne (which follows from Proposition \ref{prop:poprodleafanodyne}) and the fact that $E^\natural$ has the right lifting property with respect to leaf anodynes.  $\Box$
\end{proof}

\subsection{Marked dendroidal sets}
\label{sec:markeddsets}

In much the same way as we did for forest sets, one can establish a category of marked dendroidal sets with a corresponding model structure. All proofs can be given in analogy with what was done for forest sets, or the relevant results can be derived from those for forest sets by applying $u^*$ and using that it preserves tensor products. We briefly summarize what we need. \par 
A \emph{marked dendroidal set} is a pair $(X, \mathcal{E})$ where $X$ is a dendroidal set and $\mathcal{E}$ is a subset of the set of 1-corollas of $X$ containing all degenerate 1-corollas. Together with the maps preserving marked 1-corollas, marked dendroidal sets form a category $\mathbf{dSets}^+$. There is a forgetful functor 
\begin{equation*}
a: \mathbf{dSets}^+ \longrightarrow \mathbf{dSets}
\end{equation*}
which has left and right adjoints $(-)^\flat$ and $(-)^\sharp$ respectively. Using the tensor product on the category of dendroidal sets, we obtain a tensor product for marked dendroidal sets by defining
\begin{equation*}
(X, \mathcal{E}_X) \otimes (Y, \mathcal{E}_Y) \, := \, (X \otimes Y, \mathcal{E}_X \times \mathcal{E}_Y). 
\end{equation*}
This tensor product can be used to construct simplicial mapping objects $\mathrm{Map}^+(-,-)$, $\mathrm{Map}^\flat(-,-)$ and $\mathrm{Map}^\sharp(-,-)$ as before. We define cofibrations, normalizations and marked equivalences of marked dendroidal sets by obvious analogy with the corresponding definitions for marked forest sets. We can also define the marked anodyne maps of marked dendroidal sets to simply be the image under $u^*$ of the marked anodyne maps of forest sets. Of course, we do not have to worry about marked anodynes of type ($M_5$), since these are all sent to isomorphisms by $u^*$. For the same reason, we only have to consider the tree versions of ($M_1$) and ($M_2$).

\begin{theorem}
There exists a left proper, cofibrantly generated model structure on the category $\mathbf{dSets}^+$ such that:
\begin{itemize}
\item[(C)] The cofibrations are the normal monomorphisms.
\item[(W)] The weak equivalences are the marked equivalences.
\end{itemize}
Furthermore, this model structure enjoys the following properties:
\begin{itemize}
\item[(i)] An object is fibrant if and only if it is of the form $E^\natural$, for a dendroidal $\infty$-operad $E$.
\item[(ii)] A map $f$ between fibrant objects is a fibration if and only if it has the right lifting property with respect to all marked anodyne morphisms.
\item[(iii)] With the weak simplicial enrichment on $\mathbf{dSets}^+$ corresponding to the mapping objects $\mathrm{Map}^\sharp(-,-)$, the model structure is homotopically enriched over the Kan-Quillen model structure on simplicial sets.
\end{itemize}
\end{theorem}

\begin{corollary}
\label{cor:equivmarkedunmarked2}
In the following commutative square all functors are left Quillen and induce Quillen equivalences:
\[
\xymatrix{
\mathbf{dSets} \ar[d]_{(-)^\flat}& \mathbf{fSets} \ar[l]_{u^*} \ar[d]^{(-)^\flat} \\
\mathbf{dSets}^+ & \mathbf{fSets}^+ \ar[l]^{u^*}
}
\]
\end{corollary}
\begin{proof}
As in the proof of Corollary \ref{cor:equivmarkedunmarked} one shows that the left vertical functor is part of a Quillen equivalence. We already know that the top and right functors induce Quillen equivalences and hence so does the bottom one. $\Box$
\end{proof}

%% file: dendrification.tex
\section{The dendrification functor}
\label{chap:dendrification}

\subsection{The functor $\omega$}

In this chapter we will finally relate the category of open forest sets to Lurie's category of non-unital $\infty$-preoperads. We start by defining a functor
\begin{equation*}
\omega: \mathbf{\Delta}/N\mathbf{F}_o \longrightarrow \mathbf{fSets}_o
\end{equation*}
taking simplices in the nerve of $\mathbf{F}_o$ to forest sets. We will refer to it as the \emph{dendrification functor} (even though strictly speaking its values only become dendroidal sets after applying the functor $u^*$). We gave a heuristic description of this functor in Section \ref{sec:mainresults}; we now define it properly. Let us first describe $\omega$ on objects. Suppose we have a simplex
\begin{equation*}
A: \Delta^n \longrightarrow N\mathbf{F}_o.
\end{equation*}
If this simplex is the constant $n$-simplex with image $\langle 0 \rangle$, we define
\begin{equation*}
\omega(A) := \varnothing
\end{equation*}
Otherwise, $\omega(A)$ will be a representable forest set defined as follows:
\begin{itemize}
\item[(i)] The set of edges of the forest $\omega(A)$ is $\coprod_{i=0}^n A(i)$.
\item[(ii)] For every $a \in A(i)$, $i > 0$, there is a vertex $v_a$ with output $a$ (i.e. attached to the top of the edge $a$). An edge $b \in A(i-1)$ is an input of $v_a$, for $a \in A(i)$, if the map $A(i-1) \longrightarrow A(i)$ sends $b$ to $a$.
\end{itemize}

It might help the reader's intuition to see how this works in a picture; a typical example was already drawn in Section \ref{sec:mainresults}.

\begin{remark}
It might seem odd that we do not construct $\omega$ in such a way that $\omega(A)$ is \emph{always} representable. We could add an object $O$ to the category $\mathbf{\Phi}$ respresenting the empty forest and define $\omega(\langle 0 \rangle) = O$. However, this causes several problems elsewhere. In particular, the functor $\bar\omega^*$ we construct later will not be left Quillen.
\end{remark}

Let us now define the dendrification functor $\omega$ on morphisms in the category $\mathbf{\Delta}/N\mathbf{F}_o$. It suffices to do this on faces and degeneracies and check that the simplicial relations hold. We start with faces. So assume we have a diagram
\[
\xymatrix{
\Delta^{n-1} \ar[rr]^{\partial_i} \ar[dr]_{d_i A} & & \Delta^n \ar[dl]^A \\
& N\mathbf{F}_o &  
}
\]
The cases where $A$ or $d_i A$ is the degenerate simplex at $\langle 0 \rangle$ are uniquely determined by the fact that $\varnothing$ is the initial object in $\mathbf{fSets}_o$, so let us assume that both $\omega(A)$ and $\omega(d_i A)$ are forests. The map of forests $\omega(d_i A) \longrightarrow \omega(A)$ is induced by the evident map on edges. We can describe it explicitly as follows:
\begin{itemize}
\item[$i=0$:] The map is a composition of external faces chopping off all edges in $A(0)$ and all the vertices $v_a$ for $a \in A(1)$.  
\item[$i>0$:] The map is a composition of faces contracting all edges $e$ in the domain of definition of the partial map $A(i) \longrightarrow A(i+1)$ (these are all inner edges) and root faces chopping off $v_r$ and $r$ for all $r \in A(i)$ that are \emph{not} in the domain of definition of that partial map (note that these are indeed roots of constituent trees of the forest $\omega(A)$).
\end{itemize}

As an example, we can consider the maps $d_0 A \longrightarrow A$ and $d_1 A \longrightarrow A$ for the simplex $A$ we drew in Section \ref{sec:mainresults}. They can be pictured as follows: 

\[
\begin{tikzpicture} 
[level distance=10mm, 
every node/.style={fill, circle, minimum size=.1cm, inner sep=0pt}, 
level 1/.style={sibling distance=20mm}, 
level 2/.style={sibling distance=10mm}, 
level 3/.style={sibling distance=5mm}]

%middle simplex
%left tree
\node (lefttree)[style={color=white}] {} [grow'=up] 
child {node (level1) {} 
	child{ node (level2) {}
		child
		child
		child
	}
	child{ node {}
		child
	}
};

%right tree
\node (righttree)[style={color=white}, right = 1.5cm of lefttree] {};
\node (righttreestart)[style={color=white}, above = .88cm of righttree] {} [grow'=up] 
child {node {} 
	child
	child
};

%left simplex
%left tree
\node (Llefttree)[style={color=white}, left = 5cm of righttree] {};
\node (Llefttreestart)[style={color=white}, below = 3cm of Llefttree] {}[grow'=up] 
child {node (Llevel1) {} 
	child
	child
	child
	child
};

%right tree 1
\node (Lrighttree)[style={color=white}, right = 1.7cm of Llefttreestart] {};
\node (Lrighttreestart)[style={color=white}, above = .88cm of Lrighttree] {} [grow'=up] 
child {};

%right tree 2
\node (Lrighttree)[style={color=white}, right = 2.4cm of Llefttreestart] {};
\node (Lrighttreestart)[style={color=white}, above = .88cm of Lrighttree] {} [grow'=up] 
child {};

%right simplex
%left tree
\node (Rlefttree)[style={color=white}, right = 2cm of righttree] {};
\node (Rlefttreestart)[style={color=white}, below = 3cm of Rlefttree] {}[grow'=up] 
child {node (Rlevel1) {} 
	child
	child
};

%right tree 1
\node (Rrighttree)[style={color=white}, right = 1cm of Rlefttreestart] {};
\node (Rrighttreestart)[style={color=white}, above = .88cm of Rrighttree] {} [grow'=up] 
child {};

\tikzstyle{every node}=[]

%lines
\draw[dashed] ($(level1) + (-1.5cm, 0)$) -- ($(level1) + (2.5cm, 0)$);
\draw[dashed] ($(level1) + (-1.5cm, 1cm)$) -- ($(level1) + (2.5cm, 1cm)$);
\draw[dashed] ($(Llevel1) + (-1.5cm, 0)$) -- ($(Llevel1) + (2.8cm, 0)$);
\draw[dashed] ($(Rlevel1) + (-1.5cm, 0)$) -- ($(Rlevel1) + (2.1cm, 0)$);

%labels
\node at ($(level1) + (-1.5cm, 1.5cm)$) {$0$};
\node at ($(level1) + (-1.5cm, .5cm)$) {$1$};
\node at ($(level1) + (-1.5cm, -.5cm)$) {$2$};
\node at ($(Llevel1) + (-1.5cm, .5cm)$) {$0$};
\node at ($(Llevel1) + (-1.5cm, -.5cm)$) {$2$};
\node at ($(Rlevel1) + (2cm, .5cm)$) {$1$};
\node at ($(Rlevel1) + (2cm, -.5cm)$) {$2$};

%arrows
\draw[->] ($(Llefttreestart) + (0.4cm,2.5cm)$) -- node[above]{$\partial_1$} ($(Llefttreestart) + (1.4cm, 3.5cm)$);
\draw[->] ($(Rlefttreestart) + (0cm,2.5cm)$) -- node[above]{$\partial_0$} ($(Rlefttreestart) + (-1cm, 3.5cm)$);

\end{tikzpicture} 
\]

Let us now consider a degeneracy map. Suppose we have a diagram
\[
\xymatrix{
\Delta^{n+1} \ar[rr]^{\sigma_j} \ar[dr]_{s_j A} & & \Delta^n \ar[dl]^{A} \\
& N\mathbf{F}_o &
}
\]
The map $\sigma_j$ is the degeneracy identifying $j$ and $j+1$. Again, the map $\omega(s_j A) \longrightarrow \omega(A)$ is the evident one on edges. All vertices $v_a$ for $a \in s_j A(j+1)$ are unary and they are mapped to $\mathrm{id}_a$ in $\omega(A)$. In particular, the map $\omega(s_j A) \longrightarrow \omega(A)$ is a composition of degeneracies, one for each $a \in A(j)$. \par
It remains to verify the simplicial relations in order for $\omega$ to define a functor. But a map of forests is uniquely determined by what it does on edges, so these relations must be satisfied, simply because the maps on edges satisfy them.  

\subsection{Two Quillen pairs induced by $\omega$}
\label{sec:omegaQpairs}

In this section we will discuss how the dendrification functor
\begin{equation*}
\omega: \mathbf{\Delta}/N\mathbf{F}_o \longrightarrow \mathbf{fSets}_o
\end{equation*}
induces two adjoint pairs of functors. These pairs are in fact Quillen pairs, but the proofs that they are will be postponed until Sections \ref{sec:omega!leftQ} and \ref{sec:omega*leftQ}. \par 
First of all, by left Kan extension, the functor $\omega$ induces an adjoint pair of functors
\[
\xymatrix@C=40pt{
\omega_!: \mathbf{sSets}/N\mathbf{F}_o \ar@<.5ex>[r] & \mathbf{fSets}_o: \omega^* \ar@<.5ex>[l]
}
\]
completely determined (up to natural isomorphism) by the requirement that $\omega_!$ agrees with $\omega$ on representables $A: \Delta^n \longrightarrow N\mathbf{F}_o$. The following simple observation, which is clear from our definitions, will be of crucial importance later on:

\begin{lemma}
\label{lem:omega*pointed}
For a representable object $A: \Delta^n \longrightarrow N\mathbf{F}_o$ of $\mathbf{sSets}/N\mathbf{F}_o$, the forest set $\omega_!(A)$ is again representable, except when $A$ is the degenerate simplex at $\langle 0 \rangle$. In that case, $\omega_!(\langle 0 \rangle) = \varnothing$. So for an arbitrary forest set $X$, the set $\omega^*(X)(\langle 0 \rangle)$ is a one-point set.
\end{lemma}

We now wish to lift this adjoint pair between the categories `without markings' to an adjoint pair of functors
\[
\xymatrix@C=40pt{
\omega_!: \mathbf{POp}_o \ar@<.5ex>[r] & \mathbf{fSets}_o^+: \omega^* \ar@<.5ex>[l]
}
\]
between the categories `with markings'. If $A: (\Delta^1)^\sharp \longrightarrow N\mathbf{F}_o^\natural$ is a marked $1$-simplex over $N\mathbf{F}_o$, then $\omega_!(A)$ is (the presheaf represented by) the marked forest given by marking (the corollas corresponding to) the unary vertices $v_a$ for each $a \in A(1)$. This completely determines the functor $\omega_!$. In the other direction, for a marked forest set $X$, the marked edges of $\omega^*(X)$ are determined by adjunction. Indeed, a 1-simplex in $\omega^*(X)$ is a map
\[
\xymatrix{
(\Delta^1)^\flat \ar[rr]^{\alpha}\ar[dr]_A & & \omega^*(X) \ar[dl] \\
& N\mathbf{F}_o^\natural & 
}
\]
or equivalently, a map $\hat\alpha: \omega_!(A) \longrightarrow X$. Such a map is marked precisely when $A$ extends to a map $(\Delta^1)^\sharp \longrightarrow N\mathbf{F}_o^\natural$ while $\hat\alpha_a: u_!(C_1)^\flat \longrightarrow X$ is marked in $X$ (i.e. factors through $u_!(C_1)^\sharp$) for every $a \in A(1)$.

\begin{lemma}
\label{lem:omega!prescof}
The functor $\omega_!: \mathbf{POp}_o \longrightarrow \mathbf{fSets}_o^+$ preserves cofibrations.
\end{lemma}
\begin{proof}
The generating cofibrations in $\mathbf{POp}_o$ are
\[
\xymatrix{
(\Delta^1)^\flat \ar[rr]^{\alpha}\ar[dr] & & (\Delta^1)^\sharp \ar[dl]^A \\
& N\mathbf{F}_o^\natural & 
}
\]
and
\[
\xymatrix{
(\partial\Delta^n)^\flat \ar[rr]^{\alpha}\ar[dr] & & (\Delta^n)^\flat \ar[dl]^A \\
& N\mathbf{F}_o^\natural & 
}
\]
The functor $\omega_!$ maps the first one to a direct sum of maps which are either of the form
\begin{equation*}
C_1^\flat \longrightarrow C_1^\sharp
\end{equation*}
(one for each $a \in A(1)$) or of the form
\begin{equation*}
\eta \simeq \eta
\end{equation*}
(one for each $a \in A(0)$ at which $A(0) \longrightarrow A(1)$ is undefined).  As for the second generating cofibration: in the definition of $\omega_!$ we saw that it sends a face inclusion to a composition of face maps (and hence a cofibration) in $\mathbf{fSets}_o^+$. In fact, a face $d_i A$ is sent to the inclusion of the maximal subforest of $\omega_!(A)$ not containing the edges corresponding to elements of $A(i)$ (this uses that the maps in $\mathbf{F}_o$ are surjections, compare Remark \ref{rmk:Fo}). Using this observation one verifies that $\omega_!$ sends an intersection of faces to the intersection of the corresponding subobjects of $\omega_!(A)$, which then implies that it also sends the given boundary inclusion to a cofibration. $\Box$
\end{proof}

\begin{remark}
\label{rmk:Fo}
The observation about the effect of $\omega_!$ on a face inclusion \emph{fails} when we consider an analogous functor $\omega_!$ defined on all of $N\mathbf{F}$. A minimal counterexample is the inclusion $d_0A \rightarrow A$, for $A$ the unique 1-simplex $\langle 0 \rangle \rightarrow \langle 1 \rangle$ of $N\mathbf{F}$. As a consequence, the obvious extension of $\omega_!$ to a functor from $\mathbf{POp}$ to $\mathbf{fSets}^+$ does \emph{not} preserve cofibrations. A counterexample is given by taking $A$ to be the unique 2-simplex
\begin{equation*}
\langle 1 \rangle \rightarrow \langle 0 \rangle \rightarrow \langle 1 \rangle 
\end{equation*}
and considering the inclusion
\[
\xymatrix{
\Lambda_2^2 \ar[dr]\ar[rr] & & \Delta^2 \ar[dl]^A \\
& N\mathbf{F} & 
}
\]
Indeed, applying $\omega_!$ to this diagram yields the map of forest sets
\begin{equation*}
\eta \oplus (C_0 \cup_\eta C_0) \longrightarrow \eta \oplus C_0,
\end{equation*}
which is not a monomorphism, since it maps two different nullary operations to a single one.
\end{remark}

In Section \ref{sec:omega!leftQ} we will prove the following:

\begin{proposition*}[See Proposition \ref{prop:omega!leftQ}]
The pair
\[
\xymatrix@C=40pt{
\omega_!: \mathbf{POp}_o \ar@<.5ex>[r] & \mathbf{fSets}_o^+: \omega^* \ar@<.5ex>[l]
}
\]
is a Quillen pair.
\end{proposition*}

Had it been the case that $\omega: \mathbf{\Delta}/N\mathbf{F}_o \longrightarrow \mathbf{fSets}_o$ mapped into the representable forest sets, then $\omega^*$ would have had a further right adjoint. Now this cannot be the case because for the empty forest we have
\begin{equation*}
\omega^*(\varnothing) \, = \, \langle 0 \rangle 
\end{equation*}
where $\langle 0 \rangle$ stands for the one-point simplicial set over the vertex $\langle 0 \rangle \in N\mathbf{F}_o$. Thus, $\omega^*$ does not preserve colimits, so cannot have a right adjoint. To repair this, we will replace $\mathbf{POp}_o$ by the slice category $\langle 0 \rangle / \mathbf{POp}_o$. This is a relatively innocent change because of the following easy lemma, the proof of which we leave to the reader.

\begin{lemma}
\label{lem:Qslices}
\begin{itemize}
\item[(i)] Let $\mathcal{E}$ be a model category. Then any arrow $f: A \longrightarrow B$ in $\mathcal{E}$ induces a Quillen pair
\[
\xymatrix@C=40pt{
f_!: A/\mathcal{E} \ar@<.5ex>[r] & B/\mathcal{E}: f^* \ar@<.5ex>[l]
}
\]
for the induced model structures on these slice categories.
\item[(ii)] This Quillen pair is a Quillen equivalence if the map $f: A \longrightarrow B$ is a trivial cofibration.
\item[(iii)] A left adjoint functor $B/\mathcal{E} \longrightarrow \mathcal{F}$ into another model category $\mathcal{F}$ is left Quillen if and only if the composition
\begin{equation*}
A/\mathcal{E} \longrightarrow B/\mathcal{E} \longrightarrow \mathcal{F}
\end{equation*}
is so.
\end{itemize}
\end{lemma}

Applying this to the special case at hand, we obtain (part of) the following lemma.

\begin{lemma}
\begin{itemize}
\item[(i)] The functors
\[
\xymatrix@C=40pt{
\langle 0 \rangle_!: \mathbf{POp}_o \ar@<.5ex>[r] & \langle 0 \rangle/\mathbf{POp}_o : \langle 0 \rangle^*
}
\]
defined by $\langle 0 \rangle_!(Y) = \langle 0 \rangle \amalg Y$ and $\langle 0 \rangle^*$ the forgetful functor, form a Quillen equivalence.
\item[(ii)] The left Quillen functor $\omega_!: \mathbf{POp}_o \longrightarrow \mathbf{fSets}_o^+$ factors through a left Quillen functor $\bar\omega_!$ as in
\[
\xymatrix{
\mathbf{POp}_o\ar[d]_{\langle 0 \rangle_!}\ar[r]^{\omega_!} & \mathbf{fSets}_o^+ \\
\langle 0 \rangle/\mathbf{POp}_o \ar[ur]_{\bar\omega_!}
}
\]
\end{itemize}
\end{lemma}
\begin{proof}
Part (i) follows by applying Lemma \ref{lem:Qslices} to the map $\varnothing \longrightarrow \langle 0 \rangle$ in $\mathbf{POp}_o$. For part (ii), define
\begin{equation*}
\bar\omega_!(\langle 0 \rangle \rightarrow X) \, = \, \omega_!(X)
\end{equation*}
Since $\omega_!(\langle 0 \rangle) = \varnothing$, the diagram commutes. Moreover, $\bar\omega_!$ has a right adjoint because for any forest set $Y$, the object $\omega^*(Y)$ in $\mathbf{POp}_o$ has a \emph{unique} map $\langle 0 \rangle \longrightarrow \omega^*(Y)$ (cf. Lemma \ref{lem:omega*pointed}). Thus there is a unique functor
\begin{equation*}
\bar\omega^*: \mathbf{fSets}_o \longrightarrow \langle 0 \rangle/\mathbf{POp}_o
\end{equation*}
with the property that $\langle 0 \rangle^*\bar\omega^* = \omega^*$. It is now trivial to check that $\bar\omega^*$ is indeed right adjoint to $\bar\omega_!$. Finally, Lemma \ref{lem:Qslices} gives that $\bar\omega_!$ is left Quillen since $\omega_!$ is (cf. Proposition \ref{prop:omega!leftQ}). $\Box$
\end{proof}

As suggested already, the main reason for the change from $\omega^*: \mathbf{fSets}_o^+ \longrightarrow \mathbf{POp}_o$ to $\bar\omega^*: \mathbf{fSets}_o^+ \longrightarrow \langle 0 \rangle/\mathbf{POp}_o$ is the following.

\begin{lemma}
\label{lem:omega*prescof}
\begin{itemize}
\item[(i)] The functor $\bar\omega^*: \mathbf{fSets}_o^+ \longrightarrow \langle 0 \rangle/\mathbf{POp}_o$ has a right adjoint.
\item[(ii)] The functor $\bar\omega^*$ preserves cofibrations.
\end{itemize}
\end{lemma}
\begin{proof}
(i). On the underlying categories without markings, we can define a functor
\begin{equation*}
\bar\omega_*: \langle 0 \rangle/(\mathbf{sSets}/N\mathbf{F}_o) \longrightarrow \mathbf{fSets}_o
\end{equation*}
as follows. For an object $(X, x_0)$ where $X \in \mathbf{sSets}/N\mathbf{F}_o$ and $x_0: \langle 0 \rangle \rightarrow X$, and for a forest $F$, set
\begin{equation*}
\bar\omega_*(X,x_0)(F) \, = \, \mathrm{Hom}_*(\bar\omega^*(F), (X,x_0))
\end{equation*}
Here $\mathrm{Hom}_*$ denotes the set of pointed maps in $\mathbf{sSets}/N\mathbf{F}_o$. (Recall that $\omega^*(F)$ has a unique map $\langle 0 \rangle \rightarrow \omega^*(F)$.) In order to prove that $\bar\omega_*$ is indeed right adjoint, it suffices to prove that $\bar\omega^*$ preserves colimits. But this is clear from the way colimits are computed in the slice $\langle 0 \rangle/(\mathbf{sSets}/N\mathbf{F}_o)$, together with the fact that
\begin{equation*}
\omega^*(Y)(A) \, = \, \mathrm{Hom}(\omega_!A, Y)
\end{equation*}
where $\omega_!A$ is representable unless $A = \langle 0 \rangle$, while $\omega_!\langle 0 \rangle = \varnothing$ so that $\omega^*(Y)(\langle 0 \rangle)$ is a singleton, as already remarked. Finally, the markings on $\bar\omega_*(X, x_0)$ are determined by adjunction: the marked elements in $\bar\omega_*(X, x_0)(C_1)$ are the maps $\bar\omega^*(C_1^\sharp) \longrightarrow (X,x_0)$ in $\langle 0 \rangle/\mathbf{POp}_o$. \par 
(ii). As a right adjoint, the functor $\bar\omega^*$ preserves monomorphisms. A fortiori, it preserves cofibrations. $\Box$
\end{proof}

In Section \ref{sec:omega*leftQ} below we will in fact prove the following:

\begin{proposition*}[See Proposition \ref{prop:omega*leftQ}]
The adjoint functors
\[
\xymatrix{
\bar\omega^*: \mathbf{fSets}_o^+ \ar@<.5ex>[r] & \langle 0 \rangle/\mathbf{POp}_o: \bar\omega_* \ar@<.5ex>[l]
}
\]
form a Quillen pair.
\end{proposition*} 

We end this section with a discussion of the functor $\omega^*$. More precisely, we will discuss the simplices of the object $\omega^*(F^\flat)$, for $F$ a representable forest set. The goal of this discussion is twofold. First, it will allow us to fix terminology to be used in Sections \ref{sec:omega!equiv} and \ref{sec:omega*leftQ}. Second, by giving an explicit description of $\omega^*(F^\flat)$ in some particular cases we hope to provide the reader with some intuition regarding the behaviour of this functor, which should make subsequent sections easier to read. \par 
\emph{Notation.} For a simplex $A: \Delta^n \longrightarrow N\mathbf{F}_o$, we will often use the notation
\begin{equation*}
A = \langle a(0) \rangle \rightarrow \langle a(1) \rangle \rightarrow \cdots \rightarrow \langle a(n) \rangle,
\end{equation*}
where $A(i) = \langle a(i) \rangle$ denotes the object $\{1, \ldots, a(i)\}$ of $\mathbf{F}_o$ and the arrows are partial maps. An $n$-simplex of $\omega^*(F^\flat)$ over $A^\flat$, i.e. a diagram
\[
\xymatrix{
(\Delta^n)^\flat \ar[rr]^\zeta\ar[dr]_A & & \omega^*(F^\flat) \ar[dl] \\
& N\mathbf{F}_o^\natural &
}
\]
is by definition a map $\omega_!(A^\flat) \longrightarrow F^\flat$ and so in particular gives for each $i$ a map
\begin{equation*}
\zeta(i): \langle a(i) \rangle \longrightarrow \mathrm{edges}(F)
\end{equation*}
whose image is a set of pairwise independent edges of $F$. The $n$-simplex $\zeta$ is completely determined by the sequence of maps $\zeta(i)$, although of course not every such sequence defines an $n$-simplex. \par 

\emph{Terminology.} We consider the following types of maps in $\mathbf{F}_o$:
\begin{center}
\begin{tabular}{l l l}
(type 1) & $\widehat i: \langle n \rangle \longrightarrow \langle n-1 \rangle$ & (forget $i$) \\
(type 2) & $\sigma: \langle n \rangle \longrightarrow \langle n \rangle$ & (permutation) \\
(type 3) & $\mu_{k,l}: \langle k + l \rangle \longrightarrow \langle 1 + l \rangle$ &
\end{tabular}
\end{center}
The map of type 1 is the unique inert order-preserving partial map $\langle n \rangle \rightarrow \langle n-1 \rangle$ whose domain of definition is precisely $\{1, \ldots, \widehat i, \ldots, n \}$, the hat denoting omission. The map of type 2 is an isomorphism of finite sets given by some element $\sigma$ in the symmetric group $\Sigma_n$. The map of type 3 is the active morphism sending $\{1, \ldots, k\}$ to $\{1\}$ and $k+1, \ldots, k + l$ to $2, \ldots, 1+l$ respectively. Observe that every arrow in $\mathbf{F}_o$ is a composition of a sequence of arrows of these three types. Accordingly, any non-degenerate simplex of $N\mathbf{F}_o$ is a face (possibly of high codimension) of a simplex $A: \Delta^n \longrightarrow N\mathbf{F}_o$ whose edges $A(i) \rightarrow A(i+1)$ are all of the types just described. We call such a simplex $A$ \emph{elementary}. \par 
For a forest $F$, we will now define corresponding notions of \emph{elementary} 1-simplices of $\omega^*(F^\flat)$: \par
\emph{Type 1}. An independent set of edges $e_1, \ldots, e_n$ of $\mathbf{F}$ (together with an order on them as indicated) determines a non-degenerate \emph{marked} 1-simplex which we depict as
\[
\begin{tikzpicture} 
[level distance=7mm, 
every node/.style={fill, circle, minimum size=.1cm, inner sep=0pt}, 
level 1/.style={sibling distance=20mm}, 
level 2/.style={sibling distance=10mm}, 
level 3/.style={sibling distance=5mm}]

%left tree
\node (lefttree)[style={color=white}] {} [grow'=up] 
child {node (level1) {}
	child
};

%tree i-1
\node (tree2)[style={color=white}, right = 1.2cm of lefttree] {} [grow'=up] 
child {node {}
	child
};

%tree i
\node (treei)[style={color=white}, right = .8cm of tree2] {};
\node (treeistart)[style={color=white}, above = .58cm of treei] {} [grow'=up] 
child;

%tree i+1
\node (tree3)[style={color=white}, right = .8cm of treei] {} [grow'=up] 
child {node{}
	child
};

%tree n
\node (tree4)[style={color=white}, right = 1.2cm of tree3] {} [grow'=up] 
child {node{}
	child
};

\tikzstyle{every node}=[]

%lines
\draw[dashed] ($(level1) + (-1.5cm, 0cm)$) -- ($(level1) + (4.5cm, 0cm)$);

%labels
\node at ($(level1) + (-1.5cm, .35cm)$) {$0$};
\node at ($(level1) + (-1.5cm, -.35cm)$) {$1$};
\node at ($(level1) + (.7cm, .35cm)$) {$\cdots$};
\node at ($(level1) + (3.8cm, .35cm)$) {$\cdots$};
\node at ($(level1) + (0cm, .9cm)$) {$e_1$};
\node at ($(level1) + (1.3cm, .9cm)$) {$e_{i-1}$};
\node at ($(level1) + (2.2cm, .9cm)$) {$e_i$};
\node at ($(level1) + (3.2cm, .9cm)$) {$e_{i+1}$};
\node at ($(level1) + (4.4cm, .9cm)$) {$e_n$};

\end{tikzpicture} 
\]

It is a 1-simplex over $\widehat i: \langle n \rangle \longrightarrow \langle n-1 \rangle$. The unary vertices in $\omega_!(\widehat i)$ are sent to identities of the respective edges $e_j$, $j \neq i$. Thus, a 1-simplex of $\omega^*(F^\flat)$ of type 1 involves no nontrivial vertices of $F$ and only `forgets' a single edge. \par
\emph{Type 2.} An independent sequence of edges $e_1, \ldots, e_n$ as above and a transposition $(i,i+1) \in \Sigma_n$ determine a non-degenerate \emph{marked} 1-simplex of $\omega^*(F^\flat)$ which we picture as 
\[
\begin{tikzpicture} 
[level distance=7mm, 
every node/.style={fill, circle, minimum size=.1cm, inner sep=0pt}, 
level 1/.style={sibling distance=20mm}, 
level 2/.style={sibling distance=10mm}, 
level 3/.style={sibling distance=5mm}]

%left tree
\node (lefttree)[style={color=white}] {} [grow'=up] 
child {node (level1) {}
	child
};

%tree i
\node (tree2)[style={color=white}, right = 1.2cm of lefttree] {}; 
\node at ($(tree2) + (.4cm,.7cm)$){};
\draw [rounded corners]($(tree2) + (0cm, 0cm)$) -- ($(tree2) + (0cm, 0.4cm)$) -- ($(tree2) + (.8cm, 1cm)$) -- ($(tree2) + (.8cm, 1.4cm)$); 

%tree i+1
\draw [rounded corners]($(tree2) + (0.8cm, 0cm)$) -- ($(tree2) + (0.8cm, 0.4cm)$) -- ($(tree2) + (.5cm, .63cm)$);
\draw [rounded corners] ($(tree2) + (0cm, 1.4cm)$) -- ($(tree2) + (0cm, 1cm)$) -- ($(tree2) + (.3cm, .77cm)$);

%tree n
\node (tree4)[style={color=white}, right = 2cm of tree2] {} [grow'=up] 
child {node{}
	child
};

\tikzstyle{every node}=[]

%lines
\draw[dashed] ($(level1) + (-1.5cm, 0cm)$) -- ($(level1) + (4.5cm, 0cm)$);

%labels
\node at ($(level1) + (-1.5cm, .35cm)$) {$0$};
\node at ($(level1) + (-1.5cm, -.35cm)$) {$1$};
\node at ($(level1) + (.7cm, .35cm)$) {$\cdots$};
\node at ($(level1) + (2.7cm, .35cm)$) {$\cdots$};
\node at ($(level1) + (0cm, .9cm)$) {$e_1$};
\node at ($(level1) + (1.3cm, .9cm)$) {$e_{i}$};
\node at ($(level1) + (2.2cm, .9cm)$) {$e_{i+1}$};
\node at ($(level1) + (3.4cm, .9cm)$) {$e_n$};

\end{tikzpicture} 
\]
It is a 1-simplex lying over the transposition $(i, i+1): \langle n \rangle \rightarrow \langle n \rangle$. Again, vertices of $\omega_!\bigl((i, i+1)\bigr)$ are sent to identities and no non-trivial operations of $\mathbf{F}$ are involved. Similar 1-simplices of course exist for any permutation $\sigma \in \Sigma_n$, which we will not attempt to draw. \par
\emph{Type 3.} For a vertex $v$ in $F$ with input edges $e_1, \ldots, e_k$ and output edge $d$, and then $l$ further independent edges $a_1, \ldots, a_l$ (also independent from $e_1, \ldots, e_k$), there is a 1-simplex of $\omega^*(F^\flat)$ depicted as
\[
\begin{tikzpicture} 
[level distance=7mm, 
every node/.style={fill, circle, minimum size=.1cm, inner sep=0pt}, 
level 1/.style={sibling distance=20mm}, 
level 2/.style={sibling distance=15mm}, 
level 3/.style={sibling distance=5mm}]

%left tree
\node (lefttree)[style={color=white}] {} [grow'=up] 
child {node (level1) {}
	child
	child
};

%tree 2
\node (tree2)[style={color=white}, right = 1.3cm of lefttree] {} [grow'=up]
child{ node{}
	child
};

%tree 3
\node (tree3)[style={color=white}, right = 1.2cm of tree2] {} [grow'=up] 
child {node{}
	child
};

\tikzstyle{every node}=[]

%lines
\draw[dashed] ($(level1) + (-1.5cm, 0cm)$) -- ($(level1) + (3cm, 0cm)$);

%labels
\node at ($(level1) + (-1.5cm, .35cm)$) {$0$};
\node at ($(level1) + (-1.5cm, -.35cm)$) {$1$};
\node at ($(level1) + (0.05cm, .35cm)$) {$\cdots$};
\node at ($(level1) + (2.1cm, .35cm)$) {$\cdots$};
\node at ($(level1) + (-.7cm, .9cm)$) {$e_1$};
\node at ($(level1) + (.7cm, .9cm)$) {$e_k$};
\node at ($(level1) + (0cm, -.9cm)$) {$d$};
\node at ($(level1) + (1.5cm, .9cm)$) {$a_{1}$};
\node at ($(level1) + (2.8cm, .9cm)$) {$a_{l}$};

\end{tikzpicture} 
\]
It is a 1-simplex over $\mu_{k,l}$ sending the elements of $\langle k+l \rangle$ to $e_1, \ldots, e_k, a_1, \ldots, a_l$ (in that order) and sending the $k$-ary vertex of $\omega_!(\mu_{k,l})$ to $v$, while sending all the other (unary) vertices to the identities on $a_1, \ldots, a_l$ respectively. \par
Every non-degenerate simplex of $\omega^*(F^\flat)$ is a face of some $n$-simplex $\zeta$ such that each edge $\zeta(\Delta^{\{i, i+1\}})$ is of one of the three types described above. We will call such a simplex $\zeta$ \emph{elementary}. In the special case that all those edges are in fact of type 1, we will say that $\zeta$ and every face of $\zeta$ is an \emph{obliviant} simplex. Thus, an obliviant 1-simplex of $\omega^*(F^\flat)$ is given by an independent sequence $e_1, \ldots, e_n$ of edges of $F$ and a subset of these which one `forgets'. A typical picture of such an obliviant 1-simplex looks as follows: 
\[
\begin{tikzpicture} 
[level distance=7mm, 
every node/.style={fill, circle, minimum size=.1cm, inner sep=0pt}, 
level 1/.style={sibling distance=20mm}, 
level 2/.style={sibling distance=10mm}, 
level 3/.style={sibling distance=5mm}]

%left tree
\node (lefttree)[style={color=white}] {} [grow'=up] 
child {node (level1) {}
	child
};

%tree 2
\node (tree2)[style={color=white}, right = .7cm of lefttree] {} [grow'=up] 
child {node {}
	child
};

%tree 3
\node (tree3)[style={color=white}, right = .7cm of tree2] {};
\node (tree3start)[style={color=white}, above = .58cm of tree3] {} [grow'=up] 
child;

%tree 4
\node (tree4)[style={color=white}, right = .7cm of tree3] {} [grow'=up] 
child {node{}
	child
};

%tree 5
\node (tree5)[style={color=white}, right = .7cm of tree4] {} [grow'=up] 
child {node{}
	child
};

%tree 6
\node (tree6)[style={color=white}, right = .7cm of tree5] {};
\node (tree6start)[style={color=white}, above = .58cm of tree6] {} [grow'=up] 
child;

%tree 3
\node (tree7)[style={color=white}, right = .7cm of tree6] {};
\node (tree7start)[style={color=white}, above = .58cm of tree7] {} [grow'=up] 
child;

%tree 8
\node (tree8)[style={color=white}, right = .7cm of tree7] {} [grow'=up] 
child {node{}
	child
};

\tikzstyle{every node}=[]

%lines
\draw[dashed] ($(level1) + (-.8cm, 0cm)$) -- ($(level1) + (6cm, 0cm)$);

%labels
\node at ($(level1) + (-.8cm, .35cm)$) {$0$};
\node at ($(level1) + (-.8cm, -.35cm)$) {$1$};
\node at ($(level1) + (0cm, .9cm)$) {$e_1$};
\node at ($(level1) + (.8cm, .9cm)$) {$e_2$};
\node at ($(level1) + (1.7cm, .9cm)$) {$\cdots$};
\node at ($(level1) + (5.7cm, .9cm)$) {$e_n$};

\end{tikzpicture} 
\]
\emph{Some examples.} Let us consider the values of the functor $\omega^*$ in several simple cases:
\begin{itemize}
\item[-] $F = \eta$: In this case $\omega^*(F^\flat)$ is the marked 1-simplex $\langle 1 \rangle \longrightarrow \langle 0 \rangle$ `forgetting' the single colour of $F$.
\item[-] $F = \eta \oplus \eta$: The simplicial set $\omega^*(F^\flat)$ has two non-degenerate $n$-simplices over 
\[
\xymatrix{
\langle 2 \rangle \ar[r]^\tau & \langle 2 \rangle \ar[r]^\tau & \cdots \ar[r]^\tau & \langle 2 \rangle \ar[r]^{\rho_i} & \langle 1 \rangle \ar[r] & \langle 0 \rangle, 
}
\]
the $n$-simplex of $N\mathbf{F}$ given by several repetitions of the non-trivial permutation $\tau$ of $\langle 2 \rangle$, one of the two inert maps $\rho_i$, $i = 1,2$, and the unique inert map $\langle 1 \rangle \rightarrow \langle 0 \rangle$. These two $n$-simplices are completely determined by the two possible bijections $\langle 2 \rangle \longrightarrow \mathrm{edges}(F)$. Any other non-degenerate simplex of $\omega^*(F^\flat)$ is a face of such a simplex. In particular, $\omega^*(F^\flat)$ contains the classifying space $B\Sigma_2$.
\item[-] $F = C_2$: Again we have the simplices listed in the previous item (where $\eta \oplus \eta$ corresponds to the two leaves of the corolla $C_2$), but also simplices lying over 
\[
\xymatrix{
\langle 2 \rangle \ar[r]^\tau & \langle 2 \rangle \ar[r]^\tau & \cdots \ar[r]^\tau & \langle 2 \rangle \ar[r] & \langle 1 \rangle \ar[r] & \langle 0 \rangle, 
}
\]
where the map $\langle 2 \rangle \longrightarrow \langle 1 \rangle$ is now the unique active such map, which corresponds to the vertex of $C_2$. 
\end{itemize}

\subsection{Proof of the equivalence}
\label{sec:omega!equiv}

In the previous section we defined two pairs of adjoint functors
\[
\xymatrix@C=35pt{
\langle 0 \rangle /\mathbf{POp}_o \ar@<.5ex>[r]^{\bar\omega_!} & \mathbf{fSets}_o^+ \ar@<.5ex>[l]^{\bar\omega^*} \ar@<.5ex>[r]^{\bar\omega^*} & \langle 0 \rangle /\mathbf{POp}_o \ar@<.5ex>[l]^{\bar\omega_*}
}
\]
and stated, but did not yet prove, that these are Quillen pairs, cf. Propositions \ref{prop:omega!leftQ} and \ref{prop:omega*leftQ}. These two propositions will be proved in Sections \ref{sec:omega!leftQ} and \ref{sec:omega*leftQ} respectively. Assuming that these pairs are indeed Quillen pairs, the goal of this section is to explain how to deduce that they are in fact Quillen equivalences. Once this is done, we will have related the model category $\mathbf{dSets}_o$ of dendroidal sets and the model category $\mathbf{POp}_o$ of $\infty$-preoperads by a sequence of Quillen equivalences, which all fit into the following diagram. In this diagram, the arrows denote the left Quillen functors and the number next to an arrow indicates the section in which we prove that the functor is a left Quillen equivalence.

\[
\xymatrix@C=35pt@R=35pt{
\mathbf{dSets}_o \ar[d]_{(-)^\flat}^{\ref{sec:markeddsets}} & \mathbf{fSets}_o \ar[l]_{u^*}^{\ref{sec:equivfsetsdsets}} \ar[d]_{(-)^\flat}^{\ref{sec:markedfsets}} & \\
\mathbf{dSets}_o^+ & \mathbf{fSets}_o^+ \ar[l]_{u^*}^{\ref{sec:markeddsets}} \ar[dr]_{\bar\omega^*}^{\ref{sec:omega*leftQ}} & \mathbf{POp}_o \ar[l]_{\omega_!}^{\ref{sec:omega!leftQ}} \ar[d]_{\ref{sec:omegaQpairs}}^{\langle 0 \rangle_!} \\
& & \langle 0 \rangle /\mathbf{POp}_o
}
\]

The proofs in subsequent parts of this chapter require a detailed understanding of the trivial cofibrations in $\mathbf{POp}_o$. To state what we need, we recall Definition B.1.1 from \cite{higheralgebra}. Write $\mathbf{2} = \Delta^0 \amalg \Delta^0$ and $\mathbf{2}^\triangleleft$ for the left cone on $\mathbf{2}$. Note that $\mathbf{2}^\triangleleft \simeq \Lambda_0^2$. Denote by $\Sigma$ the collection of maps
\begin{equation*}
p: (\Lambda_0^2)^\sharp \longrightarrow N\mathbf{F}_o^\natural 
\end{equation*}
given by
\begin{equation*}
\langle k \rangle \longleftarrow \langle m \rangle \longrightarrow \langle l \rangle
\end{equation*} 
where the two inert morphisms induce a bijection $\langle m \rangle \simeq \langle k \rangle \amalg \langle l \rangle$.

\begin{definition}
\label{def:Panodynes}
The class of \emph{$\mathfrak{P}$-anodyne morphisms} is the smallest saturated class of maps in $\mathbf{POp}_o$ containing the following maps:
\begin{itemize}
\item[($A_0$)] The inclusion
\begin{equation*}
(\Lambda^2_1)^\sharp \cup_{(\Lambda^2_1)^\flat} (\Delta^2)^\flat \longrightarrow (\Delta^2)^\sharp
\end{equation*}
for any map $(\Delta^2)^\sharp \longrightarrow N\mathbf{F}_o^\natural$.
\item[($A_1$)] The map $Q^\flat \longrightarrow Q^\sharp$ (for any map $Q^\sharp \longrightarrow N\mathbf{F}_o^\natural$), where $Q = \Delta^0 \amalg_{\Delta^{\{0,2\}}} \Delta^3 \amalg_{\Delta^{\{1,3\}}} \Delta^0$.
\item[($B_0$)] The inclusion $\{0\}^\sharp \longrightarrow (\Delta^1)^\sharp$, for any map $(\Delta^1)^\sharp \longrightarrow N\mathbf{F}_o^\natural$. 
\item[($B_1$)] Maps of the form $\mathbf{2} \longrightarrow (\mathbf{2}^\triangleleft)^\sharp$, for any map $p: (\mathbf{2}^\triangleleft)^\sharp \longrightarrow N\mathbf{F}_o^\natural$ contained in $\Sigma$.
\item[($C_0$)] Maps of the form
\begin{equation*}
(\Lambda_0^n)^\flat \cup_{(\Delta^{\{0,1\}})^\flat} (\Delta^{\{0,1\}})^\sharp 
\longrightarrow (\Delta^n)^\flat \cup_{(\Delta^{\{0,1\}})^\flat} (\Delta^{\{0,1\}})^\sharp
\end{equation*}
for any map $(\Delta^n)^\flat \cup_{(\Delta^{\{0,1\}})^\flat} (\Delta^{\{0,1\}})^\sharp \longrightarrow N\mathbf{F}_o^\natural$. (Note that these are precisely leaf anodynes of marked simplicial sets.)
\item[($C_1$)] The inner horn inclusions $(\Lambda_i^n)^\flat \longrightarrow (\Delta^n)^\flat$, for any $0 < i < n$ and any map $(\Delta^n)^\flat \longrightarrow N\mathbf{F}_o^\natural$. 
\item[($C_2$)] Maps of the form
\begin{equation*}
(\partial\Delta^n \star \mathbf{2})^\flat \cup_{(\{n\} \star \mathbf{2})^\flat} (\{n\} \star \mathbf{2})^\sharp \longrightarrow (\Delta^n \star \mathbf{2})^\flat \cup_{(\{n\} \star \mathbf{2})^\flat} (\{n\} \star \mathbf{2})^\sharp ,
\end{equation*}
where $n \geq 1$ and $(\{n\} \star \mathbf{2})^\sharp \simeq (\mathbf{2}^\triangleleft)^\sharp$ maps to $N\mathbf{F}_o^\natural$ by a morphism in $\Sigma$. 
\end{itemize}
\end{definition}

\begin{proposition}
\label{prop:Panodynestrivcof}
The class of trivial cofibrations in $\mathbf{POp}_o$ is the smallest saturated class $\mathcal{C}$ of cofibrations that contains the $\mathfrak{P}$-anodynes and has the following closure property: if $i: A \longrightarrow B$ and $j: B \longrightarrow C$ are cofibrations such that $j$ and $ji$ are in $\mathcal{C}$, then $i$ is in $\mathcal{C}$ as well.
\end{proposition}
\begin{proof}
In the appendix to \cite{higheralgebra}, Lurie proves that a map between fibrant objects in $\mathbf{POp}$ is a fibration if and only if it has the right lifting property with respect to $\mathfrak{P}$-anodynes. In fact, the proof of this result follows the same standard pattern as our proof of part (ii) of Theorem \ref{thm:markedfSets}. Given this, the same proof as that of Proposition \ref{prop:basictrivcofs} will give the desired conclusion here. $\Box$
\end{proof}

\begin{remark}
\label{rmk:sumaxiom}
In Definition \ref{def:Panodynes}, one may in fact replace the $\mathfrak{P}$-anodynes of type ($B_1$) and ($C_2$) by slightly more general families of maps, let us call them ($B_1'$) and ($C_2'$) respectively, where instead of $\mathbf{2}$ one allows an arbitrary non-empty coproduct $\coprod_\mathbf{j} \Delta^0$ over $\mathbf{j} = \{1, \ldots, j\}$ and takes $\Sigma$ to be those maps $(\mathbf{j}^\triangleleft)^\sharp \longrightarrow N\mathbf{F}_o^\natural$ given by a diagram
\[
\xymatrix{
& & \langle m \rangle \ar[dll]\ar[dl]\ar[drr] & & \\ 
\langle k_1 \rangle & \langle k_2 \rangle & \cdots & & \langle k_{j} \rangle 
}
\]
in which the inert maps induce a bijection 
\begin{equation*}
\langle m \rangle \simeq \coprod_{i=1}^j \langle k_i \rangle
\end{equation*}
Indeed, the original families ($B_1$) and ($C_2$) are special cases of this (for $\mathbf{j} = \mathbf{2}$) and conversely it is a fairly straightforward exercise to show that these more general families of maps are indeed trivial cofibrations.
\end{remark}

We now begin the proof of the main result of this section by investigating the unit morphism of the adjunction $(\bar\omega_!, \bar\omega^*)$.

\begin{proposition}
\label{prop:unitequivalence}
For any object $X$ of $\langle 0 \rangle/\mathbf{POp}_o$, the unit $\eta_X: X \longrightarrow \bar\omega^*\bar\omega_!(X)$ is a weak equivalence between cofibrant objects.
\end{proposition}

From this proposition and the fact that $\bar\omega^*$ is also \emph{left} Quillen, we immediately obtain the following consequence.

\begin{corollary}
\label{cor:omegaunit}
The derived unit $\mathrm{id} \longrightarrow \mathbf{R}\bar\omega^*\circ\mathbf{L}\bar\omega_!$ is a weak equivalence.
\end{corollary}

\begin{remark}
\label{rmk:baromega}
We have replaced the adjoint pair $\omega_!$ and $\omega^*$ with $\bar\omega_!$ and $\bar\omega^*$ in order to state that $\bar\omega^*$ is also left Quillen (in addition to being right Quillen). It follows from this that $\omega^*$ acts \emph{like} a left Quillen functor in all respects, except that it does not preserve all colimits. However, it does preserve pushouts and transfinite compositions (in fact, all connected colimits), as well as weak equivalences. This is all we will need. Note, in addition, that for an object $X$ of $\langle 0 \rangle /\mathbf{POp}_o$, the unit $X \longrightarrow \bar\omega^*\bar\omega_!(X)$ is a weak equivalence in $\langle 0 \rangle /\mathbf{POp}_o$ if and only if the unit $\langle 0 \rangle^*(X) \longrightarrow \omega^*\omega_!(\langle 0 \rangle^*(X))$ is one in $\mathbf{POp}_o$. Indeed, this is clear from the fact that $\langle 0 \rangle^*$ preserves and reflects weak equivalences, together with the identity $\bar\omega_! = \omega_!\langle 0 \rangle^*$ which holds by construction of $\bar\omega_!$. It also follows from this that $X \longrightarrow \bar\omega^*\bar\omega_!(X)$ is a weak equivalence for every object of $\langle 0 \rangle /\mathbf{POp}_o$ if and only if $Y \longrightarrow \omega^*\omega_!(Y)$ is a weak equivalence in $\mathbf{POp}_o$ for every object $Y$ there. For this reason, we will from now on not drag the extra $\langle 0 \rangle$ along and in proving the proposition above often work with $\omega^*$ and $\omega_!$ instead of $\bar\omega^*$ and $\bar\omega_!$.
\end{remark}

The proof of Proposition \ref{prop:unitequivalence} will consist of several lemmas.
\begin{lemma}
\label{lem:unitequiv1}
\begin{itemize}
\item[(i)] Consider a pushout square
\[
\xymatrix{
X \ar[d]\ar[r] & Y \ar[d] \\
X' \ar[r] & Y'
}
\]
in $\mathbf{POp}_o$, in which $X \longrightarrow X'$ is a cofibration. If the unit map $\mathrm{id} \longrightarrow \omega^*\omega_!$ is a weak equivalence at $X$, $X'$ and $Y$, then it is also a weak equivalence at $Y'$.
\item[(ii)] Let $X \rightarrowtail Y$ be a trivial cofibration in $\mathbf{POp}_o$. If the unit map $X \longrightarrow \omega^*\omega_!(X)$ is a weak equivalence, then so is $Y \longrightarrow \omega^*\omega_!(Y)$.
\end{itemize}
\end{lemma}
\begin{proof}
(i). This is a well-known special case of the `cube lemma' in model categories. In one of its versions for a model category $\mathcal{E}$, consider the Reedy category
\[
\xymatrix{
\mathbf{R} = (0 & 2 \ar[l]_-{+} \ar[r]^-{-} & 1)
}
\]
A cofibrant object in $\mathcal{E}^{\mathbf{R}}$ is precisely a diagram
\[
\xymatrix{
X' & X \ar[l]\ar[r] & Y
}
\]
where $X' \longleftarrow X$ is a cofibration while $X$ and $Y$ are cofibrant. The constant functor $\mathcal{E} \longrightarrow \mathcal{E}^\mathbf{R}$ is easily seen to be right Quillen with respect to the Reedy model structure on $\mathcal{E}^\mathbf{R}$. Therefore, its left adjoint preserves weak equivalences between cofibrant objets. Part (i) of the lemma now follows by applying this to the map represented by the vertical arrows in the diagram
\[
\xymatrix{
X' \ar[d] & X \ar[r]\ar[l]\ar[d] & Y \ar[d] \\
\omega^*\omega_!(X') & \omega^*\omega_!(X) \ar[r]\ar[l] & \omega^*\omega_!(Y)
}
\]
(We use here that $\omega^*$ preserves cofibrations and pushouts, cf. Remark \ref{rmk:baromega}.) \par 
(ii). In the square
\[
\xymatrix{
X \ar[r]^\sim\ar[d] & Y \ar[d] \\
\omega^*\omega_!(X) \ar[r] & \omega^*\omega_!(Y)
}
\]
the lower arrow is also a trivial cofibration because $\bar\omega^*$ and $\omega_!$ are both left Quillen, cf. Remark \ref{rmk:baromega}. Part (ii) is clear from this. $\Box$
\end{proof}

\begin{remark}
It follows from part (i) of the lemma and the usual skeletal filtration of simplicial sets that it suffices to prove Proposition \ref{prop:unitequivalence} for the special case where $X$ is a representable object $A: (\Delta^n)^\flat \longrightarrow N\mathbf{F}_o^\natural$ and for the marked 1-simplices $A: (\Delta^1)^\sharp \longrightarrow N\mathbf{F}_o^\natural$. Moreover, since for any such $n$-simplex $A$ the inclusion
\[
\xymatrix{
\bigcup_{i=0}^{n-1} (\Delta^{i,i+1})^\flat \ar[rr]\ar[dr] & & (\Delta^n)^\flat \ar[dl] \\
& N\mathbf{F}_o^\natural &
}
\]
is a weak equivalence, it follows by part (ii) of the lemma that it suffices to prove the Proposition for representables $A: (\Delta^n)^\flat \longrightarrow N\mathbf{F}_o^\natural$ of dimensions 0 and 1 only, together with the marked 1-simplices $A: (\Delta^1)^\sharp \longrightarrow N\mathbf{F}_o^\natural$ mentioned above.
\end{remark}

We begin with the case of 0-simplices.

\begin{lemma}
\label{lem:unitequivpoints}
For any vertex $A: \Delta^0 \longrightarrow N\mathbf{F}_o$, the unit $\eta_A: A \longrightarrow \omega^*\omega_!(A)$ is a weak equivalence.
\end{lemma}
\begin{proof}
The vertex $A$ is a finite set $A(0)$. If $A(0) = \varnothing$ then the unit is an isomorphism, while if $A(0)$ has one element then $\omega^*\omega_!(A)$ is the inert (marked) 1-simplex $\langle 1 \rangle \longrightarrow \langle 0 \rangle$ of $N\mathbf{F}_o^\natural$, so that the unit is a $\mathfrak{P}$-anodyne morphism of the form $\{0\} \longrightarrow (\Delta^1)^\sharp$. If $A(0)$ has more elements, consider the `cone' $C$ constructed as the pushout in the following diagram:
\[
\xymatrix@C=35pt@R=35pt{
\coprod_{a \in A(0)} \Delta^0 \ar[r]_\sim^{\coprod \partial_1}\ar[d] & \coprod_{a \in A(0)} (\Delta^1)^\sharp \ar[d] & \coprod_{a \in A(0)} \Delta^0 \ar[l]_{\coprod \partial_0} \ar@{-->}[dl] \\
\Delta^0 \ar[d]_A \ar[r]^\sim & C & \\
N\mathbf{F}^\natural & & 
}
\]
Here the summand $(\Delta^1)^\sharp$ indexed by $a \in A(0)$ is the inert 1-simplex $\rho_a: A(0) \longrightarrow \{a\}$ over $N\mathbf{F}_o^\natural$ and the corresponding vertex $\Delta^0 \longrightarrow N\mathbf{F}_o$ on the right of the diagram is the one-point set $\{a\}$. The dotted slanted map on the right is a trivial cofibration of the form discussed in Remark \ref{rmk:sumaxiom}, i.e. a generalized version of a $\mathfrak{P}$-anodyne of type ($B_1'$). In this way, we obtain a zigzag of weak equivalences 
\[
\xymatrix{
A \ar[r]^\sim & C & \coprod_{a \in A(0)} \langle a \rangle \ar[l]_-{\sim} 
}
\]
where we have written $\langle a \rangle$ for the vertex $\{a\}: \Delta^0 \longrightarrow N\mathbf{F}_o$. Since we already know that each $\eta_{\langle a \rangle}$ is a weak equivalence, it follows by Lemma \ref{lem:unitequiv1} and two-out-of-three that $\eta_A$ is also a weak equivalence. $\Box$
\end{proof}

We next turn to 1-simplices, possibly marked. Let us call a 1-simplex $A$ \emph{connected} if $\omega_!(A)$ consists of a single tree or is empty. For a general 1-simplex $A: \Delta^1 \longrightarrow N\mathbf{F}_o$, i.e. a partial surjection of finite sets $f: A(0) \longrightarrow A(1)$, we can write
\begin{equation*}
\omega_!(A) = \bigoplus_{a \in A(1)} C_a \oplus \bigoplus_{b \in U_f} \eta.
\end{equation*}
Here $C_a$ is the corolla with vertex $v_a$ and $f^{-1}(a)$ as the set of its leaves, while $U_f \subseteq A(0)$ is the set of $b \in A(0)$ on which $f$ is undefined. Similarly, we will compare the 1-simplex $A$ to its `decomposition' into a family of connected 1-simplices
\begin{equation*}
A_a = (f^{-1}(a) \longrightarrow \{a\}) \quad\quad \text{and} \quad\quad A_b = (\{b\} \longrightarrow \varnothing)
\end{equation*}
indexed by all $a \in A(1)$ and $b \in U_f$, which are all marked if $A$ is (and in this case each $f^{-1}(a)$ is a singleton, of course). The following two lemmas now show that $\eta_A: A \longrightarrow \omega^*\omega_!(A)$ is a weak equivalence and complete the proof of Proposition \ref{prop:unitequivalence}. The first one deals with the case of a connected 1-simplex, the second reduces the general case to the connected one.

\begin{lemma}
\label{lem:connected1simplex}
Let $B: (\Delta^1)^{\flat} \longrightarrow N\mathbf{F}_o^\natural$ be a 1-simplex $B(0) \longrightarrow B(1)$ in the nerve of $\mathbf{F}$. Suppose that either $B(1) = \langle 1 \rangle$ and $B$ is active, or $B(1) = \varnothing$ and $B(0) = \langle 1 \rangle$. Then $\eta_B: B \longrightarrow \omega^*\omega_!(B)$ is a weak equivalence, and similarly when $(\Delta^1)^\flat$ is replaced by $(\Delta^1)^\sharp$.
\end{lemma}
\begin{proof}
We distinguish various cases: \\
(i). In case $B$ is $\langle 1 \rangle \longrightarrow \langle 0 \rangle$ (the second case in the statement), then
\[
\xymatrix{
\Delta^0 \ar[rr]^{\partial_1}\ar[dr] & & (\Delta^1)^{\flat/\sharp} \ar[dl]^B \\
& N\mathbf{F}_o^\natural &
}
\]
is a weak equivalence, so this case follows from Lemma \ref{lem:unitequivpoints}. \par 
(ii). In case $B$ is $\langle 1 \rangle \longrightarrow \langle 1 \rangle$ (possibly marked), $B$ is degenerate and we can again apply Lemma \ref{lem:unitequivpoints}. \par 
(iii). The more complicated case is where $B$ is an active map $\langle k \rangle \longrightarrow \langle 1 \rangle$, for $k>0$, and $\omega_!(B)$ is the corresponding corolla $C_k$. In this case $\omega^*\omega_!(B) = \omega^*(C_k)$ is quite a bit larger: for example, it contains the entire classifying space $B\Sigma_k$ of the symmetric group (cf. the example at the end of Section \ref{sec:omegaQpairs}). \par 
Let us fix an order on the leaves of $C_k$, viewed as an isomorphism $\alpha: \langle k \rangle \longrightarrow \mathrm{leaves}(C_k)$. The non-degenerate simplices of $\omega^*(C_k)$ are all faces of two kinds of simplices, which we indicate by
\[
\xymatrix{
(\text{type 1}) & \langle k_0 \rangle \ar[d]_{\alpha_0}\ar[r]^{\sigma_1}_{\sharp} & \langle k_1 \rangle \ar[r]_{\sharp} & \cdots & \ar[r]^-{\sigma_n}_{\sharp} & \langle k_n \rangle \\
& \mathrm{leaves}(C_k) &&&& \\
(\text{type 2}) & \langle k_0 \rangle \ar[d]_{\alpha_0}\ar[r]^{\sigma_1}_{\sharp} & \langle k_1 \rangle \ar[r] & \cdots & \ar[r]^-{\sigma_{n-2}}_-{\sharp} & \langle k_{n-2} \rangle \ar[r]^-{\sigma_{n-1}} & \langle 1 \rangle \ar[d]\ar[r]^{\sigma_n}_{\sharp} & \langle 0 \rangle \\
& \mathrm{leaves}(C_k) &&&&& \mathrm{root}(C_k) &
}
\]
For the simplices of type 1, we require that $k_0 = k$, that $\langle k_0 \rangle$ is mapped to the leaves of $C_k$ by the fixed map $\alpha = \alpha_0$ and that each of the $\sigma_i$ is inert and marked. For the simplices of type 2, the map $\sigma_{n-1}$ is active, each of the other $\sigma_i$ for $i < n-1$ is necessarily an isomorphism, $k_0 = k$ and $\alpha_0 = \alpha$ again and every $\sigma_i$ for $i < n-1$ is marked, as is $\sigma_n$. Let us also define the following kind of simplices:
\[
\xymatrix{
(\text{type 2}') & \langle k_0 \rangle \ar[d]_{\alpha_0}\ar[r]^{\sigma_1}_{\sharp} & \langle k_1 \rangle \ar[r]_{\sharp} & \cdots & \ar[r]^-{\sigma_{n-1}}_-{\sharp} & \langle k_{n-1} \rangle \ar[r] & \langle 1 \rangle \ar[d] \\
& \mathrm{leaves}(C_k) &&&&& \mathrm{root}(C_k)
}
\]
Obviously, these are faces of the type 2 simplices. \par 
Now, the original simplex $B$ is the unique 1-simplex of type $2'$. The object $\omega^*\omega_!(B)$ has a filtration
\begin{eqnarray*}
B \subseteq F \subseteq G & = & \omega^*\omega_!(B) \\
B & = & F_0 \subseteq F_1 \subseteq F_2 \subseteq \cdots \subseteq \bigcup_n F_n = F \\
F & = & G_0 \subseteq G_1 \subseteq G_2 \subseteq \cdots \subseteq \bigcup_n G_n = G 
\end{eqnarray*}
where $F_n$ is obtained from $F_{n-1}$ by adding all $n$-simplices of type 1 and $G_n$ is obtained from $G_{n-1}$ by adding all $n$-simplices of type 2. For $n \geq 1$, the inclusion $F_{n-1} \subseteq F_n$ is a pushout of the form
\[
\xymatrix{
\coprod (\Lambda_0^n)^\diamond \ar[d]\ar[r] & F_{n-1} \ar[d] \\
\coprod (\Delta^n)^\diamond \ar[r] & F_n
}
\]
and hence $\mathfrak{P}$-anodyne. (Here the superscript $\diamond$ indicates that the 1-simplex $\Delta^{\{0,1\}}$ is marked; the left vertical map is a coproduct of $\mathfrak{P}$-anodynes of type ($C_0$).) The inclusion $G_0 \subseteq G_1$ is given by the pushout
\[
\xymatrix{
\Delta^0 \ar[d]_{\partial_1}\ar[r] & G_0 \ar[d] \\ 
(\Delta^1)^\sharp \ar[r] & G_1
}
\]
(adjoining the inert 1-simplex $\langle 1 \rangle \longrightarrow \langle 0 \rangle$) and is therefore also $\mathfrak{P}$-anodyne. For $n \geq 2$, we factor the inclusion $G_{n-1} \subseteq G_n$ as $G_{n-1} \subseteq G'_{n-1} \subseteq G_n$, where $G_{n-1} \subseteq G'_{n-1}$ is given by adding all $(n-1)$-simplices of type $2'$ and $G'_{n-1} \subseteq G_n$ is then given by adding all $n$-simplices of type 2. There are pushout diagrams
\[
\xymatrix{
\coprod (\Lambda_0^{n-1})^\diamond \ar[r]\ar[d] & G_{n-1} \ar[d] \\
\coprod (\Delta^{n-1})^\diamond \ar[r] & G'_{n-1} 
}
\] 
and
\[
\xymatrix{
\coprod (\Lambda_0^n)^\diamond \ar[r]\ar[d] & G'_{n-1} \ar[d] \\
\coprod (\Delta^n)^\diamond \ar[r] & G_n 
}
\]
This shows that each $F_{n-1} \subseteq F_n$ and $G_{n-1} \subseteq G_n$ is a trivial cofibration and hence that $B \longrightarrow \omega^*\omega_!(B)$ is. $\Box$
\end{proof}

To complete the proof of Proposition \ref{prop:unitequivalence}, we have to reduce the case of a general 1-simplex $A$ to the case of a connected 1-simplex $B$, which was treated in the previous lemma. This reduction is given by the following lemma: 

\begin{lemma}
\label{lem:reducetoconnected}
Let $A$ be a 1-simplex of $N\mathbf{F}_o^\natural$, given by $f: A(0) \longrightarrow A(1)$, with a `decomposition' into a family of 1-simplices $A_a$ for $a \in A(1)$ and $A_b$ for $b \in U_f$ (the set of $b$'s where $f$ is undefined), as described before Lemma \ref{lem:connected1simplex}. Then there is a zig-zag of trivial cofibrations in $\mathbf{POp}_o$ as follows:
\[
\xymatrix@C=35pt{
A \ar[r]^\sim & E & (\coprod_{a \in A(1)} A_a) \amalg (\coprod_{b \in U_f} A_b) \ar[l]_-\sim
}
\]
\end{lemma}
\begin{proof}
We will explicitly construct such an $E$. As a start, construct trivial cofibrations
\[
\xymatrix@C=35pt{
A(1) \ar[r]^\sim & C_1 & \coprod_{a \in A(1)} \langle a \rangle \ar[l]_-{\sim}
}
\]
as in Lemma \ref{lem:unitequivpoints}. So $A(1)$ and each $\langle a \rangle$ are vertices of $N\mathbf{F}_o$ and $C_1$ is a wedge of marked 1-simplices connecting $A(1)$ to each $\langle a \rangle$. In the same way, we can construct 
a wedge $C_0$ which fits into a diagram
\[
\xymatrix@C=35pt{
A(0) \ar[r]^\sim & C_0 & (\coprod_{a \in A(1)} \langle f^{-1}(a) \rangle) \amalg (\coprod_{b \in U_f} \langle b \rangle) \ar[l]_-\sim
}
\]
corresponding to writing $A(0)$ as the disjoint sum of these $f^{-1}(a)$ and these $b \in U_f$. Next, attach $A$ to $C_0 \amalg C_1$ as in the pushout
\[
\xymatrix{
\partial A \ar[r]^\sim \ar[d] & C_0 \amalg C_1 \ar[d] \\
A \ar[r]^\sim & B
}
\]
Thus $B$ is a simplicial set which can be pictured as
\[
\begin{tikzpicture} 
[every node/.style={fill, circle, minimum size=.1cm, inner sep=0pt}]

\node(anchor){}; 
\node at ($(anchor) + (0cm, -1.5cm)$){};
\node at ($(anchor) + (.5cm, -.5cm)$){};
\node at ($(anchor) + (1.6cm, .6cm)$){};
\node at ($(anchor) + (2cm, 0cm)$){};
\node at ($(anchor) + (1.6cm, -.9cm)$){};
\node at ($(anchor) + (2cm, -1.5cm)$){};

\tikzstyle{every node}=[]

%labels
\node at ($(anchor) + (3cm, .1cm)$) {$\Bigr{\}} C_0$};
\node at ($(anchor) + (3cm, -1.3cm)$) {$\Bigr{\}} C_1$};

%arrows
\draw[->] ($(anchor)$) -- node[left]{$A$} ($(anchor) + (0, -1.45cm)$);
\draw[->] ($(anchor)$) -- ($(anchor) + (.45cm, -.45cm)$);
\draw[->] ($(anchor)$) -- ($(anchor) + (1.57cm, .58cm)$);
\draw[->] ($(anchor)$) -- ($(anchor) + (1.95cm, 0cm)$);
\draw[->] ($(anchor) + (0cm,-1.5cm)$) -- ($(anchor) + (1.57cm, -.92cm)$);
\draw[->] ($(anchor) + (0cm,-1.5cm)$) -- ($(anchor) + (1.95cm, -1.5cm)$);

\end{tikzpicture} 
\]

The arrows in the upper half of the picture together constitute $C_0$, the arrows in the bottom half constitute $C_1$. Next, attach (by an inner anodyne map) for each $a \in A(1)$ a 2-simplex $\sigma_a$ to $B$ with $d_2\sigma_a = A$ and $d_0\sigma_a = A(1) \longrightarrow \langle a \rangle$;
\[
\begin{tikzpicture} 
[every node/.style={fill, circle, minimum size=.1cm, inner sep=0pt}]

\node(anchor){}; 
\node at ($(anchor) + (0cm, -1.5cm)$){};
\node at ($(anchor) + (1.5cm, -1.5cm)$){};

\tikzstyle{every node}=[]

%arrows
\draw[->] ($(anchor)$) -- node[left]{$A$} ($(anchor) + (0, -1.45cm)$);
\draw[->] ($(anchor) + (0cm,-1.5cm)$) -- node[below]{$\sharp$}($(anchor) + (1.45cm, -1.5cm)$);
\draw[->,dashed] ($(anchor)$) -- ($(anchor) + (1.45cm, -1.45cm)$);

%labels
\node at ($(anchor) + (.6cm, -1.1cm)$) {$\sigma_a$};

%fill
\draw [fill=black, draw opacity=0, fill opacity=0.3] ($(anchor)$) -- ($(anchor) + (0cm, -1.5cm)$) -- ($(anchor) + (1.5cm, -1.5cm)$) -- cycle;

\end{tikzpicture} 
\]

Also attach for each $b \in U_f$ an inert $i_b: A(1) \longrightarrow \langle 0 \rangle$ (by a pushout along $\{0\} \longrightarrow (\Delta^1)^\sharp$) and a 2-simplex $\sigma_b$ with $d_2\sigma_b = A$ and $d_0\sigma_b = i_b$ (by an inner anodyne). This gives a $\mathfrak{P}$-anodyne map $B \longrightarrow D$, where $D$ looks like
\[
\begin{tikzpicture} 
[every node/.style={fill, circle, minimum size=.1cm, inner sep=0pt}]

\node(anchor){}; 
\node at ($(anchor) + (0cm, -1.5cm)$){};
\node at ($(anchor) + (.5cm, -.5cm)$){};
\node at ($(anchor) + (1.6cm, .6cm)$){};
\node at ($(anchor) + (2cm, 0cm)$){};
\node at ($(anchor) + (1.6cm, -.9cm)$){};
\node at ($(anchor) + (2cm, -1.5cm)$){};
\node at ($(anchor) + (.5cm, -2cm)$){};

\tikzstyle{every node}=[]

%arrows
\draw[->] ($(anchor)$) -- node[left]{$A$} ($(anchor) + (0, -1.45cm)$);
\draw[->] ($(anchor)$) -- ($(anchor) + (.45, -.45cm)$);
\draw[->] ($(anchor)$) -- ($(anchor) + (1.57cm, .58cm)$);
\draw[->] ($(anchor)$) -- ($(anchor) + (1.95cm, 0cm)$);
\draw[->] ($(anchor) + (0cm,-1.5cm)$) -- ($(anchor) + (1.57cm, -.92cm)$);
\draw[->] ($(anchor) + (0cm,-1.5cm)$) -- ($(anchor) + (1.95cm, -1.5cm)$);
\draw[->] ($(anchor) + (0cm,-1.5cm)$) -- ($(anchor) + (.45cm, -1.95cm)$);
\draw[->] ($(anchor)$) -- ($(anchor) + (1.57cm, -.92cm)$);
\draw[->] ($(anchor)$) -- ($(anchor) + (1.95cm, -1.5cm)$);
\draw[->] ($(anchor)$) -- ($(anchor) + (.45cm, -1.95cm)$);

%fill
\draw [fill=black, draw opacity=0, fill opacity=0.3] ($(anchor)$) -- ($(anchor) + (1.6cm, -.9cm)$) -- ($(anchor) + (0, -1.5cm)$) -- cycle;
\draw [fill=black, draw opacity=0, fill opacity=0.3] ($(anchor)$) -- ($(anchor) + (2cm, -1.5cm)$) -- ($(anchor) + (0, -1.5cm)$) -- cycle;
\draw [fill=black, draw opacity=0, fill opacity=0.3] ($(anchor)$) -- ($(anchor) + (.5cm, -2cm)$) -- ($(anchor) + (0, -1.5cm)$) -- cycle;

\end{tikzpicture} 
\]

Finally, attach for each such $a$ and $b$ a 2-simplex $\tau_a$, respectively $\tau_b$, as in
\[
\xymatrix{
A(1)\ar[d]\ar[r]^\sharp\ar[dr]^{\tau_a}_{\sigma_a} & \langle f^{-1}(a) \rangle \ar@{-->}[d] & & A(1)\ar[d]\ar[r]^\sharp\ar[dr]^{\tau_b}_{\sigma_b} & \langle b \rangle \ar@{-->}[d] \\
A(0) \ar[r]_\sharp & \langle a \rangle & & A(0) \ar[r]_\sharp & \langle 0 \rangle
}
\]
by constructing the pushout
\[
\xymatrix{
\coprod_{a,b}(\Lambda_0^2)^\diamond \ar[d]\ar[r] & D \ar[d] \\
\coprod_{a,b}(\Delta^2)^\diamond \ar[r] & E
}
\]
This gives a trivial cofibration $A \rightarrowtail E$ by composition of $A \rightarrowtail B \rightarrowtail D \rightarrowtail E$. The simplicial set $E$ looks like a book with $A$ as its spine and a page with margin $A_a$, respectively $A_b$, for each $a \in A(1)$ and $b \in U_f$:
\[
\begin{tikzpicture} 
[every node/.style={fill, circle, minimum size=.1cm, inner sep=0pt}]

\node(anchor){}; 
\node at ($(anchor) + (0cm, -1.5cm)$){};
\node at ($(anchor) + (.5cm, -.5cm)$){};
\node at ($(anchor) + (1.6cm, .6cm)$){};
\node at ($(anchor) + (2cm, 0cm)$){};
\node at ($(anchor) + (1.6cm, -.9cm)$){};
\node at ($(anchor) + (2cm, -1.5cm)$){};
\node at ($(anchor) + (.5cm, -2cm)$){};

\tikzstyle{every node}=[]

%arrows
\draw[->] ($(anchor)$) -- node[left]{$A$} ($(anchor) + (0, -1.45cm)$);
\draw[->] ($(anchor)$) -- ($(anchor) + (.45, -.45cm)$);
\draw[->] ($(anchor)$) -- ($(anchor) + (1.57cm, .58cm)$);
\draw[->] ($(anchor)$) -- ($(anchor) + (1.95cm, 0cm)$);
\draw[->] ($(anchor) + (0cm,-1.5cm)$) -- ($(anchor) + (1.57cm, -.92cm)$);
\draw[->] ($(anchor) + (0cm,-1.5cm)$) -- ($(anchor) + (1.95cm, -1.5cm)$);
\draw[->] ($(anchor) + (0cm,-1.5cm)$) -- ($(anchor) + (.45cm, -1.98cm)$);
\draw[->] ($(anchor) + (1.6cm, .6cm)$) -- ($(anchor) + (1.6cm, -.88cm)$);
\draw[->] ($(anchor) + (2cm, 0cm)$) -- node[right]{$A_a$}($(anchor) + (2cm, -1.5cm)$);
\draw[->] ($(anchor) + (.5cm, -.5cm)$) -- ($(anchor) + (.5cm, -1.98cm)$);

%label
\node at ($(anchor) + (.7cm, -1.7cm)$){$A_b$};

%fill
\draw [fill=black, draw opacity=0, fill opacity=0.3] ($(anchor)$) -- ($(anchor) + (1.6cm, .6cm)$) -- ($(anchor) + (1.6cm, -.9cm)$) -- ($(anchor) + (0, -1.5cm)$) -- cycle;
\draw [fill=black, draw opacity=0, fill opacity=0.3] ($(anchor)$) -- ($(anchor) + (2cm, 0cm)$) -- ($(anchor) + (2cm, -1.5cm)$) -- ($(anchor) + (0, -1.5cm)$) -- cycle;
\draw [fill=black, draw opacity=0, fill opacity=0.3] ($(anchor)$) -- ($(anchor) + (.5cm, -.5cm)$) -- ($(anchor) + (.5cm, -2cm)$) -- ($(anchor) + (0, -1.5cm)$) -- cycle;

\end{tikzpicture} 
\]

These embeddings of $A_a$ into $E$ as $d_0\tau_a$ and of $A_b$ into $E$ as $d_0\tau_b$ define a map
\begin{equation*}
R = \coprod_{a \in A(1)} A_a \amalg \coprod_{b \in U_f} A_b \longrightarrow E
\end{equation*}
To complete the proof of the lemma, it now suffices to show that this map is a trivial cofibration. To this end, let us reconstruct $E$ from the coproduct of 1-simplices $R$. First, we attach to $R$ a wedge of marked 1-simplices of the form
\begin{equation*}
A(1) \longrightarrow \langle a \rangle, \quad A(1) \longrightarrow \langle 0 \rangle
\end{equation*}
and a wedge of marked 1-simplices
\begin{equation*}
A(0) \longrightarrow \langle f^{-1}(a) \rangle, \quad A(0) \longrightarrow \langle b \rangle
\end{equation*}
By pushouts along maps of type $(B_1)'$ as described in Remark \ref{rmk:sumaxiom}, this results in a trivial cofibration $R \rightarrowtail S = R \cup C_0 \cup C_1$. This $S$ looks like
\[
\begin{tikzpicture} 
[every node/.style={fill, circle, minimum size=.1cm, inner sep=0pt}]

\node(anchor){}; 
\node at ($(anchor) + (0cm, -1.5cm)$){};
\node at ($(anchor) + (.5cm, -.5cm)$){};
\node at ($(anchor) + (1.6cm, .6cm)$){};
\node at ($(anchor) + (2cm, 0cm)$){};
\node at ($(anchor) + (1.6cm, -.9cm)$){};
\node at ($(anchor) + (2cm, -1.5cm)$){};
\node at ($(anchor) + (.5cm, -2cm)$){};

\tikzstyle{every node}=[]

%arrows
\draw[->] ($(anchor)$) -- ($(anchor) + (.45, -.45cm)$);
\draw[->] ($(anchor)$) -- ($(anchor) + (1.57cm, .58cm)$);
\draw[->] ($(anchor)$) -- ($(anchor) + (1.95cm, 0cm)$);
\draw[->] ($(anchor) + (0cm,-1.5cm)$) -- ($(anchor) + (1.57cm, -.92cm)$);
\draw[->] ($(anchor) + (0cm,-1.5cm)$) -- ($(anchor) + (1.95cm, -1.5cm)$);
\draw[->] ($(anchor) + (0cm,-1.5cm)$) -- ($(anchor) + (.45cm, -1.98cm)$);
\draw[->] ($(anchor) + (1.6cm, .6cm)$) -- ($(anchor) + (1.6cm, -.88cm)$);
\draw[->] ($(anchor) + (2cm, 0cm)$) -- node[right]{$A_a$}($(anchor) + (2cm, -1.5cm)$);
\draw[->] ($(anchor) + (.5cm, -.5cm)$) -- ($(anchor) + (.5cm, -1.98cm)$);

%label
\node at ($(anchor) + (.7cm, -1.7cm)$){$A_b$};

\end{tikzpicture} 
\]
We can then enlarge $S$ by an inner anodyne map $S \rightarrowtail T$ by gluing in the 2-simplices $\tau_a$ and $\tau_b$; and finally, we can construct $T \rightarrowtail E$ by gluing in the 1-simplex $A$ together with the $\sigma_a$ using a pushout along a generalized form of a $\mathfrak{P}$-anodyne of type $(C_2')$, again as described in Remark \ref{rmk:sumaxiom}. This shows that $R \rightarrowtail E$ is a trivial cofibration and completes the proof of the lemma and hence the proof of Proposition \ref{prop:unitequivalence}. $\Box$ 
\end{proof}

With Proposition \ref{prop:unitequivalence} about the unit of the adjunction at hand, it is now easy to deal with the counit:

\begin{proposition}
\label{prop:counitequiv}
For any cofibrant object $Y$ in $\mathbf{fSets}_o^+$, the counit map $\bar\omega_!\bar\omega^*(Y) \longrightarrow Y$ is a weak equivalence in $\mathbf{fSets}_o^+$.
\end{proposition}

Applying this proposition to objects $Y$ which are both fibrant and cofibrant and using that $\bar\omega^*$ is also left Quillen, we immediately deduce:

\begin{corollary}
\label{cor:omegacounit}
The derived counit $\mathbf{L}\bar\omega_!\mathbf{R}\bar\omega^* \longrightarrow \mathrm{id}$ is a weak equivalence.
\end{corollary}

\begin{proof}[Proof of Proposition \ref{prop:counitequiv}]
The initial steps in the proof are similar to those in the proof of Proposition \ref{prop:unitequivalence}. In particular, by using induction on the skeletal filtration of $Y$, one sees that it suffices to prove the proposition for the special case where $Y$ is a forest $F$ (possibly with some marked 1-corollas). Consider the Segal core $\mathrm{fSc}(F)$ of $F$. We have a commutative square
\[
\xymatrix{
\bar\omega_!\bar\omega^*\mathrm{fSc}(F) \ar[r]^\sim\ar[d] & \bar\omega_!\bar\omega^*(F) \ar[d] \\
\mathrm{fSc}(F) \ar[r]^\sim & F
}
\]
in which the horizontal arrows are weak equivalences. Indeed, we already know this for the bottom map (by Proposition \ref{prop:Segalcoreinnanod})and for the top map it then follows since both $\bar\omega_!$ and $\bar\omega^*$ are left Quillen functors. Thus, it suffices to prove the proposition in case $Y$ is of the form $\mathrm{fSc}(F)$. Such an object $\mathrm{fSc}(F)$ is a union of forests which are each direct sums of corollas and copies of the unit tree. Up to weak equivalence we may replace direct sums by coproducts, which allows us to reduce to the case of a single corolla (marked or unmarked) or the unit tree $\eta$. But for any such object $G$, we can write $G = \omega_!(A)$ for some object $A$ of $\mathbf{POp}_o$ (in fact, a marked or unmarked 1-simplex of $N\mathbf{F}_o^\natural$ or a 0-simplex of $N\mathbf{F}_o^\natural$). Thus, the unit 
\begin{equation*}
\eta_A: A \longrightarrow \omega^*\omega_!(A)
\end{equation*}
is a weak equivalence by Proposition \ref{prop:unitequivalence} (and Remark \ref{rmk:baromega}) and hence so is $\omega_!(\eta_A)$. We now conclude that the counit $\epsilon_G$ is a weak equivalence as well, by the triangle identity for the adjunction:
\[
\xymatrix{
\omega_!A \ar[r]^-{\sim} \ar[dr]_{\mathrm{id}} & \omega_!\omega^*\omega_!(A) \ar[d]\ar@{=}[r] & \bar\omega_!\bar\omega^*G \ar[d]^{\epsilon_G} \\
& \omega_!(A) \ar@{=}[r] & G
}
\]
$\Box$
\end{proof}

For the record, we combine Corollaries \ref{cor:omegaunit} and \ref{cor:omegacounit} into the main theorem.

\begin{theorem}
The Quillen pair
\[
\xymatrix@C=40pt{
\bar\omega_!: \langle 0 \rangle/\mathbf{POp}_o \ar@<.5ex>[r] & \mathbf{fSets}_o^+ : \bar\omega^* \ar@<.5ex>[l]
}
\]
is a Quillen equivalence. Therefore the two Quillen pairs
\[
\xymatrix@C=40pt{
\omega_!: \mathbf{POp}_o \ar@<.5ex>[r] & \mathbf{fSets}_o^+ : \omega^* \ar@<.5ex>[l]
}
\]
and
\[
\xymatrix@C=40pt{
\bar\omega^*: \mathbf{fSets}_o^+ \ar@<.5ex>[r] & \langle 0 \rangle/\mathbf{POp}_o: \bar\omega_* \ar@<.5ex>[l]
}
\]
are also Quillen equivalences.
\end{theorem}

\subsection{The functor $\omega_!$ is left Quillen}
\label{sec:omega!leftQ}

\begin{proposition}
\label{prop:omega!leftQ}
The pair $(\omega_!, \omega^*)$ is a Quillen pair.
\end{proposition}
\begin{proof}
We already know from Lemma \ref{lem:omega!prescof} that $\omega_!$ preserves cofibrations. It remains to show that $\omega_!$ preserves trivial cofibrations. Since the left Quillen functor $u^*: \mathbf{fSets}_o^+ \longrightarrow \mathbf{dSets}_o^+$ is part of a Quillen equivalence, it reflects weak equivalences between cofibrant objects. Observing that every object in the image of $\omega_!$ is cofibrant, we deduce that it suffices to check that the composition $u^*\omega_!$ preserves trivial cofibrations. By Proposition \ref{prop:Panodynestrivcof}, it thus suffices to check that $u^*\omega_!$ sends $\mathfrak{P}$-anodynes to trivial cofibrations. We have seven cases to handle. The first four are easy; the cases ($C_0$), ($C_1$) and ($C_2$) require some more attention. \par  
($A_0$). The functor $u^*\omega_!$ sends maps of this type to compositions of pushouts of the marked anodyne map of type ($M_3$) (or rather, the map of dendroidal sets obtained from it by applying $u^*$). \par 
($A_1$). The simplicial set $Q$ has the following important property: a map from $Q$ to an $\infty$-category must send all 1-simplices of $Q$ to equivalences. In particular, any map from $Q$ to $N\mathbf{F}$ sends all 1-simplices to isomorphisms. Therefore the map $u^*\omega_!(Q^\flat) \longrightarrow u^*\omega_!(Q^\sharp)$ is a coproduct of copies of the map $i_!Q^\flat \longrightarrow i_!Q^\sharp$, where we have included the $i_!$ in the notation for emphasis. From the property of $Q$ mentioned above, it is easy to see that this map is a marked equivalence. \par 
($B_0$). The inclusion $u^*\omega_!(\{0\}^\sharp) \longrightarrow u^*\omega_!((\Delta^1)^\sharp)$ is a coproduct of copies of the identity map of $\eta^\sharp$ and copies of the inclusion of marked dendroidal sets $\{0\}^\sharp \longrightarrow (\Delta^1)^\sharp$. The latter is a leaf anodyne map and hence a trivial cofibration, by Lemma \ref{lem:leafanodmarkedequiv}. \par 
($B_1$). Applying $u^*\omega_!$ to a map of type $B_1$ yields a coproduct of maps of the form
\begin{equation*}
\{1\}^\sharp \longrightarrow (\Delta^1)^\sharp
\end{equation*}
and is hence root anodyne, i.e. marked anodyne of type ($M_2$). \par 
($C_0$). Suppose we have a diagram
\[
\xymatrix{
(\Lambda_0^n)^\flat \cup_{(\Delta^{\{0,1\}})^\flat} (\Delta^{\{0,1\}})^\sharp \ar[rr]\ar[dr]_{A'} & & (\Delta^n)^\flat \cup_{(\Delta^{\{0,1\}})^\flat} (\Delta^{\{0,1\}})^\sharp \ar[dl]^A \\   
& N\mathbf{F}_o^\natural & 
}
\]
We will show that the map $u^*\omega_!(A') \longrightarrow u^*\omega_!(A)$ is a leaf anodyne map of dendroidal sets. First of all, note that $u^*\omega_!(A)$ is a coproduct of trees and that the stated map will in fact split as a coproduct of maps, one corresponding to each such tree. Therefore, we may restrict our attention to the case where the simplex $A$ is connected. Also, once this restriction is made, we may assume it is totally active (i.e. every 1-simplex of $A$ is active). If it isn't, then $A(n) = \langle 0 \rangle$ and it is easily verified that the map $\omega_!(A') \longrightarrow \omega_!(A)$ is an isomorphism. Indeed, we would have $\omega_!(d_n A) = \omega_!(A)$ and $d_n A$ is already contained in $A'$. \par 
With these assumptions in place, let us begin our induction. Note that $A(0)$ is exactly the set of leaves of the tree $\omega_!(A)$ and that all these leaves are attached to a unary corolla. Indeed, the edge $A(0) \longrightarrow A(1)$ is inert and by our assumption on the connectedness of $A$ it in fact maps to an isomorphism in $\mathbf{F}_o$. Furthermore, all these leaf corollas are marked. Let us define a \emph{leaf pruning} of the (marked) tree $\omega_!(A)$ to be a pruning $P$ of $A$ (as in Definition \ref{def:pruning}) satisfying the following two conditions:
\begin{itemize}
\item[-] $P$ contains at least one of the leaves of $\omega_!(A)$.
\item[-] If $P$ contains an edge corresponding to an element $a \in A(1)$, then $P$ also contains the top vertex $v_a$ of $\omega_!(A)$ attached to that edge (which is then necessarily a unary leaf vertex of $P$). 
\end{itemize}
By adjoining the leaf prunings to $\omega_!(A')$ one by one, in an order that extends the partial order of size, we obtain a filtration 
\begin{equation*}
\omega_!(A') =: F_0 \subseteq F_1 \subseteq \cdots \subseteq \bigcup_i F_i = \omega_!(A)
\end{equation*}
Consider a map $F_i \subseteq F_{i+1}$ in this filtration, given by adjoining a leaf pruning $P$. If $P$ is already contained in $F_i$, there is nothing to prove. If it doesn't, we refine our filtration further. For a subset $H \subseteq I(P)$ of the inner edges of $P$, define (as usual) $P^{[H]}$ to be the tree obtained from $P$ by contracting all inner edges in $I(P) - H$. Extend the partial order of inclusion on the subsets of $I(P)$ to a linear order and adjoin the trees $P^{[H]}$ to $F_i$ in this order to obtain a filtration
\begin{equation*}
F_i =: F_i^0 \subseteq F_i^1 \subseteq \cdots \subseteq \bigcup_j F_i^j = F_{i+1}
\end{equation*}
Consider a map $F_i^j \subseteq F_i^{j+1}$, given by adjoining a tree $P^{[H]}$. If this map is the identity there is of course nothing to prove. Note that this is in particular the case if $H$ does not contain any of the inner edges of $P$ corresponding to $e \in A(1) \cap I(P)$. Indeed, if all these are contracted, then $P^{[H]}$ is contained in $\omega_!(d_1 A)$. (Note that here we use the fact that edges $e \in A(1)$ can never be outer edges of $P$, by the second condition in the definition of leaf pruning.) Now assume the map $F_i^j \subseteq F_i^{j+1}$ is not the identity. In particular, by the previous observation, we may assume that at least one of the inner edges $e \in A(1)$ is in $H$. We find:
\begin{itemize}
\item[-] The root face of the tree $P^{[H]}$ factors through $\omega_!(d_n A)$ and hence through $\omega_!(A')$. 
\item[-] Any inner face of the tree $P^{[H]}$ factors through $F_i^j$ by induction on the size of $H$.
\item[-] A leaf face of $P^{[H]}$ chopping off a vertex $v_e$ for some $e \in A(1) \cap H$ cannot factor through an earlier stage of the filtration, since chopping off such a vertex would not yield a leaf pruning.
\item[-] Any leaf face of $P^{[H]}$ other than the ones mentioned in the previous item will factor through an earlier leaf pruning and hence through $F_i$.
\end{itemize}
We conclude that $F_i^j \subseteq F_i^{j+1}$ is a pushout of the map
\begin{equation*}
(\Lambda^L[P^{[H]}], \mathcal{E} \cap \Lambda^L[P^{[H]}]) \longrightarrow (P^{[H]}, \mathcal{E})
\end{equation*}
where $L$ denotes the set of leaves of $P^{[H]}$ attached to vertices $v_e$ for $e \in A(1) \cap H$ and $\mathcal{E}$ is the union of the set of those leaf corollas with the degenerate corollas of $P^{[H]}$. It is easily verified that this map is a composition of leaf anodynes and is therefore a trivial cofibration, by Lemma \ref{lem:leafanodmarkedequiv}. \par 
($C_1$). Suppose we have a diagram
\[
\xymatrix{
(\Lambda_i^n)^\flat \ar[rr]\ar[dr]_{A'} & & (\Delta^n)^\flat \ar[dl]^A \\   
& N\mathbf{F}_o^\natural & 
}
\]
for $0 < i < n$. We will show that the map $u^*\omega_!(A') \longrightarrow u^*\omega_!(A)$ is inner anodyne, i.e. a composition of pushouts of marked anodynes of type ($M_1$) (or rather, the image of such a map under $u^*$). First of all, by the same argument used for $(C_0)$, we may assume that the simplex $A$ is connected and totally active. Also, define $E = A(i)$, which is a subset of the inner edges of the tree $\omega_!(A)$. \par 
We will again set up an induction using the prunings $P$ of $\omega_!(A)$ (cf. Definition \ref{def:pruning}). Adjoin all these prunings to $\omega_!(A')$ in an order extending the partial order of size to obtain a filtration
\begin{equation*}
\omega_!(A') =: F_0 \subseteq F_1 \subseteq \cdots \subseteq \bigcup_i F_i = \omega_!(A)
\end{equation*} 
Consider one of the inclusions $F_i \subseteq F_{i+1}$, given by adjoining a tree $P$. Define
\begin{equation*}
\mathcal{H}_P = I(P) - (E \cap I(P))
\end{equation*}
and consider for each $H \subseteq \mathcal{H}_P$ the tree $P^{[H]}$ defined by contracting all inner edges of $P$ contained in $\mathcal{H}_P - H$. Adjoin the trees $P^{[H]}$ to $F_i$ in an order extending the natural partial order on the subsets $H$ of $\mathcal{H}_P$ to obtain a filtration
\begin{equation*}
F_i =: F_i^0 \subseteq F_i^1 \subseteq \cdots \subseteq \bigcup_j F_i^j = F_{i+1} 
\end{equation*}
Now consider one of the maps $F_i^j \subseteq F_i^{j+1}$, given by adjoining a tree $P^{[H]}$. If it is not the identity, we can say the following:
\begin{itemize}
\item[-] The root face of $P^{[H]}$ factors through $\omega_!(d_n A)$ and hence through $\omega_!(A')$.
\item[-] Any leaf face of $P^{[H]}$ factors through $F_i$ by our induction on the size of the prunings.
\item[-] Any inner face contracting an edge of $P^{[H]}$ that is \emph{not} in $E$ factors through $F_i^j$ by our induction on the size of $H$.
\item[-] Any inner face contracting an edge of $E$ cannot factor through an earlier stage of the filtration. Indeed, it cannot factor through an earlier pruning and given this, it is clear that it also cannot factor through an earlier $P^{[H']}$.
\end{itemize}
We conclude that $F_i^j \subseteq F_i^{j+1}$ is a pushout of
\begin{equation*}
(\Lambda^{E \cap I(P)}[P^{[H]}])^\flat \longrightarrow (P^{[H]})^\flat
\end{equation*}
which is inner anodyne. \par 
($C_2$). Suppose we have a diagram
\[
\xymatrix{
(\partial\Delta^n \star \mathbf{2})^\flat \cup_{(\{n\} \star \mathbf{2})^\flat} (\{n\} \star \mathbf{2})^\sharp \ar[rr]\ar[dr]_{A'} & & (\Delta^n \star \mathbf{2})^\flat \cup_{(\{n\} \star \mathbf{2})^\flat} (\{n\} \star \mathbf{2})^\sharp \ar[dl]^A \\
& N\mathbf{F}_o^\natural &
}
\]
of the form described in Definition \ref{def:Panodynes}. We will show that the map $u^*\omega_!(A') \longrightarrow u^*\omega_!(A)$ is root anodyne, i.e. a composition of marked anodynes of type ($M_2$). The marked dendroidal set $u^*\omega_!(A)$ is a coproduct of (marked) trees and it is easy to see that the map $u^*\omega_!(A') \longrightarrow u^*\omega_!(A)$ splits as a coproduct of maps, one corresponding to each component of $u^*\omega_!(A)$. Using this observation, one sees that it in fact suffices to consider diagrams of the form
\[
\xymatrix{
(\Lambda_{n+1}^{n+1})^\flat \cup_{(\Delta^{\{n,n+1\}})^\flat} (\Delta^{\{n,n+1\}})^\sharp \ar[rr]\ar[dr]_{B'} & & (\Delta^{n+1})^\flat \cup_{(\Delta^{\{n,n+1\}})^\flat} (\Delta^{\{n,n+1\}})^\sharp \ar[dl]^{B} \\
& N\mathbf{F}_o^\natural &
}
\]
where $B$ is a connected totally active simplex. Note that the root corolla of $\omega_!(B)$ is unary and is in fact marked. Now set up a filtration
\begin{equation*}
\omega_!(B') =: F_0 \subseteq F_1 \subseteq \cdots \subseteq \bigcup_i F_i = \omega_!(B)
\end{equation*}
by adjoining the prunings of $\omega_!(B)$ one by one, in an order respecting the size of prunings. (Again, prunings here in the usual sense, obtained from $\omega_!(B)$ by an iteration of leaf faces.) Consider one of the maps $F_i \subseteq F_{i+1}$ given by adjoining a pruning $P$. We filter this map again; consider subsets $H \subseteq I(P)$ of the inner edges of $P$ and adjoin the trees $P^{[H]}$ (given by contracting all edges in $I(P) - H$) one by one, in an order compatible with the natural partial order on the subsets of $I(P)$, to get
\begin{equation*}
F_i =: F_i^0 \subseteq F_i^1 \subseteq \cdots \subseteq \bigcup_j F_i^j = F_{i+1}
\end{equation*}
Consider one of the maps $F_i^j \subseteq F_i^{j+1}$ given by adjoining a tree $P^{[H]}$. If $P^{[H]}$ is already contained in $F_i^j$ there is nothing to prove. Note that this is in particular the case if $H$ does not contain the unique element of $B(n)$, which is the incoming edge of the unary root vertex of $\omega_!(B)$. Indeed, if this edge is contracted, the resulting tree is contained in $\omega_!(d_n B)$ and hence in $\omega_!(B')$. So let us now assume $H$ contains the unique edge in $B(n)$. Then:
\begin{itemize}
\item[-] Any external face chopping off a leaf corolla of $P^{[H]}$ is contained in $F_i$ by our induction on the size of prunings.
\item[-] Any inner face of $P^{[H]}$ is contained in $F_i^j$ by our induction on $H$.
\item[-] The root face chopping off the unary marked root corolla of $P^{[H]}$ cannot factor through an earlier stage of the filtration.
\end{itemize}
Therefore $F_i^j \subseteq F_i^{j+1}$ is a pushout of the map
\begin{equation*}
\Lambda^{\mathrm{root}}[P^{[H]}]^\diamond \longrightarrow (P^{[H]})^\diamond 
\end{equation*}
where the diamond, as usual, indicates that the only non-degenerate marked corolla is the unary root corolla. We conclude that this map is root anodyne, which also concludes the proof of the proposition. $\Box$
\end{proof}

\subsection{The functor $\bar{\omega}^*$ is left Quillen}
\label{sec:omega*leftQ}

We begin with a short digression on the compatibility of the functors $\omega_!$ and $\omega^*$ with the process of `taking underlying simplicial sets'. As discussed before, there is an embedding
\begin{equation*}
u_!i_!: \mathbf{sSets}^+ \longrightarrow \mathbf{fSets}_o^+
\end{equation*}
which has a right adjoint $i^*u^*$. Similarly, there is an embedding
\begin{equation*}
j_!: \mathbf{sSets}^+ \longrightarrow \mathbf{POp}_o
\end{equation*}
which simply augments a marked simplicial set $X$ with the constant map
\begin{equation*}
X \longrightarrow \langle 1 \rangle 
\end{equation*}
This functor too has a right adjoint $j^*$, which is given by taking the fiber over the vertex $\langle 1 \rangle$. The following is clear from the definitions:

\begin{lemma}
The following diagram commutes (up to natural isomorphism):
\[
\xymatrix{
\mathbf{POp}_o \ar[rr]^{\omega_!} & & \mathbf{fSets}_o^+ \\
& \mathbf{sSets}^+ \ar[ul]^{j_!}\ar[ur]_{u_!i_!} &
}
\]
\end{lemma}

The diagram
\[
\xymatrix{
\mathbf{POp}_o & & \mathbf{fSets}_o^+ \ar[ll]_{\omega^*} \\
& \mathbf{sSets}^+ \ar[ul]^{j_!}\ar[ur]_{u_!i_!} &
}
\]
does not quite commute. However, the unit of the adjunction $(\omega_!,\omega^*)$ induces a natural transformation $j_! \longrightarrow \omega^*\omega_!j_! \simeq \omega^*u_!i_!$, which we claim is a weak equivalence. To make this precise, let us introduce a construction.

\begin{definition}
Given an object $(X \longrightarrow N\mathbf{F}_o^\natural) \in \mathbf{POp}_o$, the \emph{right cone} on this object has as underlying marked simplicial set \begin{equation*}
X^\triangleright := X \star \{v\} \cup_{(X_0 \star \{v\})^\flat} (X_0 \star \{v\})^\sharp
\end{equation*}
and its map to $N\mathbf{F}_o^\natural$ is uniquely determined by the requirement that the cone vertex $\{v\}$ is sent to $\langle 0 \rangle$. In other words, the right cone on $X$ is obtained by adding, for each $n$-simplex $A$ of $X$, an $(n+1)$-simplex $A^\triangleright$, of which each edge ending in the final vertex $v$ is marked.
\end{definition}

\begin{lemma}
The inclusion $X \longrightarrow X^\triangleright$ is a trivial cofibration.
\end{lemma} 
\begin{proof}
First form the pushout
\[
\xymatrix{
\coprod_{x \in X_0} \langle x \rangle \ar[d]_{\coprod \partial_1}\ar[r] & X \ar[d] \\ 
\coprod_{x \in X_0} (\Delta^1)^\sharp \ar[r] & Y
}
\]
to adjoin, for each vertex $x$ of $X$, an inert 1-simplex with final vertex lying over $\langle 0 \rangle$ in $N\mathbf{F}_o$. Here $\langle x \rangle$ is shorthand for the vertex $x: \Delta^0 \longrightarrow N\mathbf{F}_o$. Then form the pushout
\[
\xymatrix{
\coprod_{x \in X_0} \langle 0 \rangle \ar[r]\ar[d] & Y \ar[d] \\
\langle 0 \rangle \ar[r] & Z
}
\]
crushing the final vertices of the 1-simplices just adjoined to a single vertex $v$ lying over $\langle 0 \rangle$. The left vertical map is a weak equivalence, so $Y \longrightarrow Z$ is a weak equivalence as well. This follows from the fact that $\mathbf{POp}_o$ is left proper, or one can use the fact that the pushout is in fact a homotopy pushout (all objects are cofibrant and the top horizontal map is a cofibration). Now filter the inclusion $Z \rightarrowtail X^\triangleright$ as
\begin{equation*}
Z = S_1 \subseteq S_2 \cdots \subseteq \bigcup_n S_n = X^\triangleright
\end{equation*}
where each $S_n$ is the union of $Z$ with all the $n$-simplices of $X^\triangleright$. Then every inclusion $S_{n-1} \subseteq S_n$ is a pushout along a coproduct of $\mathfrak{P}$-anodynes of type ($C_2$) and hence itself $\mathfrak{P}$-anodyne. $\Box$ 
\end{proof}

\begin{lemma}
\label{lem:j!toomega*i!}
The natural transformation $j_! \longrightarrow \omega^*u_!i_!$ is a weak equivalence.
\end{lemma}
\begin{proof}
This follows from the previous lemma, by observing that for any marked simplicial set $K$ there is a canonical isomorphism
\begin{equation*}
\omega^*u_!i_!(K) \simeq j_!(K)^\triangleright
\end{equation*}
and that under this identification $j_!(K) \longrightarrow \omega^*u_!i_!(K)$ is precisely the map $j_!(K) \longrightarrow j_!(K)^\triangleright$ considered above. $\Box$
\end{proof}

We can now move on to the main goal of this section. Lemma \ref{lem:omega*prescof} already  states that $\bar\omega^*$ preserves cofibrations, so it remains to prove that $\bar\omega^*$ preserves trivial cofibrations. It suffices to check that $\bar\omega^*$ sends the maps of Lemma \ref{lem:markedtrivcof}(a)-(e) to trivial cofibrations in $\langle 0 \rangle/\mathbf{POp}_o$. Note that this is equivalent to checking that $\omega^*$ sends those maps to trivial cofibrations in $\mathbf{POp}_o$. Let us get the easy cases out of the way first:

\begin{proposition}
\label{prop:omega*easycases}
The functor $\omega^*$ sends maps in $\mathbf{fSets}_o^+$ of either of the following forms (see Lemma \ref{lem:markedtrivcof}) to trivial cofibrations:
\begin{itemize}
\item[(c)] The inclusion
\begin{equation*}
u_!i_!\bigl((\Lambda_1^2)^\sharp \cup_{(\Lambda_1^2)^\flat} (\Delta^2)^\flat \longrightarrow (\Delta^2)^\sharp \bigr)
\end{equation*} 
\item[(d)] For any Kan complex $K$, the inclusion
\begin{equation*}
u_!i_!\bigl(K^\flat \longrightarrow K^\sharp \bigr)
\end{equation*}
\end{itemize}
\end{proposition}
\begin{proof}
By the previous lemma, it suffices to check that the maps
\begin{eqnarray*}
j_!\bigl((\Lambda_1^2)^\sharp \cup_{(\Lambda_1^2)^\flat} (\Delta^2)^\flat &\longrightarrow& (\Delta^2)^\sharp \bigr) \\
j_!\bigl(K^\flat &\longrightarrow& K^\sharp \bigr)
\end{eqnarray*}
are trivial cofibrations, which is clear. $\Box$
\end{proof}

The rest of this section treats the three remaining cases, which require a little more work.

\begin{proposition}
\label{prop:omega*sums}
For any non-empty sequence of trees $T_1, \ldots, T_k$, the map
\begin{equation*}
\omega^*\bigl((T_1 \amalg \cdots \amalg T_k)^\flat \longrightarrow (T_1 \oplus \cdots \oplus T_k)^\flat\bigr)
\end{equation*}
is a trivial cofibration in $\mathbf{POp}_o$.
\end{proposition}

This proposition is a consequence of the following two lemmas:

\begin{lemma}
\label{lem:omega*sums1}
Suppose the functor $\omega^*$ sends boundary inclusions $\partial F^\flat \longrightarrow F^\flat$ to weak equivalences, for forests $F$ which have at least two components. Then $\omega^*$ sends the maps of Proposition \ref{prop:omega*sums} to weak equivalences.
\end{lemma}

\begin{lemma}
\label{lem:omega*sums2}
Let $F$ be a disconnected forest, i.e. a forest consisting of at least two trees. Then the map
\begin{equation*}
\omega^*(\partial F^\flat \longrightarrow F^\flat)
\end{equation*}
is a trivial cofibration.
\end{lemma}

\begin{proof}[Proof of Lemma \ref{lem:omega*sums1}]
Let $\mathcal{W}$ denote the set of cofibrations in $\mathbf{fSets}_o^+$ that are sent to weak equivalences by $\omega^*$ and assume that $\mathcal{W}$ contains the maps $\partial F^\flat \longrightarrow F^\flat$, for all forests $F$ consisting of at least two trees. Now let $F = T_1 \oplus \cdots \oplus T_k$ be any such forest. We wish to show that
\begin{equation*}
(T_1 \amalg \cdots \amalg T_k)^\flat \simeq u_!u^*F^\flat \longrightarrow F^\flat
\end{equation*}
is contained in $\mathcal{W}$. We may factor the given map as
\begin{equation*}
u_!u^*F^\flat \rightarrowtail \partial F^\flat \rightarrowtail F^\flat.
\end{equation*}
The second map is in $\mathcal{W}$ by assumption, so we have to show that the first map is as well. In fact we will prove something slightly stronger, namely that for any factorization
\begin{equation*}
u_!u^*F^\flat \rightarrowtail A \rightarrowtail \partial F^\flat
\end{equation*}
where both arrows are monos, both these maps are in $\mathcal{W}$. Such an $A$ can be written as
\begin{equation*}
A = u_!u^*F^\flat \cup H_1^\flat \cup \cdots \cup H_n^\flat
\end{equation*}
for subforests $H_1 \rightarrowtail F, \ldots, H_n \rightarrowtail F$ and we may assume that each $H_i$ is disconnected, because otherwise it is contained in $u_!u^*F$. We proceed by induction on the size of $F$ and the number $n$ of forests in $A$. The smallest case is the one where $F = \eta \oplus \eta$. Then $\partial F = \eta \amalg \eta$ and $u_!u^*F \rightarrowtail \partial F$ is an isomorphism, so there is nothing to prove. For general $F$, now assume that the assertion has been proved for all forests smaller than $F$, as well as for $A' = u_!u^*F \cup H_1 \cup \cdots \cup H_{n-1}$. Consider the diagram
\[
\xymatrix{
& P \ar[d]\ar[r]^h & H_n^\flat \ar[d] \\
u_!u^*F^\flat \ar[r]^f & A' \ar[r]^g & A
}
\]
where $P$ is the pullback in the square. Then the square is also a pushout (all maps in the diagram are monos) and the map $h$ is the composition
\begin{equation*}
h: (u_!u^*F \cap H_n)^\flat \cup \bigcup_{i < n} (H_i \cap H_n)^\flat \rightarrowtail \partial H_n^\flat \rightarrowtail H_n^\flat
\end{equation*}
The first map is in $\mathcal{W}$ by the inductive hypothesis (since $H_n$ is strictly smaller than $F$ and $u_!u^*H_n$ is contained in the domain of $h$), the second map is in $\mathcal{W}$ by assumption and therefore $h$ is in $\mathcal{W}$. Thus $g$ is in $\mathcal{W}$ since $\mathcal{W}$ is closed under pushouts. Also $f \in \mathcal{W}$ by the inductive assumption on $n$ and therefore $u_!u^*F^\flat \rightarrowtail A$ is in $\mathcal{W}$. By letting $A = \partial F^\flat$ we reach the desired conclusion. $\Box$
\end{proof}

\begin{proof}[Proof of Lemma \ref{lem:omega*sums2}]
To prove that $\omega^*(\partial F^\flat \rightarrowtail F^\flat)$ is a trivial cofibration, we will show that we can use $\mathfrak{P}$-anodynes to successively adjoin certain non-degenerate simplices to $\omega^*(\partial F^\flat)$, so that at the end every non-degenerate simplex of $\omega^*(F^\flat)$ is a face of one of the simplices having been adjoined. \par 
Consider an $n$-simplex $e: A \longrightarrow \omega^*(F^\flat)$. For this simplex not to factor through $\omega^*(\partial F^\flat)$, every edge of $F$ must occur in the image of some $e(i): \langle a(i) \rangle \rightarrow \mathrm{edges}(F)$. In particular, $e(0)$ is a bijection to the set of all leaves of $F$ and the image of $e(n)$ is a subset of the set of roots of $F$. We will especially be interested in simplices where the image of $e(n)$ consists of exactly one root of $F$, say the root of one of the constituent trees $T$ of $F$. In that case there will be a smallest number $s$, $0 \leq s \leq n$, for which $\omega_!(A|_{\Delta^{\{s, \ldots, n\}}})$ is connected. If $s=0$, then $\omega_!(A)$ is connected so $e: A \longrightarrow \omega^*(F^\flat)$ must factor through $\omega^*(\partial F^\flat)$. If $s > 0$, then $e$ maps $\omega_!(A|_{\Delta^{\{0, \ldots, s-1\}}})$ into a sum of trees. In case the last vertex $e(n): \langle a(n) \rangle \rightarrow \mathrm{edges}(F)$ consists of more than one root, we will set $s = n+1$. Let us also write $t = n - (s-1)$. This number $t$ is the number of vertices $s, \ldots, n$ of the simplex $A$ mapped into $\omega^*(T)$. In this way, we have assigned to each $n$-simplex $e: A \longrightarrow \omega^*(F^\flat)$ a \emph{size} $s$ and a \emph{tail length} $t$. A typical (schematic) picture is this:
\[
\begin{tikzpicture} 
[level distance=7mm, 
every node/.style={fill, circle, minimum size=.1cm, inner sep=0pt}, 
level 1/.style={sibling distance=3mm}, 
level 2/.style={sibling distance=3mm}, 
level 3/.style={sibling distance=3mm}]

%left tree
\node (lefttree)[style={color=white}] {} [grow'=up] 
child {node (level1) {}
	child
	child
};

%tree 2
\node (tree2)[style={color=white}, right = 1.8cm of lefttree] {};
\node (tree2start)[style={color=white}, above = 2cm of tree2] {} [grow'=up] 
child{node (node2){}
	child
	child
};

%tree 3
\node (tree3start)[style={color=white}, right = 1.5cm of tree2start] {} [grow'=up] 
child{node (node3){}
	child
	child
};

\tikzstyle{every node}=[]

%lines
\draw[dashed] ($(level1) + (-1.1cm, 0cm)$) -- ($(level1) + (4cm, 0cm)$);
\draw[dashed] ($(level1) + (-1.1cm, 1.4cm)$) -- ($(level1) + (4cm, 1.4cm)$);
\draw[dashed] ($(level1) + (-1.1cm, 2.1cm)$) -- ($(level1) + (4cm, 2.1cm)$);
\draw[dashed] ($(level1) + (-1.1cm, 3.5cm)$) -- ($(level1) + (4cm, 3.5cm)$);
\draw ($(level1) + (-.15cm,.7cm)$) -- ($(level1) + (-.825cm,3.85cm)$);
\draw ($(level1) + (.15cm,.7cm)$) -- ($(level1) + (.825cm,3.85cm)$);
\draw ($(node2) + (-.15cm,.7cm)$) -- ($(node2) + (-.375cm,1.75cm)$);
\draw ($(node2) + (.15cm,.7cm)$) -- ($(node2) + (.325cm,1.75cm)$);
\draw ($(node3) + (-.15cm,.7cm)$) -- ($(node3) + (-.375cm,1.75cm)$);
\draw ($(node3) + (.15cm,.7cm)$) -- ($(node3) + (.325cm,1.75cm)$);

%labels
\node at ($(level1) + (-1.1cm, 3.85cm)$) {$0$};
\node at ($(level1) + (-1.1cm, 1.75cm)$) {$s-1$};
\node at ($(level1) + (-1.1cm, .8cm)$) {$\vdots$};
\node at ($(level1) + (-1.1cm, -.35cm)$) {$n$};

\end{tikzpicture} 
\]

Define a non-degenerate $n$-simplex $e$ to be \emph{admissible} if $t \geq 1$ (tail of length at least 1) and the edge $e|_{\Delta^{\{s-1,s\}}}$ is obliviant (recall the terminology from Section \ref{sec:omegaQpairs}). Note that any $n$-simplex of size $s \leq n$ is a face of an admissible $m$-simplex of the same size $s$ (but with a longer tail, in general). \par 
Now let $V$ denote the collection of all admissible 1-simplices of $\omega^*(F^\flat)$, necessarily having $s=t=1$, and consider the map
\begin{equation*}
\omega^*(\partial F^\flat) \longrightarrow \omega^*(\partial F^\flat) \cup V
\end{equation*}
If the forest $F$ contains a vertex, then any simplex in $V$ is in fact already contained in $\omega^*(\partial F^\flat)$. In particular, the given map is the identity. So the only non-trivial case is where $F = k \cdot \eta = \oplus_{i=1}^k \eta$, a sum of copies of the unit tree. In this case, the given map is a pushout of a generalized $\mathfrak{P}$-anodyne of type $(B_1')$, i.e. a trivial cofibration of the form described in Remark \ref{rmk:sumaxiom}. \par 
We will proceed by induction on the pair $(s,t)$, lexicographically ordered. To this end, let $W^{(s)}$ denote the union of $\omega^*(\partial F^\flat) \cup V$ with all the admissible simplices of size at most $s$ and let $W^{(s,t)}$ denote the union of $\omega^*(\partial F^\flat) \cup V$ with all admissible simplices of size at most $s$ and tail length at most $t$. This defines filtrations
\begin{eqnarray*}
\bigl(\omega^*(\partial F^\flat) \cup V \bigr) \subseteq W^{(1)} \subseteq W^{(2)} \subseteq & \cdots & \subseteq \bigcup_s W^{(s)} = \omega^*(\partial F^\flat) \\
W^{(s-1)} \subseteq W^{(s,1)} \subseteq W^{(s,2)} \subseteq & \cdots & \subseteq \bigcup_t W^{(s,t)} = W^{(s)}
\end{eqnarray*}
It now suffices to show that the maps $W^{(s-1)} \rightarrowtail W^{(s,1)}$ and $W^{(s,t-1)} \rightarrowtail W^{(s,t)}$ are all trivial cofibrations. \par
\emph{The map $W^{(s-1)} \rightarrowtail W^{(s,1)}$}: Since $\omega^*(\partial F^\flat) \cup V = W^{(1,1)}$, we may assume that $s > 1$. Consider an $(s-1)$-simplex $e: A \rightarrow \omega^*(F^\flat)$ of size $s$ that does not factor through $W^{(s-1)}$. Then $A$ is necessarily of the form
\[
\xymatrix{
\langle a(0) \rangle \ar[r] & \cdots \ar[r] & \langle a(s-1) \rangle
}
\]
where all the maps are active, $\langle a(0) \rangle$ is in bijection with the set of leaves of $F$ and $\langle a(s-1) \rangle$ is in bijection with the set of roots of $F$. Consider the collection $V_e$ of all admissible simplices $\overline{e}: \overline{A} \rightarrow W^{(s,1)}$ of the form
\[
\xymatrix{
\langle a(0) \rangle \ar[r] & \cdots \ar[r] & \langle a(s-1) \rangle \ar[r]^-{\rho^i} & \langle 1 \rangle
}
\]
which restrict to $e$ on $A|_{\Delta^{\{0, \ldots, s-1\}}}$. Every simplex of $W^{(s,1)}$ that we're adjoining is of this form, for some $e$. We have a pushout diagram
\[
\xymatrix@C=40pt{
(\partial \Delta^{s-1} \star \mathbf{j})^\flat \cup_{(\{s-1\} \star \mathbf{j})^\flat} (\{s-1\} \star \mathbf{j})^\sharp \ar[r]\ar[d] & W^{(s-1)} \ar[d] \\
(\Delta^{s-1} \star \mathbf{j})^\flat \cup_{(\{s-1\} \star \mathbf{j})^\flat} (\{s-1\} \star \mathbf{j})^\sharp \ar[r] & W^{(s-1)} \cup V_e
}
\]
where the left vertical map is a generalized $\mathfrak{P}$-anodyne of type ($C_2'$), cf. Remark \ref{rmk:sumaxiom}, with $j$ being precisely the number of roots of $F$. Indeed, the faces $d_i(e) \star \mathbf{j}$ will be admissible of smaller size for $0 \leq i \leq s-2$, the face $d_{s-1}(e) \star \mathbf{j}$ may not be admissible but is a face of an admissible simplex of smaller size (possibly with longer tail), and $e$ does not factor through $W^{(s-1)}$ by assumption. Now letting $e$ vary, all the simplices of $W^{(s,1)}$ can be adjoined in similar fashion and we see that the map $W^{(s-1)} \rightarrowtail W^{(s,1)}$ is a trivial cofibration. \par 
\emph{The map $W^{(s,t-1)} \rightarrowtail W^{(s,t)}$}: Let $e: A \rightarrow \omega^*(F^\flat)$ be an admissible $n$-simplex of size $s$ with tail length $t > 1$ that is not already contained in $W^{(s,t-1)}$. Its face $d_k(e)$ lies in $W^{(s-1)}$ for $k < s$ and in $W^{(s,t-1)}$ for $k > s$. For $k=s$, the face $d_k(e) = d_s(e)$ cannot lie in $\omega^*(\partial F^\flat)$ because $e|_{\Delta^{\{s-1,s\}}}$ is obliviant, so no edge of $F$ is deleted in passing from $e$ to $d_s(e)$. The face $d_s(e)$ is a non-admissible simplex of size $s$ and tail length $t-1$ and it occurs as a face of a \emph{unique} admissible $n$-simplex, viz. $e$ itself. Thus, $W^{(s,t)}$ can be constructed from $W^{(s,t-1)}$ by a pushout along a coproduct of inner horn inclusions
\begin{equation*}
\coprod (\Lambda_s^n)^\flat \longrightarrow (\Delta^n)^\flat
\end{equation*}
ranging over all such admissible $n$-simplices of size $s$ and tail size $t$ (so $n = s+t-1$). In particular, $W^{(s,t-1)} \rightarrowtail W^{(s,t)}$ is inner anodyne, which completes the proof of the proposition. $\Box$
\end{proof}

\begin{proposition}
\label{prop:omega*inneranodyne}
For any tree $T$, the map
\begin{equation*}
\omega^*\bigl(\mathrm{fSc}(T)^\flat \longrightarrow T^\flat\bigr)
\end{equation*}
is a trivial cofibration.
\end{proposition}
\begin{proof}
We work by induction on the size of $T$. If $T$ is $\eta$ or $T$ is a corolla, then $\mathrm{fSc}(T) = T$ and there is nothing to prove. Write $\mathcal{W}$ for the collection of all cofibrations in $\mathbf{fSets}_o^+$ that are sent to weak equivalences by $\omega^*$. Now let $T$ be an arbitrary (larger) tree and assume the statement has been proved for all trees smaller than $T$ (i.e. all trees $S$ that admit a monomorphism $S \rightarrowtail T$ that is not an isomorphism). As in the first part of the proof of Proposition \ref{prop:Segcoregeninneran}, we conclude by induction that $\mathcal{W}$ contains the map
\begin{equation*}
\mathrm{fSc}(T)^\flat \rightarrowtail \partial^{\mathrm{ext}}(T)^\flat,
\end{equation*}
where we use that $\mathcal{W}$ is closed under composition and pushout, but also under sums, invoking Proposition \ref{prop:omega*sums} above. So it remains to prove that
\begin{equation*}
\partial^{\mathrm{ext}}(T)^\flat \longrightarrow T^\flat
\end{equation*}
is in $\mathcal{W}$. Write
\begin{equation*}
T = C_p \star (T_1, \ldots, T_p)
\end{equation*}
so that $T$ is given by grafting the trees $T_1, \ldots, T_p$ onto the leaves of $C_p$. Let us label the leaves of $C_p$ by $l_1, \ldots, l_p$ (implicitly fixing an order on them). Let us consider a non-degenerate $n$-simplex $e: A \rightarrow \omega^*(T^\flat)$ that is not already contained in $\omega^*(\partial^{\mathrm{ext}}(T)^\flat)$. Then the image of $e$ must contain the root of $T$. We will say that $e$ is \emph{admissible (of size $n$)} if $A$ is of the form
\[
\xymatrix{
\langle a(0) \rangle \ar[r] & \cdots \ar[r] & \langle p \rangle \ar[r] & \langle 1 \rangle \ar[r]^-\sharp & \langle 0 \rangle
}
\]
and furthermore the following conditions are satisfied:
\begin{itemize}
\item[-] The final edge $e|_{\Delta^{\{n-1,n\}}}$ lying over $\langle 1 \rangle \rightarrow \langle 0 \rangle$ is marked, as indicated.
\item[-] The edge $\langle p \rangle \rightarrow \langle 1 \rangle$ is active and is sent to the root corolla $C_p$ by $e$.
\item[-] The map $e(n-2):\langle p \rangle \rightarrow \mathrm{leaves}(C_p)$ maps $i$ to $l_i$ for $1 \leq i \leq p$ (i.e. is order-preserving).  
\end{itemize}
Note that any simplex of $\omega^*(T^\flat)$ is a face of some admissible simplex. Write $W^{(n)}$ for the union of $\omega^*(\partial^{\mathrm{ext}}(T)^\flat)$ with all admissible simplices of size at most $n$. This gives a filtration
\begin{equation*}
\omega^*(\partial^{\mathrm{ext}}(T)^\flat) = W^{(2)} \subseteq W^{(3)} \subseteq \cdots \subseteq \bigcup_n W^{(n)} = \omega^*(T^\flat)
\end{equation*}
We wish to show that all of the maps in this filtration are $\mathfrak{P}$-anodyne. Consider an inclusion $W^{(n-1)} \subseteq W^{(n)}$ given by adjoining a collection of admissible $n$-simplices. Let us first adjoin the $n$'th faces $d_n(e)$ of all these simplices (these are also not contained in $W^{(n-1)}$, since they are not admissible, or faces of admissible simplices already adjoined). This is done by a pushout
\[
\xymatrix{
\coprod_e (\Lambda_{n-2}^{n-1})^\flat \ar[d]\ar[r] & W^{(n-1)} \ar[d] \\
\coprod_e (\Delta^{n-1})^\flat \ar[r] & \overline{W}^{(n-1)}
}
\]
Indeed, the faces $d_i(d_n(e))$ for $i < n-2$ are faces of admissible simplices of smaller size and are thus contained in $W^{(n-1)}$, whereas the face $d_{n-1}(d_n(e))$ is contained in $\omega^*(\partial^{\mathrm{ext}}(T)^\flat)$ since it `chops off the root'. The face $d_{n-2}(d_n(e))$ on the other hand is \emph{not} a face of an admissible simplex of smaller size; the smallest admissible simplex it is a face of is in fact $e$ itself. Also, it is not a face of an admissible simplex of size $n$ other than $e$; indeed, the face $d_{n-2}(d_n(e))$ in fact uniquely determines the admissible simplex $e$. So, the map $W^{(n-1)} \rightarrowtail \overline{W}^{(n-1)}$ is inner anodyne. Now, we form another pushout
\[
\xymatrix{
\coprod_e (\Lambda_{n-2}^n)^\flat \ar[d]\ar[r] & \overline{W}^{(n-1)} \ar[d] \\
\coprod_e (\Delta^n)^\flat \ar[r] & W^{(n)}
}
\] 
which is established by similar reasoning. We conclude that $W^{(n-1)} \rightarrowtail W^{(n)}$ is inner anodyne, which also concludes the proof. $\Box$
\end{proof}

\begin{proposition}
\label{prop:omega*rootanodyne}
Consider an inclusion of the form
\begin{equation*}
(\Lambda^r[T], \mathcal{E} \cap \Lambda^r[T]) \longrightarrow (T, \mathcal{E})
\end{equation*}
where $T$ is a tree with a root corolla of valence one, $\Lambda^r[T]$ is the horn of $T$ corresponding to that root and $\mathcal{E}$ consists of all degenerate 1-corollas of $T$ together with that root corolla. The functor $\omega^*$ sends this map to a trivial cofibration.
\end{proposition}
\begin{proof}
The proof is quite similar to that of the previous proposition. In what should by now be familiar notation, we will use the abbreviation
\begin{equation*}
(\Lambda^r[T])^\diamond \longrightarrow T^\diamond
\end{equation*}
for the map under consideration. Consider a non-degenerate $n$-simplex $e: A \rightarrow \omega^*(T^\diamond)$ that is not already contained in $\omega^*((\Lambda^r[T])^\diamond)$. For the purposes of this proof, we will say that $e$ is \emph{admissible (of size $n$)} if $A$ is of the form
\[
\xymatrix{
\langle a(0) \rangle \ar[r] & \cdots \ar[r] & \langle 1 \rangle \ar[r]^\sharp & \langle 1 \rangle \ar[r]^-\sharp & \langle 0 \rangle
}
\]
and furthermore the following conditions are satisfied:
\begin{itemize}
\item[-] The final edge $e|_{\Delta^{\{n-1,n\}}}$ lying over $\langle 1 \rangle \rightarrow \langle 0 \rangle$ is marked, as indicated.
\item[-] The edge $A|_{\Delta^{\{n-2,n-1\}}}: \langle 1 \rangle \rightarrow \langle 1 \rangle$ is marked and is sent to the root corolla $C_r$ of $T$.  
\end{itemize}
In fact, the second condition is automatic by the requirement that $e$ be non-degenerate and doesn't factor through $\omega^*((\Lambda^r[T])^\diamond)$, but it is worth emphasizing. Note that any simplex of $\omega^*(T^\diamond)$ is a face of an admissible simplex. Now, similar to the last proof, let us write $W^{(n)}$ for the union of $\omega^*((\Lambda^r[T])^\diamond)$ with all admissible simplices of size at most $n$. We obtain a filtration
\begin{equation*}
\omega^*((\Lambda^r[T])^\diamond) = W^{(1)} \subseteq W^{(2)} \subseteq \cdots \subseteq \bigcup_n W^{(n)} = \omega^*(T^\diamond)
\end{equation*}
Consider an inclusion $W^{(n-1)} \subseteq W^{(n)}$ given by adjoining the collection of admissible $n$-simplices. Let us (again) first adjoin the $n$'th faces $d_n(e)$ of all these simplices (these are not contained in $W^{(n-1)}$ since they're not admissible, or faces of admissible simplices already adjoined), which is achieved by forming a pushout
\[
\xymatrix{
\coprod_e (\Lambda_{n-1}^{n-1})^\diamond \ar[r]\ar[d] & W^{(n-1)} \ar[d] \\
\coprod_e (\Delta^{n-1})^\diamond \ar[r] & \overline{W}^{(n-1)}
}
\]
where the superscript $\diamond$ now indicates that the edge $\Delta^{\{n-2,n-1\}}$ is marked. This square is indeed a pushout: the faces $d_i(d_n(e))$ for $i < n-1$ are faces of admissible simplices of smaller size and hence contained in $W^{(n-1)}$, whereas the face $d_{n-1}(d_n(e))$ is \emph{not} a face of an admissible simplex of smaller size, or a face of an admissible simplex of size $n$ other than $e$. We deduce that the map $W^{(n-1)} \rightarrowtail \overline{W}^{(n-1)}$ is $\mathfrak{P}$-anodyne. To finish, we form the pushout
\[
\xymatrix{
\coprod_e (\Lambda_{n-1}^n)^\flat \ar[r]\ar[d] & \overline{W}^{(n-1)} \ar[d] \\
\coprod_e (\Delta^n)^\flat \ar[r] & W^{(n)}
}
\]
from which we see that $\overline{W}^{(n-1)} \subseteq W^{(n)}$ is inner anodyne and thus that $W^{(n-1)} \rightarrowtail W^{(n)}$ is a trivial cofibration. $\Box$
\end{proof}

Combining Propositions \ref{prop:omega*easycases}, \ref{prop:omega*sums}, \ref{prop:omega*inneranodyne} and \ref{prop:omega*rootanodyne}, we arrive at the following result:

\begin{proposition}
\label{prop:omega*leftQ}
The adjoint pair
\[
\xymatrix@C=40pt{
\bar\omega^*: \mathbf{fSets}_o^+ \ar@<.5ex>[r] & \langle 0 \rangle/\mathbf{POp}_o: \bar\omega_* \ar@<.5ex>[l]
}
\]
is a Quillen pair.
\end{proposition}

%% file: properties.tex
\section{Some additional properties of dendrification}

In the first section of this final chapter, we deduce from the existence of the Quillen equivalences of the previous chapter that the perhaps more evident nerve functor from non-unital simplicial operads to non-unital preoperads induces an equivalence on the level of homotopy categories. In the second section we prove that the Quillen equivalences between the model categories of open dendroidal sets and non-unital preoperads are compatible with tensor products. In the third section we investigate the associativity properties of the tensor product of dendroidal sets more closely. We will conclude that the Quillen equivalences of the previous chapter induce an equivalence of symmetric monoidal categories between the homotopy categories $\mathrm{Ho}(\mathbf{dSets}_o)$ and $\mathrm{Ho}(\mathbf{POp}_o)$. Sections \ref{sec:tensorproduct} and \ref{subsec:unbiased} are independent of the first section.

\subsection{Simplicial operads and $\infty$-preoperads}
\label{sec:strictification}

As we have seen, there are Quillen equivalences (left adjoints on top) as follows:
\[
\xymatrix@C=35pt{
\mathbf{POp}_o \ar@<.5ex>[r]^{u^*\omega_!} & \mathbf{dSets}_o^+ \ar@<.5ex>[l]^{\omega^*u_*}\ar@<-.5ex>[r]_{a} & \mathbf{dSets}_o \ar@<-.5ex>[l]_{(-)^\flat} \ar@<.5ex>[r]^{W_!} & \mathbf{sOp}_o \ar@<.5ex>[l]^{W^*}
}
\]
relating the category of $\infty$-preoperads to the category of simplicial operads. However, as explained in Section \ref{sec:POp}, there is also a direct functor
\begin{equation*}
\nu: (\mathbf{sOp}_o)_f \longrightarrow \mathbf{POp}_o: \mathbf{P} \longmapsto \bigl(N\mathrm{cat}(\mathbf{P})^\natural \rightarrow N\mathbf{F}_o^\natural\bigr)
\end{equation*}
where the subscript $f$ indicates the full subcategory of $\mathbf{sOp}_o$ spanned by the fibrant simplicial operads. The goal of this section is to compare these two functors
\[
\xymatrix@C=75pt{
(\mathbf{sOp}_o)_f \ar@<.5ex>[r]^{\nu} \ar@<-.5ex>[r]_{\mathbf{R}\omega^*u_*\circ\mathbf{L}(-)^\flat\circ\mathbf{R}W^*} & \mathbf{POp}_o
}
\]
and in fact show that they are weakly equivalent. In particular, this allows us to conclude that the functor $\nu$ induces an equivalence on the level of homotopy categories. \par 
Mostly, we just have to unravel the definitions. First of all, let us take a closer look at the lower of these two functors. Assume that $\mathbf{P}$ is a fibrant simplicial operad. As was noted in Remark \ref{rmk:Sigmacofibrant}, the dendroidal set $W^*\mathbf{P}$ is cofibrant if we assume that the operad $\mathbf{P}$ is $\Sigma$-cofibrant (in particular, if it is cofibrant), so that in this case we may take
\begin{equation*}
\mathbf{L}(-)^\flat\circ\mathbf{R}W^*(\mathbf{P}) = W^*(\mathbf{P})^\flat
\end{equation*}  
To compute the effect of $\mathbf{R}\omega^*u_*$ we should fibrantly replace this marked dendroidal set. But that is easy: indeed, the map
\begin{equation*}
W^*(\mathbf{P})^\flat \longrightarrow W^*(\mathbf{P})^\natural
\end{equation*}
is a weak equivalence. (This is clear directly, but one could also identify it as the derived counit of the $((-)^\flat, a)$-adjunction.) Therefore, we may take
\begin{equation*}
\mathbf{R}\omega^*u_*\circ\mathbf{L}(-)^\flat\circ\mathbf{R}W^*(\mathbf{P}) = \omega^*u_*(W^*(\mathbf{P})^\natural).
\end{equation*} 

Our claim can then be formulated as follows:

\begin{proposition}
For $\mathbf{P}$ a fibrant and $\Sigma$-cofibrant non-unital simplicial operad, there is a natural weak equivalence
\begin{equation*}
\alpha: \omega^*u_*(W^*(\mathbf{P})^\natural) \longrightarrow \nu(\mathbf{P}).
\end{equation*}
\end{proposition}
\begin{proof}
Let us construct the map $\alpha$. First we define a map
\begin{equation*}
\omega^*u_*(W^*(\mathbf{P})) \longrightarrow \bigl(N\mathrm{cat}(\mathbf{P}) \rightarrow N\mathbf{F}_o\bigr)
\end{equation*}
of underlying simplicial sets, where $N$ as usual denotes the homotopy-coherent nerve. A simplex
\[
\xymatrix{
\Delta^n \ar[dr]_A \ar[rr]^-\zeta & & \omega^*u_*(W^*(\mathbf{P})) \ar[dl] \\
& N\mathbf{F}_o & 
}
\]
corresponds by adjunction to a map 
\begin{equation*}
W_!(u^*\omega_!(A)) \longrightarrow \mathbf{P}
\end{equation*}
which in turn gives rise to a functor
\begin{equation}
\label{eq:simplicialoperads}
\mathrm{cat}(W_!(u^*\omega_!(A))) \longrightarrow \mathrm{cat}(\mathbf{P}).
\end{equation} 
Note that we may write
\begin{equation*}
W_!(u^*\omega_!(A)) = \coprod_{p \in \pi_0(A)} W_!(u^*\omega_!(A_p))
\end{equation*} 
where the coproduct is over the connected components of $A$, i.e. the trees constituting the forest $\omega_!(A)$. Also, recall that $W_!(u^*\omega_!(A_p))$ is the Boardman-Vogt resolution of the operad $\Omega(u^*\omega_!(A_p))$, the free operad in $\mathbf{Sets}$ generated by the tree $u^*\omega_!(A_p)$. For the rest of this proof, let us use the abbreviation $\hat A = u^*\omega_!(A)$ to avoid awkward expressions. A simplex
\[
\xymatrix{
\Delta^n \ar[dr]_A \ar[rr]^-\xi & & N\mathrm{cat}(\mathbf{P}) \ar[dl] \\
& N\mathbf{F}_o & 
}
\]
is the same thing as a map
\[
\xymatrix{
\mathfrak{C}(\Delta^n) \ar[dr]\ar[rr] & & \mathrm{cat}(\mathbf{P}) \ar[dl] \\
& \mathbf{F}_o & 
}
\]
of simplicial categories over $\mathbf{F}_o$. Of course, the functor $\mathfrak{C}$ is just the restriction of the functor $W_!$ to simplicial sets, but we distinguish in notation to avoid possible confusion in what follows. There is a functor between simplicial categories over $\mathbf{F}_o$
\[
\xymatrix{
\mathfrak{C}(\Delta^n) \ar[rr]^-{\phi}\ar[dr] & & \mathrm{cat}(W_!\hat A) \ar[dl] \\
& \mathbf{F}_o &
}
\]
which can be described as follows:
\begin{itemize}
\item[-] On objects, $\phi(i) = A(i)$.
\item[-] Given the description of the Boardman-Vogt resolution in terms of labelled trees with `lengths' assigned to inner edges of trees, there is an evident map
\begin{equation*}
\phi_{i,j}: \mathfrak{C}(\Delta^n)(i,j) \longrightarrow \mathrm{cat}(W_!\hat A)(A(i),A(j)).
\end{equation*}
Indeed, the simplicial set $\mathfrak{C}(\Delta^n)(i,j)$ is a cube whose vertices can be identified with maps $v: \{i+1, \ldots, j-1\} \rightarrow \{0,1\}$, assigning lengths of either 0 or 1 to the inner edges of the simplex $\Delta^{\{i, \ldots, j\}}$. To specify the corresponding vertex of $\mathrm{cat}(W_!\hat A)(A(i),A(j))$, we should specify for each inner edge of the forest $\omega_!(A|_{\Delta^{\{i, \ldots, j\}}})$ a length of either 0 or 1. Each such inner edge $e$ is an element of some $A(k)$ for $i < k < j$ and we simply assign $v(k)$. The map $\phi_{i,j}$ is completely determined by this description.
\end{itemize} 
Now precomposing the functor of (\ref{eq:simplicialoperads}) with the functor $\phi$ yields a map
\begin{equation*}
\omega^*u_*(W^*(\mathbf{P}))(A) \longrightarrow N\mathrm{cat}(\mathbf{P})(A).
\end{equation*}
Observe that this map is natural in $A$. Also, it respects markings: a marked 1-simplex of $\omega^*u_*(W^*(\mathbf{P}))$ lying over an inert morphism $f: \langle m \rangle \rightarrow \langle n \rangle$ of $\mathbf{F}_o$ corresponds to a collection of equivalences $\{f_i\}_{i \in \langle n \rangle}$ in the operad $\mathbf{P}$ and a collection of colours $\{c_i\}_{i \in U_f}$, one for each $i$ at which $f$ is undefined. Clearly, this also corresponds to an inert 1-simplex of $N\mathrm{cat}(\mathbf{P})(A)$ and hence to a marked 1-simplex of $\nu(\mathbf{P})$. \par 
It remains to show that the map
\begin{equation*}
\alpha: \omega^*u_*(W^*(\mathbf{P})^\natural) \longrightarrow \nu(\mathbf{P})
\end{equation*}
is a weak equivalence. Since it is a map between fibrant objects of $\mathbf{POp}_o$, it suffices to check the following:
\begin{itemize}
\item[(i)] The map $\alpha$ is essentially surjective.
\item[(ii)] For each active morphism $f: \langle k \rangle \rightarrow \langle 1 \rangle$ the induced map
\begin{equation*}
\alpha_{x,y}: \mathrm{Map}(x,y)_f \longrightarrow \mathrm{Map}(\alpha(x),\alpha(y))_f
\end{equation*}
is a homotopy equivalence, for any $x$ in the fiber over $\langle k \rangle$ and $y$ in the fiber over $\langle 1 \rangle$. Here, the $\mathrm{Map}$ on the left-hand side refers to a mapping space computed in $\omega^*u_*W^*(\mathbf{P})^\natural$, the right-hand side to a mapping space in $\nu(\mathbf{P})$. Equivalently, we may also check this for the map
\begin{equation*}
\alpha^L_{x,y}: \mathrm{Map}^L(x,y)_f \longrightarrow \mathrm{Map}^L(\alpha(x),\alpha(y))_f
\end{equation*}
Recall (cf. \cite{htt}) that for an $\infty$-category $\mathcal{C}$ with vertices $x$ and $y$, these mapping objects are defined as follows:
\begin{equation*}
\mathrm{Hom}(\Delta^n,\mathrm{Map}^L(x,y)) = \{x\} \times_{\mathrm{Hom}(\{0\}, \mathcal{C})} \mathrm{Hom}(\Delta^{n+1} \amalg_{\Delta^{\{1, \ldots, n+1\}}} \Delta^0, \mathcal{C}) \times_{\mathrm{Hom}(\{1\}, \mathcal{C})} \{y\}
\end{equation*}
\end{itemize}

For (i), we note that $\alpha$ induces an isomorphism on vertices and hence is in particular essentially surjective. It remains to verify (ii). 
But the map $\alpha^L_{x,y}$ above is in fact an isomorphism. Indeed, let $T$ be the following tree:
\[
\begin{tikzpicture} 
[level distance=6mm, 
every node/.style={fill, circle, minimum size=.1cm, inner sep=0pt}, 
level 1/.style={sibling distance=12mm}, 
level 2/.style={sibling distance=12mm}, 
level 3/.style={sibling distance=12mm}]
\node[style={color=white}] {} [grow'=up] 
child {node (1){} 
		child {node{} 
			child
			child
		}
};
\node[style={color=white}, below=3cm of zero] {} [grow'=up] 
child {node (n+1){}
		child{} 
};
\draw[dashed] ($(1) + (0cm,.6cm)$) --($(n+1) + (0cm,.6cm) $);

\tikzstyle{every node}=[]

\node at ($(1) + (0cm, 1.8cm)$) {$k$ leaves};
\node at ($(1) + (0cm, 1.5cm)$) {$\overbrace{\quad\quad\quad\quad}$};
\node at ($(1) + (0cm, 1cm)$) {$\cdots$};
\node at ($(1) + (.2cm,.3cm)$) {$1$};
\node at ($(1) + (-.2cm,.6cm)$) {$v$};
\node at ($(1) + (.2cm,-.3cm)$) {$2$};
\node at ($(n+1) + (.5cm,-.3cm)$) {$n+1$};

\end{tikzpicture} 
\]

Then the $n$-simplices of $\mathrm{Map}^L(x,y)_f$ (resp. $\mathrm{Map}^L(\alpha(x),\alpha(y))_f$) computed in $\omega^*u_*(W^*(\mathbf{P})^\natural)$ (resp. $\nu(\mathbf{P})$) canonically correspond to maps
\begin{equation*}
W_!(T) \amalg_{W_!(\partial_v T)} W_!(\eta) \longrightarrow \mathbf{P}
\end{equation*}
sending the leaves of $T$ to $x$ and the root of $T$ to $y$. The map $\alpha$ is compatible with these identifications. This concludes the proof. $\Box$
\end{proof}

\subsection{Compatibility with tensor products}
\label{sec:tensorproduct}

Both the category $\mathbf{POp}_o$ of non-unital $\infty$-preoperads and the category $\mathbf{dSets}_o$ of open dendroidal sets carry a tensor product. Our goal in this section is to compare these two structures. Let us begin with a brief review of the relevant definitions.

\begin{definition}
There is a functor $\wedge: \mathbf{F} \times \mathbf{F} \longrightarrow \mathbf{F}$ which can be described as follows:
\begin{itemize}
\item[(i)] On objects, we have $\langle m \rangle \wedge \langle n \rangle = \langle mn \rangle$.
\item[(ii)] For morphisms $f: \langle m \rangle \rightarrow \langle m' \rangle$ and $g: \langle n \rangle \rightarrow \langle n' \rangle$, we have
\begin{equation*}
(f \wedge g)((k-1)n + l) = (f(k)-1)n' + g(l)
\end{equation*}
where $1 \leq k \leq m$ and $1 \leq l \leq n$.
\end{itemize}
In other words, the operation $\wedge$ is given by identifying $\langle m \rangle \times \langle n \rangle$ with $\langle mn \rangle$ via the lexicographical ordering.
\end{definition}

The operation $\wedge$ is strictly associative, but manifestly not symmetric. We can use it to define a monoidal structure on $\mathbf{POp}$ as follows:

\begin{definition}
For objects $X, Y \in \mathbf{POp}$, their \emph{tensor product} $X \odot Y$ is the composite
\[
\xymatrix{
X \times Y \ar[r] & N\mathbf{F}^\natural \times N\mathbf{F}^\natural \ar[r]^-{\wedge} & N\mathbf{F}^\natural
}
\]
\end{definition}

Observe that the operation $\odot$ restricts to a monoidal structure on the category $\mathbf{POp}_0$ of non-unital preoperads. 

%The following is a consequence of Proposition 2.2.5.7 of \cite{higheralgebra}:

%\begin{proposition}
%The model structure on $\mathbf{POp}$ is compatible with the tensor product, i.e. equipped with the monoidal structure $\odot$, the category $\mathbf{POp}$ is a monoidal model category. The same is true for $\mathbf{POp}_o$, equipped with the restriction of $\odot$ to this subcategory.
%\end{proposition}

We have already used the tensor products on $\mathbf{dSets}_o$ and $\mathbf{fSets}_o$ several times in this paper. Recall that for trees $S$ and $T$, their tensor product $S \otimes T$ can be written as a colimit over the \emph{shuffles} of the trees $S$ and $T$ (cf. \cite{moerdijkweiss}). The tensor product on $\mathbf{dSets}_o$ is completely determined by this description and the fact that it preserves colimits in each variable separately. Similarly, the tensor product on $\mathbf{fSets}_o$ is determined by the formula $u_!S \otimes u_!T = u_!(S \otimes T)$, the fact that it distributes over sums and the fact that it preserves colimits in each variable separately (cf. Section \ref{sec:fsetsproperties}). These tensor products also induce tensor products on the categories $\mathbf{dSets}_o^+$ and $\mathbf{fSets}_o^+$ of open marked dendroidal and marked forest sets respectively, as explained in Chapter \ref{chap:markedfsets}. Recall that all the left Quillen equivalences in the diagram
\[
\xymatrix{
\mathbf{dSets}_o \ar[d]_{(-)^\flat} & \mathbf{fSets}_o \ar[l]_{u^*} \ar[d]_{(-)^\flat} \\
\mathbf{dSets}_o^+ & \mathbf{fSets}_o^+ \ar[l]_{u^*}
}
\]
are compatible with tensor products. The key ingredient for our comparison results on tensor products is the following:

\begin{theorem}
\label{thm:monoidal}
\begin{itemize}
\item[(i)] For $X$ and $Y$ objects in $\mathbf{POp}_o$ there is a map
\begin{equation*}
\theta_{X,Y}: \omega_!(X \odot Y) \longrightarrow \omega_!(X) \otimes \omega_!(Y)
\end{equation*}
which is natural in $X$ and $Y$.
%\item[(ii)] The natural transformation $\theta$ is compatible with the associativity isomorphisms for the two tensor products.
\item[(ii)] The natural transformation $\theta$ is a weak equivalence. 
\end{itemize}
\end{theorem}

We will now construct the map $\theta$ and establish its desired properties. The preceding theorem will follow from Proposition \ref{prop:proofmonoidal}. Since $\odot$ and $\otimes$ preserve colimits in each variable separately and since $\omega_!$ preserves colimits, it suffices to define $\theta$ on representables $X$ and $Y$ (i.e. simplices, possibly with markings) and extend its definition by colimits. In fact, if one can prove that part (ii) of the theorem holds for simplices, it holds for all $X$ and $Y$ by induction on skeletal filtrations in view of the cube lemma (cf. the proof of Lemma \ref{lem:unitequiv1}) applied to cubes of the form
\[
\xymatrix@C=10pt{
\omega_!(\partial A \odot Y) \ar[rr]\ar[dr]\ar[dd] & & \omega_!(X \odot Y) \ar'[d][dd]\ar[dr] & \\
& \omega_!(A \odot Y) \ar[dd]\ar[rr] & & \omega_!(X' \odot Y)\ar[dd] \\
\omega_!(\partial A) \otimes \omega_!(Y) \ar'[r][rr]\ar[dr] & & \omega_!(X) \otimes \omega_!(Y) \ar[dr] & \\
& \omega_!(A) \otimes \omega_!(Y) \ar[rr] & & \omega_!(X') \otimes \omega_!(Y)
}
\]
arising from a pushout
\[
\xymatrix{
\partial A \ar[r]\ar[d] & X \ar[d] \\ 
A \ar[r] & X'.
}
\]
This reduction to simplices using the skeletal filtration is standard and has already been used several times in this paper, so we omit the details. \par 
Let us turn our attention to constructing the map $\theta_{A,B}$ for two (marked) simplices 
\begin{equation*}
A: (\Delta^m, \mathcal{E}_A) \longrightarrow N\mathbf{F}^\natural \quad\quad \text{and} \quad\quad B: (\Delta^n, \mathcal{E}_B) \longrightarrow N\mathbf{F}^\natural
\end{equation*}
We first introduce some helpful terminology:
\begin{itemize}
\item[-] Recall that the cartesian product $\Delta^m \times \Delta^n$ can also be described in terms of shuffles: this product is a union of $m+n$-simplices (each of which is called a shuffle), one corresponding to every way of linearly ordering $m$ white vertices and $n$ black vertices, respecting the order already existing on each colour. There are $\binom{m+n}{m}$ such shuffles.
\item[-] A \emph{layered forest} is a forest $F$ with a function $\lambda: \mathrm{vertices}(F) \rightarrow \mathbb{N}$, such that for any path from a leaf to a root, this function increases by 1 from any vertex to the next. In other words, given an inner edge $e$ with bottom vertex $w$ and top vertex $v$, we have $\lambda(w) = \lambda(v) + 1$. A \emph{layer} of such a forest is simply a set of vertices of the form $\lambda^{-1}(i)$. The forests $\omega_!(A)$ and $\omega_!(B)$ are naturally layered in an obvious way; for example, given a vertex $v_a$ of $\omega_!(A)$ arising from some $a \in A(i)$, we set $\lambda(v_a) = i$.
\item[-] Given two layered forests $F$ and $G$ (with layerings $\lambda_F$ and $\lambda_G$), we can consider the \emph{layered shuffles} of $F$ and $G$. To be precise, consider any shuffle $S$ of $F$ and $G$. A vertex of $S$ corresponds to either a vertex of $F$ or a vertex of $G$. We say $S$ is layered if it admits the structure of a layering $\lambda_S$ in such a way that each layer of $S$ is precisely the set of vertices corresponding to either a layer of $F$ or a layer of $G$.
\end{itemize}

As an example, consider the 2-simplex $A^\flat$ and 1-simplex $B^\flat$ of $N\mathbf{F}^\natural$ pictured below:
\[
\begin{tikzpicture} 
[level distance=10mm, 
every node/.style={fill, circle, minimum size=.1cm, inner sep=0pt}, 
level 1/.style={sibling distance=20mm}, 
level 2/.style={sibling distance=10mm}, 
level 3/.style={sibling distance=5mm}]

%left simplex
%left tree
\node (leftsimplex)[style={color=white}] {} [grow'=up] 
child {node (Llevel1) {} 
	child{ node (Llevel2) {}
		child
		child
	}
	child{ node {}
		child
	}
};

%right simplex
\node (rightsimplex)[style={color=white}, right = 7cm of leftsimplex] {};
\node (rightsimplexstart)[style={color=white}, above = .4cm of rightsimplex] {} [grow'=up] 
child {node(Rlevel1)[draw,fill=none]{} 
	child
	child
};

\tikzstyle{every node}=[]

%lines
\draw[dashed] ($(Llevel1) + (-1.5cm, 0)$) -- ($(Llevel1) + (1.5cm, 0)$);
\draw[dashed] ($(Llevel1) + (-1.5cm, 1cm)$) -- ($(Llevel1) + (1.5cm, 1cm)$);
\draw[dashed] ($(Rlevel1) + (-1cm, 0)$) -- ($(Rlevel1) + (1cm, 0)$);

%labels
\node at ($(Llevel1) + (-1.5cm, 1.5cm)$) {$0$};
\node at ($(Llevel1) + (-1.5cm, .5cm)$) {$1$};
\node at ($(Llevel1) + (-1.5cm, -.5cm)$) {$2$};
\node at ($(Rlevel1) + (-1cm, .5cm)$) {$0$};
\node at ($(Rlevel1) + (-1cm, -.5cm)$) {$1$};
\node at ($(Llevel1) + (-2.5cm, .5cm)$) {$A:$};
\node at ($(Rlevel1) + (-2.5cm, 0cm)$) {$B:$};

\end{tikzpicture} 
\]

The set of shuffles of the forests $\omega_!(A)$ and $\omega_!(B)$ looks as follows:
\[
\begin{tikzpicture} 
[level distance=5mm, 
every node/.style={fill, circle, minimum size=.1cm, inner sep=0pt}, 
level 1/.style={sibling distance=20mm}, 
level 2/.style={sibling distance=10mm}, 
level 3/.style={sibling distance=5mm},
level 4/.style={sibling distance=4mm}]

%shuffle 1
\node (shuffle1)[style={color=white}] {} [grow'=up] 
child {node[draw,fill=none] {} 
	child{ node {}
		child{ node {}
			child
			child
		}
		child{ node {}
			child
		}
	}
	child{ node {}
		child{ node {}
			child
			child
		}
		child{ node {}
			child
		}
	}
};

%shuffle2
\node (shuffle2)[style={color=white}, right = 3cm of shuffle1] {} [grow'=up] 
child {node {} 
	child{node[draw,fill=none]{}
		child{node{}
			child
			child
		}
		child{node{}
			child
			child
		}
	}
	child{node[draw,fill=none]{}
		child{node{}
			child
		}
		child{node{}
			child
		}
	}
};

%shuffle3
\node (shuffle3)[style={color=white}, right = 3cm of shuffle2]{};
\node (shuffle3start)[style={color=white}, above = 1.2cm of shuffle3] {} [grow'=up] 
child {node {} 
	child{node{}
		child{node[draw,fill=none]{}
			child
			child
		}
		child{node[draw,fill=none]{}
			child
			child
		}
	}
	child{node[draw,fill=none]{}
		child{node{}
			child
		}
		child{node{}
			child
		}
	}
};

%shuffle4
\node (shuffle4)[style={color=white}, right = 3cm of shuffle2]{};
\node (shuffle4start)[style={color=white}, below = 1.2cm of shuffle4] {} [grow'=up] 
child {node {} 
	child{node[draw,fill=none]{}
		child{node{}
			child
			child
		}
		child{node{}
			child
			child
		}
	}
	child{node{}
		child{node[draw,fill=none]{}
			child
			child
		}
	}
};

%shuffle5
\node (shuffle5)[style={color=white}, right = 6cm of shuffle2] {} [grow'=up] 
child {node {} 
	child{node{}
		child{node[draw,fill=none]{}
			child
			child
		}
		child{node[draw,fill=none]{}
			child
			child
		}
	}
	child{node{}
		child{node[draw,fill=none]{}
			child
			child
		}
	}
};

\tikzstyle{every node}=[]

%lines
\draw ($(shuffle1) + (1cm, 1cm)$) -- ($(shuffle1) + (2cm, 1cm)$);
\draw ($(shuffle2) + (1.3cm, 1.7cm)$) -- ($(shuffle2) + (2cm, 2.2cm)$);
\draw ($(shuffle2) + (1.3cm, 0.7cm)$) -- ($(shuffle2) + (2cm, .2cm)$);
\draw ($(shuffle2) + (4.3cm, 2.2cm)$) -- ($(shuffle2) + (5cm, 1.7cm)$);
\draw ($(shuffle2) + (4.3cm, 0.2cm)$) -- ($(shuffle2) + (5cm, .7cm)$);

%labels
\node at ($(shuffle1) + (-.5cm, 0)$) {$S_1$};
\node at ($(shuffle2) + (-.5cm, 0cm)$) {$S_2$};
\node at ($(shuffle3start) + (-.5cm, 0cm)$) {$S_3$};
\node at ($(shuffle4start) + (-.5cm, 0cm)$) {$S_4$};
\node at ($(shuffle5) + (-.5cm, 0cm)$) {$S_5$};

\end{tikzpicture} 
\]

However, the only shuffles that are layered are $S_1$, $S_2$ and $S_5$. Now, for any two marked simplices
\begin{equation*}
A: (\Delta^m, \mathcal{E}_A) \longrightarrow N\mathbf{F}_o^\natural \quad\quad \text{and} \quad\quad B: (\Delta^n, \mathcal{E}_B) \longrightarrow N\mathbf{F}_o^\natural
\end{equation*}
it should be clear that $\omega_!(A \odot B)$ is the union of the layered shuffles of the forests $\omega_!(A)$ and $\omega_!(B)$, corresponding precisely to the shuffles of the simplices $\Delta^m$ and $\Delta^n$. On the other hand, $\omega_!(A) \otimes \omega_!(B)$ is the union of \emph{all} the shuffles of the forests $\omega_!(A)$ and $\omega_!(B)$. There is an evident inclusion
\begin{equation*}
\theta_{A,B}: \omega_!(A \odot B) \longrightarrow \omega_!(A) \otimes \omega_!(B)
\end{equation*}
which is easily seen to be natural in $A$ and $B$. This takes care of part (i) of Theorem \ref{thm:monoidal}. It remains to deal with part (ii):

\begin{proposition}
\label{prop:proofmonoidal}
For simplices $A$ and $B$ as above, the map $\theta_{A,B}$ is a weak equivalence.
\end{proposition}
\begin{proof}
Just for simplicity of notation, we will not indicate markings in this proof and leave them implicit. They play no essential role here. First of all, we consider the Segal cores of $A$ and $B$, which give trivial cofibrations
\begin{eqnarray*}
\mathrm{Sc}(A) = \bigcup_{i=0}^{m-1} \Delta^{\{i,i+1\}} & \longrightarrow & A \\
\mathrm{Sc}(B) = \bigcup_{i=0}^{n-1} \Delta^{\{i,i+1\}} & \longrightarrow & B 
\end{eqnarray*}
By the fact that $\omega_!$ is left Quillen and $\odot$, $\otimes$ are left Quillen in each variable separately, it suffices to prove that $\theta_{\mathrm{Sc}(A),\mathrm{Sc}(B)}$ is a weak equivalence. Invoking the cube lemma again, we can now reduce to the case where $A$ and $B$ are both of dimensions $0$ or $1$. By Lemma \ref{lem:reducetoconnected}, we may reduce further to the case where both $A$ and $B$ are `connected' simplices, i.e. the case where the forests $\omega_!(A)$ and $\omega_!(B)$ each have at most one component. Then we either have $\omega_!(A) \simeq \varnothing$, $\omega_!(A) \simeq \eta$ or $\omega_!(A) \simeq C_k$ for some $k \geq 1$ and similarly for $B$. In these cases, $\theta_{A,B}$ is an isomorphism. $\Box$
\end{proof}

Let us prove a corollary that was already mentioned in Section \ref{sec:mainresults}: 

\begin{corollary}
\label{cor:omega*monoidal2}
For cofibrant objects $P, Q \in \mathbf{fSets}_o^+$ there is a natural weak equivalence
\begin{equation*}
\omega^*(P) \odot \omega^*(Q) \longrightarrow \omega^*(P \otimes Q). 
\end{equation*}
\end{corollary}
\begin{proof}
There are natural weak equivalences
\begin{equation*}
\omega_!(\omega^*(P) \odot \omega^*(Q)) \longrightarrow \omega_!\omega^*(P) \otimes \omega_!\omega^*(Q) \longrightarrow P \otimes Q,
\end{equation*}
the first one coming from Theorem \ref{thm:monoidal}, the second one from Proposition \ref{prop:counitequiv} and the fact that the functors $\omega^*$ and $\bar\omega^*$ are weakly equivalent. We will denote their composition, which is a weak equivalence, by $\psi$. Now consider the diagram
\[
\xymatrix{
\omega^*(P) \odot \omega^*(Q) \ar[dr]\ar[r] & \omega^*(P \otimes Q) \ar[d] \\
& \omega^*\bigl((P \otimes Q)_f\bigr)
}
\]
where the subscript $f$ denotes a fibrant replacement. The horizontal map is the adjoint of $\psi$. The skew map is a weak equivalence since the pair $(\omega_!,\omega^*)$ is a Quillen equivalence and $\psi$ is a weak equivalence. The vertical map is a weak equivalence because $\omega^*$ preserves weak equivalences between cofibrant objects: indeed, $\omega^*$ is weakly equivalent to $\bar\omega^*$, which is a left Quillen functor. The result now follows by two-out-of-three. $\Box$
\end{proof}

\subsection{Homotopical monoidal structures}
\label{subsec:unbiased}

The tensor products on the categories $\mathbf{dSets}$ and $\mathbf{fSets}$ are not associative. We already encountered this defect in Section \ref{subsec:weakenrichment} when discussing the weak enrichments of these categories over the category of simplicial sets. There we showed that the necessary associativity constraints, although not isomorphisms, were weak equivalences. More generally, the tensor products of dendroidal sets and forest sets, while not being associative up to isomorphism, can be made associative up to weak equivalence. In particular, the homotopy categories of $\mathbf{dSets}$ and $\mathbf{fSets}$ are symmetric monoidal categories. In this section we will formalize such `weakly associative' monoidal structures.

Recall that for a coloured operad $\mathbf{P}$ in $\mathbf{Sets}$ we write $\iota^*\mathbf{P}$ for its underlying category, consisting of the unary operations of $\mathbf{P}$. We call $\mathbf{P}$ \emph{corepresentable} if, for any tuple $(x_1, \ldots, x_n)$ of colours of $\mathbf{P}$, the functor
\begin{equation*}
\mathbf{P}(x_1, \ldots, x_n; -): \iota^*\mathbf{P} \longrightarrow \mathbf{Sets}
\end{equation*}
is corepresentable. 

\begin{definition}
For a category $\mathbf{C}$, a \emph{lax symmetric monoidal structure} on $\mathbf{C}$ is a corepresentable operad $\mathbf{P}$ with $\iota^*\mathbf{P} = \mathbf{C}$. Dually (and more relevant to our examples), a \emph{colax symmetric monoidal structure} on a category $\mathbf{C}$ is a lax symmetric monoidal structure on the opposite category $\mathbf{C}^{\mathrm{op}}$.
\end{definition}
 
Our first goal in this section is to exhibit colax symmetric monoidal structures on the model categories $\mathbf{dSets}$ and $\mathbf{fSets}$ and their marked variants. We will give a more elaborate reformulation of this definition below, but first we need to introduce some notation. 

We will need the category $\mathbf{\Omega}_{\mathrm{pl}}$ of \emph{planar} trees: the objects are trees, like the objects of $\mathbf{\Omega}$, but now also equipped with a planar structure. The maps in $\mathbf{\Omega}_{\mathrm{pl}}$ are as in $\mathbf{\Omega}$, but with the extra requirement that they preserve the planar structure. In particular, every object in $\mathbf{\Omega}_{\mathrm{pl}}$ has no other automorphisms than the identity. Also, we consider the category $\mathbf{\Omega}_{\mathrm{s}}$ having the same objects as $\mathbf{\Omega}_{\mathrm{pl}}$, but now with \emph{all} maps between trees, not necessarily preserving planar structures. There is an obvious embedding $\mathbf{\Omega}_{\mathrm{pl}} \rightarrow \mathbf{\Omega}_{\mathrm{s}}$ and a functor $\mathbf{\Omega}_{\mathrm{s}} \rightarrow \mathbf{\Omega}$ forgetting the planar structure on the objects. The latter functor is an equivalence of categories. Finally, we will consider the subcategories
\begin{equation*}
\mathbf{\Omega}_{\mathrm{pl}}^{\mathrm{in}} \subseteq \mathbf{\Omega}_{\mathrm{pl}}, \quad\quad\quad  \mathbf{\Omega}_{\mathrm{s}}^{\mathrm{in}} \subseteq  \mathbf{\Omega}_{\mathrm{s}},
\end{equation*}
which have the same objects, but with arrows generated by \emph{inner} face maps, degeneracies and isomorphisms only.

Let $\mathcal{E}$ be a category with a colax symmetric monoidal structure. This structure determines, by corepresentability, a sequence of functors
\begin{equation*}
\otimes_n: \mathcal{E}^n \longrightarrow \mathcal{E}, \quad\quad n \geq 0.
\end{equation*}
For $n = 0$ this gives an object of $\mathcal{E}$ called the \emph{unit} and denoted $I$. For $n=1$, the functor $\otimes_1: \mathcal{E} \rightarrow \mathcal{E}$ is the identity functor of $\mathcal{E}$. For general $n$ we just write $X_1 \otimes \cdots \otimes X_n$ for $\otimes_n(X_1, \ldots, X_n)$. By induction on trees and composition of functors, this gives for each planar tree $T \in \mathbf{\Omega}_{\mathrm{pl}}$ a functor 
\begin{equation*}
\otimes_T: \mathcal{E}^{l(T)} \longrightarrow \mathcal{E},
\end{equation*}
with $l(T)$ denoting the (ordered) set of leaves of $T$. Explicitly, if $T$ is a corolla $C_n$ then $\otimes_T = \otimes_n$, if $T = \eta$ then $\otimes_\eta = \mathrm{id}_{\mathcal{E}}$ and if $T = C_n \star (T_1, \ldots, T_n)$ is the tree obtained by grafting trees $T_1, \ldots, T_n$ onto the leaves of the corolla $C_n$, then $\otimes_T$ is the composition
\begin{equation*}
\otimes_T = \otimes_{n} \circ (\otimes_{T_1}, \ldots, \otimes_{T_n}).
\end{equation*}

The colax symmetric monoidal structure (specifically, composition of operations in the associated operad), determines an extension of the collection of these functors $\otimes_T$ to a contravariant functor on $\mathbf{\Omega}_{\mathrm{pl}}^{\mathrm{in}}$. More precisely, the $\otimes_T$ are equipped with the following structure. Each morphism $\alpha: S \rightarrow T$ in $\mathbf{\Omega}_{\mathrm{pl}}^{\mathrm{in}}$ induces an isomorphism $l(\alpha): l(S) \rightarrow l(T)$. For each such $\alpha$ and $X \in \mathcal{E}^{l(T)}$ there is a natural map 
\begin{equation*}
\alpha^*: \otimes_T(X) \longrightarrow \otimes_S(X \circ l(\alpha)).
\end{equation*} 
These maps are functorial in $\alpha$. (One can of course encode the collection of such $\alpha$'s into a single functor from the category $\int_{\mathbf{\Omega}_{\mathrm{pl}}^{\mathrm{in}}}\mathcal{E}^{l(-)}$ back to $\mathcal{E}$, where the first denotes the fibered category associated to the functor $(\mathbf{\Omega}_{\mathrm{pl}}^{\mathrm{in}})^{\mathrm{op}} \rightarrow \mathbf{Cat}$ sending $T$ to $\mathcal{E}^{l(T)}$.)

\begin{example}
If $\alpha$ is the following map $\partial_e$
\[
\begin{tikzpicture} 
[level distance=5mm, 
every node/.style={fill, circle, minimum size=.1cm, inner sep=0pt}, 
level 1/.style={sibling distance=5mm}, 
level 2/.style={sibling distance=5mm}, 
level 3/.style={sibling distance=5mm},
level 4/.style={sibling distance=4mm}]

%shuffle 1
\node (shuffle1)[style={color=white}] {} [grow'=up] 
child {node {} 
	child
	child
	child
};

%shuffle2
\node (shuffle2)[style={color=white}, right = 3cm of shuffle1] {} [grow'=up] 
child {node {} 
	child
	child{ node {}
		child
		child
	}
};

\tikzstyle{every node}=[]

%arrow
\draw[->] ($(shuffle1) + (1cm, .7cm)$) -- ($(shuffle1) + (2cm, .7cm)$);

%labels
\node at ($(shuffle1) + (1.5cm, .9cm)$) {$\partial_e$};
\node at ($(shuffle2) + (.25cm, .7cm)$) {$e$};

\end{tikzpicture} 
\]
then
\begin{equation*}
\partial_e^*: X_1 \otimes (X_2 \otimes X_3) \longrightarrow X_1 \otimes X_2 \otimes X_3,
\end{equation*}
while for
\[
\begin{tikzpicture} 
[level distance=5mm, 
every node/.style={fill, circle, minimum size=.1cm, inner sep=0pt}, 
level 1/.style={sibling distance=5mm}, 
level 2/.style={sibling distance=5mm}, 
level 3/.style={sibling distance=5mm},
level 4/.style={sibling distance=4mm}]

%shuffle 1
\node (shuffle1)[style={color=white}] {} [grow'=up] 
child {node {} 
	child
};

%shuffle2
\node (shuffle2)[style={color=white}, right = 3cm of shuffle1] {} [grow'=up] 
child {node {} 
	child
	child{ node {}
	}
};

\tikzstyle{every node}=[]

%arrow
\draw[->] ($(shuffle1) + (1cm, .7cm)$) -- ($(shuffle1) + (2cm, .7cm)$);

%labels
\node at ($(shuffle1) + (1.5cm, .9cm)$) {$\partial_a$};
\node at ($(shuffle2) + (.25cm, .7cm)$) {$a$};

\end{tikzpicture} 
\]
we obtain a map
\begin{equation*}
\partial_a^*: X \otimes I \longrightarrow X.
\end{equation*}
\end{example}

In the examples relevant to us, all maps coming from contracting an edge below a nullary vertex are in fact isomorphisms. If this is the case, we say that the unit is \emph{strong}. The symmetry of the colax monoidal structure provides a further extension of the above to a functor on $\mathbf{\Omega}_{\mathrm{s}}^{\mathrm{in}}$. This symmetry gives for each (non-planar) isomorphism $\sigma: T \rightarrow T'$ of trees, with induced isomorphism $l(\sigma)$ on leaves, an isomorphism
\begin{equation*}
\sigma^*: \otimes_T X \longrightarrow \otimes_{T'} (X \circ l(\sigma)), \quad\quad X \in \mathcal{E}^{l(T)}.
\end{equation*}
These natural transformations $\sigma^*$ are completely determined by the maps associated to isomorphisms between corollas. The reader should observe that the colax symmetric monoidal structure on $\mathcal{E}$ is completely determined by the data of the tensor products $\otimes_n$ and the maps $\alpha^*, \sigma^*$ together with their functoriality described above.

\begin{definition}
A colax symmetric monoidal structure on a model category $\mathcal{E}$ is \emph{homotopical} if for each morphism $\alpha: S \rightarrow T$ in $\mathbf{\Omega}_{\mathrm{s}}^{\mathrm{in}}$ and each $l(T)$-indexed sequence $X \in \mathcal{E}^{l(T)}$ consisting of cofibrant objects, the map $\alpha^*$ is a weak equivalence.
\end{definition}

The reader should observe that the homotopy category of a colax symmetric monoidal model category can naturally be made a symmetric monoidal category.

\begin{proposition}
\label{prop:unbiasedtensor}
The (binary) tensor product on $\mathbf{dSets}$ can be extended to a colax symmetric monoidal structure. When restricted to the subcategory $\mathbf{dSets}_o$ of open dendroidal sets, this monoidal structure is homotopical. The analogous statements hold true for the categories of forest sets, marked dendroidal sets and marked forest sets (and their open variants).
\end{proposition}

Let us first construct the required tensor products and associativity maps $\alpha^*$. We work in the category $\mathbf{dSets}$, the other cases being analogous. The tensor products $\otimes_T$ are to preserve colimits in each variable separately. Thus, it suffices to construct the tensor product
\begin{equation*}
X_1 \otimes \cdots \otimes X_n
\end{equation*}
of a sequence of \emph{representable} dendroidal sets in functorial fashion. To this end, set
\begin{equation*}
X_1 \otimes \cdots \otimes X_n = N_d(\tau_d(X_1) \otimes \cdots \otimes \tau_d(X_n)),
\end{equation*}
where the tensor product on the right is the $n$-fold Boardman-Vogt tensor product of operads and $(\tau_d, N_d)$ is the usual adjunction relating dendroidal sets and operads in sets. Note that the functor $\tau_d$ distributes over tensor products in $\mathbf{dSets}$: this follows from the fact that it preserves colimits and that $\tau_dN_d = \mathrm{id}$, which establishes distributivity on representables.

Next, we wish to construct the relevant associativity maps. If $\alpha$ is a degeneracy we can take $\alpha^*$ to be the identity. Now, by a straightforward induction on trees, it suffices to define $\alpha^*$ in the case where $\alpha$ is the inner face map
\begin{equation*}
\partial_i: C_{n+m-1} \longrightarrow C_n \circ_i C_m,
\end{equation*}
where the right-hand side denotes the tree obtained by grafting $C_m$ onto the $i$'th leaf of $C_n$. So, we are looking for a natural map
\[
\xymatrix{
X_1 \otimes \cdots \otimes X_{i-1} \otimes (Y_1 \otimes \cdots \otimes Y_m) \otimes X_{i+1} \otimes \cdots \otimes X_n \ar[d] \\
X_1 \otimes \cdots \otimes X_{i-1} \otimes Y_1 \otimes \cdots \otimes Y_m \otimes X_{i+1} \otimes \cdots \otimes X_n.
}
\]
The codomain is the nerve of an operad in sets, namely $\tau_d(X_1) \otimes \cdots \otimes \tau_d(X_n)$. Using adjunction and the distributivity of $\tau_d$ over tensor products, it suffices to specify a map
\[
\xymatrix{
\tau_d(X_1 \otimes \cdots \otimes X_{i-1} \otimes (Y_1 \otimes \cdots \otimes Y_m) \otimes X_{i+1} \otimes \cdots \otimes X_n) \ar[d] \\
\tau_d(X_1) \otimes \cdots \otimes \tau_d(Y_1) \otimes \cdots \otimes \tau_d(Y_m) \otimes \cdots \otimes \tau_d(X_n).
}
\]
Using the distributivity of $\tau_d$ over tensor products in the domain, such a map is given by the associativity isomorphisms of the Boardman-Vogt tensor product of operads. It is routine to verify that the associativity maps thus defined for the tensor products on $\mathbf{dSets}$ have the required naturality and functoriality properties. Moreover, they can be made symmetric using the symmetry of the tensor product of operads. Also, the unit $\eta$ is strong.

We now wish to show that the colax symmetric monoidal structure on $\mathbf{dSets}_o$ is homotopical. To be able to use skeletal induction, we need the following:

\begin{lemma}
\label{lem:tensorprodhomotopical}
Let $X_1, \ldots, X_{i-1}, X_{i+1}, \ldots, X_n$ be normal dendroidal sets, which are moreover open. Then the functor
\begin{equation*}
\mathbf{dSets}_o \rightarrow \mathbf{dSets}_o: X_i \longmapsto X_1 \otimes \cdots \otimes X_{i-1} \otimes X_i \otimes X_{i+1} \otimes \cdots \otimes X_n
\end{equation*}
preserves cofibrations and trivial cofibrations.
\end{lemma}

The case $n=2$ follows from Propositions \ref{prop:newnormalmonopoprod}, \ref{prop:poprodinneranodyne} and \ref{prop:trivcofoperadic} (and Lemma \ref{lem:poprodmarkedanodyne} in the marked case). Proving the analogous statement for higher $n$ is done in completely analogous fashion; we omit the details. To prove that the colax symmetric monoidal structure is homotopical, we may now (using the previous lemma) apply the usual skeletal induction to reduce to the case where all the $X_i$ are representable dendroidal sets $T_i$. In fact, using the Segal core inclusions $\mathrm{Sc}(T_i) \rightarrow T_i$, which are trivial cofibrations, and applying the lemma again, we may reduce to the case where all the $X_i$ are simply corollas. Proposition \ref{prop:unbiasedtensor} is then a consequence of the following:

\begin{proposition}
\label{lem:alphaunbiased}
% maybe just do C_{k_1} \otimes (C_{k_2} \otimes (C_{k_3} \otimes ( \cdots (C_{k_{n-1}}\otimes C_{k_n}) \cdots ) \rightarrow C_{k_1} \otimes \cdots \otimes C_{k_n}
% for the sake of definiteness
% even better: do C_{k_1} \otimes (C_{k_2} \otimes \cdots \otimes C_{k_n})
For a tree $T \in \mathbf{\Omega}_{\mathrm{s}}^\mathrm{in}$ with $n$ leaves, let $\alpha: C_n \rightarrow T$ denote the map contracting all the inner edges of $T$. For a collection of corollas $X_1, \ldots, X_n$, the map
\begin{equation*}
\alpha^*: \otimes_T(X_1, \ldots, X_n) \longrightarrow X_1 \otimes \cdots \otimes X_n
\end{equation*}
is a trivial cofibration of dendroidal sets.
\end{proposition}
\begin{proof}
Since the unit of the tensor product is strong, we can without loss of generality assume that $T$ has no nullary vertices. Then the cases $n=1, 2$ of the statement of the proposition are trivial. For $n \geq 3$ and any collection of corollas $X_1, \ldots, X_n$, we will prove that the map
\begin{equation*}
\iota: X_1 \otimes (X_2 \otimes \cdots \otimes X_n) \longrightarrow X_1 \otimes X_2 \otimes \cdots \otimes X_n
\end{equation*}
is inner anodyne. Afterwards, we will show how the statement of the proposition can be deduced from this. Note that $\iota$ is a special case of a map of the form $\alpha^*$, for $T$ a tree with a binary vertex at the root and a vertex of valence $n-1$ attached to it.

The $n$-fold tensor product $X_1 \otimes \cdots \otimes X_n$ is a union of the shuffles of the corollas $X_1, \ldots, X_n$ as in the case of binary tensor products. The tensor product $X_1 \otimes(X_2 \otimes \cdots \otimes X_n)$ is the union of a subset of these shuffles (see the proof of Proposition \ref{prop:forestsetsH2} for a typical example). Consider the set of shuffles $\Sigma$ of $X_1 \otimes \cdots \otimes X_n$; it can be partially ordered by declaring $S_1 < S_2$ whenever $S_2$ is obtained from $S_1$ by shuffling vertices corresponding to the vertex of $X_1$ downwards (i.e. towards the root). Extend this partial order to a linear order in an arbitrary fashion. Now filter the map $\iota$ by adjoining the shuffles in $\Sigma$ one by one according to the chosen order to obtain a sequence of maps
\begin{equation*}
X_1 \otimes (X_2 \otimes \cdots \otimes X_n) =: A_0 \subseteq A_1 \subseteq \cdots \subseteq A_N = X_1 \otimes \cdots \otimes X_n.
\end{equation*}
Consider an inclusion $A_i \subseteq A_{i+1}$ given by adjoining some shuffle $S$. The tree $S$ has a set $V$ of distinguished vertices corresponding to the vertex of the corolla $X_1$. Define a \emph{$V$-pruning} to be a subtree of $S$ that contains all the vertices of $V$ and is obtained from $S$ by iteratively chopping off leaf vertices and root vertices. These $V$-prunings form a partially ordered set $\mathcal{P}$ by declaring $P_1 < P_2$ if $P_1$ is a subtree of $P_2$. This poset has a minimal element given by removing from $S$ all vertices above $V$ and all unary vertices at the root (if any). It also has a maximal element given by $S$ itself. Note that the minimal $V$-pruning is already contained in $A_0$. Extend the partial order on $\mathcal{P}$ to a linear order arbitrarily and adjoin all $V$-prunings one by one in this order to obtain a further filtration
\begin{equation*}
A_i =: A_{i}^0 \subseteq A_i^1 \subseteq \cdots \subseteq \bigcup_j A_i^j = A_{i+1}.
\end{equation*}
Consider an inclusion $A_i^j \subseteq A_i^{j+1}$ given by adjoining a $V$-pruning $P$. Denote by $\mathcal{E}(P)$ the collection of \emph{special} edges of $P$: an edge is special if it is an input edge of a vertex in $V$ and also an inner edge of $P$. Without loss of generality we may assume $\mathcal{E}(P)$ to be non-empty: if not, the pruning $P$ has no vertices above the vertices of $V$ and is therefore already contained in $A_0$. Let $I(P)$ denote the set of inner edges of $P$ and write
\begin{equation*}
\mathcal{H}(P) = I(P) - \mathcal{E}(P).
\end{equation*}
Now, for any subset $H \subset \mathcal{H}(P)$, denote by $P^{[H]}$ the tree obtained from $P$ by contracting all the edges in $\mathcal{H}(P) - H$. Pick a linear order on the subsets of $\mathcal{H}(P)$ that extends the partial order of inclusion and adjoin the trees $P^{[H]}$ to $A_i^j$ in this order to obtain a filtration
\begin{equation*}
A_i^j=: A_{i}^{j,0} \subseteq A_i^{j,1} \subseteq \cdots \subseteq \bigcup_k A_i^{j,k} = A_i^{j+1}.
\end{equation*}
Finally, consider an inclusion $A_i^{j,k} \subseteq A_i^{j,k+1}$ given by adjoining a tree $P^{[H]}$. If $P^{[H]}$ is already contained in $A_i^{j,k}$ there is nothing to prove. If not, we can say the following:
\begin{itemize}
\item[-] Any inner face of $P^{[H]}$ contracting a special edge, or a composition of inner faces contracting several special edges, is not contained in $A_i^{j,k}$. Indeed, $P^{[H]}$ is not contained in $A_i^{j,k}$ because the trees on top of the various vertices of $V$ do not all stem from the \emph{same} shuffle of $C_{k_2} \otimes \cdots \otimes C_{k_n}$. For the same reason, a tree obtained by contracting any number of special edges is also not contained in $A_i^{j,k}$.
\item[-] Any inner face of $P^{[H]}$ contracting an edge that is not special is contained in $A_i^{j,k}$ by our induction on the size of $H$.
\item[-] Any outer face of $P^{[H]}$ is contained in $A_i^{j,k}$ is contained in $A_i^j$ by our induction on the size of prunings.
\end{itemize}
We conclude that the map $A_i^{j,k} \subseteq A_i^{j,k+1}$ is a pushout of the inclusion
\begin{equation*}
\Lambda^{\mathcal{E}(P)}[P^{[H]}] \longrightarrow P^{[H]}
\end{equation*}
and hence inner anodyne.

We have proved $\iota$ is a trivial cofibration (for all $n \geq 3$). It remains to treat the general case of a map
\begin{equation*}
\alpha^*: \otimes_T(X_1, \ldots, X_n) \longrightarrow X_1 \otimes \cdots \otimes X_n.
\end{equation*}
The left-hand side is the union of a subset of the shuffles that make up the right-hand side, so that $\alpha^*$ is a normal monomorphism. Consider a \emph{maximal binary expansion} $\beta: T \rightarrow \widehat{T}$ of the tree $T$: that is, a composition of inner face maps such that the tree $\widehat{T}$ contains only binary vertices. It is clear that such an expansion always exists. The composition $\beta \circ \alpha: C_n \rightarrow \widehat{T}$ (which is itself a maximal binary expansion of $C_n$) may be factored into \emph{elementary expansions}, i.e. maps of the following form:
\[
\begin{tikzpicture} 
[level distance=5mm, 
every node/.style={fill, circle, minimum size=.1cm, inner sep=0pt}, 
level 1/.style={sibling distance=7mm}, 
level 2/.style={sibling distance=5mm}, 
level 3/.style={sibling distance=3mm}]

%vertex of T
\node(anchorT)[style={color=white}] {} [grow'=up] 
child {node(vertexT) {} 
	child
	child
	child
	child
};

%comment

\node (anchorT')[style={color=white}, right = 3cm of anchorT] {} [grow'=up] 
child {node(vertexT') {} 
	child{
	}
	child{node {}
		child
		child
		child
	}
};

\tikzstyle{every node}=[]

%lines
\draw[->] ($(shuffle1) + (1cm, .5cm)$) -- ($(shuffle1) + (2cm, .5cm)$);

%labels
%T
\node at ($(vertexT) + (-0.2cm,0cm)$) {$v$};
\node at ($(vertexT) + (-.5cm, .7cm)$) {$\cdots$};
\node at ($(vertexT) + (.5cm, .7cm)$) {$\cdots$};
\node at ($(vertexT) + (0, -.7cm)$) {$\cdots$};

%T'
\node at ($(vertexT') + (-0.2cm,0cm)$) {$w$};
\node at ($(vertexT') + (0.55cm,.5cm)$) {$u$};
\node at ($(vertexT') + (.27cm, 1.1cm)$) {$\cdots$};
\node at ($(vertexT') + (-.3cm, .7cm)$) {$\cdots$};
\node at ($(vertexT') + (0, -.7cm)$) {$\cdots$};

\end{tikzpicture} 
\]
The dots indicate that we are only picturing the relevant part of the trees in question: more may be attached to the roots and leaves of the corollas drawn. In words, an elementary expansion is a map decomposing a vertex $v$ into a composition $w \circ_2 u$, where $w$ is binary and $u$ has valence one less than $v$. Let $\gamma$ denote such an elementary expansion. By what we proved above, $\gamma^*$ is a trivial cofibration. Indeed, it is a tensor product of a map of the form $\iota$ with a sequence of normal and open dendroidal sets and hence itself a trivial cofibration by Lemma \ref{lem:tensorprodhomotopical}. We conclude that $(\beta \circ \alpha)^*$ is a trivial cofibration. Similarly, the map $\beta: T \rightarrow \widehat{T}$ may be factored into elementary expansions, so that $\beta^*$ is a trivial cofibration as well. By two-out-of-three, we conclude that $\alpha^*$ must be a trivial cofibration. $\Box$ 
\end{proof}

Since the category $\mathbf{POp}_o$ is monoidal, its binary tensor product $\odot$ can be used to construct a colax (non-symmetric) monoidal category for which all the associativity maps $\alpha^*$ are isomorphisms. A straightforward elaboration of the proof of Theorem \ref{thm:monoidal} then gives the following:

\begin{lemma}
\label{lem:unbiasedequiv}
For any objects $X_1, \ldots, X_n \in \mathbf{POp}_o$ and $T$ a tree with $n$ leaves, there exists a weak equivalence
\begin{equation*}
\theta^T_{X_1, \ldots, X_n}: \omega_!\bigl(\odot_T(X_1, \ldots, X_n)\bigr) \longrightarrow \otimes_T(\omega_! X_1, \ldots, \omega_! X_n).
\end{equation*}
Furthermore, these equivalences are natural in the $X_i$ and are compatible with the associativity maps $\alpha^*$ for the tensor products $\odot$ and $\otimes$.
\end{lemma}

To be more specific, the map $\theta^T_{X_1, \ldots, X_n}$ is constructed as in the proof of Theorem $\ref{thm:monoidal}$. The fact that it is a weak equivalence is a consequence of Theorem $\ref{thm:monoidal}$, using a maximal binary expansion of $T$ on both sides combined with the fact that the colax monoidal structures on $\mathbf{POp}_o$ and $\mathbf{fSets}^+_o$ are homotopical. Note that in this last step we are also using that all the objects involved are cofibrant (i.e. the $X_i$, $\omega_!X_i$ and tensor products of such). 

We wish to conclude that $\omega_!$ induces an equivalence of symmetric monoidal categories between $\mathrm{Ho}(\mathbf{POp}_o)$ and  $\mathrm{Ho}(\mathbf{fSets}^+_o)$. However, the tensor product on $\mathbf{POp}_o$ is not symmetric. Still, it is symmetric `up to weak equivalence' and can be used to give $\mathrm{Ho}(\mathbf{POp}_o)$ a symmetric monoidal structure. We briefly recall how this is done and then prove Proposition \ref{prop:equivsymmon} below.

\begin{definition}
Let $F: \mathbf{F}_o^{\times n} \longrightarrow \mathbf{F}_o$ be a functor. We will say $F$ is an \emph{$n$-fold smash product} if:
\begin{itemize}
\item[-] $F(\langle 1 \rangle, \ldots, \langle 1 \rangle) = \langle 1 \rangle$;
\item[-] $F$ preserves coproducts in each variable separately.
\end{itemize}
\end{definition}  

The functor $\wedge$ we used to define $\odot$ is a two-fold smash product. For every $n$, the collection of $n$-fold smash products and natural isomorphisms between them form a groupoid which is denoted $\mathfrak{S}(n)$. Since there is a unique natural isomorphism between any two $n$-fold smash products, this groupoid is contractible. In fact, by composing smash products, these groupoids fit together into a (strict) operad in groupoids, which we will denote by $\mathfrak{S}$. \par 
It is important to observe that for the construction of the natural transformation $\theta$ of Theorem \ref{thm:monoidal}, the choice of (two-fold) smash product used to construct the tensor product $\odot$ is completely irrelevant. To be more precise, the observations we made before will also prove the following: 

\begin{lemma}
\label{lem:othersmashes}
Let $P_1, \ldots, P_n$ be objects of $\mathbf{POp}_o$, let $\sigma$ be a $k$-simplex of the nerve of $\mathfrak{S}(n)$ and define $\bigodot_{\sigma}\{P_i\}_{1 \leq i \leq n}$ to be the composition
\[
\xymatrix{
(\Delta^k)^\sharp \times \prod_{i=1}^n P_i \ar[r] & (\Delta^k)^\sharp \times (N\mathbf{F}_o^\natural)^{\times n} \ar[r]^-{\bar\sigma} & N\mathbf{F}_o^\natural
}
\]
where the map $\bar\sigma$ corresponds to the simplex $\sigma$. Then there are natural weak equivalences
\begin{equation*}
\omega_!\bigl(\bigodot_{\sigma}\{P_i\}_{1 \leq i \leq n}\bigr) \longrightarrow (\Delta^k)^\sharp \otimes \omega_!\bigl(\bigodot_{i=1}^n P_i\bigr) \longrightarrow (\Delta^k)^\sharp \otimes \bigotimes_{i=1}^n \omega_!(P_i),
\end{equation*}
the second map coming from Lemma \ref{lem:unbiasedequiv}.
\end{lemma}

Define a (symmetric) simplicial coloured operad $\bigl(\mathbf{POp}_o\bigr)_{\mathfrak{S}}^\otimes$ as follows:
\begin{itemize}
\item[-] Let the colours of $\bigl(\mathbf{POp}_o\bigr)_{\mathfrak{S}}^\otimes$ be the fibrant (and automatically cofibrant) objects of $\mathbf{POp}_o$.
\item[-] For fibrant objects $X_1, \ldots, X_n$ and $Y$, let the $k$-simplices of the simplicial set $\bigl(\mathbf{POp}_o\bigr)_{\mathfrak{S}}^\otimes(X_1, \ldots, X_n ; Y)$ be commutative diagrams of the form
\[
\xymatrix{
(\Delta^k)^\sharp \times \prod_{i=1}^n X_i \ar[r]\ar[d] & Y \ar[d] \\
(\Delta^k)^\sharp \times (N\mathbf{F}_o^\natural)^{\times n} \ar[r] & N\mathbf{F}_o^\natural
}
\]
where the bottom horizontal arrow corresponds to a $k$-simplex of the nerve of $\mathfrak{S}(n)$. 
\item[-] Define composition using the operad structure on $\mathfrak{S}(n)$.
\item[-] The symmetric group $\Sigma_n$ acts by permuting the $X_i$ and through its evident action on $\mathfrak{S}(n)$.
\end{itemize}

Denote by $\mathrm{Ho}\bigl(\bigl(\mathbf{POp}_o\bigr)_{\mathfrak{S}}^\otimes\bigr)$ the operad in sets obtained by taking connected components of the simplicial sets defining the operations in $\bigl(\mathbf{POp}_o\bigr)_{\mathfrak{S}}^\otimes$. This operad has underlying category $\mathrm{Ho}(\mathbf{POp}_o)$ and defines a symmetric monoidal structure on this category; this is immediate from the observation that the functor $\mathrm{Ho}\bigl(\bigl(\mathbf{POp}_o\bigr)_{\mathfrak{S}}^\otimes\bigr)(X_1, \ldots, X_n: -)$ is corepresented by $X_1 \odot \cdots \odot X_n$ and the fact that this operad is symmetric.

Now for fibrant-cofibrant objects $X_1, \ldots, X_n, Y \in \mathbf{fSets}^+_o$, define a simplicial set as follows:
\begin{equation*}
\bigl(\mathbf{fSets}_o^+\bigr)_{\mathfrak{S}}^\otimes(X_1, \ldots, X_n; Y) \, := \, \mathrm{Map}^\sharp\bigl(\bigotimes_{i=1}^n X_i, Y\bigr) \times \mathfrak{S}(n).
\end{equation*}
These simplicial sets do \emph{not} naturally form a coloured operad. Indeed, the associativity maps $\alpha^*$ go the wrong way if one were to try to define composition  (the analogous construction for the opposite category of $\mathbf{fSets}^+$ would give a simplicial operad). However, taking the connected components of these simplicial sets does give an operad $\mathrm{Ho}\bigl(\bigl(\mathbf{fSets}_o^+\bigr)_{\mathfrak{S}}^\otimes\bigr)$ in sets; indeed, the associators $\alpha^*$ have inverses in the homotopy category $\mathrm{Ho}(\mathbf{fSets}^+_o)$ which can be used to define composition. Again, this operad encodes the symmetric monoidal structure of $\mathrm{Ho}(\mathbf{fSets}^+_o)$.

Proposition \ref{prop:omegasymmonoidal} now follows from the next result, together with the fact that $\mathbf{fSets}_o^+$ and $\mathbf{dSets}_o$ are linked by a chain of symmetric monoidal Quillen equivalences:

\begin{proposition}
\label{prop:equivsymmon}
The functor $\omega^*$ induces an equivalence of operads $\mathrm{Ho}\bigl(\bigl(\mathbf{fSets}_o^+\bigr)_{\mathfrak{S}}^\otimes\bigr) \rightarrow \mathrm{Ho}\bigl(\bigl(\mathbf{POp}_o\bigr)_{\mathfrak{S}}^\otimes\bigr)$.
\end{proposition} 
\begin{proof}
We will use $\omega^*$ to define maps of simplicial sets
\begin{equation*}
\bigl(\mathbf{fSets}_o^+\bigr)_{\mathfrak{S}}^\otimes(X_1, \ldots, X_n; Y) \longrightarrow \bigl(\mathbf{POp}_o\bigr)_{\mathfrak{S}}^\otimes(\omega^*X_1, \ldots, \omega^*X_n ; \omega^*Y)
\end{equation*}
which induce an equivalence of operads after passing to connected components. A $k$-simplex of the left-hand side is a map
\begin{equation*}
f: (\Delta^k)^\sharp \otimes \bigotimes_{i=1}^n X_i \longrightarrow Y
\end{equation*}
and a $k$-simplex $\sigma$ of $\mathfrak{S}(n)$. We have to define a diagram
\[
\xymatrix{
(\Delta^k)^\sharp \times \prod_{i=1}^n \omega^*(X_i) \ar[d]\ar[r] & \omega^*(Y) \ar[d] \\ 
(\Delta^k)^\sharp \times (N\mathbf{F}_o^\natural)^{\times n} \ar[r] & N\mathbf{F}_o^\natural
}
\]
By adjunction, this is equivalent to defining a map
\begin{equation*}
\omega_!\bigl(\bigodot_{\sigma}\{\omega^*X_i\}_{1 \leq i \leq n}\bigr) \longrightarrow Y.
\end{equation*}
Such a map is given by the composition
\begin{equation*}
\omega_!\bigl(\bigodot_{\sigma}\{\omega^*X_i\}_{1 \leq i \leq n}\bigr) \longrightarrow (\Delta^k)^\sharp \otimes \bigotimes_{i=1}^n \omega_!\omega^*(X_i) \longrightarrow (\Delta^k)^\sharp \otimes \bigotimes_{i=1}^n X_i \longrightarrow Y,
\end{equation*}
where the first arrow is the map $\theta$ provided by Lemma \ref{lem:othersmashes}, the second one is induced by the counit of the adjunction and the last one is the map $f$. It is straightforward to verify that the definition of this map is natural and yields a map of operads after passing to connected components. The fact that $\omega^*$ induces an equivalence of homotopy categories $\mathrm{Ho}(\mathbf{fSets}^+_o) \rightarrow \mathrm{Ho}(\mathbf{POp}_o)$ combined with Lemma \ref{lem:unbiasedequiv} proves that this map of operads is in fact an equivalence. $\Box$
\end{proof}